%% file: main.tex
\documentclass[11pt,a4paper,reqno]{amsart}

\usepackage{amsmath,amssymb,amsthm}
\usepackage[utf8]{inputenc}
\usepackage[T1]{fontenc}
\usepackage{geometry}
\usepackage{hyperref}
\usepackage{xurl}
\hypersetup{
    colorlinks=true,
    linkcolor=blue,
    citecolor=red,
    urlcolor=blue,
}
\usepackage{cleveref}

\theoremstyle{plain}
\newtheorem{theorem}{Theorem}[section]
\newtheorem{lemma}[theorem]{Lemma}
\newtheorem{proposition}[theorem]{Proposition}
\newtheorem{corollary}[theorem]{Corollary}

\newtheorem{mainthm}{Main Theorem}

\theoremstyle{definition}
\newtheorem{definition}[theorem]{Definition}
\newtheorem{remark}[theorem]{Remark}

\numberwithin{equation}{section}

\input{macros}

\begin{document}

\title[Scattering theory for the impedance half-space, I]%
{Scattering theory for bulk and surface waves in an impedance half-space, I:
 radiation conditions and a Rellich-type theorem}

\author{Michiyuki Watanabe}
\address{Department of Applied Mathematics,
Okayama University of Science,
1--1 Ridaicho, Kita--ku, Okayama 700--0005, Japan}
\email{mic-watanabe@ous.ac.jp}
\thanks{ORCID: 0000-0003-0612-8411}

\begin{abstract}
We study the negative Laplacian on the half-space $\mathbb{R}^n_+$, $n \ge 3$,
subject to a constant impedance (Robin) boundary condition with $\alpha > 0$.
The boundary then carries a bound transverse mode, and at each energy $E > 0$
the system propagates two channels with the distinct wavenumbers
$\sqrt{E}$ and $\sqrt{E+\alpha^2}$: bulk waves decaying like $|x|^{-(n-1)/2}$
and surface waves decaying like $|x'|^{-(n-2)/2}$. The two rates are not
comparable pointwise. Each is critical for the Agmon--H\"ormander space
$\mathcal{B}^*$, however. This Part~I treats the flat boundary, and its purpose
is to develop, at that endpoint, the scattering theory of the two channels in
full. We prove a limiting
absorption principle $\mathcal{B} \to \mathcal{B}^*$, construct the generalized
Fourier transform, derive an expansion of the resolvent at infinity in which the
two channels appear in a single relation, characterize the solutions of
$\mathcal{B}^*$ growth, and compute the scattering matrix. The radiation
condition is imposed on the solution itself, with no decomposition into
channels, and is shown to be equivalent to the vanishing of the incoming data.
The main result is a Rellich-type theorem whose hypotheses are imposed only near
infinity and which forces both channels to vanish simultaneously. The surface
channel turns out to be rigid: its block of the scattering matrix is the
identity and carries no information about $\alpha$, so that on the band
$-\alpha^2 < E < 0$ the scattering matrix is trivial. The bulk block, by
contrast, determines $\alpha$.
\end{abstract}

\maketitle

\tableofcontents

\section{Introduction and Physical Motivation}
\input{Part1_01_introduction}

\section{Scattering Theory in the Unperturbed Half-Space}
\label{sec:unperturbed}

\input{Part1_02_opening}

\subsection{Self-Adjointness and Spectral Structure}
\input{Part1_021_selfadjoint}

\subsection{The Limiting Absorption Principle in \texorpdfstring{$\Bsp$--$\Bst$}{B--B*}}
\input{Part1_022_lap}

\subsection{Generalized Eigenfunctions and the Generalized Fourier Transform}
\input{Part1_023_gft}

\subsection{Asymptotic Expansion of the Resolvent}
\input{Part1_024_asymptotics}

\subsection{Characterization of Scattering Solutions}
\input{Part1_025_characterization}

\subsection{Radiation Conditions, Rellich Uniqueness, and the Outgoing Solution}
\input{Part1_026_radiation}

\section{Inverse Scattering for the Flat Half-Space}
\label{sec:inverse}

\input{Part1_027_inverse_uniqueness}

\appendix

\section{Proof of Theorem~\ref{thm:sa-spec}}
\label{app:sa-spec}
\input{Part1_a1_app_selfadjoint}

\section{The Mourre Estimate for \texorpdfstring{$\Halpha$}{H alpha}}
\label{app:mourre}
\input{Part1_a2_app_mourre}

\section*{Statements and Declarations}

\noindent\textbf{Funding.}
The author did not receive support from any organization for the submitted work.

\smallskip
\noindent\textbf{Competing interests.}
The author has no competing interests to declare that are relevant to the
content of this article.

\smallskip
\noindent\textbf{Data availability.}
Data sharing is not applicable to this article, as no datasets were generated
or analysed during the current study.

\smallskip
\noindent\textbf{Use of AI tools.}
The author used an AI assistant (Claude, Anthropic) for language editing, for
drafting parts of the LaTeX source under the author's direction, and for
checking cross-references and citations. All mathematical content, proofs and
references were verified by the author, who takes full responsibility for the
manuscript.
\bibliographystyle{amsplain}
\bibliography{references}

\end{document}

%% file: macros.tex
\ifdefined\RobinMacrosLoaded \fi
\let\RobinMacrosLoaded\relax

\newcommand{\Rn}{\mathbb{R}^n}
\newcommand{\Rnm}{\mathbb{R}^{n-1}}          %
\newcommand{\Rnp}{\mathbb{R}^n_+}            %
\newcommand{\xpar}{x'}                       %
\newcommand{\xipar}{\xi'}                    %
\newcommand{\Rhalf}{\mathbb{R}_+}            %
\newcommand{\Sm}{S^{n-1}_+}                  %
\newcommand{\Ssurf}{S^{n-2}}                 %
\newcommand{\Ltwo}[1]{L^2(#1)}
\newcommand{\norm}[1]{\left\lVert #1 \right\rVert}

\newcommand{\Bsp}{\mathcal{B}}
\newcommand{\Bst}{\mathcal{B}^*}
\newcommand{\Bstz}{\mathcal{B}^*_0}
\newcommand{\Bnorm}[1]{\left\lVert #1 \right\rVert_{\Bsp}}
\newcommand{\Bstnorm}[1]{\left\lVert #1 \right\rVert_{\Bst}}
\newcommand{\Bspd}[1]{\mathcal{B}(#1)}       %
\newcommand{\Bstd}[1]{\mathcal{B}^*(#1)}     %
\newcommand{\shell}[1]{\Omega_{#1}}          %
\newcommand{\shellp}[1]{\Omega_{#1}^+}       %
\newcommand{\tshell}[1]{\omega_{#1}}         %
\newcommand{\slab}[1]{S_{#1}}                %

\newcommand{\Halpha}{H_\alpha}               %
\newcommand{\Ralpha}[1]{R_\alpha(#1)}        %
\newcommand{\RD}[1]{R_{\mathrm{D}}(#1)}      %
\newcommand{\Rfree}[2]{R_0^{(#1)}(#2)}       %
\newcommand{\Krem}[1]{K(#1)}                 %
\newcommand{\Psurf}{P_{\mathrm{surf}}}       %
\newcommand{\DN}{\Lambda}                    %
\newcommand{\Dom}{\mathcal{D}}               %
\newcommand{\hform}{\mathfrak{h}_\alpha}     %
\newcommand{\hfib}{h}                        %
\newcommand{\Lop}[2]{\mathcal{L}(#1,#2)}     %
\DeclareMathOperator{\Ran}{Ran}

\newcommand{\sac}{\sigma_{\mathrm{ac}}}
\newcommand{\spp}{\sigma_{\mathrm{pp}}}
\newcommand{\ssc}{\sigma_{\mathrm{sc}}}

\newcommand{\Thr}{\mathcal{T}}               %
\newcommand{\Ereg}{\mathcal{E}}              %

\newcommand{\kbulk}{\sqrt{E}}                %
\newcommand{\ksurf}{\sqrt{E+\alpha^2}}       %
\newcommand{\Echan}[1]{\mathcal{H}(#1)}      %
\newcommand{\Hdirect}{\mathcal{H}}           %

\newcommand{\Psib}{\Psi_{\mathrm{bulk}}}     %
\newcommand{\Psis}{\Psi_{\mathrm{surf}}}     %
\newcommand{\Fzero}{\mathcal{F}_0}           %
\newcommand{\Fbulk}{\mathcal{F}_{\mathrm{bulk}}}
\newcommand{\Fsurf}{\mathcal{F}_{\mathrm{surf}}}
\newcommand{\Ftbulk}{\widetilde{\mathcal{F}}_{\mathrm{bulk}}}  %
\newcommand{\Ftsurf}{\widetilde{\mathcal{F}}_{\mathrm{surf}}}  %
\newcommand{\Adens}[1]{A(#1)}                %
\newcommand{\cph}[1]{c(#1)}                  %
\newcommand{\refext}[1]{#1^{\mathrm{r}}}     %
\newcommand{\zext}[1]{#1^{\circ}}            %
\newcommand{\Utrans}{\mathcal{U}}            %

\newcommand{\Smat}{S}

\newcommand{\refl}[1]{\mathcal{R}(#1)}   %

\newcommand{\usurf}{u_{\mathrm{surf}}}
\newcommand{\ubulk}{u_{\mathrm{bulk}}}
\newcommand{\Usurf}{U_{\mathrm{surf}}}
\newcommand{\Ubulk}{U_{\mathrm{bulk}}}
\newcommand{\psizero}{\psi_0}                %
\newcommand{\gphi}{\varphi}                  %
\newcommand{\Hsurf}{\mathcal{H}_{\mathrm{surf}}}  %
\newcommand{\Hbulk}{\mathcal{H}_{\mathrm{bulk}}}  %

\newcommand{\Aav}{\mathcal{A}_\alpha}

\newcommand{\Csharp}{C_\sharp}               %
\newcommand{\CAH}[1]{C_{\mathrm{AH}}^{(#1)}} %

\newcommand{\Mav}[1]{M[#1]}
\newcommand{\MavR}[2]{M_{#2}[#1]}
\newcommand{\Mavt}[1]{M'[#1]}
\newcommand{\MavtR}[2]{M'_{#2}[#1]}

\newcommand{\Jbulk}{J_{\mathrm{bulk}}}
\newcommand{\Jsurf}{J_{\mathrm{surf}}}
\newcommand{\Jchan}{J}
\newcommand{\yout}{y^{(+)}}                  %
\newcommand{\yin}{y^{(-)}}                   %
\newcommand{\gambulk}{\gamma_{\mathrm{bulk}}}
\newcommand{\gamsurf}{\gamma_{\mathrm{surf}}}
\newcommand{\Csph}[2]{C_{#1}(#2)}
\newcommand{\Ball}[1]{B_{#1}^{+}}

\newcommand{\Nsol}[1]{\mathcal{N}(#1)}

\newcommand{\Wb}[2]{W^{(#1)}_{\mathrm{bulk}}[#2]}
\newcommand{\Ws}[2]{W^{(#1)}_{\mathrm{surf}}[#2]}

\newcommand{\Vsurf}{V}
\newcommand{\VsurfA}[1]{V_{#1}}              %

\newcommand{\Pibulk}{\Pi_{\mathrm{bulk}}}

\newcommand{\Ffree}[1]{\mathcal{F}^{\mathrm{free}}_{#1}}

\newcommand{\SmatA}[2]{S_{#1}(#2)}                   %
\newcommand{\SmatbulkA}[2]{S_{#1,\mathrm{bulk}}(#2)} %
\newcommand{\SmatsurfA}[2]{S_{#1,\mathrm{surf}}(#2)} %
\newcommand{\reflA}[2]{\mathcal{R}_{#1}(#2)}         %
\newcommand{\cphA}[2]{c_{#1}(#2)}                    %
\newcommand{\psizeroA}[1]{\psi_{0,#1}}               %
\newcommand{\HalphaA}[1]{H_{#1}}                     %
\newcommand{\HsurfA}[1]{\mathcal{H}_{#1,\mathrm{surf}}}
\newcommand{\HbulkA}[1]{\mathcal{H}_{#1,\mathrm{bulk}}}
\newcommand{\FsurfA}[1]{\mathcal{F}_{#1,\mathrm{surf}}}
\newcommand{\EregA}[1]{\mathcal{E}_{#1}}
\newcommand{\ksurfA}[1]{k_{#1,\mathrm{surf}}}
\newcommand{\Arec}[1]{A_{#1}}
\newcommand{\Wint}{W}

\newcommand{\eps}{\varepsilon}
\newcommand{\dist}{\operatorname{dist}}
\DeclareMathOperator{\re}{Re}                %
\DeclareMathOperator{\im}{Im}                %

%% file: Part1_01_introduction.tex
\subsection{The Impedance Half-Space}
\label{subsec:intro-setting}

Let $n \ge 3$. Let $\Rnp = \{x = (\xpar, x_n) \in \Rn : x_n > 0\}$ be the
half-space, and identify its boundary $\Gamma = \partial\Rnp$ with $\Rnm$. In
its physical form the problem of this paper is the Helmholtz equation in
$\Rnp$, subject to a constant impedance, or Robin, boundary condition:
\begin{equation}
  \label{eq:intro-helmholtz}
  (\Delta + k^2)u = 0 \quad\text{in } \Rnp,
  \qquad
  \partial_\nu u + \lambda u = 0 \quad\text{on } \Gamma .
\end{equation}
Here $k > 0$ is the wavenumber of the ambient medium. For a time-harmonic wave
of frequency $\omega$ in a medium of speed $c$ one has $k = \omega/c$, and
$2\pi/k$ is the wavelength. The vector $\nu = -e_n$ is the outward unit normal,
and the constant $\lambda \in \mathbb{R}$ is the impedance of the boundary. Its
two extreme values are the familiar ones: $\lambda = 0$ gives the Neumann, or
sound-hard, boundary, and $\lambda = \infty$ the Dirichlet, or sound-soft, one.

The sign of $\lambda$ determines the nature of the problem.
Consider the transverse operator $-d^2/dx_n^2$ on $\Ltwo{\Rhalf}$ with the
boundary condition $\psi'(0) = \lambda\psi(0)$. Any eigenfunction of it is a
multiple of $e^{\lambda x_n}$, with eigenvalue $-\lambda^2$, and this lies in
$\Ltwo{\Rhalf}$ exactly when $\lambda < 0$. So for $\lambda \ge 0$ there is no
eigenvalue at all. The transverse spectrum is then purely continuous, and the
theory is a variant of the Dirichlet or Neumann theory in the half-space. The
classical dissipative regime $\lambda > 0$, that of Newton's law of cooling,
falls on this side. The case treated here is the other one.

We therefore set $\alpha = -\lambda > 0$ once and for all. The boundary
condition becomes $\partial_{x_n}u + \alpha u = 0$ on $\Gamma$, which is the
form used in the applied literature, and the impedance $\lambda$ does not
occur again. The transverse eigenvalue is $-\alpha^2$ and the normalized
eigenfunction is
\begin{equation}
  \label{eq:intro-psizero}
  \psizero(x_n) = \sqrt{2\alpha}\,e^{-\alpha x_n} .
\end{equation}
The boundary thus carries a bound mode. It is confined to a layer of thickness
$1/\alpha$, and that layer runs along a boundary which is unbounded. The mode
is a second channel of propagation, and the whole of the analysis below is a
response to its presence. The length $1/\alpha$ is the only scale in the
problem.

Boundary conditions of the form \eqref{eq:intro-helmholtz} are the standard
macroscopic description of a surface that is neither perfectly reflecting nor
transparent. They occur throughout acoustics, electromagnetism and continuum
mechanics \cite{ChandlerWilde1998,Kirsch2018}.

A flat rigid wall is a different matter. It reflects an incident wave
completely and carries no wave of its own. It is the Neumann boundary. What has
changed the standing of \eqref{eq:intro-helmholtz} in recent years is that this
can be undone by structure alone. Cut the wall with grooves whose depth, width
and spacing all lie well below the wavelength. Each groove is then a small
closed resonator, and the field inside it forms a standing wave. The incident
field does not resolve the individual grooves, so their effect may be averaged.
In the subwavelength limit the corrugated wall is replaced, exactly, by a flat
wall carrying an effective impedance. If the grooves are shallower than a
quarter wavelength, that impedance falls on the binding side, which in the
present normalization is $\alpha > 0$. A wall with no surface wave of its own
has acquired one. This is the mechanism behind spoof surface plasmon polaritons
in electromagnetism \cite{Pendry2004}, and behind acoustic surface waves on
structured plates \cite{Cummer2016}.

The impedance half-space is therefore an object of design, and not only of
modelling. A designer fixes $\alpha$ by choosing the depth of the grooves and
the fraction of the surface they occupy. Through $\alpha$ he fixes the
penetration depth $1/\alpha$ of the surface wave, and the amount by which that
wave is slowed relative to the ambient one.\footnote{An impedance obtained by
such a homogenization depends on the frequency. We treat $\alpha$ as a constant
and let the energy vary. This suffices here, since the reconstruction of
Corollary~\ref{cor:inv-alpha} uses the scattering matrix at a single energy.
The dispersion of $\alpha$ is not otherwise discussed.} Equation
\eqref{eq:intro-helmholtz} is the simplest model in which this confinement can
be analysed exactly, and it is the model adopted here. It also gives a concrete
meaning to the inverse problem of \S\ref{sec:inverse}. To recover $\alpha$
from measurements made far from the surface is to determine the effective
impedance of a manufactured surface without dismantling it.

The work is divided into two parts. The present paper, Part~I, treats the flat
half-space, and develops its spectral and scattering theory completely: the
limiting absorption principle, the generalized Fourier transform, the
asymptotic behaviour of the resolvent at infinity, the radiation condition,
uniqueness, and the inverse problem for $\alpha$. Part~II treats local
perturbations of this configuration, a deformation of the boundary, an
obstacle placed in the interior, or both. There nothing forces the two
channels to decouple, and a bulk wave may convert into a surface wave.

The flat case is not a corollary of the perturbed one, and it cannot be
relegated to a preliminary section. It is where the channel structure, the
topology in which the two channels are to be compared, and the radiation
condition have all to be settled, and each of these requires a theorem. Beyond that, the Rellich-type theorem proved here is the step on
which the perturbed theory rests. At positive energy it excludes
eigenvalues embedded in the continuous spectrum of the perturbed operator. It
also shows that the far-field pattern determines the near field, so that data
taken at infinity are data on any surface enclosing the perturbation. That
pattern is here a pair of amplitudes, one for each channel. On the band it is unavailable.

\subsection{The Operator and Its Two-Channel Structure}
\label{subsec:intro-two-channel}

The analysis is carried out on the self-adjoint operator behind
\eqref{eq:intro-helmholtz}, namely
\begin{equation}
  \label{eq:intro-operator}
  \Halpha = -\Delta \quad\text{on } \Rnp,
  \qquad
  \gamma_1 u - \alpha\,\gamma_0 u = 0 \quad\text{on } \Gamma,
  \qquad \alpha > 0,
\end{equation}
where $\gamma_0$ is the Dirichlet trace and $\gamma_1 = \partial_\nu$ the
conormal derivative. Then \eqref{eq:intro-helmholtz} is the eigenvalue equation
$\Halpha u = Eu$ at $E = k^2$, and from here on the energy $E$ replaces the
wavenumber $k$. The gain is not one of language only. Real values of $k$ reach
the energies $E \ge 0$ and no others, whereas the spectrum of $\Halpha$ extends
down to $-\alpha^2$. The band $-\alpha^2 < E < 0$ is invisible in
\eqref{eq:intro-helmholtz}, and it is there that the surface wave propagates
alone.

The structure of the problem is visible after a partial Fourier transform in the
tangential variable $\xpar$. This reduces $\Halpha$ to the family of fiber
operators
\begin{equation}
  \label{eq:intro-fiber}
  \hfib(\xipar) = -\frac{d^2}{dx_n^2} + |\xipar|^2
  \quad\text{on } \Ltwo{\Rhalf},
  \qquad \psi'(0) = -\alpha\psi(0),
\end{equation}
each of which has the simple eigenvalue $|\xipar|^2 - \alpha^2$ with
eigenfunction \eqref{eq:intro-psizero}, together with the continuous spectrum
$[\,|\xipar|^2,\infty)$. Taking the union over $\xipar \in \Rnm$ gives
$\sigma(\Halpha) = [-\alpha^2,\infty)$. More to the point, it shows how a fixed
energy $E$ is carried. The two branches carry it on sets of different type:
\begin{equation}
  \label{eq:intro-wavenumbers}
  |\xipar| \le \sqrt{E} \quad\text{(continuous branch)},
  \qquad
  |\xipar| = \sqrt{E+\alpha^2} \quad\text{(eigenvalue branch)}.
\end{equation}
The continuous branch fills a ball, whose radius $\kbulk$ is the wavenumber of
the bulk channel. The eigenvalue branch is carried on a single sphere, whose
radius $\ksurf$ is the tangential wavenumber of the surface channel. The first
requires $E > 0$, the second only $E > -\alpha^2$. Two channels are therefore
open when $E > 0$, and exactly one on the band $-\alpha^2 < E < 0$. Since
$\ksurf > \kbulk$ whenever both are defined, the surface wave has the shorter
wavelength and the slower phase velocity, as it must.

The two words are the physical ones. In physics and engineering the
\emph{bulk} is the interior of a medium, where the influence of any boundary is
negligible. Bulk waves are the waves that propagate freely into it. They are
the modes of the free Laplacian on $\Rn$, restricted to the half-space and
reflected at $\Gamma$, and their band is $[0,\infty)$. Surface waves are made by
the wall. Their energy is trapped at $\Gamma$, they travel along it and decay
exponentially away from it, and their band is $[-\alpha^2,\infty)$. The theory
below is the theory of how the two families interact.

The set of thresholds is $\Thr = \{-\alpha^2, 0\}$, the first being the point at
which the surface channel opens and the second that at which the bulk channel
opens. We write $\Ereg = (-\alpha^2,0)\cup(0,\infty)$ for the non-threshold
energies. The two thresholds turn out to play symmetric roles: the resolvent
bounds of \S\ref{subsec:lap} diverge at the rate $\dist(E,\Thr)^{-1/2}$ at each,
and both must be excluded.

At infinity the two channels appear as spherical waves of different dimensions.
Writing $r = |x|$, $\theta = x/r$, $\rho = |\xpar|$ and $\omega = \xpar/\rho$,
an outgoing solution at energy $E > 0$ behaves as a superposition of
\begin{equation}
  \label{eq:intro-two-waves}
  A(\theta)\,r^{-\frac{n-1}{2}}e^{i\kbulk r},
  \qquad
  B(\omega)\,\rho^{-\frac{n-2}{2}}e^{i\ksurf\rho}\,\psizero(x_n),
\end{equation}
with $A$ on the upper hemisphere $\Sm$ and $B$ on the equatorial sphere
$\Ssurf$. The exponents are those of a spherical wave in dimension $n$ and in
dimension $n-1$ respectively: the surface wave is a spherical wave of the
boundary, not of the interior.

The difficulty which organizes the rest of the paper is that these two
expressions cannot be compared pointwise. Each is governed by its own radial
variable. Near the boundary the two variables are comparable, and there the
surface term is the larger. Away from the boundary the factor $\psizero(x_n)$
makes the surface term exponentially small, and the bulk term is the larger.
Neither ordering holds uniformly in the direction, and no pointwise estimate
covers both.

A single framework is nevertheless possible, because the two
rates in \eqref{eq:intro-two-waves} are each exactly critical for their own
dimension. Average the square of each over the ball $|x| < R$, divide by $R$ and
let $R \to \infty$. Since $dx = r^{n-1}dr\,dS(\theta)$, the bulk term gives
\[
  \frac1R\int_{|x|<R} |A(\theta)|^2 r^{-(n-1)}\,dx
  \;\longrightarrow\; \int_{\Sm}|A|^2\,dS,
\]
and since $dx = \rho^{n-2}d\rho\,dS(\omega)\,dx_n$ with
$\norm{\psizero}_{\Ltwo{\Rhalf}} = 1$, the surface term gives
\[
  \frac1R\int_{|x|<R} |B(\omega)|^2 \rho^{-(n-2)}\psizero(x_n)^2\,dx
  \;\longrightarrow\; \int_{\Ssurf}|B|^2\,dS .
\]
Both limits are finite and nonzero, and each returns the far-field amplitude of
its own channel. The two channels therefore sit at a common borderline, in spite
of their different rates and dimensions, and that borderline is the one of the
Agmon--H\"ormander spaces \cite{Agmon1975,AgmonHormander1976}. With
$D_j = \{2^j \le |x| < 2^{j+1}\}$, the space $\Bst$ is defined by the finiteness
of $\sup_j 2^{-j/2}\norm{u}_{\Ltwo{D_j}}$ and $\Bstz$ by the vanishing of these
quantities. Now $\Bst$ contains both terms of \eqref{eq:intro-two-waves} with
nonzero norm, and $\Bstz$ contains neither. A weighted space
$\Ltwo{\langle x\rangle^{-2s}dx}$ cannot do this. For $s > 1/2$ it contains both
terms but does not see their amplitudes. For $s \le 1/2$ it excludes them.

This is why the framework is not a matter of presentation. The limiting
absorption principle has to be proved as a bound $\Bsp \to \Bst$ rather than
between weighted spaces, since it is only at the endpoint that both channels are
resolved. The radiation condition has to be a statement of membership in
$\Bstz$. And the hypothesis of the Rellich-type uniqueness theorem has to be of
the same kind, since anything stronger would exclude the very solutions the
theorem is about. These are the results described next.

\subsection{Main Results}
\label{subsec:intro-main-results}

The aim of this section is to state, in the order in which they are proved, the
four results that together constitute the scattering theory of $\Halpha$ in the
Agmon--H\"ormander spaces $\Bsp$--$\Bst$. All of them are proved for $n \ge 3$
and for non-threshold energies $E \in \Ereg$. See
Remark~\ref{rem:standing-dimension} on the dimension. The four statements are
displayed as Main Theorems~A--D. The paragraphs between them describe what
carries one to the next. The comparison with the literature is deferred to
\S\ref{subsec:intro-method} and \S\ref{subsec:intro-novelty}.

The starting point is the spectral structure of $\Halpha$
(Theorem~\ref{thm:sa-spec}): the operator is self-adjoint and bounded below by
$-\alpha^2$, its spectrum is $[-\alpha^2,\infty)$, and it is purely absolutely
continuous with no eigenvalues. Absolute continuity is proved by explicit
diagonalization rather than by commutator methods, since the fibers
\eqref{eq:intro-fiber} have continuous spectrum and the analytically fibered
framework is unavailable.

\begin{mainthm}[Limiting absorption principle; Theorem~\ref{thm:lap}]
\label{mainthm:lap}
For $E \in \Ereg$ the boundary values $\Ralpha{E \pm i0}$ exist as bounded
operators $\Bsp \to \Bst$, locally uniformly in $E$, and depend continuously on
$E$. The bounds diverge at the rate $\dist(E,\Thr)^{-1/2}$ at each of the two
thresholds.
\end{mainthm}

The proof rests on an exact algebraic identity
(Proposition~\ref{prop:splitting}),
\begin{equation}
  \label{eq:intro-splitting}
  \Ralpha{z} = \RD{z} + \Rfree{n-1}{z+\alpha^2} \otimes \Psurf + \Krem{z},
  \qquad \Psurf = \psizero \otimes \overline{\psizero},
\end{equation}
in which the first term carries the bulk channel, the second is the surface
channel in tensor form, and the remainder satisfies a fiber bound uniform in
$x_n$, with constant $1/\alpha$ (Lemma~\ref{lem:K-kernel-bound}). The
comparison is made against the Dirichlet Laplacian, not the Neumann one. With
$\kappa = (|\xipar|^2 - z)^{1/2}$, the weight produced by the reflection
coefficient is $(\kappa-\alpha)^{-1}$ in the first case and
$\kappa^{-1}(\kappa-\alpha)^{-1}$ in the second, and only the former stays
bounded at the bulk threshold $\kappa \to 0$, where the reflection coefficient
tends to $-1$. The two thresholds are then excluded symmetrically,
at the same rate, although one is an endpoint of the spectrum and the other is
embedded in it.

The limiting absorption principle yields a generalized Fourier transform
$\Fzero$ diagonalizing $\Halpha$ (Theorem~\ref{thm:spectral-representation}),
with the two-channel representation space
\begin{equation}
  \label{eq:intro-channel}
  \Echan{E} = \Ltwo{\Sm} \oplus \Ltwo{\Ssurf}
  \quad (E > 0),
  \qquad
  \Echan{E} = \Ltwo{\Ssurf}
  \quad (-\alpha^2 < E < 0),
\end{equation}
and Stone's formula (Theorem~\ref{thm:stone}). The band $(-\alpha^2,0)$ is a
single-channel window in which only the surface channel is open.

It yields also the behaviour of $\Ralpha{E+i0}f$ at infinity. Write
$u \simeq v$ for $u - v \in \Bstz$, that is, for the vanishing of the dyadic
shell averages of $u - v$ (Definition~\ref{def:simeq}). Then for $E > 0$ and
$f \in \Bsp$ (Theorem~\ref{thm:asymp-uniform}),
\begin{equation}
  \label{eq:intro-uniform}
  \Ralpha{E+i0}f
  \;\simeq\;
  c_n \bigl(\Fbulk^{(+)}(E)f\bigr)(\theta)\,
      r^{-\frac{n-1}{2}} e^{i\kbulk r}
  \;+\;
  c_n' \bigl(\Fsurf(E)f\bigr)(\omega)\,
      \rho^{-\frac{n-2}{2}} e^{i\ksurf\rho}\,\psizero(x_n),
\end{equation}
with explicit constants $c_n$, $c_n'$. On the surface band $-\alpha^2 < E < 0$
the second term stands alone (Theorem~\ref{thm:characterization}). A single
relation covers every direction, the grazing region $x_n = o(|x|)$ included,
where by \eqref{eq:intro-two-waves} the two terms are not comparable pointwise.
Expansions of this form are due to Isozaki, Kadowaki and Watanabe
\cite{IsozakiKadowakiWatanabe2020,IsozakiKadowakiWatanabe2023}, and no novelty
is claimed for \eqref{eq:intro-uniform} itself. What is different is the route
to it, which is described in \S\ref{subsec:intro-method}.

\begin{mainthm}[Solutions of $\Bst$ growth, and the scattering matrix;
Theorems~\ref{thm:characterization} and \ref{thm:shell-average-identity}]
\label{mainthm:char}
Let $E \in \Ereg$. A solution $u \in \Bst$ of $(\Halpha - E)u = 0$ is precisely
an element of $\Ran \Fzero(E)^*$, and is determined by a pair
$\yin, \yout \in \Echan{E}$ of incoming and outgoing far-field amplitudes
$($Corollary~\ref{cor:solution-far-field}$)$. The scattering matrix is
\begin{equation}
  \label{eq:intro-smatrix}
  \Smat(E) = \refl{\kbulk\theta_n} \oplus \mathrm{Id}_{\mathrm{surf}} .
\end{equation}
\end{mainthm}

The surface block is the identity because the surface mode is a bound state of
the transverse problem, and a bound state is not reflected. The lower bound
needed for the characterization does not follow from Stone's formula, which
controls $\Fzero(E)$ on $\Bsp$ but says nothing about the invertibility of
$\Fzero(E)^*$ from below. It comes instead from a shell-average identity for the
generalized eigenfunctions, whose source is total reflection,
$|\refl{\cdot}| = 1$, at the impedance boundary. The separation of the two
channels rests on an exact orthogonality: in $\Ltwo{\Rhalf}$ the bound mode
$\psizero$ is orthogonal to every generalized eigenfunction of the fiber
\eqref{eq:intro-fiber}. This gives $\Bst$-orthogonality of the channels
(Proposition~\ref{prop:char-channel-orth}), and it does not rest on the
difference $\kbulk \neq \ksurf$ of the wavenumbers.

\begin{mainthm}[Rellich-type theorem; Theorem~\ref{thm:rad-rellich}]
\label{mainthm:rellich}
Let $E > 0$ and let $u$ solve the Robin problem outside a ball. If
\[
  \lim_{R\to\infty}\frac1R\int_{\{|x|<R\}\cap\Rnp}|u(x)|^2\,dx = 0 ,
\]
then $u$ vanishes outside some larger ball.
\end{mainthm}

The hypothesis is one scalar statement about $u$, imposed only near infinity.
There is no decomposition into channels, no pointwise decay, and no condition on
the boundary trace. And yet both channels are forced to vanish at once. Why this
is not a formality, and how it stands to the existing literature, is the subject
of \S\ref{subsec:intro-novelty}.

Combined with unique continuation \cite{Aronszajn1957}, the theorem gives
uniqueness for the homogeneous problem
(Theorem~\ref{thm:rad-uniqueness-hom}), and with it existence and uniqueness for
$(\Halpha - E)u = f$ under the radiation condition
(Theorem~\ref{thm:rad-main}). It also yields an equivalence
(Theorem~\ref{thm:rad-equivalence}). For $u \in \Bst$ solving that problem with
$f \in C_c^\infty(\Rnp)$, that is, with $f$ smooth and compactly supported in
$\Rnp$, three conditions coincide. They are the two-region
pointwise condition of Definition~\ref{def:rad-condition}, the purely outgoing
form $u \simeq \Wb{+}{A} + \Ws{+}{C}$, which is to say the vanishing of the
incoming data, and the identity $u = \Ralpha{E+i0}f$.

\begin{mainthm}[Rigidity of the surface channel, and reconstruction;
Proposition~\ref{prop:inv-rigidity} and Corollary~\ref{cor:inv-alpha}]
\label{mainthm:inverse}
For every $\alpha > 0$ the surface block of \eqref{eq:intro-smatrix} is the
identity, and the surface data of two distinct impedances are unitarily
intertwined. On the band $-\alpha^2 < E < 0$ the scattering matrix is therefore
the identity in its entirety. The bulk block, by contrast, determines $\alpha$:
for $E > 0$ and $\theta \in \Sm$ with $\theta_n > 0$,
\[
  \alpha = \sqrt{E}\,\theta_n\,
    \tan\Bigl( \tfrac12 \arg \SmatbulkA{\alpha}{E}(\theta) \Bigr).
\]
\end{mainthm}

The rigidity is intrinsic: the surface transform
factors through an energy-independent isometry and the free generalized Fourier
transform of $\Rnm$ (Lemma~\ref{lem:inv-surface-free}). So the two channels
behave in opposite ways, and the whole of the information about $\alpha$ sits in
the bulk block. In the reconstruction formula the argument lies in $(0,\pi)$, so
that no branch has to be chosen, and the conditioning is best in the
directions with $\sqrt{E}\,\theta_n = \alpha$.

\subsection{Method and Background}
\label{subsec:intro-method}

\subsubsection*{The splitting identity}

Almost everything in Part~I rests on the identity \eqref{eq:intro-splitting}. It
is exact and algebraic rather than perturbative, obtained by computing the fiber
resolvent of \eqref{eq:intro-fiber} in closed form and recognizing the Dirichlet
resolvent and the surface projection inside it. Its two features recorded above
are used repeatedly: the comparison with the Dirichlet rather than the Neumann
Laplacian, and the bound on $\Krem{z}$ uniform in $x_n$, which allows the
remainder to be controlled on regions about which nothing pointwise is known.
The limiting absorption principle, the shell-negligibility of the grazing region, and
the collapse of the analysis on the surface band $-\alpha^2 < E < 0$ all follow
from these two facts.

\subsubsection*{The far field}

The far-field analysis is classical in its ingredients. The fiber kernel is put
in reflected form, and the resulting integrals are evaluated by the asymptotics
of Hankel functions and by stationary phase. The step that is not routine is the
passage from the expansion to a characterization of all solutions of $\Bst$
growth, and it rests on the lower bound on $\Fzero(E)^*$ described after Main
Theorem~\ref{mainthm:char}. Its consequence is that $\Fzero(E)^*$ has closed
range.

\subsubsection*{Uniqueness}

The uniqueness proof begins from an exact energy identity. Since $\alpha$ is
real, the boundary integrand
$u\,\overline{\partial_\nu u} - \bar u\,\partial_\nu u$ vanishes identically on
$\Gamma$, and not merely in its imaginary part. The two channels are then
separated at the level of $\Bst$ by the orthogonality described after Main
Theorem~\ref{mainthm:char}, and the order in which they are treated is forced.
The surface component has to be disposed of first, by the classical Rellich
theorem \cite{Rellich1943} in dimension $n-1$ followed by unique continuation in
$\Rnm$, and only then can the bulk component be treated in dimension $n$. This
is where the standing assumption $n \ge 3$ is used
(Remark~\ref{rem:standing-dimension}). The Robin condition makes itself felt at
exactly this point: the passage from the half-space to the whole space cannot be
made by reflection, and is made instead by an explicit integral extension
formula, which is usable only because the surface component has already been
shown to vanish (\S\ref{subsec:radiation}).

\subsubsection*{Background}

For the half-space Robin Laplacian itself the literature is piecemeal. Its
spectral theory is highly developed
\cite{LevitinParnovski2008,BehrndtLanger2007,LotoreichikRohleder2012,%
BehrndtSchlosser2018,CossettiKrejcirik2020}, but is concerned with eigenvalue
asymptotics, with the absence of eigenvalues, and with the dependence of the
spectrum on the geometry. We use it only for the qualitative facts recorded in
Theorem~\ref{thm:sa-spec}. There is also a scattering theory, due to Frank and
to Frank and Shterenberg \cite{Frank2003,FrankShterenberg2004,Frank2006}, for an
impedance periodic along $\Gamma$, of which ours is a degenerate case, and its
picture agrees with ours (Remark~\ref{rem:literature}). Its route, a Floquet
decomposition followed by wave operators for the fibers, produces neither of the
two objects this paper is built on: the resolvent bound in the endpoint spaces
$\Bsp$, $\Bst$ and the pointwise asymptotics of the generalized
eigenfunctions. The
outgoing problem under an impedance condition was analysed by Dur\'an, Muga and
N\'ed\'elec \cite{DuranMugaNedelec2009} in weighted $L^2$ spaces, and their
two-region radiation condition is the starting point of
\S\ref{subsec:radiation}. For a Dirichlet condition on a half-space perturbed
by a graph tending to zero at infinity, Eskin and Ralston
\cite{EskinRalston1999} worked with outgoing solutions in the sense of Agmon
\cite{Agmon1975}, and related the Dirichlet-to-Neumann map at a fixed energy to
the scattering amplitude of the operator extended to $\Rn$. Transmission and Besov-space
methods for wave propagation in perturbed stratified media are due to Zhang
\cite{Zhang1994,Zhang2000}.

The framework comes from elsewhere. The structural template is Isozaki's
treatment of multichannel stationary scattering theory
\cite{Isozaki2001,Isozaki2023}. For the general stationary theory see
\cite{Yafaev}. It is a template in which the limiting absorption principle, the
eigenfunction expansion, the radiation condition and the Rellich-type theorem
form a single chain, each link furnishing the hypotheses of the next. The
closest analogue of the present problem is the elastic half-space, treated in
the Agmon--H\"ormander spaces by Isozaki, Kadowaki and Watanabe
\cite{IsozakiKadowakiWatanabe2023,IsozakiKadowakiWatanabe2025,%
IsozakiKadowakiWatanabe2025JOS}, where the body
waves and the Rayleigh surface waves occur together in the expansion and in the
uniqueness theorem. \S\ref{subsec:intro-novelty} explains where the present
treatment differs and why the difference matters. Rellich-type theorems have
been obtained in that school in settings far removed from this one: on
asymptotically hyperbolic manifolds by Isozaki and Kurylev
\cite{IsozakiKurylev2014}, and on the square lattice by Isozaki and Morioka
\cite{IsozakiMorioka2014}, where the theorem yields unique continuation and the
absence of embedded eigenvalues. What those results share with the present one
is the place the theorem occupies in the chain, not the geometry.

Two smaller debts remain to be recorded. Unique continuation, used
inside the Rellich-type theorem and again to pass from it to uniqueness for the
homogeneous problem, is taken from Aronszajn \cite{Aronszajn1957}. See
\cite{KochTataru2001} for the sharp form under weaker regularity. And the inverse problem
of \S\ref{sec:inverse} is elementary beside the inverse scattering literature
for potentials \cite{SylvesterUhlmann1987,Nachman1996,Paivarinta2014}, for which
\cite{ColtonKress2013,Isakov2017} are the standard references, since only a
single parameter is to be recovered.

\subsection{Novelty and Relation to Earlier Work}
\label{subsec:intro-novelty}

\subsubsection*{The radiation condition}

No novelty is claimed for the Agmon--H\"ormander framework, which is standard
\cite{Agmon1975,AgmonHormander1976} and, as recalled above, has already been
used for the elastic half-space. What is new is the shape of the condition
inside it. The half-space must be divided, since the
two channels carry different wavenumbers and no single expression
$\partial_r u - iku$ can serve in all directions. What is not divided is the
solution. The same quantity appears in both regions of
Definition~\ref{def:rad-condition}, with only the wavenumber changing, and it is
applied to $u$ itself. In the elastic half-space the corresponding conditions
are imposed instead on the components of a polarization decomposition
\cite[Thm.~5.1]{IsozakiKadowakiWatanabe2025JOS}. That difference decides
everything that follows: a condition imposed on the components of a
decomposition presupposes the asymptotic analysis which produces that
decomposition, and so cannot serve as the hypothesis of a Rellich-type theorem,
where nothing beyond decay at infinity may be assumed.

The division of the half-space is governed by an exponent $\alpha_1 \in (0,1)$,
which cuts $\Rnp$ along the surface $x_n = r^{\alpha_1}$ into an interior region,
where the bulk wave dominates, and a boundary layer, where the surface wave
does. Dur\'an, Muga and N\'ed\'elec \cite{DuranMugaNedelec2009} require
$\alpha_1 \in (1/4,1/2)$. Here both endpoints disappear, for different reasons
explained in Remark~\ref{rem:rad-alpha1}, and the uniqueness theorem, the
linearity of the radiation condition, and the verification that the resolvent
satisfies it all hold for every $\alpha_1 \in (0,1)$, uniformly in $n$. That series of papers
\cite{DuranMugaNedelec2005a,DuranMugaNedelec2005b,DuranMugaNedelec2009} is
moreover confined to two and three dimensions. The treatment given here is intrinsic and holds for every
$n \ge 3$: the channel spaces \eqref{eq:intro-channel} come from the spectral
decomposition of the fibers \eqref{eq:intro-fiber}, and the two wavenumbers
\eqref{eq:intro-wavenumbers} arise as the branches at which a given energy is
carried, so that $n = 3$ is an instance of the picture and not its foundation.
That matters for Part~II, where the off-diagonal blocks of
\eqref{eq:intro-smatrix} are no longer forced to vanish and the channel spaces have to be compared with one another.

\subsubsection*{The Rellich-type theorem}

This is the main contribution. The natural comparison is with the conical
domains of Bonnet-Ben Dhia, Fliss, Hazard and Tonnoir
\cite{BonnetBenDhia2016}. For an axisymmetric cone of vertex angle strictly
greater than $\pi$ they prove, in any dimension and with no boundary condition
whatever, that a square-integrable solution of the Helmholtz equation vanishes.
Their theorem is sharp, and the half-space is precisely where it breaks: at
vertex angle exactly $\pi$ the statement is false, as they observe,
square-integrable solutions being available there.

Equation \eqref{eq:intro-evanescent} exhibits them. For
$a \in C_c^\infty(\{|\xipar| > \kbulk\})$ the evanescent superposition
\begin{equation}
  \label{eq:intro-evanescent}
  u(x) = \int_{|\xipar| > \kbulk} a(\xipar)\,e^{i\xipar\cdot\xpar}\,
         e^{-\sqrt{|\xipar|^2 - E}\,x_n}\,d\xipar
\end{equation}
solves $(-\Delta - E)u = 0$ in $\Rnp$, is rapidly decreasing, and satisfies
every hypothesis imposed at infinity anywhere in this paper. The geometry
treated here is therefore the borderline one, and a boundary condition is not a
technical convenience but the only thing that can restore a uniqueness theorem
at all. What \S\ref{subsec:radiation} proves is that the Robin condition does
restore it, and under a hypothesis weaker than square integrability: $u \in
\Ltwo{\Rnp}$ implies $u \in \Bstz$, and not conversely.

The two channels make this weakening necessary. In the junction of
stratified media of Al Humaikani, Bonnet-Ben Dhia, Fliss and Hazard
\cite{AlHumaikani2025} a guided wave is evanescent across the guide but
oscillates along it, and so fails to be square-integrable. That geometry is
two-dimensional, it does carry guided modes, and again there is no boundary
condition.
The global $L^2$ hypothesis therefore disposes of the guided channel outright,
before any interaction between channels has to be examined. Under a
shell-average hypothesis no such thing happens here. The surface channel is
exactly critical: it lies in $\Bst$ with finite nonzero norm, and it has to be
removed by an argument rather than by the hypothesis. That argument is the
$\Bst$-orthogonality of the two channels, which forces them to vanish
together.

A second tradition is $N$-body scattering, where Isozaki \cite{Isozaki1994}
generalized the Sommerfeld radiation condition and Ito and Skibsted
\cite{ItoSkibsted2016} proved a Rellich theorem under exactly the hypothesis
used here, membership in $\Bstz$, admitting Dirichlet conditions on bounded
obstacles. See also \cite{AdachiItakuraItoSkibsted2021}. There the channels sit
in conical regions of a single configuration space and are measured against one
and the same radial weight $|x|$. Here they have different effective dimensions,
$n$ and $n-1$, and there is no common radial variable in which to compare them.
Related uniqueness theorems for open waveguides and perturbed stratified media,
under an $L^2$ or eigenfunction hypothesis, are due to Weder \cite{Weder1991},
Bonnet-Ben Dhia, Dakhia, Hazard and Chorfi \cite{BonnetBenDhiaDakhia2009},
Ciraolo and Magnanini \cite{CiraoloMagnanini2009} and Hazard
\cite{Hazard2015}, and a single-channel half-space result to Vesalainen
\cite{Vesalainen2014}.

With these delimitations recorded, the claim is the following. \emph{To the best
of our knowledge, no Rellich-type theorem has previously been established for a
system in which two channels of different effective dimension propagate at a
common energy along an unbounded boundary, under a hypothesis imposed solely
through the shell averages of the solution near infinity.} For the elastic
half-space, whose channel geometry is the closest to this one, the corresponding
extension is recorded as an open problem in
\cite[\S5]{IsozakiKadowakiWatanabe2025JOS}, and is identified there as the step
which the inverse problem would require.

\subsubsection*{Rigidity in the inverse problem}

The recovery of $\alpha$ itself is elementary. What is substantive is the
asymmetry between the two blocks in Main Theorem~\ref{mainthm:inverse}, and it
is this rigidity of the surface channel that makes the perturbed problem of
Part~II necessary rather than merely more general.

\subsection{Organization of the Paper}
\label{subsec:intro-plan}

Section~\ref{sec:unperturbed} develops the theory of the flat half-space, in the
order in which each result is needed by the next.
\S\ref{subsec:selfadjoint} establishes self-adjointness, the spectrum
$[-\alpha^2,\infty)$ and absolute continuity, by explicit diagonalization.
\S\ref{subsec:lap} proves the limiting absorption principle
(Main Theorem~\ref{mainthm:lap}) from the exact splitting
\eqref{eq:intro-splitting}, with a remark on what Mourre theory does and does
not give here. \S\ref{subsec:gft} constructs the generalized Fourier transform,
first abstractly and then by identifying it through $\delta$-orthogonality, and
proves Stone's formula. \S\ref{subsec:asymptotics} derives the asymptotic
expansion \eqref{eq:intro-uniform}. \S\ref{subsec:characterization}
characterizes the solutions of $\Bst$ growth and defines the scattering matrix
(Main Theorem~\ref{mainthm:char}), by way of the shell-average identity and the
two-sided estimate it yields. \S\ref{subsec:radiation} proves the Rellich-type
theorem (Main Theorem~\ref{mainthm:rellich}), the equivalence of the radiation
conditions, and existence and uniqueness for the radiation problem.
Section~\ref{sec:inverse} treats the inverse problem for $\alpha$ and the
rigidity of the surface channel (Main Theorem~\ref{mainthm:inverse}). The proof
of Theorem~\ref{thm:sa-spec} is given in Appendix~\ref{app:sa-spec}, and
Appendix~\ref{app:mourre} develops the Mourre theory of $\Halpha$: it is used
nowhere in the body of the paper, but it exhibits both thresholds as the zeros
of a single explicit constant.

Part~II will treat local perturbations of this configuration: a deformation of
the boundary, an obstacle in the interior, or both. There the scattering
matrix is no longer that of the flat problem. By
Proposition~\ref{prop:inv-rigidity} the flat surface-to-surface block is the
identity for every $\alpha > 0$, so any departure from the identity is
attributable to the perturbation and not to the impedance. The inverse problem
in that setting is the subject of Part~II.

Two further directions may be mentioned. One is the elastic half-space,
mentioned in \S\ref{subsec:intro-novelty}: the mechanism used below is not
specific to the scalar setting, though the elastic problem has a feature with no
scalar analogue, namely the singularities carried by the critical-angle cone
between the two body-wave speeds. The other is a variable impedance
$\alpha = \alpha(\xpar)$, or an impedance changing sign, which lies outside the
method used here, since that method depends on the exact fibration
\eqref{eq:intro-fiber}.

%% file: Part1_02_opening.tex
\subsection{Notation, Function Spaces, and the Robin Laplacian}
\label{sec:setting}

In this section we fix the notation and the function spaces used throughout the
paper, and define the Robin, or impedance, Laplacian that is its object of
study.

\subsubsection{Notation}
\label{sec:notation}
\label{sec:fourier}   %

Throughout the paper we work in the open half-space
\begin{equation*}
  \Rnp = \{x = (\xpar, x_n) \in \Rn : x_n > 0\},
  \qquad n \ge 3,
  \qquad
  \Gamma = \partial\Rnp = \{(\xpar, 0) : \xpar \in \Rnm\} \cong \Rnm .
\end{equation*}
A point of $\Rnp$ is written $x = (\xpar, x_n)$, where
$\xpar = (x_1, \dots, x_{n-1}) \in \Rnm$ is the tangential variable and
$x_n > 0$ is the normal variable. The normal coordinate is always denoted $x_n$
and never $z$. The dual variable is split in the same way,
$\xi = (\xipar, \xi_n)$. We split the gradient and the Laplacian accordingly,
\begin{equation*}
  \nabla = (\nabla', \partial_{x_n}),
  \qquad
  \nabla' = (\partial_{x_1}, \dots, \partial_{x_{n-1}}),
  \qquad
  \Delta = \Delta' + \partial_{x_n}^2,
  \qquad
  \Delta' = \sum_{j=1}^{n-1} \partial_{x_j}^2 .
\end{equation*}
The outward unit normal on $\Gamma$ is $\nu = -e_n$, so that the normal
derivative is $\partial_\nu = -\partial_{x_n}$.

\begin{remark}[Standing assumption on the dimension]
\label{rem:standing-dimension}
Throughout this paper we assume $n \ge 3$. Where a result holds for every
$n \ge 2$ we state it in that generality. This is the case for
Theorem~\ref{thm:sa-spec}, Proposition~\ref{prop:fibration} and
Theorem~\ref{thm:lap}.

The restriction concerns the surface channel, whose direction set $\Ssurf$
reduces to $S^0$ when $n = 2$. The restriction theorem of \S\ref{subsec:gft}
and the stationary-phase analysis of \S\ref{subsec:asymptotics} would then have
to be replaced by their one-dimensional counterparts, but the essential
obstruction lies in the Rellich-type theorem of \S\ref{subsec:radiation}. There
the surface component of a solution satisfies the free Helmholtz equation on
$\Rnm$ outside a compact set, and the argument applies the classical Rellich
theorem in dimension $n-1$ and then propagates the vanishing by unique
continuation through $\{\xpar \in \Rnm : |\xpar| > R_0\}$. Both steps require
$n - 1 \ge 2$, whereas for $n = 2$ that region is a pair of half-lines.
\end{remark}

We collect the remaining symbol conventions. The hat is reserved exclusively
for the Fourier transform of functions, so that $\widehat{u}$ always denotes a
transformed function and never a unit vector. Directions on the unit spheres
are denoted by separate letters: $\theta \in \Sm$ for the bulk hemisphere and
$\omega \in \Ssurf$ for the surface sphere. The letter $\alpha > 0$ is the
Robin, or impedance, parameter fixed in \S\ref{subsec:intro-setting}. The
original impedance $\lambda = -\alpha$ is not used again. Since $\alpha$ is thus
unavailable, multi-indices are written $\beta = (\beta_1, \dots, \beta_n)$, with
$|\beta| = \beta_1 + \cdots + \beta_n$ and
$\partial^\beta = \partial_{x_1}^{\beta_1} \cdots \partial_{x_n}^{\beta_n}$. The
spectral, or energy, parameter is denoted $E$ when real and $z$ when complex,
and the letter $z$ is kept free for that purpose. Finally, we write $|\xipar|$
for the Euclidean norm of $\xipar \in \Rnm$, and $C$ for a generic positive
constant whose value may change from line to line. Such constants are
independent of the energy $E$, of the radius $R$ and of the functions involved,
but not in general of $n$ or of $\alpha$. The dependence on $\alpha$ is
displayed whenever the behaviour as $\alpha \to 0$ or $\alpha \to \infty$ is at
issue.

Fourier transforms are normalized symmetrically. On $\Rn$ the transform and its
inverse both carry the factor $(2\pi)^{-n/2}$, and on $\Rnm$ both carry
$(2\pi)^{-(n-1)/2}$. With this choice the transform is unitary on $L^2$ and the
Plancherel identity holds with no extra constant. Because the half-space is
translation invariant in the tangential variable, the transform we actually use
is the partial one, in $\xpar$ alone:
\begin{equation}
  \label{eq:fourier-partial}
  (\mathcal{F}' u)(\xipar, x_n) = \widehat{u}(\xipar, x_n)
    = \left(\frac{1}{2\pi}\right)^{(n-1)/2}
       \int_{\Rnm} e^{-i \xpar \cdot \xipar}\, u(\xpar, x_n)\, d\xpar ,
\end{equation}
its inverse being given by the same formula with $e^{+i\xpar \cdot \xipar}$ in
place of $e^{-i\xpar \cdot \xipar}$. The constants in
\eqref{eq:fourier-partial} enter the normalization of the generalized
eigenfunctions and of the scattering matrix constructed in \S\ref{subsec:gft},
and are tracked throughout.

\subsubsection{Function spaces}
\label{sec:function-spaces}

All function spaces are complex, and derivatives are understood in the weak
sense unless stated otherwise. We use the Hilbert space $\Ltwo{\Rnp}$ with
inner product and norm
\begin{equation*}
  (u, v)_{\Ltwo{\Rnp}} = \int_{\Rnp} u(x)\,\overline{v(x)}\, dx,
  \qquad
  \norm{u}_{\Ltwo{\Rnp}} = (u, u)_{\Ltwo{\Rnp}}^{1/2},
\end{equation*}
where $dx = d\xpar\, dx_n$ denotes Lebesgue measure on $\Rnp$. Regarding
$\Ltwo{\Rnp}$ as $L^2(\Rnm_{\xpar}; \Ltwo{\Rhalf})$, the partial Fourier
transform $\mathcal{F}'$ of \eqref{eq:fourier-partial} is unitary on
$\Ltwo{\Rnp}$.

For an integer
$m \in \mathbb{N}$ the Sobolev space $H^m(\Rnp)$ consists of those
$u \in \Ltwo{\Rnp}$ for which $\partial^\beta u \in \Ltwo{\Rnp}$ for every
multi-index $\beta$ with $|\beta| \le m$, and it carries the norm
\begin{equation*}
  \norm{u}_{H^m(\Rnp)}
    = \Bigl( \sum_{|\beta| \le m}
       \norm{\partial^\beta u}_{\Ltwo{\Rnp}}^2 \Bigr)^{1/2}.
\end{equation*}

For an open set $\Omega$, in a Euclidean space or in a smooth manifold such as
the sphere $S^{n-1}$, we write $C_c^\infty(\Omega)$ for the space of smooth
functions on $\Omega$ whose support is a compact subset of $\Omega$. The
notation $C_0^\infty(\Omega)$ is also common. In particular an element of
$C_c^\infty(\Rnp)$ vanishes in a neighbourhood of $\Gamma$, and
$C_c^\infty(\Rnp)$ is dense in $\Ltwo{\Rnp}$. We write
$C_c^\infty(\overline{\Rnp})$ for the space of restrictions to
$\overline{\Rnp}$ of functions in $C_c^\infty(\Rn)$. Its elements are smooth
up to the boundary and need not vanish on $\Gamma$, as is needed for the
Robin condition, and $C_c^\infty(\overline{\Rnp})$ is dense in
$H^m(\Rnp)$ for every $m$.

Boundary values are taken through the trace operators. For $m \ge 1$ the
Dirichlet trace $\gamma_0 u = u|_{\Gamma}$ extends to a bounded surjection
$\gamma_0 : H^m(\Rnp) \to H^{m - 1/2}(\Gamma)$, and for $m \ge 2$ the Neumann
trace
$\gamma_1 u = \partial_\nu u|_{\Gamma} = -\partial_{x_n} u|_{\Gamma}$ is bounded
from $H^m(\Rnp)$ to $H^{m - 3/2}(\Gamma)$. Here
$H^s(\Gamma) = H^s(\Rnm)$, $s \in \mathbb{R}$, denotes the usual Sobolev scale
on the boundary, defined through the Fourier transform on $\Rnm$ normalized as
in \S\ref{sec:notation}.

\subsubsection{The Robin Laplacian}
\label{sec:operator}

We can now define precisely the operator written informally in
\eqref{eq:intro-operator}. The boundary condition of \S\ref{subsec:intro-setting},
namely $\partial_{x_n}u + \alpha u = 0$ on $\Gamma$, reads in the language of the
trace operators
\begin{equation}
  \label{eq:robin-alpha}
  \gamma_1 u - \alpha\, \gamma_0 u = 0
  \qquad\text{on } \Gamma,
  \qquad \alpha > 0,
\end{equation}
the change of sign being that of the outward normal $\nu = -e_n$. The
\emph{half-space Robin Laplacian} is then the operator
\begin{equation}
  \label{eq:Halpha-def}
  \Halpha u = -\Delta u,
  \qquad
  \Dom(\Halpha)
    = \bigl\{\, u \in H^2(\Rnp) \;\bigm|\;
       \gamma_1 u - \alpha\, \gamma_0 u = 0 \text{ on } \Gamma \,\bigr\} .
\end{equation}
The impedance $\alpha > 0$ is fixed once and for all in
\S\ref{sec:unperturbed}, and the dependence of the various objects on it is
suppressed from the notation. It is displayed only in \S\ref{sec:inverse}, where
two impedances are compared.

%% file: Part1_021_selfadjoint.tex
\label{subsec:selfadjoint}

Throughout this subsection $\Halpha$ denotes the half-space Robin Laplacian
\eqref{eq:Halpha-def} of \S\ref{sec:operator}. Its natural sesquilinear form is
\begin{equation}
  \label{eq:form}
  \hform[u,v]
  = (\nabla u, \nabla v)_{\Ltwo{\Rnp}}
    - \alpha\,(\gamma_0 u, \gamma_0 v)_{\Ltwo{\Gamma}},
  \qquad u,v \in H^1(\Rnp),
\end{equation}
and we write $\hform[u] := \hform[u,u]$. That \eqref{eq:form} is closed and
semibounded on $H^1(\Rnp)$, and that the associated self-adjoint operator is
exactly \eqref{eq:Halpha-def}, is part of the following statement.

\begin{theorem}[Self-adjointness and spectrum of $\Halpha$]
\label{thm:sa-spec}
Let $n \ge 2$ and $\alpha > 0$.
\begin{enumerate}
  \item[\textup{(a)}] $\Halpha$ is self-adjoint and bounded below. It is the
  self-adjoint operator associated with the closed, symmetric, semibounded form
  \eqref{eq:form} with form domain $H^1(\Rnp)$, and
  \begin{equation}
    \label{eq:lower-bound}
    \hform[u] \ge -\alpha^2 \,\norm{u}_{\Ltwo{\Rnp}}^2
    \qquad \bigl(u \in H^1(\Rnp)\bigr).
  \end{equation}
  \item[\textup{(b)}] The spectrum of $\Halpha$ is the half-line
  \begin{equation}
    \label{eq:spectrum}
    \sigma(\Halpha) = [-\alpha^2, \infty),
    \qquad \inf \sigma(\Halpha) = -\alpha^2,
  \end{equation}
  so that the bound \eqref{eq:lower-bound} is sharp.
  \item[\textup{(c)}] The spectrum is purely absolutely continuous:
  \begin{equation}
    \label{eq:ac}
    \sac(\Halpha) = [-\alpha^2, \infty),
    \qquad
    \spp(\Halpha) = \ssc(\Halpha) = \varnothing .
  \end{equation}
  In particular $\Halpha$ has no eigenvalues. None is embedded in the
  continuous spectrum, and the spectral bottom $-\alpha^2$ is not attained.
\end{enumerate}
\end{theorem}

\begin{remark}[Relation to the literature]
\label{rem:literature}
All three parts are known. Lotoreichik and Rohleder
\cite[Thm.~2.7]{LotoreichikRohleder2012} treat a variable coefficient
$\alpha \in W^{1,\infty}(\Gamma)$, of which our constant is a special case, and
\cite[Prop.~3.8]{BehrndtSchlosser2018} carries out the form approach for
$\alpha \in L^p(\Gamma) + L^\infty(\Gamma)$. The value of $\inf\sigma(\Halpha)$
is \cite[Lem.~2.8]{LevitinParnovski2008}, and the essential spectrum is stated
without proof in the introduction of \cite{LotoreichikRohleder2012}.

Part~(c) is contained in the work of Frank on a periodic impedance. He studies
the half-space Laplacian under the condition $\gamma_1 u + \sigma\gamma_0 u = 0$
with $\sigma$ periodic along $\Gamma$, our case being the constant
$\sigma = -\alpha$, and proves that the spectrum is then purely absolutely
continuous \cite[Thm.~1.4]{Frank2006}. A constant satisfies the hypothesis of
that theorem in every dimension. The surface channel appears there as the extra
band below the bulk threshold produced by \cite[Prop.~4.1]{Frank2006}, and the
case $n = 2$ is treated in \cite{Frank2003,FrankShterenberg2004}. The
scattering picture drawn there agrees with the one drawn in this paper: the wave
operators comparing the operator with the Neumann Laplacian exist but are not
complete, the defect being the surface states. (The unitary
equivalence of the absolutely continuous part of $\Halpha$ with the Neumann
Laplacian obtained in \cite[Cor.~3.11]{LotoreichikRohleder2012} is a different
statement. It requires a decay hypothesis on $\alpha$ that a nonzero constant
does not satisfy.)

We nevertheless give a complete and self-contained proof in
Appendix~\ref{app:sa-spec} rather than assemble the statement from these
sources. Frank obtains (c) by a Floquet decomposition, which a constant
impedance does not call for, and the specialisation is longer than a direct
argument. The proof given there also fixes the fiber objects
$\bigl(\hfib(\xipar), \psizero, \gphi_k\bigr)$, that is, the half-line Robin
operator, its bound state and its generalized eigenfunctions, in the explicit
normalisation that is used throughout the rest of the paper.
\end{remark}

The proof rests on the reduction of $\Halpha$ to a family of half-line
operators indexed by the tangential frequency. We record this reduction here,
since the objects it produces are used repeatedly in the sequel.

\begin{proposition}[Fiber decomposition]
\label{prop:fibration}
Let $\mathcal{F}'$ be the partial Fourier transform in the tangential variable
$\xpar$, as in \S\ref{sec:fourier}. Then $\mathcal{F}'$ is a unitary map of
$\Ltwo{\Rnp}$ onto $\int_{\Rnm}^{\oplus} \Ltwo{\Rhalf}\, d\xipar$. It
diagonalizes $\Halpha$ into the direct integral of half-line Robin operators
\begin{equation}
  \label{eq:fibration}
  \mathcal{F}'\,\Halpha\,(\mathcal{F}')^{-1}
  = \int_{\Rnm}^{\oplus} \hfib(\xipar)\, d\xipar,
  \qquad
  \hfib(\xipar) = -\frac{d^2}{dx_n^2} + |\xipar|^2 ,
\end{equation}
Each $\hfib(\xipar)$ is self-adjoint on $\Ltwo{\Rhalf}$ with domain
\begin{equation}
  \label{eq:fiber-domain}
  \Dom\bigl(\hfib(\xipar)\bigr)
  = \bigl\{\, v \in H^2(\Rhalf) \;\bigm|\; v'(0) + \alpha\, v(0) = 0 \,\bigr\} .
\end{equation}
For every $\xipar \in \Rnm$ the fiber spectrum consists of one simple eigenvalue
lying below a half-line of absolutely continuous spectrum,
\begin{equation}
  \label{eq:fiber-spectrum}
  \sigma\bigl(\hfib(\xipar)\bigr)
  = \bigl\{\, |\xipar|^2 - \alpha^2 \,\bigr\}
    \,\cup\, \bigl[\, |\xipar|^2, \infty \bigr),
\end{equation}
the eigenvalue $|\xipar|^2 - \alpha^2$ having the \emph{$\xipar$-independent}
normalized eigenfunction
\begin{equation}
  \label{eq:psi0}
  \psizero(x_n) = \sqrt{2\alpha}\; e^{-\alpha x_n} .
\end{equation}
\end{proposition}

We call $\lambda_{\mathrm{surf}}(\xipar) = |\xipar|^2 - \alpha^2$ the
\emph{surface band function}. It sweeps out $[-\alpha^2,\infty)$ as $\xipar$
runs over $\Rnm$, whereas the continuous part of each fiber begins at
$|\xipar|^2$, so that the fiber continua together fill $[0,\infty)$.

The two branches of \eqref{eq:fiber-spectrum} are the two scattering channels of
\S\ref{subsec:intro-two-channel}: the surface channel, of internal energy
$-\alpha^2$ and transverse profile $\psizero$, open for $E > -\alpha^2$, and the
bulk channel, of internal energy $0$, open for $E > 0$. We write
$\Thr = \{-\alpha^2, 0\}$ for the set of thresholds and
\begin{equation}
  \label{eq:Ereg}
  \Ereg = \sigma(\Halpha) \setminus \Thr = (-\alpha^2, 0) \cup (0, \infty)
\end{equation}
for the set of non-threshold energies. The limiting absorption principle, the
eigenfunction expansion and the scattering theory below are carried out on
$\Ereg$.

\begin{remark}[Thresholds and the change of spectral type]
\label{rem:thresholds}
The two thresholds occupy different positions in the spectrum. The surface
threshold $-\alpha^2$ is the bottom of $\sigma(\Halpha)$, whereas the bulk
threshold $0$ is embedded in the surface band. The surface channel does not
degenerate there, since $\nabla\lambda_{\mathrm{surf}}(\xipar) = 2\xipar \ne 0$
on the sphere $\{|\xipar| = \alpha\}$ that carries the energy $E = 0$. The rate
at which the resolvent bounds diverge at either threshold is identified in
\S\ref{subsec:lap}.

The spectral type also changes when the fibers are assembled. At fixed
$\xipar$ the value $\lambda_{\mathrm{surf}}(\xipar)$ is a simple eigenvalue of
$\hfib(\xipar)$, isolated below the fiber continuum, but the superposition over
$\xipar \in \Rnm$ turns the surface band into absolutely continuous spectrum,
so that $\Halpha$ has no eigenvalues (Theorem~\ref{thm:sa-spec}(c)). The
mechanism is made explicit in Remark~\ref{rem:no-eigenvalues-direct}.
\end{remark}

%% file: Part1_022_lap.tex
\label{subsec:lap}

The purpose of this subsection is to establish the limiting absorption principle
for $\Halpha$ in the Agmon--H\"ormander spaces $\Bsp$ and $\Bst$, at every energy
outside the threshold set $\Thr = \{-\alpha^2, 0\}$. The proof proceeds by an
exact decomposition of the resolvent, without cut-offs in the dual variable and
without any commutator estimate. The reasons for this choice, and the relation
to the Mourre theory, are discussed in Remark~\ref{rem:mourre} below and in
Appendix~\ref{app:mourre}. Since the spaces $\Bsp$ and $\Bst$ occur here for the
first time, we begin by fixing our conventions.

\subsubsection{The spaces \texorpdfstring{$\Bsp$, $\Bst$ and $\Bstz$}{B, B* and B*0}}
\label{sec:AH-spaces}

We first fix the dyadic decompositions used throughout. In $\Rn$ we put
\begin{equation}
  \label{eq:shells}
  \shell{0} = \{x \in \Rn : |x| < 1\},
  \qquad
  \shell{j} = \{x \in \Rn : 2^{j-1} \le |x| < 2^{j}\}
  \quad (j \ge 1),
\end{equation}
and we write $\shellp{j} = \shell{j} \cap \Rnp$. The same notation is used in
$\Rnm$, where the shells are denoted by
\begin{equation*}
  \tshell{0} = \{\xpar \in \Rnm : |\xpar| < 1\},
  \qquad
  \tshell{i} = \{\xpar \in \Rnm : 2^{i-1} \le |\xpar| < 2^{i}\}
  \quad (i \ge 1),
\end{equation*}
and in the normal variable, where we use the intervals
\begin{equation*}
  \slab{0} = (0,1),
  \qquad
  \slab{l} = [\,2^{l-1},\, 2^{l}) \subset \Rhalf
  \quad (l \ge 1),
\end{equation*}
so that the corresponding slab of $\Rnp$ is $\Rnm \times \slab{l}$.

\begin{definition}[Agmon--H\"ormander spaces]
\label{def:AH-spaces}
For $u \in L^2_{\mathrm{loc}}(\overline{\Rnp})$ we set
\begin{equation}
  \label{eq:B-norms}
  \Bnorm{u} = \sum_{j \ge 0} 2^{j/2}\,\norm{u}_{\Ltwo{\shellp{j}}},
  \qquad
  \Bstnorm{u} = \sup_{j \ge 0} 2^{-j/2}\,\norm{u}_{\Ltwo{\shellp{j}}},
\end{equation}
and we let $\Bsp$ and $\Bst$ be the spaces of those $u$ for which the respective
quantity is finite. We further set
\begin{equation}
  \label{eq:Bzero}
  \Bstz = \Bigl\{ u \in \Bst \;\Bigm|\;
    \lim_{j \to \infty} 2^{-j/2}\norm{u}_{\Ltwo{\shellp{j}}} = 0 \Bigr\}.
\end{equation}
The same definitions are used on $\Rn$ and on $\Rnm$. When the underlying space
has to be displayed we write $\Bspd{\Rnm}$, $\Bstd{\Rnm}$, and so on.
\end{definition}

Thus $\Bsp$ carries an $\ell^1$-type norm and $\Bst$ an $\ell^\infty$-type norm.
Neither is a Hilbert space. The norm of $\Bst$ is equivalent to the
scale-invariant expression
\begin{equation}
  \label{eq:Bstar-equivalent}
  \Bigl( \sup_{R > 1} \frac{1}{R}
    \int_{\{|x| < R\} \cap \Rnp} |u(x)|^2\,dx \Bigr)^{1/2},
\end{equation}
and $\Bstz$ is the subspace on which the quantity in \eqref{eq:Bstar-equivalent}
tends to $0$ as $R \to \infty$. It plays the role of $c_0$ inside $\ell^\infty$
and is the space in which the radiation condition of \S\ref{subsec:radiation}
will be formulated. With $L^{2,s}$ denoting the weighted space with norm
$\norm{(1+|x|)^{s}u}_{\Ltwo{\Rnp}}$ one has the standard chain of continuous
inclusions
\begin{equation}
  \label{eq:AH-inclusions}
  L^{2,s} \subset \Bsp \subset L^{2,1/2}
  \subset \Ltwo{\Rnp} \subset
  L^{2,-1/2} \subset \Bstz \subset \Bst \subset L^{2,-s}
  \qquad (s > 1/2),
\end{equation}
all of them strict. Under the pairing $(u,g) = \int_{\Rnp} u\overline{g}\,dx$ the
space $\Bst$ is the continuous dual of $\Bsp$ \cite{AgmonHormander1976}. We shall
need only the elementary half of this statement: for $u \in \Bst$ and
$g \in \Bsp$ the pairing converges absolutely, with
\begin{equation}
  \label{eq:AH-pairing}
  |(u,g)| \le \Bstnorm{u}\,\Bnorm{g}.
\end{equation}
Indeed, put $c_j = 2^{-j/2}\norm{u}_{\Ltwo{\shellp{j}}}$ and
$b_j = 2^{j/2}\norm{g}_{\Ltwo{\shellp{j}}}$. Then
$\sum_j \bigl| \int_{\shellp{j}} u\overline{g}\,dx \bigr|
\le \sum_j c_j b_j \le (\sup_j c_j)\sum_j b_j$. The inclusions
$\Bsp \subset \Ltwo{\Rnp} \subset \Bst$ hold, moreover, with norm at most one.
Indeed, writing $u_j = \norm{u}_{\Ltwo{\shellp{j}}}$, one has
$\norm{u}_{\Ltwo{\Rnp}} = \bigl( \sum_j u_j^2 \bigr)^{1/2}$, which is at most
$\sum_j 2^{j/2}u_j$ and at least $\sup_j 2^{-j/2}u_j$. Every operator bounded on
$\Ltwo{\Rnp}$ is therefore bounded from $\Bsp$ to $\Bst$ with the same bound.

Finally we record the version of the classical Agmon--H\"ormander resolvent
estimate that will be quoted. For $m \ge 1$ let
$\Rfree{m}{\zeta} = (-\Delta_{\mathbb{R}^m} - \zeta)^{-1}$.

\begin{theorem}[Agmon--H\"ormander \cite{AgmonHormander1976}]
\label{thm:AH-classical}
Let $m \ge 1$. There is a constant $\CAH{m}$ such that
\begin{equation}
  \label{eq:AH-classical}
  \bigl\| \Rfree{m}{\zeta}
  \bigr\|_{\Lop{\Bspd{\mathbb{R}^m}}{\Bstd{\mathbb{R}^m}}}
  \le \CAH{m}\,|\zeta|^{-1/2},
  \qquad \zeta \in \mathbb{C}\setminus\mathbb{R}.
\end{equation}
Moreover, for every $\zeta > 0$ and every $F \in \Bspd{\mathbb{R}^m}$ the limits
$\Rfree{m}{\zeta \pm i0}F = \lim_{\eps \downarrow 0}\Rfree{m}{\zeta \pm i\eps}F$
exist in the weak-$*$ topology of $\Bstd{\mathbb{R}^m}$, they obey
\eqref{eq:AH-classical}, and for all $F,G \in \Bspd{\mathbb{R}^m}$ the maps
$\zeta \mapsto (\Rfree{m}{\zeta \pm i0}F, G)$ are continuous on
$(0,\infty)$.
\end{theorem}

\begin{remark}[The case $m = 1$]
\label{rem:AH-exponent}
For $m = 1$ the estimate \eqref{eq:AH-classical} is elementary and does not
require the theory of \cite{AgmonHormander1976}: the kernel of
$\Rfree{1}{\zeta}$ is $(i/2\zeta^{1/2})e^{i\zeta^{1/2}|t-t'|}$, and one has
$\norm{F}_{L^1(\mathbb{R})} \le \sqrt2\,\norm{F}_{\Bspd{\mathbb{R}}}$ together
with $\norm{u}_{\Bstd{\mathbb{R}}} \le \sqrt2\,\norm{u}_{L^\infty(\mathbb{R})}$.
This makes Theorem~\ref{thm:lap} below available when $n = 2$: the
proof uses Theorem~\ref{thm:AH-classical} in dimensions $n$ and $n-1$ only, and
both are available for every $m \ge 1$.
\end{remark}

\subsubsection{The main result}
\label{sec:lap-theorem}

\begin{theorem}[Limiting absorption principle]
\label{thm:lap}
Let $n \ge 2$ and $\alpha > 0$, and recall
$\Ereg = \sigma(\Halpha)\setminus\Thr = (-\alpha^2,0) \cup (0,\infty)$ from
\eqref{eq:Ereg}.
\begin{enumerate}
  \item[\textup{(a)}] \emph{(Uniform bound.)} For every compact set
  $\mathcal{K} \subset \Ereg$ there is a constant
  $C = C(n,\alpha,\mathcal{K})$ such that
  \begin{equation}
    \label{eq:lap-bound}
    \Bstnorm{\Ralpha{z}f} \le C\,\Bnorm{f},
    \qquad f \in \Bsp,
  \end{equation}
  for all $z \in \mathbb{C}\setminus\mathbb{R}$ with $\re z \in \mathcal{K}$.
  \item[\textup{(b)}] \emph{(Boundary values.)} For every $E \in \Ereg$ and every
  $f \in \Bsp$ the limits
  \begin{equation*}
    \Ralpha{E \pm i0}f := \lim_{\eps \downarrow 0} \Ralpha{E \pm i\eps}f
  \end{equation*}
  exist in the weak-$*$ topology of $\Bst$, that is,
  $\lim_{\eps \downarrow 0}(\Ralpha{E \pm i\eps}f,g)$ exists for every
  $g \in \Bsp$. The operators so defined belong to $\Lop{\Bsp}{\Bst}$ and satisfy
  \eqref{eq:lap-bound} with the same constant.
  \item[\textup{(c)}] \emph{(Continuity.)} For all $f,g \in \Bsp$ the maps
  $\Ereg \ni E \mapsto (\Ralpha{E \pm i0}f,g)$ are continuous. Equivalently,
  $z \mapsto \Ralpha{z}$ is continuous in the weak-$*$ operator topology on each
  of the sets $\{z : \re z \in \Ereg,\ \pm\im z \ge 0\}$.
  \item[\textup{(d)}] \emph{(Adjoints and positivity.)}
  $\Ralpha{E \mp i0} = \Ralpha{E \pm i0}^{*}$ in the sense that
  $(\Ralpha{E \pm i0}f,g) = (f,\Ralpha{E \mp i0}g)$ for $f,g \in \Bsp$, and
  \begin{equation*}
    \im\,\bigl(\Ralpha{E+i0}f,f\bigr) \ge 0,
    \qquad f \in \Bsp .
  \end{equation*}
  \item[\textup{(e)}] \emph{(The equation.)} For $f \in \Bsp$ and $E \in \Ereg$
  the function $u = \Ralpha{E \pm i0}f$ belongs to
  $\Bst \cap H^2_{\mathrm{loc}}(\overline{\Rnp})$ and solves
  \begin{equation*}
    (-\Delta - E)u = f \ \text{ in } \Rnp,
    \qquad
    \gamma_1 u - \alpha\gamma_0 u = 0 \ \text{ on } \Gamma .
  \end{equation*}
\end{enumerate}
\end{theorem}

The proof occupies \S\S\ref{sec:dyadic-lemmas}--\ref{sec:lap-proof}, and its
whole content is the following exact decomposition, established in
Proposition~\ref{prop:splitting}: for every $z \in \mathbb{C}\setminus\mathbb{R}$,
\begin{equation*}
  \Ralpha{z} = \RD{z}
    + \Rfree{n-1}{z+\alpha^2} \otimes \Psurf
    + \Krem{z},
\end{equation*}
where $\RD{z}$ is the resolvent of the Dirichlet Laplacian on $\Rnp$, $\Psurf$
is the orthogonal projection of $\Ltwo{\Rhalf}$ onto the surface mode
$\mathbb{C}\psizero$, and $\Krem{z}$ is an explicit remainder. There is neither a
cut-off nor an approximation here: the identity is an algebraic one, between
fiber Green kernels. Its three terms are then treated
separately. The Dirichlet term is reduced, by the odd reflection across
$\Gamma$, to the free Laplacian on $\Rn$. The surface term is, by
Proposition~\ref{prop:surface-reduction}, a tensor product, and it
reduces to the free Laplacian on $\Rnm$ at the shifted energy $z + \alpha^2$.
The remainder is estimated directly, its kernel being bounded by $1/\alpha$
uniformly in all of its arguments and in $z$ (Lemma~\ref{lem:K-kernel-bound}).

The decomposition also has a structural consequence. Since the remainder is bounded uniformly on all of
$\mathbb{C}\setminus\mathbb{R}$, in particular uniformly up to and across both
thresholds, the entire dependence of the resolvent estimate on the spectral
parameter is carried by two free resolvents, in two different dimensions:
\begin{equation}
  \label{eq:two-channels-preview}
  \Rfree{n}{z}
  \quad \text{with its threshold at } 0,
  \qquad\text{and}\qquad
  \Rfree{n-1}{z+\alpha^2}
  \quad \text{with its threshold at } -\alpha^2 .
\end{equation}
The two thresholds of $\Halpha$ are thus the thresholds of two distinct free
problems, of dimensions $n$ and $n-1$, and they do not interact: their singular
loci in the dual variable are the spheres $|\xipar|^2 = \re z$ and
$|\xipar|^2 = \re z + \alpha^2$, which are disjoint whenever $\alpha > 0$.
This separation of the channels is used again in
\S\S\ref{subsec:gft}--\ref{subsec:radiation}: in the generalized Fourier
transform, in the block structure of the scattering matrix, and in the
radiation condition.

\begin{remark}[Comparison with the Mourre theory]
\label{rem:mourre}
The limiting absorption principle is often derived from a positive commutator
estimate \cite{Mourre1981,ABG1996}, and this method applies here: the generator
of dilations is a conjugate operator $A$ for $\Halpha$, with
$i[\Halpha,A] = 2\Halpha + \alpha\,\gamma_0^{*}\gamma_0$ in the form sense
(Appendix~\ref{app:mourre}). Its conclusion, however, is a bound on
$\langle x\rangle^{-s}\Ralpha{z}\langle x\rangle^{-s}$ for $s > 1/2$ only,
whereas Theorem~\ref{thm:lap} holds at the endpoint $\Bsp \to \Bst$ of that
scale \eqref{eq:AH-inclusions}, which interpolation does not reach. Moreover,
the sections that follow use the kernel of $\Ralpha{E \pm i0}$ supplied by
\eqref{eq:splitting-preview}, not merely its norm. The behaviour of the two
bounds at the thresholds is compared in Appendix~\ref{app:mourre}.
\end{remark}

\subsubsection{Two elementary lemmas on the dyadic decomposition}
\label{sec:dyadic-lemmas}

Both lemmas of this subsection compare the isotropic shells $\shellp{j}$ of
$\Rnp$ with the product decomposition $\tshell{i} \times \slab{l}$. The only
geometric input is the elementary pair of inclusions
\begin{equation}
  \label{eq:geometry}
  \tshell{i} \times \slab{l} \subset \shellp{m} \cup \shellp{m+1}
  \quad \text{with } m = \max(i,l),
  \qquad
  \shellp{j} \subset \{|\xpar| < 2^{j}\} \times \Rhalf,
\end{equation}
which follow from
$\max(|\xpar|,x_n) \le |x| \le \sqrt2\,\max(|\xpar|,x_n)$.

\begin{lemma}[Slab estimate]
\label{lem:slab}
For every $u \in \Bsp$,
\begin{equation}
  \label{eq:slab-1}
  \sum_{l \ge 0} 2^{l/2}\,\norm{u}_{\Ltwo{\Rnm \times \slab{l}}}
  \le \sqrt2\,\Bnorm{u},
\end{equation}
and consequently
\begin{equation}
  \label{eq:slab-2}
  \int_0^{\infty} \norm{u(\cdot,x_n)}_{\Ltwo{\Rnm}}\,dx_n
  \le \sqrt2\,\Bnorm{u}.
\end{equation}
\end{lemma}

\begin{proof}
On $\Rnm \times \slab{l}$ one has $|x| \ge x_n \ge 2^{l-1}$, hence
$(\Rnm \times \slab{l}) \cap \shellp{j} = \varnothing$ for $j < l$. Since the
slabs are pairwise disjoint, the Cauchy--Schwarz inequality in the variable $l$
gives
\begin{align*}
  \sum_{l \ge 0} 2^{l/2} \norm{u}_{\Ltwo{\Rnm \times \slab{l}}}
  &\le \sum_{l \ge 0} 2^{l/2} \sum_{j \ge l}
       \norm{u}_{\Ltwo{(\Rnm \times \slab{l}) \cap \shellp{j}}}
   = \sum_{j \ge 0} \sum_{l \le j} 2^{l/2}
       \norm{u}_{\Ltwo{(\Rnm \times \slab{l}) \cap \shellp{j}}}
  \\
  &\le \sum_{j \ge 0} \Bigl( \sum_{l \le j} 2^{l} \Bigr)^{1/2}
       \Bigl( \sum_{l \le j}
         \norm{u}_{\Ltwo{(\Rnm \times \slab{l}) \cap \shellp{j}}}^2
       \Bigr)^{1/2}
  \\
  &\le \sum_{j \ge 0} 2^{(j+1)/2}\,\norm{u}_{\Ltwo{\shellp{j}}}
   = \sqrt2\,\Bnorm{u},
\end{align*}
which is \eqref{eq:slab-1}. For \eqref{eq:slab-2} apply the Cauchy--Schwarz
inequality on each slab,
\begin{equation*}
  \int_{\slab{l}} \norm{u(\cdot,x_n)}_{\Ltwo{\Rnm}}\,dx_n
  \le |\slab{l}|^{1/2}\,\norm{u}_{\Ltwo{\Rnm \times \slab{l}}}
  \le 2^{l/2}\,\norm{u}_{\Ltwo{\Rnm \times \slab{l}}},
\end{equation*}
$|\slab{l}|$ being the length of the interval $\slab{l}$, and sum over $l$.
\end{proof}

The second lemma compares the isotropic norms of $\Bsp$ and $\Bst$ with mixed
norms in which the tangential and the normal variable are separated. It gives
access to the tensor structure of the surface term.

\begin{lemma}[Mixed-norm embeddings]
\label{lem:mixed}
Let $n \ge 2$ and put $\Csharp = (1+\sqrt2)^2$. Then
\begin{equation}
  \label{eq:mixed-B}
  \Bigl( \int_0^{\infty}
    \norm{f(\cdot,x_n)}_{\Bspd{\Rnm}}^2\,dx_n \Bigr)^{1/2}
  \le \Csharp\,\Bnorm{f},
  \qquad f \in \Bsp,
\end{equation}
and
\begin{equation}
  \label{eq:mixed-Bstar}
  \Bstnorm{u}
  \le \sqrt2 \Bigl( \int_0^{\infty}
    \norm{u(\cdot,x_n)}_{\Bstd{\Rnm}}^2\,dx_n \Bigr)^{1/2}
\end{equation}
for every measurable $u$ on $\Rnp$ for which the right-hand side is finite.
\end{lemma}

\begin{proof}
We begin with \eqref{eq:mixed-Bstar}, which is the shorter of the two and uses
only the second inclusion in \eqref{eq:geometry}. By that inclusion and Fubini's
theorem,
\begin{equation*}
  \norm{u}_{\Ltwo{\shellp{j}}}^2
  \le \int_0^{\infty}
      \norm{u(\cdot,x_n)}^2_{\Ltwo{\{|\xpar| < 2^{j}\}}}\,dx_n .
\end{equation*}
Since $\{|\xpar| < 2^{j}\}$ is the disjoint union of
$\tshell{0},\dots,\tshell{j}$, and since
$\norm{v}_{\Ltwo{\tshell{i}}} \le 2^{i/2}\norm{v}_{\Bstd{\Rnm}}$ for every $i$,
\begin{equation*}
\begin{aligned}
  \norm{u(\cdot,x_n)}^2_{\Ltwo{\{|\xpar| < 2^{j}\}}}
  &= \sum_{i \le j} \norm{u(\cdot,x_n)}^2_{\Ltwo{\tshell{i}}}
   \le \Bigl(\sum_{i \le j} 2^{i}\Bigr)
       \norm{u(\cdot,x_n)}^2_{\Bstd{\Rnm}}
  \\
  &< 2^{j+1}\,\norm{u(\cdot,x_n)}^2_{\Bstd{\Rnm}} .
\end{aligned}
\end{equation*}
Multiplying by $2^{-j}$ and taking the supremum over $j$ gives
\eqref{eq:mixed-Bstar}.

We turn to \eqref{eq:mixed-B}, and abbreviate
\begin{equation*}
  G_m = \norm{f}_{\Ltwo{\shellp{m}}} \quad (m \ge 0),
  \qquad
  a_{i,l} = \norm{f}_{\Ltwo{\tshell{i} \times \slab{l}}}
  \quad (i,l \ge 0),
\end{equation*}
so that $\Bnorm{f} = \sum_{m \ge 0} 2^{m/2}G_m$. We shall use twice the
elementary consequence
\begin{equation}
  \label{eq:shifted-sum}
  \sum_{m \ge 0} 2^{m/2}\bigl(G_m + G_{m+1}\bigr)
  \le \Bnorm{f} + \frac{1}{\sqrt2}\sum_{m \ge 0} 2^{(m+1)/2}G_{m+1}
  \le \Bigl(1 + \frac{1}{\sqrt2}\Bigr)\Bnorm{f}.
\end{equation}
The proof of \eqref{eq:mixed-B} is in four steps.

\emph{Step 1: reduction to a single sum over the tangential shells.} Put
$h_i(x_n) = 2^{i/2}\norm{f(\cdot,x_n)}_{\Ltwo{\tshell{i}}}$, so that
$\norm{f(\cdot,x_n)}_{\Bspd{\Rnm}} = \sum_{i \ge 0} h_i(x_n)$ by the definition
\eqref{eq:B-norms} of the $\Bspd{\Rnm}$-norm. The triangle inequality in
$L^2(\Rhalf, dx_n)$, applied to this sum, and then Fubini's theorem give
\begin{align*}
  \Bigl( \int_0^{\infty}
    \norm{f(\cdot,x_n)}_{\Bspd{\Rnm}}^2\,dx_n \Bigr)^{1/2}
  &= \Bigl\| \sum_{i \ge 0} h_i \Bigr\|_{L^2(\Rhalf)}
   \le \sum_{i \ge 0} \norm{h_i}_{L^2(\Rhalf)}
  \\
  &= \sum_{i \ge 0} 2^{i/2}
     \Bigl( \int_0^{\infty}
       \norm{f(\cdot,x_n)}^2_{\Ltwo{\tshell{i}}}\,dx_n \Bigr)^{1/2}
  \\
  &= \sum_{i \ge 0} 2^{i/2}\,
     \norm{f}_{\Ltwo{\tshell{i} \times \Rhalf}} .
\end{align*}
It therefore suffices to bound the last sum by $\Csharp\,\Bnorm{f}$.

\emph{Step 2: splitting one tangential shell along the slabs.} Fix $i \ge 0$.
The slabs decompose $\tshell{i} \times \Rhalf$ into the disjoint union of the
sets $\tshell{i} \times \slab{l}$, $l \ge 0$, so that
$\norm{f}^2_{\Ltwo{\tshell{i} \times \Rhalf}} = \sum_{l \ge 0} a_{i,l}^2$. We
split this sum at $l = i$, which is where the value of $\max(i,l)$ in
\eqref{eq:geometry} changes, and use $\sqrt{A^2+B^2} \le A + B$:
\begin{equation*}
  \norm{f}_{\Ltwo{\tshell{i} \times \Rhalf}} \le P_i + Q_i,
  \qquad
  P_i = \Bigl(\sum_{l \le i} a_{i,l}^2\Bigr)^{1/2},
  \qquad
  Q_i = \Bigl(\sum_{l > i} a_{i,l}^2\Bigr)^{1/2} .
\end{equation*}
Steps 3 and 4 bound $\sum_{i \ge 0} 2^{i/2}P_i$ and $\sum_{i \ge 0} 2^{i/2}Q_i$
separately.

\emph{Step 3: the slabs with $l \le i$.} Here $\max(i,l) = i$, so by
\eqref{eq:geometry} the sets $\tshell{i} \times \slab{l}$ with $0 \le l \le i$
are pairwise disjoint subsets of $\shellp{i} \cup \shellp{i+1}$. Their squared
norms therefore add up to at most $\norm{f}^2_{\Ltwo{\shellp{i} \cup
\shellp{i+1}}} = G_i^2 + G_{i+1}^2$, whence
\begin{equation*}
  P_i \le \bigl(G_i^2 + G_{i+1}^2\bigr)^{1/2} \le G_i + G_{i+1},
  \qquad\text{and so}\qquad
  \sum_{i \ge 0} 2^{i/2}P_i
  \le \Bigl(1 + \frac{1}{\sqrt2}\Bigr)\Bnorm{f}
\end{equation*}
by \eqref{eq:shifted-sum}.

\emph{Step 4: the slabs with $l > i$.} Here $\max(i,l) = l$, so
$\tshell{i} \times \slab{l} \subset \shellp{l} \cup \shellp{l+1}$ and hence
$a_{i,l} \le G_l + G_{l+1}$, a bound independent of $i$. Using
$(\sum_k c_k^2)^{1/2} \le \sum_k c_k$ we obtain
$Q_i \le \sum_{l > i} a_{i,l} \le \sum_{l > i}(G_l + G_{l+1})$. All terms being
nonnegative, we may now sum over $i$ and exchange the order of summation. The
pair $(i,l)$ occurs exactly when $0 \le i < l$, so that fixing $l$ leaves the
geometric sum $\sum_{i=0}^{l-1}2^{i/2}
= (2^{l/2}-1)/(\sqrt2-1) \le 2^{l/2}/(\sqrt2-1)$, and therefore
\begin{equation*}
  \sum_{i \ge 0} 2^{i/2}Q_i
  \le \sum_{l \ge 1} \Bigl( \sum_{i=0}^{l-1} 2^{i/2} \Bigr)(G_l + G_{l+1})
  \le \frac{1}{\sqrt2-1} \sum_{l \ge 0} 2^{l/2}(G_l + G_{l+1})
  \le \frac{1}{\sqrt2-1}\Bigl(1 + \frac{1}{\sqrt2}\Bigr)\Bnorm{f},
\end{equation*}
again by \eqref{eq:shifted-sum}.

\emph{Conclusion.} Adding the bounds of Steps 3 and 4 and inserting them into
Steps 2 and 1,
\begin{equation*}
  \Bigl( \int_0^{\infty}
    \norm{f(\cdot,x_n)}_{\Bspd{\Rnm}}^2\,dx_n \Bigr)^{1/2}
  \le \Bigl(1 + \frac{1}{\sqrt2}\Bigr)
      \Bigl(1 + \frac{1}{\sqrt2-1}\Bigr)\Bnorm{f}
  = (1+\sqrt2)^2\,\Bnorm{f},
\end{equation*}
since $1 + (\sqrt2-1)^{-1} = 2 + \sqrt2$ and
$\bigl(1+2^{-1/2}\bigr)\bigl(2+\sqrt2\bigr) = 3 + 2\sqrt2 = (1+\sqrt2)^2$.
\end{proof}

\begin{corollary}[Tensor lemmas]
\label{cor:tensor}
Let $\Psi \in \Ltwo{\Rhalf}$ and define
\begin{equation*}
  (T_\Psi f)(\xpar) = \int_0^{\infty} \Psi(y_n)\,f(\xpar,y_n)\,dy_n,
  \qquad
  (M_\Psi w)(\xpar,x_n) = \Psi(x_n)\,w(\xpar).
\end{equation*}
Then
\begin{equation}
  \label{eq:tensor-bounds}
  \norm{T_\Psi}_{\Lop{\Bsp}{\Bspd{\Rnm}}}
    \le \Csharp\,\norm{\Psi}_{\Ltwo{\Rhalf}},
  \qquad
  \norm{M_\Psi}_{\Lop{\Bstd{\Rnm}}{\Bst}}
    \le \sqrt2\,\norm{\Psi}_{\Ltwo{\Rhalf}} .
\end{equation}
\end{corollary}

\begin{proof}
For $T_\Psi$, apply Minkowski's integral inequality in the $\Bspd{\Rnm}$-norm,
then the Cauchy--Schwarz inequality in $x_n$ followed by \eqref{eq:mixed-B}:
\begin{equation*}
  \norm{T_\Psi f}_{\Bspd{\Rnm}}
  \le \int_0^{\infty} |\Psi(y_n)|\,
      \norm{f(\cdot,y_n)}_{\Bspd{\Rnm}}\,dy_n
  \le \norm{\Psi}_{\Ltwo{\Rhalf}}\,\Csharp\,\Bnorm{f}.
\end{equation*}
For $M_\Psi$ the statement is \eqref{eq:mixed-Bstar} applied to $u = M_\Psi w$,
whose inner norm is
$\norm{u(\cdot,x_n)}_{\Bstd{\Rnm}} = |\Psi(x_n)|\,\norm{w}_{\Bstd{\Rnm}}$.
\end{proof}

\subsubsection{The surface channel and the exact splitting}
\label{sec:splitting}

This subsection produces the decomposition announced in \S\ref{sec:lap-theorem}.
It has two ingredients. The first is that the surface mode $\psizero$ spans a
reducing subspace of $\Halpha$, on which the operator is a free Laplacian on
$\Rnm$, shifted by $-\alpha^2$. This is Proposition~\ref{prop:surface-reduction}.
The second is an algebraic identity satisfied by the Green kernel of the fiber
operator $\hfib(\xipar)$, which separates that surface contribution from a
Dirichlet part and a remainder. This is Lemma~\ref{lem:fiber-green}. Assembling
the fibers then gives the resolvent identity, Proposition~\ref{prop:splitting}.
Throughout we use the identification
$\Ltwo{\Rnp} = \Ltwo{\Rnm} \otimes \Ltwo{\Rhalf}$.

\begin{proposition}[Reduction of the surface channel]
\label{prop:surface-reduction}
Let $\Psurf = (\cdot,\psizero)_{\Ltwo{\Rhalf}}\,\psizero$ denote the orthogonal
projection of $\Ltwo{\Rhalf}$ onto $\mathbb{C}\psizero$, and put
\begin{equation*}
  \Hsurf = \Ltwo{\Rnm} \otimes \mathbb{C}\psizero,
  \qquad
  \Hbulk = \Hsurf^{\perp},
\end{equation*}
so that $I \otimes \Psurf$ is the orthogonal projection of $\Ltwo{\Rnp}$ onto
$\Hsurf$. Let $U : \Ltwo{\Rnm} \to \Hsurf$ be the unitary map
$Ug = g \otimes \psizero$. Then both $\Hsurf$ and $\Hbulk$ reduce $\Halpha$, the
map $U$ intertwines
\begin{equation}
  \label{eq:surface-reduction}
  \Halpha\,U = U\,\bigl( -\Delta' - \alpha^2 \bigr)
  \qquad\text{on } H^2(\Rnm),
\end{equation}
and consequently, for every $z \in \mathbb{C}\setminus\mathbb{R}$,
\begin{equation}
  \label{eq:surface-resolvent}
  \Ralpha{z}\,\bigl( I \otimes \Psurf \bigr)
  = \Rfree{n-1}{z + \alpha^2} \otimes \Psurf .
\end{equation}
\end{proposition}

\begin{proof}
For $g \in H^2(\Rnm)$ the function $u = Ug$ lies in $H^2(\Rnp)$ and satisfies
$\psizero'(0) = -\alpha\,\psizero(0)$, hence
$\gamma_1 u - \alpha\gamma_0 u = 0$ and $u \in \Dom(\Halpha)$. Moreover
$\Halpha u = \bigl( (-\Delta' - \alpha^2)g \bigr) \otimes \psizero$, because
$\psizero'' = \alpha^2\psizero$. This is \eqref{eq:surface-reduction}, and it
shows that $\Hsurf$ is invariant. Since $\Halpha$ is self-adjoint
(Theorem~\ref{thm:sa-spec}(a)), so is $\Hbulk = \Hsurf^{\perp}$. The resolvent
therefore commutes with $I \otimes \Psurf$ and acts on $\Hsurf$ as
$U(-\Delta' - \alpha^2 - z)^{-1}U^{*}$, which is
$\Rfree{n-1}{z+\alpha^2} \otimes \Psurf$. The latter operator vanishes on
$\Hbulk$, so \eqref{eq:surface-resolvent} holds on all of $\Ltwo{\Rnp}$.
\end{proof}

We turn to the fiber operator. The next lemma is a computation in one variable:
$\xipar$ enters only through $\kappa$, and no operator on $\Rnp$ occurs in it.

\begin{lemma}[The fiber Green kernel and its splitting]
\label{lem:fiber-green}
Let $z \in \mathbb{C}\setminus\mathbb{R}$ and set
\begin{equation*}
  \kappa = \kappa(\xipar,z) = \bigl( |\xipar|^2 - z \bigr)^{1/2},
  \qquad \re \kappa > 0 .
\end{equation*}
Then the Green kernel of $\hfib(\xipar) - z$ on $\Rhalf$ is
\begin{equation}
  \label{eq:fiber-green}
  g(x_n,y_n;\xipar,z)
  = \frac{1}{2\kappa}
    \Bigl( e^{-\kappa|x_n - y_n|} + r(\kappa)\,e^{-\kappa(x_n+y_n)} \Bigr),
  \qquad
  r(\kappa) = \frac{\kappa + \alpha}{\kappa - \alpha},
\end{equation}
and, writing $s = x_n + y_n$, it admits the exact decomposition
\begin{equation}
  \label{eq:fiber-splitting}
  g = \frac{1}{2\kappa}
        \Bigl( e^{-\kappa|x_n-y_n|} - e^{-\kappa s} \Bigr)
    + \frac{\psizero(x_n)\psizero(y_n)}{|\xipar|^2 - (z+\alpha^2)}
    + k(\xipar,x_n,y_n;z),
\end{equation}
where
\begin{equation}
  \label{eq:K-kernel}
  k(\xipar,x_n,y_n;z)
  = \frac{(\kappa+\alpha)e^{-\kappa s} - 2\alpha e^{-\alpha s}}
         {|\xipar|^2 - (z+\alpha^2)} .
\end{equation}
\end{lemma}

\begin{proof}
A decaying solution of $(-d^2/dx_n^2 + \kappa^2)v = 0$ on $\Rhalf$ is
$\psi(x_n) = e^{-\kappa x_n}$, and
$\varphi(x_n) = e^{\kappa x_n} + r\,e^{-\kappa x_n}$ satisfies
$\varphi'(0) = -\alpha\varphi(0)$ precisely when
$\kappa(1 - r) = -\alpha(1+r)$, that is $r = (\kappa+\alpha)/(\kappa-\alpha)$.
Their Wronskian is $W = \varphi\psi' - \varphi'\psi = -2\kappa$, and the Green
kernel of $-d^2/dx_n^2 + \kappa^2$ under the Robin condition at $x_n = 0$ is
\begin{equation*}
  g(x_n,y_n;\xipar,z)
  = -\frac{1}{W}
  \begin{cases}
    \varphi(x_n)\,\psi(y_n), & x_n < y_n, \\[1mm]
    \varphi(y_n)\,\psi(x_n), & x_n > y_n,
  \end{cases}
\end{equation*}
which is \eqref{eq:fiber-green}. (The sign is fixed by the free case: for
$r = 0$ one recovers $(2\kappa)^{-1}e^{-\kappa|x_n-y_n|}$, the kernel of
$(-d^2/dx_n^2 + \kappa^2)^{-1}$ on $\mathbb{R}$.)

For \eqref{eq:fiber-splitting}, insert
$r(\kappa) = -1 + 2\kappa/(\kappa-\alpha)$ into \eqref{eq:fiber-green}:
\begin{equation}
  \label{eq:g-minus-gD}
  g = \frac{1}{2\kappa}
      \Bigl( e^{-\kappa|x_n-y_n|} - e^{-\kappa s} \Bigr)
    + \frac{e^{-\kappa s}}{\kappa - \alpha} .
\end{equation}
Since $\kappa^2 - \alpha^2 = |\xipar|^2 - (z+\alpha^2)$ and
$\psizero(x_n)\psizero(y_n) = 2\alpha e^{-\alpha s}$, the last term of
\eqref{eq:g-minus-gD} equals
\begin{equation*}
  \frac{e^{-\kappa s}}{\kappa-\alpha}
  = \frac{(\kappa+\alpha)e^{-\kappa s}}{\kappa^2 - \alpha^2}
  = \frac{2\alpha e^{-\alpha s}}{|\xipar|^2 - (z+\alpha^2)}
    + k(\xipar,x_n,y_n;z)
\end{equation*}
with $k$ as in \eqref{eq:K-kernel}. Substituting this into
\eqref{eq:g-minus-gD} gives \eqref{eq:fiber-splitting}.
\end{proof}

Each of the three terms of \eqref{eq:fiber-splitting} is the fiber kernel of an
operator on $\Ltwo{\Rnp}$, and assembling the fibers turns the lemma into the
identity announced in \S\ref{sec:lap-theorem}.

\begin{proposition}[Exact splitting of the resolvent]
\label{prop:splitting}
Let $\Krem{z}$ denote the operator on $\Ltwo{\Rnp}$ which, after partial Fourier
transform in $\xpar$, acts in each fiber as the integral operator on
$\Ltwo{\Rhalf}$ with kernel \eqref{eq:K-kernel}. Then, for every
$z \in \mathbb{C}\setminus\mathbb{R}$,
\begin{equation}
  \label{eq:splitting-preview}
  \Ralpha{z} = \RD{z}
    + \Rfree{n-1}{z+\alpha^2} \otimes \Psurf
    + \Krem{z},
\end{equation}
where $\RD{z}$ is the resolvent of the Dirichlet Laplacian on $\Rnp$.
\end{proposition}

\begin{proof}
By Proposition~\ref{prop:fibration} the partial Fourier transform $\mathcal{F}'$
carries $\Ralpha{z}$ into the direct integral of the fiber resolvents
$\bigl(\hfib(\xipar) - z\bigr)^{-1}$, whose integral kernels are the functions
$g(\cdot,\cdot;\xipar,z)$ of \eqref{eq:fiber-green}. It therefore suffices to
identify the three terms of \eqref{eq:fiber-splitting} fiberwise.

The Dirichlet Laplacian on $\Rnp$ fibers in the same way, its fiber being
$-d^2/dx_n^2 + |\xipar|^2$ with the boundary condition $v(0) = 0$, and the
corresponding Green kernel is obtained from the free one by the odd image:
it is $(2\kappa)^{-1}\bigl(e^{-\kappa|x_n-y_n|} - e^{-\kappa s}\bigr)$, the first
term of \eqref{eq:fiber-splitting}. The second term is the fiber kernel of
$\Rfree{n-1}{z+\alpha^2} \otimes \Psurf$, since $\Psurf$ has kernel
$\psizero(x_n)\psizero(y_n)$ while $\Rfree{n-1}{z+\alpha^2}$ becomes,
after partial Fourier transform, multiplication by
$\bigl(|\xipar|^2 - (z+\alpha^2)\bigr)^{-1}$. This is
\eqref{eq:surface-resolvent} read fiberwise. The third term is the fiber kernel
of $\Krem{z}$ by definition.
\end{proof}

\begin{remark}[Krein's formula and the choice of comparison operator]
\label{rem:krein}
Identity \eqref{eq:g-minus-gD} is the half-space instance of the Krein-type
resolvent formula for the pair (Dirichlet, Robin),
\begin{equation*}
  \Ralpha{z} - \RD{z}
  = \mathcal{P}_z\,(\DN_z - \alpha)^{-1}\,\mathcal{P}_z^{\top},
\end{equation*}
where $\mathcal{P}_z$ is the Poisson operator of $-\Delta - z$ and $\DN_z$ the
Dirichlet-to-Neumann operator, whose symbol on the flat boundary is
$\kappa(\xipar,z)$ (see \cite{BehrndtLanger2007,LotoreichikRohleder2012}). The
comparison with the Dirichlet rather than the Neumann realization keeps the
weight bounded at the bulk threshold $\kappa \to 0$: it is
$(\kappa - \alpha)^{-1}$ here, against $\kappa^{-1}(\kappa-\alpha)^{-1}$ for the
Neumann realization, which is as singular as the free resolvent itself.
\end{remark}

\subsubsection{Estimates for the three terms}
\label{sec:three-terms}

\smallskip
\noindent\textbf{The Dirichlet term.}
For $f$ on $\Rnp$ let $f_{\mathrm{odd}}$ denote its extension to $\Rn$ that is
odd in $x_n$. Then
$\RD{z}f = \bigl( \Rfree{n}{z}f_{\mathrm{odd}} \bigr)\big|_{\Rnp}$, and since
$\norm{f_{\mathrm{odd}}}_{\Bspd{\Rn}} = \sqrt2\,\Bnorm{f}$ while restriction to
$\Rnp$ does not increase the $\Bst$-norm, Theorem~\ref{thm:AH-classical} gives
at once
\begin{equation}
  \label{eq:dirichlet-bound}
  \norm{\RD{z}}_{\Lop{\Bsp}{\Bst}} \le \sqrt2\,\CAH{n}\,|z|^{-1/2},
  \qquad z \in \mathbb{C}\setminus\mathbb{R},
\end{equation}
together with the existence of the weak-$*$ boundary values $\RD{E \pm i0}$ for
$E > 0$ and their continuity in $E$. For $E < 0$ the point $z = E$ lies in the
resolvent set of the Dirichlet Laplacian and $\RD{\cdot}$ is analytic near it.

\smallskip
\noindent\textbf{The surface term.}
By Proposition~\ref{prop:surface-reduction} and Corollary~\ref{cor:tensor}
applied with $\Psi = \psizero$, first to reduce $f$ to the boundary and then to
lift the result back,
\begin{align}
  \label{eq:surface-bound}
  \bigl\| \Rfree{n-1}{z+\alpha^2} \otimes \Psurf
    \bigr\|_{\Lop{\Bsp}{\Bst}}
  &\le \sqrt2\,\Csharp
       \bigl\| \Rfree{n-1}{z+\alpha^2}
       \bigr\|_{\Lop{\Bspd{\Rnm}}{\Bstd{\Rnm}}}
  \notag \\
  &\le \sqrt2\,\Csharp\,\CAH{n-1}\,|z+\alpha^2|^{-1/2},
\end{align}
again with weak-$*$ boundary values and continuity for $E > -\alpha^2$, all
inherited from Theorem~\ref{thm:AH-classical} in dimension $n-1$.

\smallskip
\noindent\textbf{The remainder.}
Only this term requires an argument specific to the present problem, and it
rests on the following elementary bound.

\begin{lemma}[Uniform bound for the remainder kernel]
\label{lem:K-kernel-bound}
For all $\xipar \in \Rnm$, all $x_n,y_n > 0$ and all
$z \in \mathbb{C}\setminus\mathbb{R}$, the kernel \eqref{eq:K-kernel} satisfies
\begin{equation}
  \label{eq:K-kernel-bound}
  |k(\xipar,x_n,y_n;z)| \le \frac{1}{\alpha}.
\end{equation}
\end{lemma}

\begin{proof}
The denominator of \eqref{eq:K-kernel} factors as
$\kappa^2 - \alpha^2 = (\kappa-\alpha)(\kappa+\alpha)$, so we split the
numerator to match, writing it as
$(\kappa+\alpha)\bigl(e^{-\kappa s} - e^{-\alpha s}\bigr)
 + (\kappa-\alpha)\,e^{-\alpha s}$. This gives
\begin{equation*}
  k = \frac{e^{-\kappa s} - e^{-\alpha s}}{\kappa - \alpha}
      + \frac{e^{-\alpha s}}{\kappa + \alpha} ,
\end{equation*}
and the second term is already harmless, since
$\re(\kappa+\alpha) \ge \alpha > 0$. The first is a difference quotient of
$w \mapsto e^{-ws}$ between $w = \alpha$ and $w = \kappa$, hence an average of
the derivative along the segment joining them. Parametrizing that segment by
$w_t = \alpha + t(\kappa-\alpha)$, $t \in [0,1]$, and using
$\frac{d}{dt}e^{-w_t s} = -s(\kappa-\alpha)e^{-w_t s}$,
\begin{equation*}
  \frac{e^{-\kappa s} - e^{-\alpha s}}{\kappa - \alpha}
  = -\,s\int_0^1 e^{-w_t s}\,dt .
\end{equation*}
The whole segment lies in the half-plane $\re w > 0$, because
$\re w_t = (1-t)\alpha + t\,\re\kappa \ge (1-t)\alpha$, and therefore
\begin{equation*}
  \Bigl| s\int_0^1 e^{-w_t s}\,dt \Bigr|
  \le s \int_0^1 e^{-(1-t)\alpha s}\,dt
  = \frac{1 - e^{-\alpha s}}{\alpha} .
\end{equation*}
Since $|e^{-\alpha s}/(\kappa+\alpha)| \le e^{-\alpha s}/\alpha$, adding the two
bounds gives \eqref{eq:K-kernel-bound}.
\end{proof}

\begin{proposition}[The remainder is uniformly bounded]
\label{prop:K-bounded}
Let $n \ge 2$ and $\alpha > 0$. Then $\Krem{z}$ extends from
$C_c^{\infty}(\overline{\Rnp})$ to a bounded operator $\Bsp \to \Bst$ with
\begin{equation}
  \label{eq:K-bound}
  \norm{\Krem{z}}_{\Lop{\Bsp}{\Bst}} \le \frac{\sqrt2}{\alpha},
  \qquad z \in \mathbb{C}\setminus\mathbb{R},
\end{equation}
uniformly in $z$, in particular uniformly up to and across both thresholds.
Moreover, for all $f,g \in \Bsp$ the map $z \mapsto (\Krem{z}f,g)$ extends
continuously to each of the closed half-planes $\{\pm\im z \ge 0\}$, and the
limiting operators $\Krem{E \pm i0}$, $E \in \mathbb{R}$, obey
\eqref{eq:K-bound}.
\end{proposition}

\begin{proof}
Let $f \in \Bsp$ and put
$\Theta_f(\xipar) = \int_0^{\infty} |\widehat{f}(\xipar,y_n)|\,dy_n$. By
Minkowski's integral inequality, Plancherel's theorem in $\xpar$ and
Lemma~\ref{lem:slab},
\begin{equation}
  \label{eq:Theta-bound}
  \norm{\Theta_f}_{\Ltwo{\Rnm}}
  \le \int_0^{\infty} \norm{\widehat{f}(\cdot,y_n)}_{\Ltwo{\Rnm}}\,dy_n
  = \int_0^{\infty} \norm{f(\cdot,y_n)}_{\Ltwo{\Rnm}}\,dy_n
  \le \sqrt2\,\Bnorm{f}.
\end{equation}
By Lemma~\ref{lem:K-kernel-bound},
\begin{equation}
  \label{eq:K-pointwise}
  \bigl| \widehat{\Krem{z}f}(\xipar,x_n) \bigr|
  = \Bigl| \int_0^{\infty} k(\xipar,x_n,y_n;z)\,
      \widehat{f}(\xipar,y_n)\,dy_n \Bigr|
  \le \frac{1}{\alpha}\,\Theta_f(\xipar)
\end{equation}
for every $x_n > 0$, the right-hand side being independent of $x_n$. Using
$\shellp{j} \subset \{|\xpar| < 2^{j}\} \times \{0 < x_n < 2^{j}\}$ together with
Plancherel's theorem in $\xpar$,
\begin{equation*}
  \norm{\Krem{z}f}_{\Ltwo{\shellp{j}}}^2
  \le \int_0^{2^{j}}
      \bigl\| \widehat{\Krem{z}f}(\cdot,x_n) \bigr\|_{\Ltwo{\Rnm}}^2\,dx_n
  \le \frac{2^{j}}{\alpha^2}\,\norm{\Theta_f}_{\Ltwo{\Rnm}}^2
  \le \frac{2^{j+1}}{\alpha^2}\,\Bnorm{f}^2 .
\end{equation*}
Multiplying by $2^{-j}$ and taking the supremum over $j$ gives
\eqref{eq:K-bound}. Neither a decomposition of the dual variable nor
any lower bound on $\re\kappa$ has been used.

For the boundary values, let $f,g \in \Bsp$. By \eqref{eq:K-kernel-bound},
\eqref{eq:Theta-bound} and the Cauchy--Schwarz inequality, the triple integral
\begin{equation*}
  (\Krem{z}f,g)
  = \int_{\Rnm} \int_0^{\infty} \int_0^{\infty}
    k(\xipar,x_n,y_n;z)\,\widehat{f}(\xipar,y_n)\,
    \overline{\widehat{g}(\xipar,x_n)}\,dy_n\,dx_n\,d\xipar
\end{equation*}
converges absolutely, with the $z$-independent majorant
$\alpha^{-1}\Theta_f\Theta_g \in L^1(\Rnm)$ of total mass at most
$2\alpha^{-1}\Bnorm{f}\Bnorm{g}$. For every $\xipar$ with
$|\xipar|^2 \ne \re z$ the function $z \mapsto \kappa(\xipar,z)$, hence
$z \mapsto k$, extends continuously to $\{\pm\im z \ge 0\}$. Since the
exceptional set is Lebesgue-null, dominated convergence yields both the
existence of $\lim_{\eps \downarrow 0}(\Krem{E \pm i\eps}f,g)$ and its
continuity in $E$. That the limit is again given by an operator bounded as in
\eqref{eq:K-bound} follows by repeating the first part of the proof with
$k(\cdot\,;E \pm i0)$, which satisfies the same pointwise bound.
\end{proof}

\subsubsection{Proof of the main result, and explicit constants}
\label{sec:lap-proof}

\begin{proof}[Proof of Theorem~\ref{thm:lap}]
(a) By Proposition~\ref{prop:splitting} together with
\eqref{eq:dirichlet-bound}, \eqref{eq:surface-bound} and \eqref{eq:K-bound},
\begin{equation}
  \label{eq:lap-explicit}
  \norm{\Ralpha{z}}_{\Lop{\Bsp}{\Bst}}
  \le \sqrt2\,\CAH{n}\,|z|^{-1/2}
    + \sqrt2\,\Csharp\,\CAH{n-1}\,\bigl|z+\alpha^2\bigr|^{-1/2}
    + \frac{\sqrt2}{\alpha} .
\end{equation}
If $\re z \in K$ with $K \subset \Ereg$ compact, then $\dist(K,\Thr) > 0$, so
$|z|$ and $|z+\alpha^2|$ are bounded from below uniformly in $\im z$, and
\eqref{eq:lap-bound} follows.

(b), (c) Each of the three terms of \eqref{eq:splitting-preview} has weak-$*$
boundary values, continuous in $E$: the first two by
Theorem~\ref{thm:AH-classical} in dimensions $n$ and $n-1$ respectively (for
$E > 0$, resp. $E > -\alpha^2$, the remaining values of $E \in \Ereg$ being ones
at which the term in question is analytic near $z = E$), and the third by
Proposition~\ref{prop:K-bounded}. The bound \eqref{eq:lap-explicit} is preserved
in the limit.

(d) Let $\im z \ne 0$ and $f \in \Bsp \subset \Ltwo{\Rnp}$. The identity
$\Ralpha{z}^{*} = \Ralpha{\overline{z}}$ is standard for the resolvent of a
self-adjoint operator. For the second statement put $u = \Ralpha{z}f$, so that
$u \in \Dom(\Halpha)$ and $f = (\Halpha - z)u$. Then
\begin{equation*}
  \bigl( \Ralpha{z}f, f \bigr)
  = \bigl( u, (\Halpha - z)u \bigr)
  = \bigl( u, \Halpha u \bigr)
    - \overline{z}\,\norm{u}_{\Ltwo{\Rnp}}^2 ,
\end{equation*}
and $(u,\Halpha u)$ is real because $\Halpha$ is self-adjoint. Taking imaginary
parts,
\begin{equation*}
  \im\bigl( \Ralpha{z}f, f \bigr)
  = -\im \overline{z}\,\norm{u}_{\Ltwo{\Rnp}}^2
  = \im z\,\norm{\Ralpha{z}f}_{\Ltwo{\Rnp}}^2 ,
\end{equation*}
which is nonnegative for $\im z > 0$. Both statements pass to the limit
$\eps \downarrow 0$ by (b).

(e) Let $f \in \Bsp \subset \Ltwo{\Rnp}$ and
$\varphi \in C_c^{\infty}(\overline{\Rnp})$ with
$\gamma_1\varphi - \alpha\gamma_0\varphi = 0$, so that
$\varphi \in \Dom(\Halpha) \cap \Bsp$. For $\im z \ne 0$,
\begin{equation*}
  \bigl( \Ralpha{z}f, (\Halpha - \overline{z})\varphi \bigr)
  = \bigl( (\Halpha - z)\Ralpha{z}f, \varphi \bigr)
  = (f,\varphi),
\end{equation*}
and letting $\eps \downarrow 0$ along $z = E \pm i\eps$ gives, by (b),
$(u,(\Halpha - E)\varphi) = (f,\varphi)$. Thus $u$ solves $(-\Delta - E)u = f$ in
the distributional sense in $\Rnp$ and satisfies the Robin condition weakly.
Interior and boundary elliptic regularity for the Robin problem on the
half-space (see \cite[\S6.3.1]{Evans2010} and
\cite[Thm.~4.18(ii)]{McLean2000}, the latter applied with the Neumann datum
$\alpha\gamma_0 u$ as in Remark~\ref{rem:rad-regularity}) upgrades this to
$u \in H^2_{\mathrm{loc}}(\overline{\Rnp})$, and $u \in \Bst$ by (b).
\end{proof}

\begin{proposition}[Explicit constants]
\label{prop:lap-explicit}
Estimate \eqref{eq:lap-explicit} holds for every
$z \in \mathbb{C}\setminus\mathbb{R}$, with $\Csharp = (1+\sqrt2)^2$ and with
$\CAH{n}$, $\CAH{n-1}$ the constants of Theorem~\ref{thm:AH-classical} in
dimensions $n$ and $n-1$. The same bound holds for $\Ralpha{E \pm i0}$ with $z$
replaced by $E \in \Ereg$. In particular the Robin parameter enters only through
the last term and through the shift $z + \alpha^2$, and the dependence on the
spectral parameter is carried entirely by the two free resolvents
\eqref{eq:two-channels-preview}.
\end{proposition}

The two thresholds therefore enter \eqref{eq:lap-explicit} with the same rate,
$\dist(E,\Thr)^{-1/2}$. What distinguishes them is their position in the
spectrum and the dimension of the free problem from which each originates
(Remark~\ref{rem:thresholds}). Their behaviour as $E \to 0$ and as
$E \to -\alpha^{2+}$ is not pursued here. The eigenfunction expansion, the
scattering theory and the inverse theory below are all carried out on compact
subsets of $\Ereg$.

%% file: Part1_023_gft.tex
\label{subsec:gft}

The limiting absorption principle determines the spectral representation of
$\Halpha$ abstractly, but it does not exhibit the space in which that
representation takes place. This subsection supplies the missing object, in two
halves. In \S\ref{sec:gft-abstract} we record what Theorem~\ref{thm:lap} gives
for free: the boundary values of the resolvent produce a nonnegative,
continuous family of operators $\Adens{E} \in \Lop{\Bsp}{\Bst}$ whose range
consists of generalized eigenfunctions and which is the density of the spectral
measure. Nothing about the half-space enters there, and nothing about the two
channels can be read off. The remainder of the subsection identifies
$\Adens{E}$, and with it the representation space, in explicit terms: we
construct the bulk and surface generalized eigenfunctions from the fiber
eigenfunctions of Lemma~\ref{lem:1d-transform}, prove that they are mutually
$\delta$-orthogonal, assemble them into the generalized Fourier transform
$\Fzero$, and show that Stone's formula holds in the form
$\Adens{E} = \Fzero(E)^{*}\Fzero(E)$.

\subsubsection{Consequences of the limiting absorption principle}
\label{sec:gft-abstract}

Throughout, $\Ralpha{z} = (\Halpha - z)^{-1}$, and $P(\,\cdot\,)$ denotes the
spectral measure, that is the resolution of the identity, of $\Halpha$.

\begin{definition}
For $E \in \Ereg$ we set
\begin{equation*}
  \Adens{E} := \frac{1}{2\pi i}
    \bigl( \Ralpha{E+i0} - \Ralpha{E-i0} \bigr),
\end{equation*}
the boundary values being those of Theorem~\ref{thm:lap}(b).
\end{definition}

\begin{proposition}
\label{prop:A-properties}
Let $E \in \Ereg$. Then the following hold.
\begin{enumerate}
  \item[\textup{(i)}] $\Adens{E} \in \Lop{\Bsp}{\Bst}$, and
  $\norm{\Adens{E}}_{\Lop{\Bsp}{\Bst}}$ is bounded uniformly on compact subsets
  of $\Ereg$. Moreover $\Adens{E} = \pi^{-1}\im\Ralpha{E+i0}$, and
  \begin{equation}
    \label{eq:A-nonneg}
    \bigl( \Adens{E}f, f \bigr) \ge 0,
    \qquad f \in \Bsp .
  \end{equation}
  \item[\textup{(ii)}] For all $f,g \in \Bsp$ the map
  $\Ereg \ni E \mapsto (\Adens{E}f,g)$ is continuous.
  \item[\textup{(iii)}] For every $f \in \Bsp$ the function $u = \Adens{E}f$
  belongs to $\Bst \cap H^2_{\mathrm{loc}}(\overline{\Rnp})$ and satisfies
  \begin{equation}
    \label{eq:A-eigen}
    (-\Delta - E)u = 0 \ \text{ in } \Rnp,
    \qquad
    \gamma_1 u - \alpha\gamma_0 u = 0 \ \text{ on } \Gamma ;
  \end{equation}
  that is, $\Ran \Adens{E}$ consists of generalized eigenfunctions of $\Halpha$
  for the eigenvalue $E$.
  \item[\textup{(iv)}] For every interval $(a,b) \subset \Ereg$ with
  $\overline{(a,b)} \subset \Ereg$ and all $f,g \in \Bsp$,
  \begin{equation}
    \label{eq:stone-measure}
    \bigl( P((a,b))f, g \bigr) = \int_a^b \bigl( \Adens{E}f, g \bigr)\,dE .
  \end{equation}
\end{enumerate}
\end{proposition}

\begin{proof}
(i) The membership and the uniform bound are Theorem~\ref{thm:lap}(a),(b). By
Theorem~\ref{thm:lap}(d) one has $\Ralpha{E-i0} = \Ralpha{E+i0}^{*}$ in the
sense of the pairing, so that
$\Adens{E} = (2\pi i)^{-1}(\Ralpha{E+i0} - \Ralpha{E+i0}^{*})
 = \pi^{-1}\im\Ralpha{E+i0}$, and \eqref{eq:A-nonneg} is the positivity
statement of Theorem~\ref{thm:lap}(d).

(ii) is Theorem~\ref{thm:lap}(c).

(iii) By Theorem~\ref{thm:lap}(e) both $u_{\pm} = \Ralpha{E \pm i0}f$ lie in
$\Bst \cap H^2_{\mathrm{loc}}(\overline{\Rnp})$ and solve
$(-\Delta - E)u_{\pm} = f$ together with the Robin condition. Subtracting, the
inhomogeneity cancels and \eqref{eq:A-eigen} follows.

(iv) Let $f,g \in \Bsp \subset \Ltwo{\Rnp}$. Stone's formula for the spectral
measure of $\Halpha$ gives
\begin{equation*}
  \tfrac12 \bigl( P([a,b]) + P((a,b)) \bigr)
  = \operatorname*{s-lim}_{\eps \downarrow 0} \frac{1}{2\pi i}
    \int_a^b \bigl( \Ralpha{E+i\eps} - \Ralpha{E-i\eps} \bigr)\,dE,
\end{equation*}
and the two projections on the left coincide because
$P(\{a\}) = P(\{b\}) = 0$, $\Halpha$ having no eigenvalues
(Theorem~\ref{thm:sa-spec}(c)). The integrand converges pointwise in $E$ by
Theorem~\ref{thm:lap}(b), and is dominated, uniformly in $\eps \in (0,1]$, by
the constant $2C\Bnorm{f}\Bnorm{g}$ of Theorem~\ref{thm:lap}(a) applied to the
compact set $\mathcal{K} = \overline{(a,b)}$. The dominated convergence theorem gives
\eqref{eq:stone-measure}.
\end{proof}

Properties (i)--(iv) already determine a spectral representation of $\Halpha$
on $\Ereg$, as follows. By
\eqref{eq:A-nonneg} the sesquilinear form $(f,g) \mapsto (\Adens{E}f,g)$ is
nonnegative on $\Bsp$. Let $\Echan{E}$ denote the Hilbert space obtained by
quotienting $\Bsp$ by the null space of this form and completing, and let
$\Fzero(E) : \Bsp \to \Echan{E}$ be the canonical map. Then
$\Fzero(E)^{*}\Fzero(E) = \Adens{E}$ by construction, and
\eqref{eq:stone-measure} says exactly that
$E \mapsto \norm{\Fzero(E)f}_{\Echan{E}}^2$ is the density of the spectral
measure of $f$. Since $\Halpha$ is purely absolutely continuous with spectrum
$[-\alpha^2,\infty)$ (Theorem~\ref{thm:sa-spec}(b),(c)), the associated global
map $\Fzero$ is isometric on $\Bsp$ and extends to a unitary operator onto
$\int_{\Ereg}^{\oplus}\Echan{E}\,dE$. The notation anticipates
Definition~\ref{def:gft}. That the two constructions agree is
Theorem~\ref{thm:stone}.

This is all that abstraction yields, and it is not enough. The space
$\Echan{E}$ so obtained is anonymous: its dimension is unknown, the two
scattering channels are invisible in it, and no generalized eigenfunction has
been produced. Indeed \eqref{eq:A-eigen} asserts that $\Ran\Adens{E}$ consists
of generalized eigenfunctions, but says nothing about which ones, nor how many
parameters are needed to label them. The content of the rest of this subsection
is the identification
\begin{equation*}
  \Echan{E} \;\cong\;
  \begin{cases}
    \Ltwo{\Sm} \oplus \Ltwo{\Ssurf}, & E > 0, \\[1mm]
    \Ltwo{\Ssurf}, & -\alpha^2 < E < 0,
  \end{cases}
\end{equation*}
together with the explicit eigenfunctions realising it. It is here, and only
here, that the half-space and the Robin condition enter.

\subsubsection{The generalized eigenfunctions}
\label{sec:gft-eigenfunctions}

Recall from Appendix~\ref{app:sa-spec} the fiber eigenfunctions given by
Lemma~\ref{lem:1d-transform}, namely
\begin{equation}
  \label{eq:fiber-recall}
  \psizero(t) = \sqrt{2\alpha}\,e^{-\alpha t},
  \qquad
  \gphi_k(t) = \sqrt{\frac{2}{\pi}}\;
    \frac{k\cos(kt) - \alpha\sin(kt)}{\sqrt{k^2+\alpha^2}}
  \qquad (k > 0),
\end{equation}
both real-valued, both satisfying the Robin condition $v'(0) + \alpha v(0) = 0$
on $\Rhalf$, and orthogonal to one another in $\Ltwo{\Rhalf}$.

\begin{definition}[Generalized eigenfunctions]
\label{def:eigenfunctions}
For $\xi = (\xipar,\xi_n) \in \Rnm \times (0,\infty)$ and $\eta' \in \Rnm$ we set
\begin{align}
  \label{eq:Psib-def}
  \Psib(x,\xi)
    &= (2\pi)^{-\frac{n-1}{2}}\,
       e^{i\xpar \cdot \xipar}\,\gphi_{\xi_n}(x_n),
  \\[1mm]
  \label{eq:Psis-def}
  \Psis(x,\eta')
    &= (2\pi)^{-\frac{n-1}{2}}\,
       e^{i\xpar \cdot \eta'}\,\psizero(x_n),
\end{align}
for $x = (\xpar,x_n) \in \Rnp$.
\end{definition}

\begin{proposition}[Basic properties]
\label{prop:eigenfunctions}
Let $\xi = (\xipar,\xi_n) \in \Rnm \times (0,\infty)$ and $\eta' \in \Rnm$.
\begin{enumerate}
  \item[\textup{(i)}] The functions $\Psib(\cdot,\xi)$ and $\Psis(\cdot,\eta')$
  are smooth and bounded on $\overline{\Rnp}$, satisfy the Robin condition
  $\gamma_1 - \alpha\gamma_0 = 0$ pointwise on $\Gamma$, and solve
  \begin{equation}
    \label{eq:eigen-equations}
    -\Delta\,\Psib(\cdot,\xi) = |\xi|^2\,\Psib(\cdot,\xi),
    \qquad
    -\Delta\,\Psis(\cdot,\eta') = \bigl( |\eta'|^2 - \alpha^2 \bigr)\,
      \Psis(\cdot,\eta') .
  \end{equation}
  \item[\textup{(ii)}] Put
  \begin{equation}
    \label{eq:phase-and-reflection}
    \cph{k} = \frac{k+i\alpha}{\sqrt{k^2+\alpha^2}},
    \qquad
    \refl{k} = \frac{\cph{k}}{\overline{\cph{k}}}
             = \frac{ik-\alpha}{ik+\alpha},
    \qquad k > 0,
  \end{equation}
  so that $|\cph{k}| = |\refl{k}| = 1$. Then
  \begin{equation}
    \label{eq:Psib-exponential}
    \Psib(x,\xi)
    = (2\pi)^{-\frac{n}{2}}\,\overline{\cph{\xi_n}}\;
      e^{i\xpar \cdot \xipar}
      \Bigl[ e^{-i\xi_n x_n} + \refl{\xi_n}\,e^{i\xi_n x_n} \Bigr].
  \end{equation}
  \item[\textup{(iii)}] Conversely, let $E \in \mathbb{R}$ and let
  $u(x) = e^{i\xpar \cdot \xipar}w(x_n)$ be a nonzero bounded solution of
  $-\Delta u = Eu$ on $\Rnp$ satisfying the Robin condition. If
  $E > |\xipar|^2$, then $w$ is a constant multiple of $\gphi_{\xi_n}$ with
  $\xi_n = (E - |\xipar|^2)^{1/2}$. If $w(x_n) \to 0$ as $x_n \to \infty$, then
  $E = |\xipar|^2 - \alpha^2$ and $w$ is a constant multiple of $\psizero$. In
  either case the eigenfunctions \eqref{eq:Psib-def}--\eqref{eq:Psis-def} are
  unique up to a multiplicative constant.
\end{enumerate}
\end{proposition}

\begin{proof}
(i) Both functions are products of a smooth bounded function of $\xpar$ with a
smooth bounded function of $x_n$. Applying
$-\Delta = -\Delta' - \partial_{x_n}^2$ and using $-\gphi_k'' = k^2\gphi_k$ and
$-\psizero'' = -\alpha^2\psizero$ gives \eqref{eq:eigen-equations} with
$|\xipar|^2 + \xi_n^2 = |\xi|^2$ and $|\eta'|^2 - \alpha^2$ respectively. The
boundary condition is inherited from the corresponding one-dimensional
condition, since $\gamma_1 u - \alpha\gamma_0 u
= -\bigl( \partial_{x_n}u(\xpar,0) + \alpha u(\xpar,0) \bigr)$.

(ii) Writing $\cos(kt) = \tfrac12(e^{ikt} + e^{-ikt})$ and
$\sin(kt) = \tfrac{1}{2i}(e^{ikt} - e^{-ikt})$ in \eqref{eq:fiber-recall} gives
$k\cos(kt) - \alpha\sin(kt)
 = \tfrac12 (k+i\alpha)e^{ikt} + \tfrac12 (k-i\alpha)e^{-ikt}$, so that, since
$\sqrt{2/\pi}\cdot\tfrac12 = (2\pi)^{-1/2}$,
\begin{equation}
  \label{eq:phi-exponential}
  \gphi_k(t)
  = \frac{1}{\sqrt{2\pi}}
    \Bigl[ \overline{\cph{k}}\,e^{-ikt} + \cph{k}\,e^{ikt} \Bigr]
  = \frac{\overline{\cph{k}}}{\sqrt{2\pi}}
    \Bigl[ e^{-ikt} + \refl{k}\,e^{ikt} \Bigr] .
\end{equation}
Multiplying by $(2\pi)^{-(n-1)/2}e^{i\xpar\cdot\xipar}$ gives
\eqref{eq:Psib-exponential}. The identity
$(k+i\alpha)/(k-i\alpha) = (ik-\alpha)/(ik+\alpha)$ follows on multiplying
numerator and denominator by $i$.

(iii) Substituting $u = e^{i\xpar\cdot\xipar}w$ into $-\Delta u = Eu$ gives
$-w'' = (E - |\xipar|^2)w$ on $\Rhalf$, with $w'(0) + \alpha w(0) = 0$. If
$E - |\xipar|^2 = \xi_n^2 > 0$ the solution space is two-dimensional and the
Robin condition cuts it down to one dimension, spanned by $\gphi_{\xi_n}$. Every
element of it is bounded, so boundedness imposes nothing further. Suppose now
that $w(x_n) \to 0$. Then $E - |\xipar|^2 < 0$, since for $E - |\xipar|^2 \ge 0$
the bounded solutions are trigonometric or constant, and none of those tends to
$0$ without vanishing. Write $E - |\xipar|^2 = -\kappa^2$ with $\kappa > 0$, so
that $w$ is a multiple of $e^{-\kappa x_n}$. The Robin condition reads
$(-\kappa + \alpha)w(0) = 0$ and forces $\kappa = \alpha$, whence
$E = |\xipar|^2 - \alpha^2$ and $w \propto \psizero$. This is the computation of
Lemma~\ref{lem:1d-transform}(i) applied to the shifted fiber.
\end{proof}

\begin{remark}[The hemisphere, and the phase convention]
\label{rem:standing-wave}
Two consequences of Proposition~\ref{prop:eigenfunctions}(ii) will be used
repeatedly.

\emph{(a) The hemisphere, and not the sphere.} For each fixed $\xipar$ the
Robin condition determines the ratio of the two amplitudes in
\eqref{eq:Psib-exponential}, so there is exactly \emph{one} degree of freedom
per $\xipar$, namely $\xi_n > 0$, and not two. Writing $\xi = |\xi|\theta$, this
means that $\theta$ ranges over the open hemisphere $\Sm$ only, just as the
Robin problem on the half-line is diagonalized by a single continuous parameter
(Lemma~\ref{lem:1d-transform}(iii)).

\emph{(b) $\Fzero$ is a standing-wave transform.} Since $\gphi_k$ is
real-valued, $\Psib(\cdot,\xi)$ is by \eqref{eq:Psib-exponential} the
superposition of a plane wave $e^{-i\xi_n x_n}$ travelling towards the boundary
and its reflection, with reflection coefficient $\refl{\xi_n}$ of modulus one,
and it carries no net flux in the normal direction. The real normalization
agrees with Lemma~\ref{lem:1d-transform} and is the one used for Stone's
formula, but it is not the one adapted to scattering. For that one takes
$\cph{\xi_n}\Psib(\cdot,\xi)$, whose incoming amplitude in
\eqref{eq:Psib-exponential} is $1$. The two differ by a unimodular factor and
satisfy the same orthogonality relations. The second is used for the outgoing
and incoming transforms of \S\ref{subsec:asymptotics} and \S\ref{sec:inverse}.
A reader comparing with \cite[Ch.~1, \S7]{IsozakiKurylev2014} should keep this
distinction in mind.
\end{remark}

\begin{remark}[The surface channel as a tensor factor]
\label{rem:surface-family}
The eigenvalue $|\eta'|^2 - \alpha^2$ attached to $\Psis(\cdot,\eta')$ depends
only on the tangential variable, because the Robin condition fixes the normal
decay rate to be $\alpha$ independently of $\eta'$
(Proposition~\ref{prop:eigenfunctions}(iii)): the transverse profile $\psizero$
is the same for every surface wave. This makes $\Hsurf$ a tensor factor, and
with it the exact splitting \eqref{eq:splitting-preview} and
Proposition~\ref{prop:surface-reduction} of \S\ref{subsec:lap}. In a layered
medium, by contrast, the decay rate of a trapped mode varies with the energy
and no such decomposition is available.
\end{remark}

\subsubsection{\texorpdfstring{$\delta$}{delta}-orthogonality and completeness}
\label{sec:gft-orthogonality}

The next theorem is the analytic heart of this subsection. Formally it states
\begin{equation}
  \label{eq:delta-orthogonality-formal}
\begin{aligned}
  \int_{\Rnp} \Psib(x,\xi)\,\overline{\Psib(x,\eta)}\,dx
    &= \delta(\xi - \eta),
  &\quad
  \int_{\Rnp} \Psis(x,\xipar)\,\overline{\Psis(x,\eta')}\,dx
    &= \delta(\xipar - \eta'),
  \\
  \int_{\Rnp} \Psib(x,\xi)\,\overline{\Psis(x,\eta')}\,dx &= 0, &&
\end{aligned}
\end{equation}
and we state it in the form of an isometry, which is the form used below.

\begin{theorem}[$\delta$-orthogonality]
\label{thm:delta-orthogonality}
For $F \in C_c^\infty(\Rnm \times (0,\infty))$ and $G \in C_c^\infty(\Rnm)$ put
\begin{equation}
  \label{eq:wave-packets}
  u_F(x) = \int_{\Rnm}\!\!\int_0^{\infty} F(\xi)\,\Psib(x,\xi)\,d\xi_n\,d\xipar,
  \qquad
  v_G(x) = \int_{\Rnm} G(\eta')\,\Psis(x,\eta')\,d\eta' .
\end{equation}
Then $u_F, v_G \in \Ltwo{\Rnp}$, and for all such $F,F',G,G'$,
\begin{equation}
  \label{eq:isometry}
\begin{aligned}
  (u_F, u_{F'})_{\Ltwo{\Rnp}} &= (F,F')_{\Ltwo{\Rnm \times (0,\infty)}},
  &\quad
  (v_G, v_{G'})_{\Ltwo{\Rnp}} &= (G,G')_{\Ltwo{\Rnm}},
  \\
  (u_F, v_G)_{\Ltwo{\Rnp}} &= 0 . &&
\end{aligned}
\end{equation}
\end{theorem}

\begin{proof}
Let $\mathcal{F}'$ be the partial Fourier transform \eqref{eq:fourier-partial},
a unitary operator on $\Ltwo{\Rnp} = \Ltwo{\Rnm} \otimes \Ltwo{\Rhalf}$. By
\eqref{eq:Psib-def} the inner integral in \eqref{eq:wave-packets} is, for each
fixed $x_n$, the inverse partial Fourier transform of
$\xipar \mapsto \int_0^{\infty} F(\xipar,\xi_n)\gphi_{\xi_n}(x_n)\,d\xi_n$.
Hence
\begin{equation}
  \label{eq:uF-hat}
  \widehat{u_F}(\xipar,x_n)
  = \int_0^{\infty} F(\xipar,\xi_n)\,\gphi_{\xi_n}(x_n)\,d\xi_n
  = \bigl( \Utrans^{*}F(\xipar,\cdot) \bigr)(x_n),
\end{equation}
where $\Utrans$ is the one-dimensional spectral transform
\eqref{eq:1d-transform}. The interchange of the order of integration is
justified because $F$ has compact support and $\gphi_k$ is bounded uniformly
for $k$ in compact subsets of $(0,\infty)$. Likewise
\begin{equation}
  \label{eq:vG-hat}
  \widehat{v_G}(\xipar,x_n) = G(\xipar)\,\psizero(x_n).
\end{equation}
By Lemma~\ref{lem:1d-transform}(iii) the operator $\Utrans^{*}$ is an isometry
of $\Ltwo{\Rhalf,dk}$ onto $\{\psizero\}^{\perp} \subset \Ltwo{\Rhalf}$.
Therefore, for a.e.\ $\xipar$,
\begin{equation*}
  \bigl( \widehat{u_F}(\xipar,\cdot), \widehat{u_{F'}}(\xipar,\cdot)
  \bigr)_{\Ltwo{\Rhalf}}
  = \bigl( F(\xipar,\cdot), F'(\xipar,\cdot) \bigr)_{\Ltwo{\Rhalf,dk}},
  \qquad
  \bigl( \widehat{u_F}(\xipar,\cdot), \psizero \bigr)_{\Ltwo{\Rhalf}} = 0 .
\end{equation*}
Integrating the first identity in $\xipar$ and using the unitarity of
$\mathcal{F}'$ gives the first relation in \eqref{eq:isometry}, and in
particular $u_F \in \Ltwo{\Rnp}$. The second follows from
$\norm{\psizero}_{\Ltwo{\Rhalf}} = 1$, and the third from the second identity
above together with \eqref{eq:vG-hat}:
\begin{equation*}
  (u_F,v_G)_{\Ltwo{\Rnp}}
  = \int_{\Rnm} \overline{G(\xipar)}\,
    \bigl( \widehat{u_F}(\xipar,\cdot), \psizero \bigr)_{\Ltwo{\Rhalf}}
    \,d\xipar
  = 0 . \qedhere
\end{equation*}
\end{proof}

By duality, Theorem~\ref{thm:delta-orthogonality} yields the transforms
themselves. For $f \in \Ltwo{\Rnp}$ with compact support set
\begin{equation}
  \label{eq:Ftilde-def}
  (\Ftbulk f)(\xi) = \int_{\Rnp} \overline{\Psib(x,\xi)}\,f(x)\,dx,
  \qquad
  (\Ftsurf f)(\eta') = \int_{\Rnp} \overline{\Psis(x,\eta')}\,f(x)\,dx .
\end{equation}

\begin{corollary}[Completeness]
\label{cor:completeness}
The map $f \mapsto (\Ftbulk f, \Ftsurf f)$ extends from compactly supported
functions to a unitary operator
\begin{equation}
  \label{eq:Ftilde-unitary}
  \Ltwo{\Rnp} \;\longrightarrow\;
  \Ltwo{\Rnm \times (0,\infty)} \oplus \Ltwo{\Rnm},
\end{equation}
and for $f \in \Dom(\Halpha)$,
\begin{equation}
  \label{eq:Ftilde-diagonal}
  (\Ftbulk \Halpha f)(\xi) = |\xi|^2\,(\Ftbulk f)(\xi),
  \qquad
  (\Ftsurf \Halpha f)(\eta') = \bigl( |\eta'|^2 - \alpha^2 \bigr)\,
    (\Ftsurf f)(\eta') .
\end{equation}
\end{corollary}

\begin{proof}
The operators \eqref{eq:Ftilde-def} are the adjoints of $F \mapsto u_F$ and
$G \mapsto v_G$, so \eqref{eq:isometry} says that
$\Ftbulk^{*}\Ftbulk + \Ftsurf^{*}\Ftsurf$ acts as the identity on the closed
span of the ranges of $F \mapsto u_F$ and $G \mapsto v_G$, while
$\Ftbulk\Ftbulk^{*} = I$ and $\Ftsurf\Ftsurf^{*} = I$ on
$\Ltwo{\Rnm \times (0,\infty)}$ and $\Ltwo{\Rnm}$ respectively, the cross terms
vanishing. It remains to see that the said span is all of $\Ltwo{\Rnp}$, and
this is \eqref{eq:uF-hat}--\eqref{eq:vG-hat} again: under the unitary
$\mathcal{F}'$ the two ranges become $\Ltwo{\Rnm} \otimes \{\psizero\}^{\perp}$
and $\Ltwo{\Rnm} \otimes \mathbb{C}\psizero$, whose sum is
$\Ltwo{\Rnm} \otimes \Ltwo{\Rhalf} = \Ltwo{\Rnp}$ by
Lemma~\ref{lem:1d-transform}(iii). Identity \eqref{eq:Ftilde-diagonal} follows
from \eqref{eq:1d-diag} and $\hfib(\xipar) = h_0 + |\xipar|^2$, that is, from
Proposition~\ref{prop:fibration}.
\end{proof}

\subsubsection{The generalized Fourier transform}
\label{sec:gft-transform}

We now pass to polar coordinates, so as to fiber over the spectral parameter.
For $E > 0$ the bulk variable is written $\xi = \sqrt{E}\,\theta$ with
$\theta \in \Sm$, and for $E > -\alpha^2$ the surface variable is written
$\eta' = \sqrt{E+\alpha^2}\,\omega$ with $\omega \in \Ssurf$. The corresponding
Jacobians are
\begin{equation}
  \label{eq:jacobians}
  d\xi = \tfrac12\,E^{\frac{n-2}{2}}\,dE\,d\theta,
  \qquad
  d\eta' = \tfrac12\,(E+\alpha^2)^{\frac{n-3}{2}}\,dE\,d\omega .
\end{equation}

\begin{definition}[Channel spaces and the generalized Fourier transform]
\label{def:gft}
For $E \in \Ereg$ put
\begin{equation}
  \label{eq:channel-space}
  \Echan{E} =
  \begin{cases}
    \Ltwo{\Sm} \oplus \Ltwo{\Ssurf}, & E > 0, \\[1mm]
    \Ltwo{\Ssurf}, & -\alpha^2 < E < 0,
  \end{cases}
  \qquad
  \Hdirect = \int_{\Ereg}^{\oplus} \Echan{E}\,dE ,
\end{equation}
and define, for $f$ in a suitable dense class,
\begin{align}
  \label{eq:Fbulk-def}
  \bigl( \Fbulk(E)f \bigr)(\theta)
    &= \frac{E^{\frac{n-2}{4}}}{\sqrt2}\,
       \int_{\Rnp} \overline{\Psib(x,\sqrt{E}\,\theta)}\,f(x)\,dx,
    &&\theta \in \Sm,\ E > 0,
  \\[1mm]
  \label{eq:Fsurf-def}
  \bigl( \Fsurf(E)f \bigr)(\omega)
    &= \frac{(E+\alpha^2)^{\frac{n-3}{4}}}{\sqrt2}
       \int_{\Rnp} \overline{\Psis(x,\ksurf\,\omega)}\,f(x)\,dx,
    &&\omega \in \Ssurf,\ E > -\alpha^2 .
\end{align}
We set $\Fzero(E) = \bigl( \Fbulk(E), \Fsurf(E) \bigr)$ for $E > 0$ and
$\Fzero(E) = \Fsurf(E)$ for $-\alpha^2 < E < 0$, and
$(\Fzero f)(E) = \Fzero(E)f$.
\end{definition}

\begin{theorem}[Spectral representation]
\label{thm:spectral-representation}
$\Fzero$ extends to a unitary operator from $\Ltwo{\Rnp}$ onto $\Hdirect$ which
diagonalizes $\Halpha$:
\begin{equation}
  \label{eq:diagonalization}
  \bigl( \Fzero \Halpha f \bigr)(E) = E\,\bigl( \Fzero f \bigr)(E),
  \qquad f \in \Dom(\Halpha),
\end{equation}
for a.e.\ $E \in \Ereg$. In particular the inversion formula
\begin{equation}
  \label{eq:inversion}
  f = \operatorname*{s-lim}_{N \to \infty}
    \int_{K_N} \Fzero(E)^{*}\bigl( \Fzero f \bigr)(E)\,dE,
  \qquad
  K_N = \bigl[ -\alpha^2 + \tfrac1N,\ -\tfrac1N \bigr]
      \cup \bigl[ \tfrac1N,\ N \bigr],
\end{equation}
holds in $\Ltwo{\Rnp}$.
\end{theorem}

\begin{proof}
The change of variables \eqref{eq:jacobians} turns the unitary map
\eqref{eq:Ftilde-unitary} of Corollary~\ref{cor:completeness} into a unitary map
onto
$\int_{(0,\infty)}^{\oplus}\Ltwo{\Sm}\,dE
 \oplus \int_{(-\alpha^2,\infty)}^{\oplus}\Ltwo{\Ssurf}\,dE$,
the constants in \eqref{eq:Fbulk-def}--\eqref{eq:Fsurf-def} being exactly the
square roots of the Jacobian densities. Regrouping the two direct integrals over
a common spectral parameter gives $\Hdirect$ as in \eqref{eq:channel-space}, the
bulk summand being absent for $E < 0$. Identity \eqref{eq:diagonalization} is
\eqref{eq:Ftilde-diagonal} rewritten in the new variables, since
$|\xi|^2 = E$ on the bulk sphere and $|\eta'|^2 - \alpha^2 = E$ on the surface
sphere. Finally \eqref{eq:inversion} is the statement that
$\Fzero^{*}\Fzero = I$, together with the fact that $P(\Thr) = 0$ because
$\Halpha$ has no eigenvalues (Theorem~\ref{thm:sa-spec}(c)).
\end{proof}

The two channels enter \eqref{eq:channel-space} asymmetrically. On
$(-\alpha^2,0)$ only the surface channel is open. At $E = 0$ the bulk channel
opens and the fiber gains the summand $\Ltwo{\Sm}$. No continuous family of
unitary identifications can cross such a jump in the dimension of the fiber,
which is why $E = 0$ is excluded here. In Theorem~\ref{thm:lap} it was excluded
for a different reason, the degeneration of the resolvent bound. The surface
threshold $-\alpha^2$, by contrast, is the bottom of the spectrum, and below it
there is no channel to open (Remark~\ref{rem:thresholds}).

\subsubsection{\texorpdfstring{$\Bsp$--$\Bst$}{B--B*} properties and Stone's
formula}
\label{sec:gft-BBstar}

It remains to connect Definition~\ref{def:gft} with \S\ref{sec:gft-abstract}.
The tool is the second classical theorem of Agmon and H\"ormander, the
restriction theorem, which we quote alongside its dual.

\begin{theorem}[Agmon--H\"ormander \cite{AgmonHormander1976}]
\label{thm:AH-restriction}
Let $m \ge 2$. For every $F \in \Bspd{\mathbb{R}^m}$ the restriction of
$\widehat{F}$ to the sphere $\{|\xi| = \rho\}$ is well defined for every
$\rho > 0$, and, after enlarging the constant $\CAH{m}$ of
Theorem~\ref{thm:AH-classical} if necessary,
\begin{equation}
  \label{eq:AH-restriction}
  \sup_{\rho > 0} \int_{|\xi| = \rho} \bigl| \widehat{F}(\xi) \bigr|^2\,dS(\xi)
  \le \bigl( \CAH{m} \bigr)^2\,\norm{F}_{\Bspd{\mathbb{R}^m}}^2 ,
\end{equation}
the left-hand side being scale invariant. Moreover
$\rho \mapsto \widehat{F}(\rho\,\cdot\,)|_{S^{m-1}}$ is continuous from
$(0,\infty)$ to $\Ltwo{S^{m-1}}$. Dually, for $g \in \Ltwo{S^{m-1}}$ the
function
\begin{equation}
  \label{eq:AH-extension}
  x \longmapsto \int_{S^{m-1}} e^{i\rho\,x \cdot \theta}\,g(\theta)\,d\theta
\end{equation}
belongs to $\Bstd{\mathbb{R}^m}$, with norm at most
$C\rho^{-\frac{m-1}{2}}\norm{g}_{\Ltwo{S^{m-1}}}$.
\end{theorem}

The next lemma is the elementary but essential observation that $\Ftbulk$ is, up
to unimodular factors, the restriction to a sphere in $\Rn$ of the Fourier
transform of the zero extension. For $f$ defined on $\Rnp$ we write $\zext{f}$
for its extension by zero to $\Rn$, and for $\xi = (\xipar,\xi_n)$ we write
$\refext{\xi} = (\xipar,-\xi_n)$ for its reflection in the hyperplane
$\{\xi_n = 0\}$. The same notation is used for $\theta \in \Sm$, so that
$\theta \mapsto \refext{\theta}$ is a bijection from $\Sm$ onto the lower
hemisphere.

\begin{lemma}[Reflection formula]
\label{lem:reflection-formula}
For $f \in \Bsp$ one has $\zext{f} \in \Bspd{\Rn}$ with
$\norm{\zext{f}}_{\Bspd{\Rn}} = \Bnorm{f}$, and for
$\xi \in \Rnm \times (0,\infty)$,
\begin{equation}
  \label{eq:reflection-formula}
  (\Ftbulk f)(\xi)
  = \overline{\cph{\xi_n}}\,\widehat{\zext{f}}(\xi)
    + \cph{\xi_n}\,\widehat{\zext{f}}(\refext{\xi}),
\end{equation}
where $\widehat{\phantom{f}}$ denotes the Fourier transform on $\Rn$. Moreover,
with $T_{\psizero}$ as in Corollary~\ref{cor:tensor},
\begin{equation}
  \label{eq:surface-formula}
  (\Ftsurf f)(\eta') = \widehat{T_{\psizero}f}(\eta'),
  \qquad
  (T_{\psizero}f)(\xpar) = \int_0^{\infty} \psizero(x_n)\,f(\xpar,x_n)\,dx_n,
\end{equation}
the transform on the right being that of $\Rnm$.
\end{lemma}

\begin{proof}
The dyadic shells of $\Rn$ intersect $\Rnp$ exactly in the sets $\shellp{j}$ of
\eqref{eq:B-norms}, so the norm identity is immediate from
Definition~\ref{def:AH-spaces}. For \eqref{eq:reflection-formula}, carry out the
$\xpar$-integration in \eqref{eq:Ftilde-def} first. Since $\gphi_{\xi_n}$ is
real,
\begin{equation*}
  (\Ftbulk f)(\xi)
  = \int_0^{\infty} \gphi_{\xi_n}(x_n)\,\widehat{f}(\xipar,x_n)\,dx_n,
\end{equation*}
with $\widehat{f}$ the partial transform \eqref{eq:fourier-partial}. Inserting
\eqref{eq:phi-exponential} and using
\[
  \int_0^{\infty}e^{\mp i\xi_n x_n}\widehat{f}(\xipar,x_n)\,dx_n
  = \sqrt{2\pi}\,\widehat{\zext{f}}(\xipar,\pm\xi_n)
\]
gives \eqref{eq:reflection-formula}. Formula \eqref{eq:surface-formula} is
\eqref{eq:Psis-def} with the $x_n$-integration performed first, $\psizero$ being
real.
\end{proof}

\begin{proposition}[$\Bsp$-boundedness]
\label{prop:F-bounded}
Let $E \in \Ereg$. Then $\Fzero(E) \in \Lop{\Bsp}{\Echan{E}}$, with
\begin{equation}
  \label{eq:F-bound}
\begin{aligned}
  \norm{\Fbulk(E)}_{\Lop{\Bsp}{\Ltwo{\Sm}}}^2
    &\le \bigl(\CAH{n}\bigr)^2\,E^{-\frac12},
  \\[1mm]
  \norm{\Fsurf(E)}_{\Lop{\Bsp}{\Ltwo{\Ssurf}}}^2
    &\le \tfrac12\,\Csharp^2\,\bigl(\CAH{n-1}\bigr)^2\,
       (E+\alpha^2)^{-\frac12} .
\end{aligned}
\end{equation}
In particular $\norm{\Fzero(E)}$ is bounded uniformly on compact subsets of
$\Ereg$, and $E \mapsto \Fzero(E)f$ is continuous from each of $(-\alpha^2,0)$
and $(0,\infty)$ into the corresponding channel space, for every $f \in \Bsp$.
\end{proposition}

\begin{proof}
By \eqref{eq:reflection-formula} and the inequality
$|a+b|^2 \le 2|a|^2 + 2|b|^2$, together with $|\cph{\cdot}| = 1$,
\begin{equation*}
\begin{aligned}
  \int_{\Sm} \bigl| (\Ftbulk f)(\sqrt{E}\,\theta) \bigr|^2\,d\theta
  &\le 2 \int_{\Sm}
    \Bigl( \bigl| \widehat{\zext{f}}(\sqrt{E}\,\theta) \bigr|^2
         + \bigl| \widehat{\zext{f}}(\sqrt{E}\,\refext{\theta}) \bigr|^2
    \Bigr) d\theta
  \\
  &= 2 \int_{S^{n-1}}
      \bigl| \widehat{\zext{f}}(\sqrt{E}\,\theta) \bigr|^2 \,d\theta,
\end{aligned}
\end{equation*}
since $\theta \mapsto \refext{\theta}$ carries $\Sm$ onto the complementary
hemisphere. Rewriting the last integral as a surface integral over
$\{|\xi| = \sqrt{E}\}$ costs a factor $E^{-\frac{n-1}{2}}$, and
\eqref{eq:AH-restriction} in dimension $n$ together with
$\norm{\zext{f}}_{\Bspd{\Rn}} = \Bnorm{f}$ gives
\begin{align*}
  \norm{\Fbulk(E)f}_{\Ltwo{\Sm}}^2
  &= \frac{E^{\frac{n-2}{2}}}{2}
    \int_{\Sm} \bigl| (\Ftbulk f)(\sqrt{E}\,\theta) \bigr|^2\,d\theta \\
  &\le E^{\frac{n-2}{2}}\,E^{-\frac{n-1}{2}}\,
      \bigl(\CAH{n}\bigr)^2\,\Bnorm{f}^2
   = \bigl(\CAH{n}\bigr)^2\,E^{-\frac12}\,\Bnorm{f}^2 .
\end{align*}
For the surface channel, Corollary~\ref{cor:tensor} gives
$\norm{T_{\psizero}f}_{\Bspd{\Rnm}} \le \Csharp\Bnorm{f}$, since
$\norm{\psizero}_{\Ltwo{\Rhalf}} = 1$. Applying \eqref{eq:AH-restriction} in
dimension $n-1$ to $T_{\psizero}f$ at the radius $\sqrt{E+\alpha^2}$ and
proceeding as above yields the second bound in \eqref{eq:F-bound}. The
continuity statement follows from the continuity assertion in
Theorem~\ref{thm:AH-restriction}, the maps $E \mapsto E^{(n-2)/4}$ and
$E \mapsto \cph{\sqrt{E}\,\theta_n}$ being continuous as well.
\end{proof}

Both bounds in \eqref{eq:F-bound} degenerate as $E \to 0^{+}$ and as
$E \to -\alpha^{2+}$ respectively. They do so with the same exponent
$-\frac12$, and in the same two dimensions $n$ and $n-1$, as the resolvent
bound \eqref{eq:lap-explicit}. By Theorem~\ref{thm:stone} below the two
are forms of the same estimate, and in both the channels are separated by the exact
splitting \eqref{eq:splitting-preview}.
Accordingly, and as in Theorem~\ref{thm:lap}, everything that follows is carried
out on compact subsets of $\Ereg$.

\begin{proposition}[The adjoint and the eigenoperator property]
\label{prop:F-adjoint}
Let $E \in \Ereg$. Then $\Fzero(E)^{*} \in \Lop{\Echan{E}}{\Bst}$, and for
$g = (g_{\mathrm{bulk}},g_{\mathrm{surf}}) \in \Echan{E}$,
\begin{align}
  \label{eq:Fbulk-adjoint}
  \bigl( \Fbulk(E)^{*}g_{\mathrm{bulk}} \bigr)(x)
    &= \frac{E^{\frac{n-2}{4}}}{\sqrt2}
       \int_{\Sm} \Psib(x,\sqrt{E}\,\theta)\,g_{\mathrm{bulk}}(\theta)
       \,d\theta,
  \\[1mm]
  \label{eq:Fsurf-adjoint}
  \bigl( \Fsurf(E)^{*}g_{\mathrm{surf}} \bigr)(x)
    &= \frac{(E+\alpha^2)^{\frac{n-3}{4}}}{\sqrt2}
       \int_{\Ssurf} \Psis(x,\sqrt{E+\alpha^2}\,\omega)\,
       g_{\mathrm{surf}}(\omega)\,d\omega,
\end{align}
with $\Fzero(E)^{*}g = \Fbulk(E)^{*}g_{\mathrm{bulk}}
+ \Fsurf(E)^{*}g_{\mathrm{surf}}$, a sum in $\Bst$ and not an orthogonal one.
Moreover $u = \Fzero(E)^{*}g$ satisfies
\begin{equation}
  \label{eq:eigenoperator}
  (-\Delta - E)u = 0 \ \text{ in } \Rnp,
  \qquad
  \gamma_1 u - \alpha\gamma_0 u = 0 \ \text{ on } \Gamma ;
\end{equation}
that is, $\Fzero(E)^{*}$ is an eigenoperator of $\Halpha$.
\end{proposition}

\begin{proof}
Formulas \eqref{eq:Fbulk-adjoint}--\eqref{eq:Fsurf-adjoint} are obtained by
transposing \eqref{eq:Fbulk-def}--\eqref{eq:Fsurf-def} against the pairing
$(\cdot,\cdot)$, the interchange of integrations being justified for $g$
continuous and then extended by density. That the result lies in $\Bst$ follows
from \eqref{eq:AH-extension}: inserting \eqref{eq:Psib-exponential} into
\eqref{eq:Fbulk-adjoint} exhibits $\Fbulk(E)^{*}g_{\mathrm{bulk}}$, as a
function of $x \in \Rnp$, as the sum of two integrals of the form
\eqref{eq:AH-extension} over hemispheres of $S^{n-1}$ at radius $\sqrt{E}$, one
in the variable $x$ and one in the variable $\refext{x} = (\xpar,-x_n)$. Both
are restrictions to $\Rnp$ of elements of $\Bstd{\Rn}$, and restriction does not
increase the $\Bst$-norm. For \eqref{eq:Fsurf-adjoint} one applies
\eqref{eq:AH-extension} in dimension $n-1$ and then the second half of
\eqref{eq:tensor-bounds} with $\Psi = \psizero$. Finally
\eqref{eq:eigenoperator} holds because each $\Psib(\cdot,\sqrt{E}\theta)$ and
each $\Psis(\cdot,\sqrt{E+\alpha^2}\,\omega)$ does, by
Proposition~\ref{prop:eigenfunctions}(i), and the integrals converge locally
uniformly together with all derivatives.
\end{proof}

We can now identify the abstract objects of \S\ref{sec:gft-abstract}.

\begin{theorem}[Stone's formula]
\label{thm:stone}
For every $E \in \Ereg$,
\begin{equation}
  \label{eq:stone}
  \frac{1}{2\pi i}\bigl( \Ralpha{E+i0} - \Ralpha{E-i0} \bigr)
  = \Fzero(E)^{*}\,\Fzero(E)
  \qquad \text{in } \Lop{\Bsp}{\Bst}.
\end{equation}
Consequently the abstract representation space constructed in
\S\ref{sec:gft-abstract} is the channel space \eqref{eq:channel-space}, and for
all $f \in \Bsp$,
\begin{equation}
  \label{eq:flux}
  \im \bigl( \Ralpha{E+i0}f, f \bigr)
  = \pi\,\norm{\Fzero(E)f}_{\Echan{E}}^2 .
\end{equation}
\end{theorem}

\begin{proof}
Fix $f,g \in \Bsp$ and an interval $(a,b)$ with
$\overline{(a,b)} \subset \Ereg$. By Proposition~\ref{prop:A-properties}(iv),
the pairing $(P((a,b))f,g)$ equals $\int_a^b (\Adens{E}f,g)\,dE$. On the other
hand $\Fzero$ is
unitary and diagonalizes $\Halpha$
(Theorem~\ref{thm:spectral-representation}), so that
$P((a,b)) = \Fzero^{*}\mathbf{1}_{(a,b)}\Fzero$ and therefore
\begin{equation*}
  \bigl( P((a,b))f, g \bigr)
  = \int_a^b \bigl( \Fzero(E)f, \Fzero(E)g \bigr)_{\Echan{E}}\,dE
  = \int_a^b \bigl( \Fzero(E)^{*}\Fzero(E)f, g \bigr)\,dE .
\end{equation*}
Since $(a,b)$ was arbitrary, the two integrands agree for a.e.\ $E \in \Ereg$.
Both are continuous on each of $(-\alpha^2,0)$ and $(0,\infty)$, the first by
Proposition~\ref{prop:A-properties}(ii) and the second by
Proposition~\ref{prop:F-bounded}. Hence they agree for every $E \in \Ereg$. As
$f$ and $g$ range over $\Bsp$ this is \eqref{eq:stone}. Identity
\eqref{eq:flux} follows on taking $g = f$ and using
$\Adens{E} = \pi^{-1}\im\Ralpha{E+i0}$ from
Proposition~\ref{prop:A-properties}(i).
\end{proof}

Identity \eqref{eq:flux} is the energy-flux identity of the problem: the
imaginary part of the resolvent at the boundary value measures the
$\Echan{E}$-norm of the scattering data, split into a bulk part and a surface
part. It is in this form that Theorem~\ref{thm:stone} is used in
\S\ref{subsec:radiation}, where the vanishing of the left-hand side for a
solution satisfying the radiation condition forces both channels on the right to
vanish, and again in \S\ref{sec:inverse}. Note also that the limiting absorption
principle enters the proof only through Proposition~\ref{prop:A-properties}: the
construction of $\Fzero$ itself, in
\S\S\ref{sec:gft-eigenfunctions}--\ref{sec:gft-transform}, is independent of
\S\ref{subsec:lap} and rests on Appendix~\ref{app:sa-spec} alone.

%% file: Part1_024_asymptotics.tex
\label{subsec:asymptotics}

The limiting absorption principle (Theorem~\ref{thm:lap}) places
$u = \Ralpha{E+i0}f$ in $\Bst \cap H^2_{\mathrm{loc}}(\overline{\Rnp})$ and
identifies, through Stone's formula (Theorem~\ref{thm:stone}), the spectral
density of $\Halpha$. It says nothing about how $u$ propagates to infinity.
This subsection supplies that information. The half-space carries two channels
with distinct effective wavenumbers, $\kbulk$ in the interior and $\ksurf$
along the boundary, and the far field of $u$ is a superposition of an outgoing
bulk wave and an outgoing surface wave, the first spread over $\overline{\Rnp}$
and the second confined to a boundary layer with the fixed transverse profile
$\psizero$.

We proceed in two steps. First we derive a \emph{pointwise} expansion of $u$.
Because the two channels occupy geometrically incompatible parts of space, it
is naturally stated region by region, as in the Green-function analysis of
Dur\'an, Muga and N\'ed\'elec \cite{DuranMugaNedelec2009} for $n = 3$, which it
recovers. Secondly, we show that the regional pictures glue into a
\emph{single} expansion valid on all of $\Rnp$, in the sense of the space
$\Bstz$. This second step is where the Agmon--H\"ormander framework becomes essential. Expansions of this
last form are due to Isozaki, Kadowaki and Watanabe
\cite{IsozakiKadowakiWatanabe2020,IsozakiKadowakiWatanabe2023}, and we claim no
novelty for the shape of Theorem~\ref{thm:asymp-uniform}. What is different is
the route: the uniform statement is insensitive to the transition zone between
the channels, and, as the proof shows, the expansion itself requires no
analysis of the grazing directions (Remark~\ref{rem:no-grazing}). It is also
the form in which the intrinsic radiation condition of \S\ref{subsec:radiation}
is stated (Theorem~\ref{thm:rad-equivalence}(b)). The two-region condition of
\cite{DuranMugaNedelec2009}, in which the two regions are separated by a fixed
power of $r$ with exponent in $(1/4,1/2)$, is shown there to be equivalent to
it for every exponent in $(0,1)$.

The precise dictionary with \cite{DuranMugaNedelec2009} for $n = 3$ is
recorded in Remark~\ref{rem:dmn-comparison}, once the objects it refers to have
been constructed.

Everything below rests on two explicitly displayed integrals and two classical
devices: the asymptotic expansion of the Hankel function
(Lemma~\ref{lem:hankel}), applied in dimension $n$ to the Dirichlet part and in
dimension $n-1$ to the surface part, and the stationary-phase formula
(Lemma~\ref{lem:statphase}), applied to the reflected integral
\eqref{eq:v-def}. A third and more elementary device, integration by parts in
the normal wavenumber (Lemma~\ref{lem:branch}), disposes of the circle
$|\xipar| = \kbulk$ on which the two branches of that variable meet.

\subsubsection{Notation and the relation \texorpdfstring{$\simeq$}{≃}}
\label{sec:asymp-notation}

Throughout, $x = (\xpar,x_n) \in \Rnp$ and
\begin{equation}
  \label{eq:asymp-polar}
  r = |x|,
  \qquad
  \rho = |\xpar|,
  \qquad
  \theta = \frac{x}{r} \in \Sm,
  \qquad
  \omega = \frac{\xpar}{\rho} \in \Ssurf,
\end{equation}
so that $\theta = (\sqrt{1-\theta_n^2}\,\omega,\ \theta_n)$ with
$\theta_n = x_n/r \in [0,1]$. The directions with $\theta_n$ small are called
\emph{grazing}. We keep the standing assumption $n \ge 3$ of
Remark~\ref{rem:standing-dimension}, and we fix $E > 0$ once and for all, so
that both channels are open. The bulk wavenumber $\kbulk$ and the surface
wavenumber $\ksurf$ never coincide:
\begin{equation}
  \label{eq:asymp-wavenumbers}
  \ksurf - \kbulk = \frac{\alpha^2}{\ksurf + \kbulk} > 0 .
\end{equation}
By Theorem~\ref{thm:lap}(e), $u = \Ralpha{E+i0}f$ solves
$(-\Delta - E)u = f$ in $\Rnp$ under the boundary condition
$\gamma_1 u - \alpha\gamma_0 u = 0$.

Two Fourier transforms occur, and in this subsection we keep them
typographically apart. The partial transform in the tangential variable is
written $\mathcal{F}'$, which \eqref{eq:fourier-partial} also allows one to
write with a hat. Here the plain $\widehat{\phantom{f}}$ is reserved for a full
transform, on $\Rn$ or on $\Rnm$ according to its argument. Further,
$\zext{g}$ is the extension of $g$ by zero to $\Rn$,
$\refext{\xi} = (\xipar,-\xi_n)$ is the reflection of
Lemma~\ref{lem:reflection-formula}, $f_{\mathrm{odd}}$ is the odd extension of
$f$ across $\Gamma$, so that
$\widehat{f_{\mathrm{odd}}}(\xi) = \widehat{\zext{f}}(\xi)
- \widehat{\zext{f}}(\refext{\xi})$, and $T_{\psizero}$ is the map of
\eqref{eq:surface-formula}.

The far-field statements are made in the space $\Bstz$ of \eqref{eq:Bzero}.
Recall from \eqref{eq:Bstar-equivalent} that the $\Bst$-norm is equivalent to
$\bigl(\sup_{R>1}R^{-1}\int_{\{|x|<R\}\cap\Rnp}|u|^2\bigr)^{1/2}$ and that
$\Bstz$ is the subspace on which this shell average tends to $0$.

\begin{definition}[Asymptotic equivalence]
\label{def:simeq}
For $u,v \in \Bst$ we write $u \simeq v$ when $u - v \in \Bstz$, that is,
\begin{equation}
  \label{eq:simeq}
  \lim_{R \to \infty} \frac{1}{R}
    \int_{\{|x| < R\} \cap \Rnp} |u(x) - v(x)|^2\,dx = 0 .
\end{equation}
\end{definition}

Two functions are equivalent when their difference is negligible in every
shell on average \cite{AgmonHormander1976, Isozaki2001}. This is weaker than
equality of pointwise leading terms, and it is this weaker notion that will
allow a single formula to hold across the boundary layer and across the
grazing directions.

\subsubsection{The resolvent in the fiber representation}
\label{sec:asymp-fiber}

By Proposition~\ref{prop:fibration} the partial Fourier transform
$\mathcal{F}'$ diagonalizes $\Halpha$ into the direct integral of the half-line
Robin operators $\hfib(\xipar)$. Consequently, for $z \in
\mathbb{C}\setminus\mathbb{R}$ and $f \in \Ltwo{\Rnp}$,
\begin{equation}
  \label{eq:fiber-rep}
  \mathcal{F}'\bigl(\Ralpha{z}f\bigr)(\xipar,x_n)
  = \int_0^{\infty}
      g(x_n,y_n;\xipar,z)\,
      \bigl(\mathcal{F}'f\bigr)(\xipar,y_n)\,dy_n ,
\end{equation}
where $g$ is the fiber Green kernel of \eqref{eq:fiber-green}. Every
statement of this subsection is obtained from \eqref{eq:fiber-rep} by
inserting one of the three summands of the exact splitting
\eqref{eq:fiber-splitting} and analysing the resulting integral. It is
convenient to record that splitting in a variable adapted to propagation.

\begin{lemma}[The fiber kernel in the normal wavenumber]
\label{lem:fiber-mu}
For $\xipar \in \Rnm$ and $z \in \mathbb{C}\setminus\mathbb{R}$ put
\begin{equation}
  \label{eq:mu-def}
  \mu = \mu(\xipar,z) = \sqrt{z - |\xipar|^2},
  \qquad \im\mu > 0,
\end{equation}
so that $\mu = i\kappa$ with $\kappa$ as in Lemma~\ref{lem:fiber-green}. In
this variable the fiber kernel \eqref{eq:fiber-green} is a free wave plus a
wave reflected from the boundary,
\begin{equation}
  \label{eq:fiber-kernel}
  g(x_n,y_n;\xipar,z)
  = \frac{i}{2\mu}
    \Bigl[
      e^{i\mu|x_n - y_n|}
      + \refl{\mu}\;e^{i\mu(x_n + y_n)}
    \Bigr],
  \qquad
  \refl{\mu} = \frac{i\mu - \alpha}{i\mu + \alpha},
\end{equation}
the coefficient $\refl{\cdot}$ being the one of
\eqref{eq:phase-and-reflection}, continued to $\{\im\mu \ge 0\}$. The Dirichlet
kernel $g_{\mathrm{D}}$ is \eqref{eq:fiber-kernel} with $\refl{\mu}$ replaced
by $-1$, and
\begin{equation}
  \label{eq:fiber-difference}
  g - g_{\mathrm{D}}
  = -\,\frac{e^{i\mu(x_n+y_n)}}{i\mu + \alpha}
  = \frac{(\alpha - i\mu)\,e^{i\mu(x_n+y_n)}}
         {|\xipar|^2 - (z+\alpha^2)} .
\end{equation}
The three summands of \eqref{eq:fiber-splitting} are $g_{\mathrm{D}}$, the
surface term $\psizero(x_n)\psizero(y_n)\bigl(|\xipar|^2-(z+\alpha^2)\bigr)^{-1}$,
and the remainder \eqref{eq:K-kernel}, which reads
\begin{equation}
  \label{eq:K-explicit}
  k(\xipar,x_n,y_n;z)
  = \frac{(\alpha-i\mu)\,e^{i\mu(x_n+y_n)}
          - \psizero(x_n)\psizero(y_n)}
         {|\xipar|^2 - (z+\alpha^2)} .
\end{equation}
\end{lemma}

\begin{proof}
Since $\re\kappa > 0$, the root $i\kappa$ has positive imaginary part, so
$\mu = i\kappa$ and $\kappa = -i\mu$. Substituting this in
\eqref{eq:fiber-green} gives $(2\kappa)^{-1} = i(2\mu)^{-1}$,
$e^{-\kappa t} = e^{i\mu t}$ and
\begin{equation*}
  r(\kappa) = \frac{\kappa+\alpha}{\kappa-\alpha}
            = \frac{\alpha - i\mu}{-(\alpha + i\mu)}
            = \frac{i\mu-\alpha}{i\mu+\alpha}
            = \refl{\mu},
\end{equation*}
which is \eqref{eq:fiber-kernel}. For real $\mu = \xi_n > 0$ this is the
coefficient of \eqref{eq:phase-and-reflection}. Subtracting the Dirichlet
kernel,
\begin{equation*}
  g - g_{\mathrm{D}}
  = \frac{i}{2\mu}\bigl(\refl{\mu} + 1\bigr)e^{i\mu(x_n+y_n)}
  = \frac{i}{2\mu}\cdot\frac{2i\mu}{i\mu+\alpha}\,e^{i\mu(x_n+y_n)}
  = -\,\frac{e^{i\mu(x_n+y_n)}}{i\mu+\alpha},
\end{equation*}
and multiplying numerator and denominator by $\alpha - i\mu$ gives the second
form of \eqref{eq:fiber-difference}, because
$(i\mu+\alpha)(\alpha-i\mu) = \alpha^2 + \mu^2
= -\bigl( |\xipar|^2 - (z+\alpha^2) \bigr)$, the two signs cancelling.
Finally \eqref{eq:K-kernel} becomes
\eqref{eq:K-explicit} on writing $\kappa + \alpha = \alpha - i\mu$,
$e^{-\kappa(x_n+y_n)} = e^{i\mu(x_n+y_n)}$ and
$2\alpha e^{-\alpha(x_n+y_n)} = \psizero(x_n)\psizero(y_n)$.
\end{proof}

The variable $\mu$ is the wavenumber in the normal direction: by
\eqref{eq:fiber-kernel} the kernel is built from the plane waves
$e^{\pm i\mu x_n}$, and the same $\mu$ will appear as the normal component of
the phase in \eqref{eq:v-def}, ready for stationary phase. The decaying
variable $\kappa$ of \S\ref{subsec:lap} is better adapted to uniform bounds,
and the two are used side by side.

\subsubsection{The boundary value \texorpdfstring{$z \to E+i0$}{z -> E+i0}}
\label{sec:asymp-boundary}

Fix $E > 0$. As $z = E + i\varepsilon$ with $\varepsilon \downarrow 0$, the
root \eqref{eq:mu-def} has the two branches
\begin{equation}
  \label{eq:mu-branches}
  \mu = \sqrt{E - |\xipar|^2} > 0
  \quad (|\xipar| < \kbulk),
  \qquad
  \mu = i\tau,
  \quad \tau = \sqrt{|\xipar|^2 - E} > 0
  \quad (|\xipar| > \kbulk) .
\end{equation}
We call $\{|\xipar| < \kbulk\}$ the \emph{propagating disc} and its complement
the \emph{evanescent region}. On the first the fiber kernel oscillates in
$x_n$, on the second it decays.

The denominator $|\xipar|^2 - (z+\alpha^2)$ of
\eqref{eq:fiber-difference}--\eqref{eq:K-explicit} does not vanish for
$\im z > 0$. Its zero descends to the sphere $|\xipar| = \ksurf$ from the lower
half-plane as $\varepsilon \downarrow 0$, so at the boundary value it is to be
read as
\begin{equation}
  \label{eq:v-pole}
  \frac{1}{|\xipar|^2 - (E+\alpha^2) - i0}
  = \operatorname{p.v.}\frac{1}{|\xipar|^2 - \ksurf^2}
    + i\pi\,\delta\bigl(|\xipar|^2 - \ksurf^2\bigr) .
\end{equation}
By \eqref{eq:asymp-wavenumbers} the singular sphere lies in the evanescent
region, where $|\xipar| = \ksurf$ forces $\mu = i\alpha$ and hence
\begin{equation}
  \label{eq:polar-part}
  (\alpha - i\mu)\,e^{i\mu(x_n+y_n)}\Big|_{\mu = i\alpha}
  = 2\alpha\,e^{-\alpha(x_n+y_n)}
  = \psizero(x_n)\,\psizero(y_n) .
\end{equation}
This is the purpose of the splitting: the residue of
\eqref{eq:fiber-difference} at the pole is exactly the surface term, so the
numerator of \eqref{eq:K-explicit} vanishes there and $k$ is regular across
$|\xipar| = \ksurf$. In particular the bound $|k| \le \alpha^{-1}$ used in
\S\ref{subsec:lap} survives the limit, and the abstract remainder $\Krem{z}$
of \eqref{eq:splitting-preview} has become an explicit kernel.

\begin{proposition}[The three terms at the boundary value]
\label{prop:three-terms}
Let $E > 0$ and let $f \in \Bsp$. Then the boundary value
$u = \Ralpha{E+i0}f \in \Bst$ of Theorem~\ref{thm:lap} splits as
\begin{equation}
  \label{eq:u-split}
  u = u_{\mathrm{D}} + \widetilde{u}_{\mathrm{surf}} + \Krem{E+i0}f,
\end{equation}
where
\begin{equation}
  \label{eq:usurf-exact}
  u_{\mathrm{D}}
    = \Bigl(\Rfree{n}{E+i0}\,f_{\mathrm{odd}}\Bigr)\Big|_{\Rnp},
  \qquad
  \widetilde{u}_{\mathrm{surf}}(x)
    = \Bigl(\Rfree{n-1}{E+\alpha^2+i0}\,T_{\psizero}f\Bigr)(\xpar)\;
      \psizero(x_n),
\end{equation}
each of the three limits existing separately in $\Bst$.
\end{proposition}

\begin{proof}
Apply Proposition~\ref{prop:splitting} at $z = E+i\varepsilon$ and let
$\varepsilon \downarrow 0$. The left-hand side converges in $\Bst$ by
Theorem~\ref{thm:lap}. On the right, $\Krem{z}f$ converges by
Proposition~\ref{prop:K-bounded}, and the two free resolvents converge by the
Agmon--H\"ormander limiting absorption principle on $\Rn$ and on $\Rnm$,
applied at the energies $E > 0$ and $E + \alpha^2 > 0$ respectively. Here
$f_{\mathrm{odd}} \in \Bspd{\Rn}$ and $T_{\psizero}f \in \Bspd{\Rnm}$ by
Lemma~\ref{lem:reflection-formula} and Corollary~\ref{cor:tensor}. The
identification of the first term uses that the Dirichlet Laplacian on $\Rnp$ is
carried to $-\Delta$ on $\Rn$ by odd reflection. The identification of the
second is \eqref{eq:surface-resolvent} together with
$\bigl(I\otimes\Psurf\bigr)f = \bigl(T_{\psizero}f\bigr)\otimes\psizero$.
\end{proof}

\subsubsection{The reflected integral}
\label{sec:asymp-reflected}

The Dirichlet part of \eqref{eq:u-split} is a free wave and needs no further
preparation. The other two terms are best handled together, as the two parts
of a single oscillatory integral.

\begin{proposition}[The reflected integral]
\label{prop:reflected}
Let $f \in C_c^{\infty}(\Rnp)$ and $E > 0$, and put
$v = u - u_{\mathrm{D}} = \widetilde{u}_{\mathrm{surf}} + \Krem{E+i0}f$. Then,
for every $x \in \Rnp$,
\begin{equation}
  \label{eq:v-def}
  v(x)
  = -\,(2\pi)^{-\frac{n-2}{2}}
    \int_{\Rnm}
      \frac{\widehat{\zext{f}}\bigl(\xipar,-\mu(\xipar)\bigr)}
           {i\mu(\xipar) + \alpha}\;
      e^{i(\xpar\cdot\xipar + \mu(\xipar)x_n)}\,d\xipar,
\end{equation}
with $\mu$ as in \eqref{eq:mu-branches} and the quotient understood through
\begin{equation*}
  \frac{1}{i\mu+\alpha}
  = -\,\frac{\alpha - i\mu}{|\xipar|^2 - \ksurf^2 - i0}
\end{equation*}
and \eqref{eq:v-pole}. The integral converges absolutely away from the sphere
$|\xipar| = \ksurf$, and the contribution of the singular term in
\eqref{eq:v-pole} is exactly $\widetilde{u}_{\mathrm{surf}}$.
\end{proposition}

\begin{proof}
Insert \eqref{eq:fiber-difference} into \eqref{eq:fiber-rep}:
\begin{equation*}
  \mathcal{F}'v(\xipar,x_n)
  = -\,\frac{e^{i\mu x_n}}{i\mu+\alpha}
      \int_0^{\infty} e^{i\mu y_n}\,
        \bigl(\mathcal{F}'f\bigr)(\xipar,y_n)\,dy_n
  = -\,\frac{\sqrt{2\pi}}{i\mu+\alpha}\,
      \widehat{\zext{f}}(\xipar,-\mu)\;e^{i\mu x_n},
\end{equation*}
where the identity
$\int_0^{\infty}e^{-i\xi_n y_n}\bigl(\mathcal{F}'f\bigr)(\xipar,y_n)\,dy_n
= \sqrt{2\pi}\,\widehat{\zext{f}}(\xipar,\xi_n)$, which is Fubini's theorem and
appears in the proof of Lemma~\ref{lem:reflection-formula}, has been used with
$\xi_n = -\mu$.
Inverting $\mathcal{F}'$ and using
$(2\pi)^{-\frac{n-1}{2}}\sqrt{2\pi} = (2\pi)^{-\frac{n-2}{2}}$ gives
\eqref{eq:v-def}. On the evanescent branch $\mu = i\tau$ one has, for every
$N$,
\begin{equation}
  \label{eq:PW-halfplane}
  \bigl|\widehat{\zext{f}}(\xipar,-i\tau)\bigr|
  \le C_N\,(1+|\xipar|)^{-N}
  \qquad (\tau \ge 0),
\end{equation}
with $C_N$ independent of $\tau$. This is the half-space form of the
Paley--Wiener theorem. Indeed
$\widehat{\zext f}(\xipar,-i\tau) = (2\pi)^{-\frac{n}{2}}\int_{\Rnp}
e^{-i\xpar\cdot\xipar}\,e^{-\tau y_n}f(y)\,dy$, and $e^{-\tau y_n} \le 1$ on the
support of $f$, so that integration by parts in $\xpar$ alone produces the
decay: the transform of a function vanishing for $y_n < 0$ is bounded on the
lower half-plane $\{\im\xi_n \le 0\}$, where $\xi_n = -\mu$ lies. Since
$|e^{i\mu x_n}| = e^{-\tau x_n} \le 1$ for $x_n \ge 0$, the integral converges
absolutely at every $x \in \Rnp$. On the propagating disc the integrand is
bounded. Finally, at $\mu = i\alpha$,
\begin{equation*}
  \widehat{\zext{f}}(\xipar,-i\alpha)
  = (2\pi)^{-\frac{n}{2}}\int_{\Rnp}
      e^{-i\xpar\cdot\xipar}e^{-\alpha y_n}f(y)\,dy
  = (2\pi)^{-\frac12}(2\alpha)^{-\frac12}\,
    \widehat{T_{\psizero}f}(\xipar),
\end{equation*}
so that the singular term of \eqref{eq:v-pole} contributes
\begin{equation*}
  (2\pi)^{-\frac{n-1}{2}}\psizero(x_n)
  \int_{\Rnm}
    \frac{\widehat{T_{\psizero}f}(\xipar)}
         {|\xipar|^2 - \ksurf^2 - i0}\,
    e^{i\xpar\cdot\xipar}\,d\xipar,
\end{equation*}
which is $\widetilde{u}_{\mathrm{surf}}$ as given in \eqref{eq:usurf-exact}.
\end{proof}

Formula \eqref{eq:v-def} is the half-space counterpart of Sommerfeld's
classical representation of the field of a vertical dipole over a plane
earth \cite[\S3]{Sommerfeld1909}, developed at length in his lectures
\cite[Chap.~VI, \S32]{Sommerfeld1949}.
In Sommerfeld's problem the earth is a second medium with a large complex
refractive index. His transmission condition then reduces approximately to
a Robin condition with a complex coefficient, and Sommerfeld himself
compares this approximation with the Robin problem of heat conduction
\cite[\S32, (10b)--(10d)]{Sommerfeld1949}. Our real $\alpha>0$ is the
lossless counterpart. Sommerfeld also rewrote his integral as a
superposition of homogeneous and inhomogeneous plane waves
\cite[\S11]{Sommerfeld1909}, which is the form \eqref{eq:v-def} takes here.
We shall refer to it as the \emph{reflected integral}.
The pole of the second factor is Sommerfeld's
\cite[\S5]{Sommerfeld1909}, \cite[\S32, (15)--(19)]{Sommerfeld1949},
and its residue is the surface wave.

Thus the resolvent is a sum of three explicit objects. They are a free
resolvent in $\Rn$ applied to the odd extension of $f$, a free resolvent in
$\Rnm$ at the shifted energy $E + \alpha^2$ tensored with $\psizero$, and the
reflected integral with the surface pole removed. The first two are handled by
Lemma~\ref{lem:hankel} below, the third by Lemmas~\ref{lem:statphase} and
\ref{lem:branch}.

\subsubsection{Three far-field lemmas}
\label{sec:asymp-lemmas}

\begin{lemma}[Hankel asymptotics of the free resolvent]
\label{lem:hankel}
Let $m \ge 2$, $k > 0$. The outgoing free resolvent
$\Rfree{m}{k^2+i0}$ on $\mathbb{R}^m$ is convolution with
\begin{equation}
  \label{eq:hankel-kernel}
  G_m(X;k)
  = \frac{i}{4}
    \Bigl(\frac{k}{2\pi|X|}\Bigr)^{\frac{m-2}{2}}
    H^{(1)}_{\frac{m-2}{2}}\bigl(k|X|\bigr),
\end{equation}
which for $m = 3$ is $e^{ik|X|}/(4\pi|X|)$ and for $m = 2$ is
$\tfrac{i}{4}H^{(1)}_0(k|X|)$. Consequently, for
$g \in C_c^{\infty}(\mathbb{R}^m)$ and $\varsigma = X/|X|$,
\begin{equation}
  \label{eq:hankel-far}
  \bigl(\Rfree{m}{k^2+i0}g\bigr)(X)
  = \sqrt{\frac{\pi}{2}}\;k^{\frac{m-3}{2}}\,
    e^{-i\frac{(m-3)\pi}{4}}\;
    \widehat{g}(k\varsigma)\;
    |X|^{-\frac{m-1}{2}}e^{ik|X|}
    \bigl(1 + O(|X|^{-1})\bigr)
\end{equation}
as $|X| \to \infty$, uniformly in $\varsigma \in S^{m-1}$, and the same holds
for each first derivative after differentiation of the leading term.
\end{lemma}

\begin{proof}
Formula \eqref{eq:hankel-kernel} is the standard Fourier inversion of
$(|\xi|^2 - k^2 - i0)^{-1}$ \cite[Ch.~13]{Watson1944}. Insert the asymptotic
expansion
$H^{(1)}_{\nu}(t) = \sqrt{2/(\pi t)}\;e^{i(t - \nu\pi/2 - \pi/4)}
\bigl(1 + O(t^{-1})\bigr)$ with $\nu = (m-2)/2$, and use
$|X - Y| = |X| - \varsigma\cdot Y + O(|X|^{-1})$ uniformly for $Y$ in a
compact set:
\begin{equation*}
  \int G_m(X-Y;k)g(Y)\,dY
  = \frac{i}{4}\Bigl(\frac{k}{2\pi}\Bigr)^{\frac{m-2}{2}}
    \sqrt{\frac{2}{\pi k}}\;
    e^{i\left(k|X| - \frac{(m-1)\pi}{4}\right)}|X|^{-\frac{m-1}{2}}
    \int e^{-ik\varsigma\cdot Y}g(Y)\,dY
\end{equation*}
up to a relative error $O(|X|^{-1})$. Since
$\int e^{-ik\varsigma\cdot Y}g\,dY = (2\pi)^{m/2}\widehat{g}(k\varsigma)$ and
\begin{equation*}
  \frac{i}{4}\Bigl(\frac{k}{2\pi}\Bigr)^{\frac{m-2}{2}}
  \sqrt{\frac{2}{\pi k}}\,(2\pi)^{\frac{m}{2}}
  = i\sqrt{\frac{\pi}{2}}\;k^{\frac{m-3}{2}},
  \qquad
  i\,e^{-i\frac{(m-1)\pi}{4}} = e^{-i\frac{(m-3)\pi}{4}},
\end{equation*}
this is \eqref{eq:hankel-far}. For $m = 3$ the constant is
$(2\pi)^{3/2}/(4\pi) = \sqrt{\pi/2}$, as it must be.
\end{proof}

\begin{lemma}[Stationary phase for the reflected integral]
\label{lem:statphase}
Let $\delta \in (0,1)$ and let
$\Phi(\xipar) = \xpar\cdot\xipar + \mu(\xipar)x_n$ with $\mu$ as in
\eqref{eq:mu-branches}, on the propagating disc $\{|\xipar| < \kbulk\}$. For
$\theta_n \ge \delta$ the phase $\Phi$ has in that disc exactly one critical
point, namely
\begin{equation}
  \label{eq:statpoint}
  \xipar_{*} = \kbulk\,\theta',
  \qquad
  \mu(\xipar_{*}) = \kbulk\theta_n,
  \qquad
  \Phi(\xipar_{*}) = \kbulk\,r,
\end{equation}
it is nondegenerate, and
\begin{equation}
  \label{eq:hessian}
  \Phi''(\xipar_{*})
  = -\frac{r}{\kbulk}
    \Bigl(
      I_{n-1} + \frac{\theta'\otimes\theta'}{\theta_n^2}
    \Bigr),
  \qquad
  \det\Phi''(\xipar_{*})
  = \Bigl(-\frac{r}{\kbulk}\Bigr)^{n-1}\frac{1}{\theta_n^{2}},
  \qquad
  \operatorname{sgn}\Phi''(\xipar_{*}) = -(n-1).
\end{equation}
Consequently, if $A$ is smooth near $\xipar_{*}$ and $\chi$ is a cut-off equal
to $1$ on $\{|\xipar| \le \kbulk\sqrt{1-\delta^2/2}\}$ and supported in
$\{|\xipar| < \kbulk\sqrt{1-\delta^2/4}\}$, then
\begin{equation}
  \label{eq:statphase-result}
  \int_{\Rnm} \chi(\xipar)A(\xipar)e^{i\Phi(\xipar)}\,d\xipar
  = (2\pi)^{\frac{n-1}{2}}
    \Bigl(\frac{\kbulk}{r}\Bigr)^{\frac{n-1}{2}}\theta_n\,
    e^{-i\frac{(n-1)\pi}{4}}\,A(\xipar_{*})\,e^{i\kbulk r}
    \bigl(1 + O(r^{-1})\bigr),
\end{equation}
uniformly on $\{\theta_n \ge \delta\}$.
\end{lemma}

\begin{proof}
On the disc, $\nabla\Phi = \xpar - x_n\xipar/\mu$, which vanishes if and only
if $\xipar = \mu\xpar/x_n$. Combined with $|\xipar|^2 + \mu^2 = \kbulk^2$ this
forces $|\xipar| = \kbulk|\theta'|$ and hence \eqref{eq:statpoint}, and
$\Phi(\xipar_{*}) = \kbulk r(|\theta'|^2 + \theta_n^2) = \kbulk r$. From
$\nabla\mu = -\xipar/\mu$ and
$\nabla^2\mu = -\tfrac{1}{\mu}\bigl(I_{n-1} + \xipar\otimes\xipar/\mu^2\bigr)$
one gets
$\Phi'' = -\tfrac{x_n}{\mu}\bigl(I_{n-1} + \xipar\otimes\xipar/\mu^2\bigr)$,
which at $\xipar_{*}$ is \eqref{eq:hessian} since
$x_n/\mu(\xipar_{*}) = r/\kbulk$. The determinant follows from
$\det(I + w\otimes w) = 1 + |w|^2$ and $|\theta'|^2 + \theta_n^2 = 1$, and the
matrix is negative definite, whence the signature. Since $\theta_n \ge \delta$
gives $|\xipar_*| = \kbulk\sqrt{1-\theta_n^2} \le \kbulk\sqrt{1-\delta^2}$, the
critical point lies in the interior of $\{\chi = 1\}$, at distance at least
$c\delta^2\kbulk$ from its boundary. Now \eqref{eq:statphase-result} is the
stationary-phase formula in $n-1$ variables, with
$|\det\Phi''|^{-1/2} = (\kbulk/r)^{(n-1)/2}\theta_n$ and phase factor
$e^{i\pi\operatorname{sgn}\Phi''/4} = e^{-i(n-1)\pi/4}$.
\end{proof}

The propagating part of \eqref{eq:v-def} may equivalently be written as an
integral over the hemisphere: parametrizing the disc by $\xi = \kbulk\theta$,
$\theta \in \Sm$, one has $d\xipar = \kbulk^{n-1}\theta_n\,dS(\theta)$, and the
factor $\theta_n$ cancels against the $\mu^{-1} = (\kbulk\theta_n)^{-1}$
carried by the amplitude, so that the integral becomes a Herglotz wave with
amplitude smooth up to the equator. We use the disc form, in which the Hessian
is displayed explicitly.

The circle $|\xipar| = \kbulk$, where the two branches of \eqref{eq:mu-branches}
meet, is not covered by either lemma, and it is here that the classical
difficulties of the half-space problem are located. They are, however,
difficulties of the coordinate $|\xipar|$ rather than of the integral. As
$|\xipar|$ increases from $0$ to $\infty$, the normal wavenumber $\mu$ traverses
the path
\begin{equation}
  \label{eq:branch-contour}
  \Gamma:\qquad
  \kbulk \;\longrightarrow\; 0 \;\longrightarrow\; i\infty
\end{equation}
in the closed upper half-plane, consisting of a segment of the real axis
followed by the positive imaginary axis. Across the corner $\mu = 0$ of
$\Gamma$ the integrand of \eqref{eq:v-def} extends holomorphically. The next
lemma exploits nothing else. In particular no contour is deformed, so the
region of integration in the original variable is untouched.

\begin{lemma}[The branch circle is not a singularity]
\label{lem:branch}
Let $E > 0$, $\delta \in (0,1)$ and $0 < \epsilon \le \tfrac14\kbulk\delta$.
Put $s(\mu) = \sqrt{E - \mu^2}$, holomorphic on $\{|\mu| < \kbulk\}$, and, for
$x \in \Rnp$ and $\omega' \in \Ssurf$,
\begin{equation}
  \label{eq:branch-phase}
  \Phi_{\omega'}(\mu)
  = \rho\,s(\mu)\,(\omega'\cdot\theta') + \mu x_n .
\end{equation}
Let $b(\cdot,\omega')$ be holomorphic on $\{|\mu| < 2\epsilon\}$ for each
$\omega'$, with all $\mu$-derivatives bounded uniformly in $\omega'$, and let
$\zeta \in C_c^{\infty}\bigl((-\epsilon^2,\epsilon^2)\bigr)$ be equal to $1$
near $0$. Then
\begin{equation}
  \label{eq:branch-integral}
  I(x)
  = \int_{\Ssurf}\!\!
    \left[
      -\int_0^{\epsilon}\!\zeta(\mu^2)\,b(\mu,\omega')\,
         e^{i\Phi_{\omega'}(\mu)}\,d\mu
      + i\!\int_0^{\epsilon}\!\zeta(-\tau^2)\,b(i\tau,\omega')\,
         e^{i\Phi_{\omega'}(i\tau)}\,d\tau
    \right] dS(\omega')
\end{equation}
satisfies, for every $N \in \mathbb{N}$,
\begin{equation*}
  |I(x)| \le C_N\,r^{-N},
  \qquad x \in \Rnp,\quad x_n \ge \delta r,\quad r \ge 1,
\end{equation*}
with $C_N$ depending on $N$, $\delta$, $\epsilon$, $E$, $n$ and on finitely
many derivatives of $b$.
\end{lemma}

\begin{proof}
The bracket in \eqref{eq:branch-integral} is the integral of
$\zeta(\mu^2)b\,e^{i\Phi_{\omega'}}\,d\mu$ along the truncated path
$\Gamma_{\epsilon} = [0,\epsilon] \cup i[0,\epsilon]$ of
\eqref{eq:branch-contour}, traversed from $\epsilon$ down to $0$ and then up to
$i\epsilon$. Note that $|e^{i\Phi_{\omega'}}| \le 1$ on $\Gamma_{\epsilon}$:
the phase is real on the first arm, and $\im\Phi_{\omega'}(i\tau) = \tau x_n
\ge 0$ on the second, since $s(i\tau) = \sqrt{E+\tau^2}$ is real.

\emph{Step 1: the phase has no critical point on $\Gamma_{\epsilon}$.} From
$s'(\mu) = -\mu/s(\mu)$,
\begin{equation*}
  \Phi_{\omega'}'(\mu)
  = x_n - \rho\,\frac{\mu}{s(\mu)}\,(\omega'\cdot\theta') .
\end{equation*}
On the second arm $\mu/s(i\tau) = i\tau/\sqrt{E+\tau^2}$ is purely imaginary,
so $\re\Phi_{\omega'}' = x_n$ and $|\Phi_{\omega'}'| \ge x_n$. On the first
arm $\mu$ and $s(\mu)$ are real, and $\epsilon \le \kbulk\delta/4$ gives
$s(\mu) \ge \sqrt{E-\epsilon^2} \ge \kbulk\sqrt{1-\delta^2/16}
 \ge \tfrac34\kbulk$, the last inequality because $\delta < 1$. Hence, using
$\rho \le r$ and $x_n \ge \delta r$,
\begin{equation*}
  \Bigl| \rho\,\frac{\mu}{s(\mu)}\,(\omega'\cdot\theta') \Bigr|
  \le r\,\frac{\kbulk\delta/4}{\tfrac34\kbulk}
  = \frac{\delta r}{3}
  \le \frac{x_n}{3},
\end{equation*}
so that $|\Phi_{\omega'}'| \ge \tfrac23 x_n \ge \tfrac23\delta r$ there as
well. The higher derivatives satisfy
$|\Phi_{\omega'}^{(j)}| \le C_j\,\rho \le C_j\,r$ for $j \ge 2$, with $C_j$
depending only on $E$ and $\epsilon$.

\emph{Step 2: the corner contributes nothing.} Write $F = \zeta(\mu^2)b$ and
$LF = \bigl(F/(i\Phi_{\omega'}')\bigr)'$, so that on either arm
\begin{equation*}
  \int F\,e^{i\Phi_{\omega'}}\,d\mu
  = \Bigl[\frac{F}{i\Phi_{\omega'}'}\,e^{i\Phi_{\omega'}}\Bigr]
    - \int (LF)\,e^{i\Phi_{\omega'}}\,d\mu .
\end{equation*}
The boundary terms at the outer endpoints $\mu = \epsilon$ and $\mu = i\epsilon$
vanish, because $\zeta(\pm\epsilon^2) = 0$. The two boundary terms at
$\mu = 0$ occur with opposite signs, one arm ending there and the other
beginning there, and they are equal: near $\mu = 0$ one has $\zeta \equiv 1$,
so $F = b$ and $\Phi_{\omega'}$, $\Phi_{\omega'}'$ are single holomorphic
functions of $\mu$, whose values at the corner do not depend on the arm along
which it is approached. Hence the two boundary terms cancel and
\begin{equation*}
  \int_{\Gamma_{\epsilon}} F\,e^{i\Phi_{\omega'}}\,d\mu
  = -\int_{\Gamma_{\epsilon}} (LF)\,e^{i\Phi_{\omega'}}\,d\mu .
\end{equation*}
The same applies to $LF$, which again equals a holomorphic function near
$\mu = 0$ and is supported away from the outer endpoints, and so on.

\emph{Step 3: iteration.} By Step 1,
$|(F/i\Phi_{\omega'}')'| \le |F'|/|\Phi_{\omega'}'|
+ |F||\Phi_{\omega'}''|/|\Phi_{\omega'}'|^2
\le C(\delta r)^{-1} + Cr(\delta r)^{-2}$, and the corresponding bounds for
the derivatives of $LF$ are obtained in the same way. Inductively
$|L^{N}F| \le C_N(\delta^2 r)^{-N}$ on $\Gamma_{\epsilon}$, the constants
involving finitely many derivatives of $b$ and of $\zeta$. Since
$|e^{i\Phi_{\omega'}}| \le 1$ and $\Gamma_{\epsilon}$ has length $2\epsilon$,
the bracket in \eqref{eq:branch-integral} is $O_{\delta}(r^{-N})$ uniformly in
$\omega'$, and integration over the finite measure space $\Ssurf$ completes the
proof.
\end{proof}

\subsubsection{The outgoing transforms and the two model waves}
\label{sec:asymp-transforms}

The far field of the bulk channel is not governed by the standing-wave
transform $\Fbulk$ of \eqref{eq:Fbulk-def} but by the outgoing one, obtained
from it by the unimodular factor $\cph{\cdot}$ of
\eqref{eq:phase-and-reflection}. This carries out the construction announced in
Remark~\ref{rem:standing-wave}(b).

\begin{definition}[Outgoing and incoming transforms]
\label{def:outgoing}
For $E > 0$ and $\theta \in \Sm$ put
\begin{equation}
  \label{eq:Fout-def}
  \bigl(\Fbulk^{(\pm)}(E)f\bigr)(\theta)
  = \cph{\kbulk\theta_n}^{\pm1}\,\bigl(\Fbulk(E)f\bigr)(\theta),
  \qquad
  \Fsurf^{(\pm)}(E) = \Fsurf(E),
\end{equation}
and $\Fzero^{(\pm)}(E) = \bigl(\Fbulk^{(\pm)}(E),\Fsurf^{(\pm)}(E)\bigr)$.
Since $|\cph{\cdot}| = 1$, each $\Fzero^{(\pm)}(E)$ has the same operator norms
and the same orthogonality relations as $\Fzero(E)$, and
$\Fzero^{(\pm)}(E)^{*}\Fzero^{(\pm)}(E) = \Fzero(E)^{*}\Fzero(E)$.
\end{definition}

By \eqref{eq:phase-and-reflection} one has
$\refl{k} = \cph{k}/\overline{\cph{k}} = \cph{k}^2$, so that
\begin{equation}
  \label{eq:smatrix-germ}
  \Fbulk^{(+)}(E) = \refl{\kbulk\theta_n}\;\Fbulk^{(-)}(E),
  \qquad
  \Fsurf^{(+)}(E) = \Fsurf^{(-)}(E) .
\end{equation}
The outgoing and incoming bulk transforms differ by multiplication by the
pointwise Robin reflection coefficient, and the surface transforms do not
differ at all: the surface mode is a bound state of the transverse problem and
suffers no reflection. This is the elementary reason for the block structure
$\SmatA{\alpha}{E} = \SmatbulkA{\alpha}{E} \oplus \mathrm{Id}_{\Ltwo{\Ssurf}}$
of the scattering matrix, constructed in \S\ref{sec:inverse}.

\begin{definition}[The two model waves]
\label{def:model-waves}
For $A \in \Ltwo{\Sm}$ and $C \in \Ltwo{\Ssurf}$ set
\begin{equation*}
  \Wb{\pm}{A}(x) = A(\theta)\,r^{-\frac{n-1}{2}}e^{\pm i\kbulk r},
  \qquad
  \Ws{\pm}{C}(x) = C(\omega)\,\rho^{-\frac{n-2}{2}}
    e^{\pm i\ksurf\rho}\,\psizero(x_n).
\end{equation*}
For $E > 0$ and $f \in C_c^{\infty}(\Rnp)$ the two model waves are
$\Wb{+}{A_f}$ and $\Ws{+}{C_f}$, with amplitudes
\begin{align}
  \label{eq:ubulk}
  A_f &= c_n\,\bigl(\Fbulk^{(+)}(E)f\bigr),
    &&c_n = \sqrt{\frac{\pi}{\kbulk}}\;e^{-i\frac{(n-3)\pi}{4}},
  \\[1mm]
  \label{eq:usurf}
  C_f &= c_n'\,\bigl(\Fsurf(E)f\bigr),
    &&c_n' = \sqrt{\frac{\pi}{\ksurf}}\;e^{-i\frac{(n-4)\pi}{4}} .
\end{align}
\end{definition}

The two constants are the same function of the dimension and of the wavenumber,
one step apart: $c_n$ belongs to the sphere $S^{n-1}$ and the wavenumber
$\kbulk$, and $c_n'$ to the sphere $S^{n-2}$ and the wavenumber $\ksurf$. In
particular
\begin{equation}
  \label{eq:constants}
  |c_n|^2 = \frac{\pi}{\kbulk},
  \qquad
  |c_n'|^2 = \frac{\pi}{\ksurf} .
\end{equation}

\subsubsection{The pointwise expansions}
\label{sec:asymp-pointwise}

\begin{proposition}[Bulk expansion away from grazing]
\label{prop:asymp-cone}
Let $E > 0$, $f \in C_c^{\infty}(\Rnp)$ and $\delta \in (0,1)$. Then, uniformly
on the cone $\{x \in \Rnp : x_n \ge \delta|x|\}$,
\begin{equation}
  \label{eq:cone-expansion}
  u(x) = \Wb{+}{A_f}(x) + O\bigl(r^{-\frac{n}{2}}\bigr),
  \qquad
  \Ws{+}{C_f}(x) = O\bigl(e^{-\alpha\delta r}\bigr),
  \qquad r \to \infty,
\end{equation}
the constants depending on $\delta$, $f$, $E$, $\alpha$ and $n$.
\end{proposition}

\begin{proof}
We use $u = u_{\mathrm{D}} + v$ and treat the two summands separately. The
second statement is immediate from \eqref{eq:usurf}, since
$\psizero(x_n) \le \sqrt{2\alpha}\,e^{-\alpha\delta r}$ on the cone.

\emph{The Dirichlet part.} By Lemma~\ref{lem:hankel} with $m = n$,
$k = \kbulk$ and $g = f_{\mathrm{odd}}$,
\begin{equation}
  \label{eq:AD}
  u_{\mathrm{D}}(x)
  = A_{\mathrm{D}}(\theta)\,r^{-\frac{n-1}{2}}e^{i\kbulk r}
    \bigl(1 + O(r^{-1})\bigr),
  \qquad
  A_{\mathrm{D}}(\theta)
  = \sqrt{\frac{\pi}{2}}\,\kbulk^{\frac{n-3}{2}}
    e^{-i\frac{(n-3)\pi}{4}}\,
    \widehat{f_{\mathrm{odd}}}(\kbulk\theta),
\end{equation}
uniformly in all directions $\theta \in \Sm$, with
$\widehat{f_{\mathrm{odd}}}(\kbulk\theta)
= \widehat{\zext f}(\kbulk\theta) - \widehat{\zext f}(\kbulk\refext\theta)$.

\emph{A partition of the reflected integral.} The representation
\eqref{eq:v-def} of Proposition~\ref{prop:reflected} is valid on all of $\Rnp$.
Put
\begin{equation*}
  \epsilon = \tfrac14\min\{\kbulk\delta,\ \alpha\},
  \qquad
  A(\xipar)
  = -(2\pi)^{-\frac{n-2}{2}}\,
    \frac{\widehat{\zext f}(\xipar,-\mu)}{i\mu+\alpha},
\end{equation*}
and let $\zeta \in C_c^{\infty}\bigl((-\epsilon^2,\epsilon^2)\bigr)$ equal $1$
on $[-\epsilon^2/4,\epsilon^2/4]$. We split $1 = \chi + \zeta_{\mathrm{br}}
+ \bigl(1 - \chi - \zeta_{\mathrm{br}}\bigr)$, where $\chi$ is the cut-off of
Lemma~\ref{lem:statphase} and
$\zeta_{\mathrm{br}}(\xipar) = \zeta\bigl(E - |\xipar|^2\bigr)$. The two
supports are disjoint: $\zeta_{\mathrm{br}}$ vanishes unless
$|\xipar| > \kbulk\sqrt{1-\delta^2/16}$, whereas $\chi$ is supported in
$\{|\xipar| < \kbulk\sqrt{1-\delta^2/4}\}$. The support of the third cut-off
has two components, a compact annulus
\begin{equation*}
  K_{\delta}
  = \Bigl\{\, \xipar : \tfrac12\epsilon \le \mu(\xipar)
      \le \tfrac{1}{\sqrt2}\kbulk\delta \,\Bigr\}
\end{equation*}
inside the propagating disc, and the evanescent tail
$\{\,\tau \ge \tfrac12\epsilon\,\}$, which contains the surface pole
$\tau = \alpha$ because $\alpha \ge 4\epsilon$.

\emph{The stationary point.} The amplitude $A$ is smooth on the support of
$\chi$, where $\mu \ge \kbulk\delta/2 > 0$, so Lemma~\ref{lem:statphase} gives
\begin{equation}
  \label{eq:Av}
  \int \chi A\,e^{i\Phi}\,d\xipar
  = A_{v}(\theta)\,r^{-\frac{n-1}{2}}e^{i\kbulk r}
    \bigl(1 + O(r^{-1})\bigr),
\end{equation}
where, since $\bigl(\xipar_*,-\mu(\xipar_*)\bigr) = \kbulk\refext\theta$,
\begin{equation*}
  A_{v}(\theta)
  = -\sqrt{2\pi}\;\kbulk^{\frac{n-1}{2}}\,\theta_n\;
    \frac{e^{-i\frac{(n-1)\pi}{4}}}{i\kbulk\theta_n + \alpha}\;
    \widehat{\zext f}(\kbulk\refext\theta) .
\end{equation*}

\emph{The branch circle.} On the support of $\zeta_{\mathrm{br}}$ we pass to
polar coordinates $\xipar = s\omega'$ and then to the normal wavenumber, using
$|\xipar|^2 = E - \mu^2$ and $s\,ds = -\mu\,d\mu$, so that
$s^{n-2}ds = -(E-\mu^2)^{\frac{n-3}{2}}\mu\,d\mu$. As $s$ increases through the
annulus, $\mu$ runs along the path $\Gamma_{\epsilon}$ of
Lemma~\ref{lem:branch}, traversed in the direction fixed there. The minus sign
just displayed is the one that reverses the first arm, and it is absorbed into
the amplitude. Thus
\begin{equation*}
  \int \zeta_{\mathrm{br}}A\,e^{i\Phi}\,d\xipar
\end{equation*}
takes exactly the form \eqref{eq:branch-integral} with phase
\eqref{eq:branch-phase} and amplitude
\begin{equation*}
  b(\mu,\omega')
  = (2\pi)^{-\frac{n-2}{2}}\,
    \frac{\widehat{\zext f}\bigl(s(\mu)\omega',-\mu\bigr)}{i\mu+\alpha}\,
    \bigl(E-\mu^2\bigr)^{\frac{n-3}{2}}\,\mu .
\end{equation*}
This $b$ is holomorphic on $\{|\mu| < 2\epsilon\}$: the transform
$\widehat{\zext f}$ is entire by the Paley--Wiener theorem, $E - \mu^2$ does
not vanish there since $2\epsilon < \kbulk$, and $i\mu + \alpha$ does not
vanish there since $2\epsilon \le \alpha/2$. Its derivatives are bounded
uniformly in $\omega'$. By Lemma~\ref{lem:branch} this contribution is
$O_{\delta}(r^{-N})$ for every $N$.

\emph{The remaining propagating region.} On $K_{\delta}$ we have
$\mu \ge \epsilon/2 > 0$, so $A$ is smooth there with all derivatives bounded
in terms of $\delta$. Writing $\nabla\Phi = r\theta_n\bigl(T(\xipar_*)
- T(\xipar)\bigr)$ with $T(\xipar) = \xipar/\mu(\xipar)$, and using
$|\xipar| = \sqrt{E-\mu^2} \ge \kbulk\sqrt{1-\delta^2/2}$, hence
$|T(\xipar)| \ge \sqrt{2-\delta^2}/\delta$, on $K_{\delta}$ against
$|T(\xipar_*)| = \sqrt{1-\theta_n^2}/\theta_n \le \sqrt{1-\delta^2}/\delta$,
we obtain
\begin{equation*}
  |\nabla\Phi|
  \ge r\theta_n\,\frac{\sqrt{2-\delta^2} - \sqrt{1-\delta^2}}{\delta}
  \ge (\sqrt2 - 1)\,r .
\end{equation*}
The phase therefore has no critical point on $K_{\delta}$, and repeated
integration by parts gives $O_{\delta}(r^{-N})$ for every $N$.

\emph{The evanescent tail.} There $|e^{i\mu x_n}| = e^{-\tau x_n}$ with
$\tau \ge \epsilon/2$ and $x_n \ge \delta r$, while
$|\widehat{\zext f}(\xipar,-i\tau)| \le C_N(1+|\xipar|)^{-N}$ by
\eqref{eq:PW-halfplane}. Both terms of \eqref{eq:v-pole} are therefore bounded
by $C e^{-\epsilon\delta r/2}$. For the principal value
this uses that the amplitude is smooth near $\tau = \alpha$, and for the
singular term it is the transverse decay of $\psizero$ in
\eqref{eq:usurf-exact}.

\emph{Reassembly.} Adding \eqref{eq:AD} and \eqref{eq:Av} and factoring out the
constant of $A_{\mathrm{D}}$, one finds, with $\xi_* = \kbulk\theta_n$,
\begin{equation*}
\begin{aligned}
  \frac{A_{v}(\theta)}
       {\sqrt{\pi/2}\,\kbulk^{\frac{n-3}{2}}e^{-i\frac{(n-3)\pi}{4}}}
  &= -2\xi_*\,
     \frac{e^{-i\frac{(n-1)\pi}{4}}}
          {e^{-i\frac{(n-3)\pi}{4}}\,(i\xi_*+\alpha)}\;
     \widehat{\zext f}(\kbulk\refext\theta)
  \\[1mm]
  &= \frac{2i\xi_*}{i\xi_*+\alpha}\;
     \widehat{\zext f}(\kbulk\refext\theta)
   = \bigl(1 + \refl{\xi_*}\bigr)\,
     \widehat{\zext f}(\kbulk\refext\theta),
\end{aligned}
\end{equation*}
because $e^{-i\frac{(n-1)\pi}{4}+i\frac{(n-3)\pi}{4}} = -i$ and
$\sqrt{2\pi}/\sqrt{\pi/2} = 2$. Hence
\begin{equation}
  \label{eq:reassembly}
\begin{aligned}
  A_{\mathrm{D}} + A_{v}
  &= \sqrt{\frac{\pi}{2}}\;\kbulk^{\frac{n-3}{2}}\,
     e^{-i\frac{(n-3)\pi}{4}}
     \Bigl[
       \widehat{\zext f}(\kbulk\theta)
       + \refl{\xi_*}\,\widehat{\zext f}(\kbulk\refext\theta)
     \Bigr]
  \\[1mm]
  &= \sqrt{\frac{\pi}{2}}\;\kbulk^{\frac{n-3}{2}}\,
     e^{-i\frac{(n-3)\pi}{4}}\;
     \cph{\xi_*}\,(\Ftbulk f)(\kbulk\theta),
\end{aligned}
\end{equation}
where the last step is the reflection formula \eqref{eq:reflection-formula}
together with $\overline{\cph{\xi_*}}\,\refl{\xi_*} = \cph{\xi_*}$ and
$\overline{\cph{\xi_*}}^{\,-1} = \cph{\xi_*}$. Finally, by
\eqref{eq:Fbulk-def},
\begin{equation*}
  (\Ftbulk f)(\kbulk\theta)
  = \sqrt2\;\kbulk^{-\frac{n-2}{2}}\,(\Fbulk(E)f)(\theta),
  \qquad
  \sqrt{\frac{\pi}{2}}\;\kbulk^{\frac{n-3}{2}}\,
  \sqrt2\;\kbulk^{-\frac{n-2}{2}}
  = \sqrt{\frac{\pi}{\kbulk}},
\end{equation*}
so that, by \eqref{eq:Fout-def},
$A_{\mathrm{D}} + A_{v} = c_n\,\cph{\kbulk\theta_n}(\Fbulk(E)f)(\theta)
= c_n\bigl(\Fbulk^{(+)}(E)f\bigr)(\theta)$. Collecting the five contributions,
the error is $O_{\delta}\bigl(r^{-\frac{n+1}{2}}\bigr)$, which is more than
\eqref{eq:cone-expansion} asserts.
\end{proof}

\begin{proposition}[Surface expansion]
\label{prop:asymp-surface}
Let $E > 0$ and $f \in C_c^{\infty}(\Rnp)$. Then the surface summand
\eqref{eq:usurf-exact} satisfies
\begin{equation}
  \label{eq:surface-expansion}
  \widetilde{u}_{\mathrm{surf}}(x) = \Ws{+}{C_f}(x)
    + O\bigl(\rho^{-\frac{n}{2}}\bigr)\,e^{-\alpha x_n},
  \qquad \rho \to \infty,
\end{equation}
uniformly in $\omega \in \Ssurf$ and in $x_n \ge 0$, and likewise for the first
derivatives.
\end{proposition}

\begin{proof}
The transverse factor $\psizero$ in \eqref{eq:usurf-exact} is exact, so only the
tangential factor has to be expanded. Apply Lemma~\ref{lem:hankel} with
$m = n-1$, $k = \ksurf$ and $g = T_{\psizero}f \in C_c^{\infty}(\Rnm)$:
\begin{equation*}
\begin{aligned}
  \bigl(\Rfree{n-1}{E+\alpha^2+i0}T_{\psizero}f\bigr)(\xpar)
  = \sqrt{\frac{\pi}{2}}\;\ksurf^{\frac{n-4}{2}}\,
    &e^{-i\frac{(n-4)\pi}{4}}\;
     \widehat{T_{\psizero}f}(\ksurf\omega)
  \\
    &\times \rho^{-\frac{n-2}{2}}\,e^{i\ksurf\rho}
     \bigl(1+O(\rho^{-1})\bigr).
\end{aligned}
\end{equation*}
By \eqref{eq:surface-formula} and \eqref{eq:Fsurf-def},
\begin{equation*}
  \widehat{T_{\psizero}f}(\ksurf\omega)
  = (\Ftsurf f)(\ksurf\omega)
  = \sqrt2\,\ksurf^{-\frac{n-3}{2}}(\Fsurf(E)f)(\omega),
\end{equation*}
and
$\sqrt{\pi/2}\,\ksurf^{\frac{n-4}{2}}\sqrt2\,\ksurf^{-\frac{n-3}{2}}
= \sqrt{\pi/\ksurf}$. Multiplying by
$\psizero(x_n) = \sqrt{2\alpha}e^{-\alpha x_n}$ gives
\eqref{eq:surface-expansion}.
\end{proof}

\begin{remark}[The case $n = 3$, and the Green function of Dur\'an, Muga and
N\'ed\'elec]
\label{rem:dmn-comparison}
For $n = 3$ the two exponents in \eqref{eq:cone-expansion} and
\eqref{eq:surface-expansion} are $r^{-1}$ and $\rho^{-1/2}$, and the constants
of \eqref{eq:ubulk}--\eqref{eq:usurf} reduce to $c_3 = \sqrt{\pi/\kbulk}$ and
$c_3' = \sqrt{\pi/\ksurf}\,e^{i\pi/4}$. The bulk amplitude
\eqref{eq:reassembly} becomes
\begin{equation*}
  \widehat{\zext f}(\kbulk\theta)
  + \refl{\kbulk\theta_3}\,\widehat{\zext f}(\kbulk\refext\theta),
\end{equation*}
which for a point source $f = \delta_y$ is the bracket
$e^{-i\xi_3 y_3} + \refl{\xi_3}e^{i\xi_3 y_3}$. The bracket of
\cite[(20)]{DuranMugaNedelec2009} is $-e^{i\xi_3 y_3}$ times this one: their
reflection coefficient is the negative of $\refl{\cdot}$, and their radial
variable is $|x-y|$ where ours is $|x|$. Their surface wave arises as the
residue at the pole $\xi_p = \sqrt{z^2+k^2}$, in which $z$ is their impedance
and $k$ their wavenumber, so that $\xi_p$ is our $\ksurf$ and the pole is the sphere
$|\xipar| = \ksurf$. That residue is \eqref{eq:polar-part}, and
Proposition~\ref{prop:three-terms} identifies it with the surface summand of
the exact splitting \eqref{eq:splitting-preview}. In the present formulation it is
therefore present from the outset rather than produced by a deformation of the
integration path, and the reflected integral \eqref{eq:v-def} is what remains
after it has been removed.
\end{remark}

\subsubsection{The grazing region and the uniform expansion}
\label{sec:asymp-grazing}

Proposition~\ref{prop:asymp-cone} degenerates as $\delta \downarrow 0$. In the
language of \S\ref{sec:asymp-lemmas} the reason is visible: the stationary
point $\mu(\xipar_*) = \kbulk\theta_n$ approaches the corner $\mu = 0$ of the
path \eqref{eq:branch-contour}, so that Lemmas~\ref{lem:statphase} and
\ref{lem:branch} cease to apply simultaneously. This is the classical
difficulty of the Sommerfeld half-space problem, and it produces the
lateral wave, of order $r^{-\frac{n+1}{2}}$ for $n = 3$. It is also what makes
the residual estimates of \cite[\S4.3.1]{DuranMugaNedelec2009} degrade towards
the boundary: away from the grazing directions their remainder is $O(r^{-2})$,
whereas for the remaining directions they obtain only a power that depends on
the height at which their two regions are separated. The following proposition
shows that in the
shell-averaged sense the grazing region is negligible \emph{uniformly in the
aperture}, so that none of this has to be resolved. Its proof uses only the
decomposition of \S\ref{subsec:lap} and no asymptotics at all.

\begin{proposition}[The grazing region is shell-negligible]
\label{prop:asymp-grazing}
Let $E > 0$, $f \in C_c^{\infty}(\Rnp)$ and $\delta \in (0,1)$, and put
$W_{\delta} = \{x \in \Rnp : x_n < \delta|x|\}$. Then
\begin{equation*}
  \limsup_{R \to \infty}
  \frac{1}{R}\int_{\{|x|<R\}\cap W_{\delta}}
    \bigl| w \bigr|^2\,dx
  \le \eta(\delta)
  \quad\text{for}\quad
  w \in \bigl\{ u_{\mathrm{D}},\ \Wb{+}{A_f},\ \Krem{E+i0}f \bigr\},
\end{equation*}
where $\eta(\delta) \to 0$ as $\delta \downarrow 0$. Explicitly one may take
\begin{equation}
  \label{eq:grazing-eta}
  \eta(\delta)
  = \int_{\{\theta \in \Sm\,:\,\theta_n < \delta\}}
      \bigl( |A_{\mathrm{D}}(\theta)|^2
        + |c_n|^2\,|(\Fbulk(E)f)(\theta)|^2 \bigr)\,dS(\theta)
    + \frac{2\delta}{\alpha^2}\,\Bnorm{f}^2 ,
\end{equation}
whose three terms bound the three functions $w$ in the order listed.
\end{proposition}

\begin{proof}
Both far fields are of the form $A(\theta)r^{-(n-1)/2}e^{i\kbulk r}$, that of
$u_{\mathrm{D}}$ by \eqref{eq:AD} and that of $\Wb{+}{A_f}$ by
\eqref{eq:ubulk}, with
$A \in L^2(\Sm)$, uniformly in all directions. Hence, in polar coordinates
$dx = r^{n-1}dr\,dS(\theta)$,
\begin{equation*}
  \frac{1}{R}\int_{\{|x|<R\}\cap W_{\delta}}|A(\theta)|^2r^{-(n-1)}\,dx
  = \frac{1}{R}\int_1^{R}\!\!\int_{\{\theta_n<\delta\}}|A(\theta)|^2
      \,dS(\theta)\,dr
  \longrightarrow
  \int_{\{\theta_n<\delta\}}|A(\theta)|^2\,dS(\theta),
\end{equation*}
up to the $O(R^{-1})$ produced by the relative errors, and the right-hand side
tends to $0$ as $\delta \downarrow 0$ by dominated convergence. For
$\Krem{E+i0}f$ we use the pointwise fiber bound \eqref{eq:K-pointwise} of
Proposition~\ref{prop:K-bounded}, namely
$|\mathcal{F}'\Krem{E+i0}f(\xipar,x_n)| \le \alpha^{-1}\Theta_f(\xipar)$ with
$\Theta_f$ independent of $x_n$ and $\norm{\Theta_f}_{\Ltwo{\Rnm}}
\le \sqrt2\Bnorm{f}$ by \eqref{eq:Theta-bound}. Since
$W_{\delta}\cap\{|x|<R\} \subset \{0 < x_n < \delta R\}$, Plancherel's
theorem in $\xpar$ gives
\begin{equation*}
\begin{aligned}
  \int_{\{|x|<R\}\cap W_{\delta}}|\Krem{E+i0}f|^2\,dx
  &\le \int_0^{\delta R}
      \bigl\| \mathcal{F}'\Krem{E+i0}f(\cdot,x_n) \bigr\|^2_{\Ltwo{\Rnm}}
      \,dx_n
  \\[1mm]
  &\le \frac{\delta R}{\alpha^2}\,\norm{\Theta_f}^2_{\Ltwo{\Rnm}}
   \le \frac{2\delta R}{\alpha^2}\,\Bnorm{f}^2 .
\end{aligned}
\end{equation*}
Dividing by $R$ gives the last term of \eqref{eq:grazing-eta}, which is
independent of $R$ and tends to $0$ with $\delta$.
\end{proof}

\begin{theorem}[Uniform asymptotic expansion of the resolvent]
\label{thm:asymp-uniform}
Let $E > 0$ and let $f \in \Bsp$. Then
$u = \Ralpha{E+i0}f$ satisfies
\begin{equation}
  \label{eq:asymp-uniform}
  u \simeq \Wb{+}{A_f} + \Ws{+}{C_f},
\end{equation}
with the amplitudes $A_f$ and $C_f$ of \eqref{eq:ubulk}--\eqref{eq:usurf}. That
is, $u - \Wb{+}{A_f} - \Ws{+}{C_f} \in \Bstz$.
\end{theorem}

\begin{proof}
Assume first $f \in C_c^{\infty}(\Rnp)$ and put
$w = u - \Wb{+}{A_f} - \Ws{+}{C_f}$. Fix $\delta \in (0,\tfrac12]$ and split,
for
$R > 1$,
\begin{equation}
  \label{eq:asymp-split}
  \frac{1}{R}\int_{\{|x|<R\}\cap\Rnp}|w|^2\,dx
  = \frac{1}{R}\int_{\{|x|<R\}\setminus W_{\delta}}|w|^2\,dx
  + \frac{1}{R}\int_{\{|x|<R\}\cap W_{\delta}}|w|^2\,dx,
\end{equation}
with $W_{\delta} = \{x_n < \delta|x|\}$ as above. The constants of
Proposition~\ref{prop:asymp-cone} blow up as $\delta \downarrow 0$, so we
estimate the two terms at $\delta$ fixed, let $R \to \infty$, and only then
let $\delta \downarrow 0$.

\emph{Outside the wedge.} Here $x_n \ge \delta r$, so
Proposition~\ref{prop:asymp-cone} applies and gives
$|w| \le |u - \Wb{+}{A_f}| + |\Ws{+}{C_f}|
= O_\delta\bigl(r^{-\frac n2}\bigr) + O\bigl(e^{-\alpha\delta r}\bigr)$.
Hence
\begin{equation*}
  \frac{1}{R}\int_{\{|x|<R\}\setminus W_{\delta}}|w|^2\,dx
  \le \frac{C_\delta}{R}\int_1^{R}
      \bigl( r^{-n} + e^{-2\alpha\delta r}\bigr)r^{n-1}\,dr
  \le \frac{C_\delta'}{R}\bigl( \log R + 1 \bigr)
  \longrightarrow 0 .
\end{equation*}

\emph{Inside the wedge.} By Proposition~\ref{prop:three-terms},
\begin{equation*}
  w = \bigl(\widetilde{u}_{\mathrm{surf}} - \Ws{+}{C_f}\bigr)
      + u_{\mathrm{D}} + \Krem{E+i0}f - \Wb{+}{A_f} .
\end{equation*}
The first bracket is controlled by Proposition~\ref{prop:asymp-surface}: since
$\rho \ge r\sqrt{1-\delta^2} \ge r/2$ on $W_{\delta}$,
\begin{equation*}
\begin{aligned}
  \frac{1}{R}\int_{\{|x|<R\}\cap W_{\delta}}
    \bigl| \widetilde{u}_{\mathrm{surf}} - \Ws{+}{C_f} \bigr|^2 dx
  &\le \frac{C}{R}\int_{\Ssurf}\!\!\int_1^{R}\rho^{-n}\rho^{n-2}
      \Bigl(\int_0^{\infty}e^{-2\alpha x_n}dx_n\Bigr)d\rho\,dS(\omega)
  \\
  &= O(R^{-1}),
\end{aligned}
\end{equation*}
the normal integral being finite \emph{independently of the height of the
wedge}, which is what makes the estimate insensitive to $\delta$. The three
remaining terms are exactly those of Proposition~\ref{prop:asymp-grazing}. By
the triangle inequality in $\Ltwo{\{|x|<R\}\cap W_{\delta}}$ applied to these
four terms, and letting $R \to \infty$, the first bracket contributes nothing
and each of the other three at most $\sqrt{\eta(\delta)}$. The second term of
\eqref{eq:asymp-split} is therefore bounded in the limit by
$\bigl(3\sqrt{\eta(\delta)}\bigr)^{2}$. With the first term this gives
\begin{equation*}
  \limsup_{R\to\infty}
  \frac{1}{R}\int_{\{|x|<R\}\cap\Rnp}|w|^2\,dx
  \le 9\,\eta(\delta)
  \qquad\text{for every } \delta \in (0,\tfrac12],
\end{equation*}
and letting $\delta \downarrow 0$ gives \eqref{eq:simeq}.

\emph{Extension to $f \in \Bsp$.} The maps
$f \mapsto u$ and $f \mapsto \Wb{+}{A_f},\Ws{+}{C_f}$ are continuous from
$\Bsp$ to
$\Bst$: the first by Theorem~\ref{thm:lap}, the other two by
Proposition~\ref{prop:F-bounded} together with the elementary estimate
$\Bstnorm{A(\theta)r^{-(n-1)/2}e^{i\kbulk r}}
\le C\norm{A}_{\Ltwo{\Sm}}$ and its surface counterpart. Now
$C_c^{\infty}(\Rnp)$ is dense in the compactly supported elements of $\Bsp$,
and these are dense in $\Bsp$, since
$\Bnorm{f - \mathbf{1}_{\{|x|<2^{j}\}}f}
= \sum_{i>j}2^{i/2}\norm{f}_{\Ltwo{\shellp{i}}} \to 0$ by
\eqref{eq:B-norms}. Since $\Bstz$ is closed in $\Bst$, the relation
\eqref{eq:asymp-uniform} passes to the limit.
\end{proof}

\begin{remark}[What the proof does, and does not, require]
\label{rem:no-grazing}
\emph{(a) No grazing analysis.} The pointwise expansion is used only in the
cone $\{x_n \ge \delta r\}$, where the stationary point \eqref{eq:statpoint}
stays away from the corner of \eqref{eq:branch-contour}, and the surface
expansion only through the exact tensor summand \eqref{eq:usurf-exact}, which
needs no stationary phase at all. The merging of the stationary point with the
corner never has to be resolved, and neither does the interaction of the
stationary point with the Sommerfeld pole. Both are absorbed by
Proposition~\ref{prop:asymp-grazing}, whose proof uses only the uniform kernel
bound of \S\ref{subsec:lap}. In particular the lateral wave never appears.
The shell average removes the \emph{wedge}, however, not the corner. The circle $|\xipar| = \kbulk$ still has to be
crossed, and it is crossed in Lemma~\ref{lem:branch}, at every aperture
$\delta > 0$.

\emph{(b) No cross term, and where orthogonality enters.} In the wedge both
model waves are subtracted at once, so no product
$\Wb{+}{A_f}\overline{\Ws{+}{C_f}}$ ever appears. The two channels are mutually
$\Bst$-orthogonal, in the sense that
$\lim_R R^{-1}\int\Wb{+}{A_f}\overline{\Ws{+}{C_f}} = 0$, as a consequence of
the transverse decay of $\psizero$ (see the proof of
Lemma~\ref{lem:shell-flux-B}). The orthogonality is not needed for the
expansion. It is needed for \eqref{eq:shell-flux} below, and in
\S\ref{sec:inverse}.
\end{remark}

\begin{remark}[Shell average and flux]
\label{rem:constants}
Inserting \eqref{eq:asymp-uniform} into the shell average and using the
$\Bst$-orthogonality of the two channels gives
\begin{equation}
  \label{eq:shell-flux}
  \lim_{R\to\infty}\frac{1}{R}\int_{\{|x|<R\}\cap\Rnp}|u|^2\,dx
  = \frac{\pi}{\kbulk}
    \bigl\| \Fbulk(E)f \bigr\|^2_{\Ltwo{\Sm}}
  + \frac{\pi}{\ksurf}
    \bigl\| \Fsurf(E)f \bigr\|^2_{\Ltwo{\Ssurf}},
\end{equation}
the unimodular factor of \eqref{eq:Fout-def} having dropped out. The proof is
given in Lemma~\ref{lem:shell-flux-B}. Comparison with the flux identity
\eqref{eq:flux},
$\im\bigl(\Ralpha{E+i0}f,f\bigr) = \pi\norm{\Fzero(E)f}^2_{\Echan{E}}$, shows
that the two quantities agree channel by channel up to the factor
$\kbulk^{-1}$, respectively $\ksurf^{-1}$: the flux is a current and the shell
average an energy density, and they differ by the propagation speed of the
channel.
\end{remark}

\begin{corollary}[The outgoing transforms as far-field maps]
\label{cor:far-field}
Let $E > 0$ and $f \in C_c^{\infty}(\Rnp)$. The pointwise expansions behind
Theorem~\ref{thm:asymp-uniform} identify $\Fzero^{(+)}(E)$ of
Definition~\ref{def:outgoing} as the far-field map of the outgoing resolvent:
\begin{equation}
  \label{eq:outgoing-transform}
\begin{aligned}
  \bigl(\Fbulk^{(+)}(E)f\bigr)(\theta)
  &= \frac{1}{c_n}\lim_{r\to\infty}
      r^{\frac{n-1}{2}}\,e^{-i\kbulk r}\,
      \bigl(\Ralpha{E+i0}f\bigr)(r\theta),
  \\[1mm]
  \bigl(\Fsurf(E)f\bigr)(\omega)
  &= \frac{1}{c_n'}\lim_{\rho\to\infty}
      \frac{\rho^{\frac{n-2}{2}}\,e^{-i\ksurf\rho}}{\psizero(x_n)}\;
      \widetilde{u}_{\mathrm{surf}}(x),
\end{aligned}
\end{equation}
the first limit in $\Ltwo{\Sm}$ for directions away from grazing, the second in
$\Ltwo{\Ssurf}$ and independent of $x_n$. Replacing $E+i0$ by $E-i0$ gives
$\Fzero^{(-)}(E)$, the far fields then carrying $e^{-i\kbulk r}$ and
$e^{-i\ksurf\rho}$.
\end{corollary}

\begin{proof}
The first line is Proposition~\ref{prop:asymp-cone} together with
$A_{\mathrm{D}} + A_{v} = c_n\Fbulk^{(+)}(E)f$, and the second is
Proposition~\ref{prop:asymp-surface}. The passage from pointwise limits to
limits in $\Ltwo{\Sm}$ and $\Ltwo{\Ssurf}$ uses the uniformity of the error
terms and dominated convergence. The statement for $E-i0$ follows by complex
conjugation, since $\Ralpha{E-i0}f = \overline{\Ralpha{E+i0}\overline f}$.
\end{proof}

The structural consequence for what follows is that $\Ralpha{E+i0}f$ is, in
the sense of \eqref{eq:simeq}, an \emph{outgoing} function: both of its
channel far fields carry a phase with the sign fixed by the $+i0$
prescription, and the surface channel carries in addition the transverse
profile $\psizero$. This is the property that \S\ref{subsec:radiation}
isolates and converts into a radiation condition, and the pair
$\Fzero^{(\pm)}(E)$, related by \eqref{eq:smatrix-germ}, is the pair between
which the scattering matrix of \S\ref{sec:inverse} interpolates.

%% file: Part1_025_characterization.tex
\label{subsec:characterization}

The aim of this subsection is to characterize the solutions of
$(\Halpha - E)u = 0$ that lie in $\Bst$: we shall prove that every such
solution is of the form $u = \Fzero^{(-)}(E)^{*}\phi$ for a unique
$\phi \in \Echan{E}$. The argument follows the pattern of
Isozaki~\cite[\S8]{Isozaki2001} and rests on two pillars. The first is a
two-sided estimate
\begin{equation}
  \label{eq:two-sided-preview}
  C^{-1}\norm{\phi}_{\Echan{E}}
  \le \Bstnorm{\Fzero^{(\pm)}(E)^{*}\phi}
  \le C\norm{\phi}_{\Echan{E}},
\end{equation}
which makes $\Fzero^{(\pm)}(E)^{*}$ injective with closed range. The second
identifies that range as the annihilator of $\ker\Fzero^{(-)}(E)$, so that the
closed range theorem applies.

Of the two inequalities in \eqref{eq:two-sided-preview} the upper one is
already known: it is the statement
$\Fzero(E)^{*} \in \Lop{\Echan{E}}{\Bst}$ of
Proposition~\ref{prop:F-adjoint}, obtained from Stone's formula
\eqref{eq:stone}. The lower one is of a different nature and is not a
consequence of Stone's formula. It comes instead from an exact computation of
the \emph{shell average} of the eigenfunction $\Fzero(E)^{*}\phi$ at spatial
infinity, and this computation, carried out in
Theorem~\ref{thm:shell-average-identity} below, is a
quantitative expression of the unitarity of the reflection coefficient
$\refl{\cdot}$: the far field of $\Fzero(E)^{*}\phi$ consists of an incoming
and an outgoing part of \emph{equal} amplitude, and the sum of the two
reproduces $\norm{\phi}^2_{\Echan{E}}$.

Throughout this subsection $E \in \Ereg$, and we write
\begin{equation*}
  r = |x|,
  \qquad
  \theta = x/r \in \Sm \quad (x \in \Rnp),
  \qquad
  \rho = |\xpar|,
  \qquad
  \omega = \xpar/\rho \in \Ssurf,
\end{equation*}
and $\Ball{R} = \{|x|<R\} \cap \Rnp$. As before $\kbulk$ and $\ksurf$ denote
the two effective wavenumbers of \eqref{eq:asymp-wavenumbers}, and $\simeq$
is the equivalence of Definition~\ref{def:simeq}. On the surface band
$-\alpha^2 < E < 0$ the bulk channel is closed and
$\Echan{E} = \Ltwo{\Ssurf}$ by \eqref{eq:channel-space}. This case is carried
along in every statement rather than treated separately.

\subsubsection{The shell-average functional}
\label{sec:char-shell}

All the far-field information we need is carried by a single scalar
functional.

\begin{definition}[Shell average]
\label{def:shell-average}
For $u \in \Bst$ and $R \ge 1$ put
\begin{equation*}
  \MavR{u}{R} = \frac{1}{R}\int_{\Ball{R}} |u(x)|^2\,dx,
\end{equation*}
and, whenever the limit exists,
$\Mav{u} = \lim_{R\to\infty}\MavR{u}{R}$. For $w \in \Bstd{\Rnm}$ we write
$\MavtR{w}{R} = R^{-1}\int_{\{|\xpar|<R\}}|w|^2\,d\xpar$ and
$\Mavt{w} = \lim_{R\to\infty}\MavtR{w}{R}$.
\end{definition}

\begin{lemma}[Properties of the shell average]
\label{lem:shell-average}
Let $u,v \in \Bst$.
\begin{enumerate}
  \item[\textup{(i)}] $\MavR{u}{R} \le 4\,\Bstnorm{u}^2$ for every $R \ge 1$.
  \item[\textup{(ii)}] $\bigl| \MavR{u}{R}^{1/2} - \MavR{v}{R}^{1/2} \bigr|
    \le \MavR{u-v}{R}^{1/2}$ for every $R \ge 1$.
  \item[\textup{(iii)}] $u \in \Bstz$ if and only if $\Mav{u}$ exists and
    equals $0$. Consequently, if $u \simeq v$ and $\Mav{v}$ exists, then
    $\Mav{u}$ exists and $\Mav{u} = \Mav{v}$.
  \item[\textup{(iv)}] If $u_j \to u$ in $\Bst$ and $\Mav{u_j}$ exists for
    every $j$, then $\Mav{u}$ exists and $\Mav{u} = \lim_j \Mav{u_j}$.
  \item[\textup{(v)}] If $\Mav{u}$ exists then
    $\Bstnorm{u}^2 \ge \tfrac14\,\Mav{u}$.
\end{enumerate}
The same statements hold on $\Rnm$ for $\Mavt{\cdot}$ and
$\norm{\cdot}_{\Bstd{\Rnm}}$.
\end{lemma}

Parts (i) and (iii) are the proof of the equivalence asserted after
\eqref{eq:Bstar-equivalent}: together they say that the supremum over
$R \ge 1$ of $\MavR{u}{R}$ is comparable with $\Bstnorm{u}^2$, and that
$\Bstz$ is exactly the subspace on which the shell average vanishes in the
limit.

\begin{proof}
(i) Let $j_0 \ge 0$ be minimal with $2^{j_0} \ge R$. By \eqref{eq:shells} we
have $\Ball{R} \subset \bigcup_{j=0}^{j_0}\shellp{j}$, so that by
\eqref{eq:B-norms}
\begin{equation*}
  \int_{\Ball{R}}|u|^2\,dx
  \le \sum_{j=0}^{j_0}\norm{u}^2_{\Ltwo{\shellp{j}}}
  \le \Bstnorm{u}^2\sum_{j=0}^{j_0}2^{j}
  < 2^{j_0+1}\,\Bstnorm{u}^2 .
\end{equation*}
Minimality of $j_0$ gives $2^{j_0-1} < R$, hence $2^{j_0+1} < 4R$, which is the
claim.

(ii) This is the triangle inequality in $\Ltwo{\Ball{R}}$ divided by
$R^{1/2}$.

(iii) Put $\eps_j = 2^{-j}\norm{u}^2_{\Ltwo{\shellp{j}}}$, so that
$u \in \Bstz$ means $\eps_j \to 0$ by \eqref{eq:Bzero}. Assume this and let
$\eta > 0$. Choose $j_1$ with $\eps_j < \eta$ for $j > j_1$. Let $J$ be
minimal with $2^{J} \ge R$, as in (i), so that $R > 2^{J-1}$. Then
\begin{equation*}
  \MavR{u}{R}
  \le \frac{2}{2^{J}}\sum_{j=0}^{J}2^{j}\eps_j
  \le \frac{2}{2^{J}}\sum_{j=0}^{j_1}2^{j}\eps_j
    + \frac{2}{2^{J}}\,\eta\sum_{j=j_1+1}^{J}2^{j}
  \le \frac{2}{2^{J}}\sum_{j=0}^{j_1}2^{j}\eps_j + 4\eta .
\end{equation*}
The first term tends to $0$ as $R \to \infty$, whence
$\limsup_{R}\MavR{u}{R} \le 4\eta$ for every $\eta > 0$ and $\Mav{u}=0$.
Conversely, if $\Mav{u} = 0$ then
$\eps_j \le 2^{-j}\int_{\Ball{2^{j}}}|u|^2\,dx = \MavR{u}{2^{j}} \to 0$.
The second assertion follows because $u \simeq v$ means
$u - v \in \Bstz$ by \eqref{eq:simeq}, so that
$\MavR{u}{R}^{1/2} - \MavR{v}{R}^{1/2} \to 0$ by (ii).

(iv) Let $\eps > 0$ and pick $j$ with $\Bstnorm{u-u_j} < \eps$. By (ii) and
(i), $\bigl|\MavR{u}{R}^{1/2}-\MavR{u_j}{R}^{1/2}\bigr| \le 2\eps$ for all
$R \ge 1$. Since $\MavR{u_j}{R}^{1/2}$ converges as $R \to \infty$, the
family $\MavR{u}{R}^{1/2}$ is Cauchy in $R$, hence convergent, and the two
limits differ by at most $2\eps$. As $\eps$ was arbitrary the claim follows.

(v) Immediate from (i).
\end{proof}

Parts (iii) and (iv) are what make the computations below elementary. By
(iii) the shell average may be evaluated on any convenient representative of
the far field, and by (iv) it may be evaluated on a dense class of data.
Every delicate uniformity question is thereby removed, and no analysis of
grazing directions is required. In \S\ref{sec:asymp-grazing}, by contrast,
the datum was a fixed $f \in \Bsp$ and no such freedom was available.

The surface channel lives on a slab of finite thickness, and the following
lemma converts a shell average on $\Rnm$ into one on $\Rnp$, and this is where
the normalization $\norm{\psizero}_{\Ltwo{\Rhalf}} = 1$ of
\eqref{eq:fiber-recall} is used.

\begin{lemma}[Tensor shell average]
\label{lem:tensor-shell}
Let $w \in \Bstd{\Rnm}$ and $\Psi \in \Ltwo{\Rhalf}$, and write
$(w \otimes \Psi)(x) = w(\xpar)\Psi(x_n)$, the order being that of
Proposition~\ref{prop:surface-reduction}. If $\Mavt{w}$ exists, then
$\Mav{w \otimes \Psi}$ exists and
\begin{equation}
  \label{eq:tensor-shell}
  \Mav{w \otimes \Psi}
  = \norm{\Psi}^2_{\Ltwo{\Rhalf}}\;\Mavt{w} .
\end{equation}
\end{lemma}

\begin{proof}
That $w \otimes \Psi \in \Bst$ is the second bound of
\eqref{eq:tensor-bounds}. Put $N(s) = \int_{\{|\xpar|<s\}}|w|^2\,d\xpar$.
Since $\Ball{R} = \{(\xpar,x_n) : 0 < x_n < R,\ |\xpar| < \sqrt{R^2-x_n^2}\}$,
Fubini's theorem gives
\begin{equation*}
  \MavR{w \otimes \Psi}{R}
  = \int_0^{R}|\Psi(x_n)|^2\,
      \frac{N\bigl(\sqrt{R^2-x_n^2}\bigr)}{R}\,dx_n .
\end{equation*}
For each fixed $x_n$ we have, with $s = \sqrt{R^2-x_n^2}$,
\begin{equation*}
  \frac{N(s)}{R} = \frac{N(s)}{s}\cdot\frac{s}{R}
  \longrightarrow \Mavt{w}
  \qquad (R \to \infty),
\end{equation*}
because $s \to \infty$ and $s/R \to 1$. Moreover
$N(s)/R \le N(R)/R \le 4\norm{w}^2_{\Bstd{\Rnm}}$ by
Lemma~\ref{lem:shell-average}(i) applied on $\Rnm$, so the integrand is
dominated by $4\norm{w}^2_{\Bstd{\Rnm}}|\Psi|^2 \in L^1(\Rhalf)$.
Dominated convergence yields \eqref{eq:tensor-shell}.
\end{proof}

\subsubsection{Stationary phase on the sphere}
\label{sec:char-statphase}

We shall need the classical far-field expansion of a Herglotz wave, in the
form given by Agmon--H\"ormander~\cite{AgmonHormander1976} and
Isozaki~\cite[Lemma~2.4]{Isozaki2001}. We state it with an explicit uniform
remainder, which allows us to dispense with any separate
treatment of grazing directions.

\begin{lemma}[Full sphere]
\label{lem:sphere-statphase}
Let $m \ge 2$, $\kappa > 0$ and $y \in C^{\infty}(S^{m-1})$. Put
\begin{equation}
  \label{eq:Csph-def}
  \Csph{m}{\kappa}
  = \Bigl( \frac{2\pi}{\kappa} \Bigr)^{\frac{m-1}{2}}
    e^{-i\frac{(m-1)\pi}{4}} .
\end{equation}
Then, for $X \in \mathbb{R}^{m}$ with $|X| \ge 1$,
\begin{equation}
  \label{eq:sphere-statphase}
  \int_{S^{m-1}} e^{i\kappa\sigma\cdot X}\,y(\sigma)\,d\sigma
  = |X|^{-\frac{m-1}{2}}
    \Bigl[ \Csph{m}{\kappa}\,e^{i\kappa|X|}\,y(\widehat X)
      + \overline{\Csph{m}{\kappa}}\,e^{-i\kappa|X|}\,y(-\widehat X)
    \Bigr]
    + \mathcal{Q}(X),
\end{equation}
where $\widehat X = X/|X|$ and
$|\mathcal{Q}(X)| \le C(m,\kappa,y)\,|X|^{-\frac{m+1}{2}}$ uniformly in the
direction $\widehat X$.
\end{lemma}

This is \cite[Thm.~6.11]{Isozaki2023}.

The bulk eigenfunctions are parametrized by the \emph{half}-sphere $\Sm$, and
one might expect to need a version of \eqref{eq:sphere-statphase} for
integrals over $\Sm$, whose stationary point approaches the boundary
$\partial\Sm$ as $\theta_n \downarrow 0$ and whose remainder is therefore not
uniform. This is avoided as follows: we prove the expansion first for
densities supported in the \emph{open} half-sphere, extend such a density by
zero to a smooth function on $S^{n-1}$, and apply
Lemma~\ref{lem:sphere-statphase}. The antipodal stationary point then lies
outside the support and contributes nothing. The general case follows by
density through Lemma~\ref{lem:shell-average}(iv). Both the boundary term at
the equator and the grazing region disappear from the argument.

\subsubsection{Far-field expansion of the eigenoperator}
\label{sec:char-farfield}

\begin{definition}[Reflection involutions]
\label{def:reflection-involutions}
Define
\begin{equation*}
  (\Jbulk F)(\theta) = F(-\theta',\theta_n) = F(-\refext{\theta}),
  \qquad
  (\Jsurf G)(\omega) = G(-\omega),
\end{equation*}
for $F$ on $\Sm$ and $G$ on $\Ssurf$, and put
$\Jchan = \Jbulk \oplus \Jsurf$ on $\Echan{E}$ (with
$\Jchan = \Jsurf$ when $-\alpha^2 < E < 0$). Both $\Jbulk$ and $\Jsurf$ are
measure-preserving involutions, hence unitary. The map $\Jbulk$ reverses
only the tangential components and therefore maps $\Sm$ onto itself, whereas
$\Jsurf$ is the antipodal map of the boundaryless sphere $\Ssurf$.
\end{definition}

\begin{definition}[Incoming and outgoing far-field data]
\label{def:far-field-data}
Let $E \in \Ereg$ and $g = (g_{\mathrm{bulk}},g_{\mathrm{surf}}) \in
\Echan{E}$. Put
\begin{align}
  \label{eq:ydata-bulk}
  \yout_{\mathrm{bulk}}(\theta)
    &= \cph{\kbulk\theta_n}\,g_{\mathrm{bulk}}(\theta),
  &
  \yin_{\mathrm{bulk}}(\theta)
    &= \overline{\cph{\kbulk\theta_n}}\,
       (\Jbulk g_{\mathrm{bulk}})(\theta),
  \\[1mm]
  \label{eq:ydata-surf}
  \yout_{\mathrm{surf}}(\omega)
    &= g_{\mathrm{surf}}(\omega),
  &
  \yin_{\mathrm{surf}}(\omega)
    &= (\Jsurf g_{\mathrm{surf}})(\omega),
\end{align}
where $\cph{\cdot}$ is the unimodular phase factor of
\eqref{eq:phase-and-reflection}. Since $|\cph{\cdot}| = 1$ and $\Jbulk$,
$\Jsurf$ are unitary,
\begin{equation}
  \label{eq:ydata-norms}
  \begin{aligned}
  \norm{\yout_{\mathrm{bulk}}}_{\Ltwo{\Sm}}
  &= \norm{\yin_{\mathrm{bulk}}}_{\Ltwo{\Sm}}
   = \norm{g_{\mathrm{bulk}}}_{\Ltwo{\Sm}},
  \\[1mm]
  \norm{\yout_{\mathrm{surf}}}_{\Ltwo{\Ssurf}}
  &= \norm{\yin_{\mathrm{surf}}}_{\Ltwo{\Ssurf}}
   = \norm{g_{\mathrm{surf}}}_{\Ltwo{\Ssurf}} .
  \end{aligned}
\end{equation}
\end{definition}

The names are justified by the theorem below, and are consistent with the
outgoing normalization of \S\ref{subsec:asymptotics}: if $g = \Fzero(E)f$ for
some $f \in \Bsp$, then by \eqref{eq:Fout-def} the two data of
\eqref{eq:ydata-bulk} are $\yout_{\mathrm{bulk}} = \Fbulk^{(+)}(E)f$ and
$\yin_{\mathrm{bulk}} = \Jbulk\Fbulk^{(-)}(E)f$, the phase factor commuting
with $\Jbulk$ because $\cph{\kbulk\theta_n}$ depends on $\theta_n$ alone.
Thus $\Fzero^{(\pm)}(E)$ are precisely the transforms that read off the
outgoing and the incoming far-field amplitude of the standing wave
$\Fzero(E)^{*}g$, just as, by Corollary~\ref{cor:far-field}, they read off
the far field of the outgoing resolvent.

\begin{theorem}[Far field of the eigenoperator]
\label{thm:eigen-far-field}
Let $E \in \Ereg$ and $g \in \Echan{E}$, and set
\begin{equation}
  \label{eq:gamma-constants}
  \gambulk(E) = \frac{1}{2\sqrt{\pi}}\,\kbulk^{-1/2}\,
    e^{-i\frac{(n-1)\pi}{4}},
  \qquad
  \gamsurf(E) = \frac{1}{2\sqrt{\pi}}\,\ksurf^{-1/2}\,
    e^{-i\frac{(n-2)\pi}{4}} .
\end{equation}
Then, in the sense of $\simeq$,
\begin{equation}
  \label{eq:eigen-far-field}
  \begin{aligned}
  \Fzero(E)^{*}g
  \;\simeq\;
  & r^{-\frac{n-1}{2}}
    \Bigl[ \gambulk(E)\,e^{i\kbulk r}\,\yout_{\mathrm{bulk}}(\theta)
      + \overline{\gambulk(E)}\,e^{-i\kbulk r}\,
        \yin_{\mathrm{bulk}}(\theta) \Bigr]
  \\[1mm]
  + \;& \rho^{-\frac{n-2}{2}}
    \Bigl[ \gamsurf(E)\,e^{i\ksurf\rho}\,\yout_{\mathrm{surf}}(\omega)
      + \overline{\gamsurf(E)}\,e^{-i\ksurf\rho}\,
        \yin_{\mathrm{surf}}(\omega) \Bigr]\,\psizero(x_n),
  \end{aligned}
\end{equation}
the first line being absent when $-\alpha^2 < E < 0$. In particular
\begin{equation}
  \label{eq:gamma-moduli}
  |\gambulk(E)|^2 = \frac{1}{4\pi\kbulk},
  \qquad
  |\gamsurf(E)|^2 = \frac{1}{4\pi\ksurf} .
\end{equation}
\end{theorem}

\begin{proof}
\emph{Step 1: the bulk channel, for smooth data supported away from the
equator.} Let $E > 0$ and $g_{\mathrm{bulk}} \in C_c^{\infty}(\Sm)$.
Substituting \eqref{eq:Psib-exponential} into \eqref{eq:Fbulk-adjoint} and
using
$e^{i\kbulk\xpar\cdot\sigma'}e^{\mp i\kbulk\sigma_n x_n}
= e^{i\kbulk\refext{\sigma}\cdot x}$ resp.\ $e^{i\kbulk\sigma\cdot x}$
together with
$\overline{\cph{k}}\,\refl{k} = \cph{k}$, which is immediate from
\eqref{eq:phase-and-reflection}, we obtain
\begin{equation*}
  \Fbulk(E)^{*}g_{\mathrm{bulk}}
  = a_n(E)\bigl[ I_{-} + I_{+} \bigr],
  \qquad
  a_n(E) = \frac{E^{\frac{n-2}{4}}}{\sqrt2}\,(2\pi)^{-\frac{n}{2}} > 0,
\end{equation*}
where
\begin{equation*}
  I_{+} = \int_{\Sm}e^{i\kbulk\sigma\cdot x}\,
    \cph{\kbulk\sigma_n}\,g_{\mathrm{bulk}}(\sigma)\,d\sigma,
  \qquad
  I_{-} = \int_{\Sm}e^{i\kbulk\refext{\sigma}\cdot x}\,
    \overline{\cph{\kbulk\sigma_n}}\,g_{\mathrm{bulk}}(\sigma)\,d\sigma .
\end{equation*}
In $I_{+}$ extend the amplitude by zero and apply
Lemma~\ref{lem:sphere-statphase} with $m = n$, $\kappa = \kbulk$, $X = x$ and
$y_{+} = \mathbf{1}_{\Sm}\,\cph{\kbulk\sigma_n}g_{\mathrm{bulk}}$, which is
smooth on $S^{n-1}$ because $\operatorname{supp}g_{\mathrm{bulk}}$ is a
compact subset of the open half-sphere. Since $x \in \Rnp$ forces
$\theta \in \Sm$ and hence $-\theta \notin \operatorname{supp}y_{+}$, only the
outgoing stationary point survives:
\begin{equation*}
  I_{+} = r^{-\frac{n-1}{2}}\,\Csph{n}{\kbulk}\,e^{i\kbulk r}\,
    \cph{\kbulk\theta_n}\,g_{\mathrm{bulk}}(\theta)
    + O\bigl( r^{-\frac{n+1}{2}} \bigr).
\end{equation*}
In $I_{-}$ substitute $\tilde\sigma = \refext{\sigma}$, a measure-preserving
bijection of $\Sm$ onto the open lower half-sphere
$S^{n-1}_{-} = \{\tilde\sigma \in S^{n-1} : \tilde\sigma_n < 0\}$, with inverse
$\sigma = \refext{\tilde\sigma}$ and $\sigma_n = -\tilde\sigma_n$. This turns
$I_{-}$ into a direct-phase integral,
\begin{equation*}
  I_{-} = \int_{S^{n-1}}e^{i\kbulk\tilde\sigma\cdot x}\,
    y_{-}(\tilde\sigma)\,d\tilde\sigma,
  \qquad
  y_{-} = \mathbf{1}_{S^{n-1}_{-}}\,
    \overline{\cph{-\kbulk\tilde\sigma_n}}\,
    g_{\mathrm{bulk}}(\refext{\tilde\sigma}),
\end{equation*}
again smooth on $S^{n-1}$. Now $\theta \notin \operatorname{supp}y_{-}$,
while $-\theta \in S^{n-1}_{-}$ does lie in the support, so that this time
only the incoming stationary point survives:
\begin{equation*}
  I_{-} = r^{-\frac{n-1}{2}}\,\overline{\Csph{n}{\kbulk}}\,e^{-i\kbulk r}\,
    y_{-}(-\theta) + O\bigl( r^{-\frac{n+1}{2}} \bigr),
  \qquad
  y_{-}(-\theta) = \overline{\cph{\kbulk\theta_n}}\,
    (\Jbulk g_{\mathrm{bulk}})(\theta),
\end{equation*}
where we used $(-\theta)_n = -\theta_n$ and
$\refext{(-\theta)} = (-\theta',\theta_n)$. Since $a_n(E) > 0$ the two
constants remain conjugate after multiplication, and
\begin{equation*}
  a_n(E)\,\Csph{n}{\kbulk}
  = \frac{E^{\frac{n-2}{4}}}{\sqrt2}(2\pi)^{-\frac12}\kbulk^{-\frac{n-1}{2}}
    e^{-i\frac{(n-1)\pi}{4}}
  = \frac{1}{2\sqrt{\pi}}\kbulk^{-\frac12}e^{-i\frac{(n-1)\pi}{4}}
  = \gambulk(E),
\end{equation*}
because $E^{(n-2)/4} = \kbulk^{(n-2)/2}$. Finally the remainder is
$\simeq 0$: it is bounded on $\{r \le 1\}$, and in polar coordinates
\begin{equation*}
  \frac{1}{R}\int_{\Ball{R}\cap\{r>1\}}O\bigl(r^{-(n+1)}\bigr)\,dx
  \lesssim \frac{1}{R}\int_1^{R}r^{-(n+1)}r^{n-1}\,dr
  = \frac{1}{R}\int_1^{R}r^{-2}\,dr = O(R^{-1}).
\end{equation*}
This proves the first line of \eqref{eq:eigen-far-field} for
$g_{\mathrm{bulk}} \in C_c^{\infty}(\Sm)$.

\emph{Step 2: the surface channel.} Let
$g_{\mathrm{surf}} \in C^{\infty}(\Ssurf)$. By \eqref{eq:Psis-def} and
\eqref{eq:Fsurf-adjoint} the surface part factorizes \emph{exactly},
\begin{equation}
  \label{eq:surf-factorization}
  \Fsurf(E)^{*}g_{\mathrm{surf}} = w \otimes \psizero,
  \qquad
  w(\xpar) = a_n'(E)\int_{\Ssurf}
    e^{i\ksurf\sigma\cdot\xpar}\,g_{\mathrm{surf}}(\sigma)\,d\sigma,
\end{equation}
with $a_n'(E) = 2^{-1/2}(E+\alpha^2)^{(n-3)/4}(2\pi)^{-(n-1)/2} > 0$. Since
$\Ssurf$ is a closed manifold, Lemma~\ref{lem:sphere-statphase} with
$m = n-1$, $\kappa = \ksurf$ and $X = \xpar$ applies directly and both
stationary points $\pm\omega$ contribute, giving
\begin{equation*}
  w = \rho^{-\frac{n-2}{2}}
    \Bigl[ \gamsurf(E)e^{i\ksurf\rho}g_{\mathrm{surf}}(\omega)
      + \overline{\gamsurf(E)}e^{-i\ksurf\rho}
        g_{\mathrm{surf}}(-\omega) \Bigr]
    + O\bigl( \rho^{-\frac{n}{2}} \bigr),
\end{equation*}
with $a_n'(E)\Csph{n-1}{\ksurf} = \gamsurf(E)$ by the same computation as
above, using $(E+\alpha^2)^{(n-3)/4} = \ksurf^{(n-3)/2}$. The remainder
satisfies $\Mavt{\cdot} = 0$ by the argument of Step 1 carried out in
$\Rnm$, and Lemma~\ref{lem:tensor-shell} transports this to
$\Mav{\,\cdot \otimes \psizero\,} = 0$. Recalling
$\yin_{\mathrm{surf}}(\omega) = g_{\mathrm{surf}}(-\omega)$ we obtain
the second line of \eqref{eq:eigen-far-field}.

\emph{Step 3: density.} Both sides of \eqref{eq:eigen-far-field} are, for
fixed $E$, values of maps $\Echan{E} \to \Bst$ that are bounded: the left
side by Proposition~\ref{prop:F-adjoint}, the right side by the elementary
estimates
\begin{equation*}
  \bigl\| A(\theta)\,r^{-\frac{n-1}{2}}e^{\pm i\kbulk r}
  \bigr\|_{\Bst}
  \le C\norm{A}_{\Ltwo{\Sm}},
  \qquad
  \bigl\| B(\omega)\,\rho^{-\frac{n-2}{2}}e^{\pm i\ksurf\rho}
    \otimes \psizero \bigr\|_{\Bst}
  \le C\norm{B}_{\Ltwo{\Ssurf}},
\end{equation*}
the second via \eqref{eq:tensor-bounds}, both of which follow from the
polar-coordinate computation used for the remainders. Since
$C_c^{\infty}(\Sm) \oplus C^{\infty}(\Ssurf)$ is dense in $\Echan{E}$ and
$\Bstz$ is a closed subspace of $\Bst$, the difference of the two sides,
being a bounded map $\Echan{E} \to \Bst$ vanishing into $\Bstz$ on a dense
subspace, takes values in $\Bstz$ throughout. The case
$-\alpha^2 < E < 0$ is Step 2 alone, since then
$\Echan{E} = \Ltwo{\Ssurf}$ by \eqref{eq:channel-space}. Finally
\eqref{eq:gamma-moduli} is immediate from \eqref{eq:gamma-constants}.
\end{proof}

Since $\Jbulk$ commutes with multiplication by $\cph{\kbulk\theta_n}$ and
$\Jbulk^2 = \mathrm{id}$, the data
\eqref{eq:ydata-bulk}--\eqref{eq:ydata-surf} satisfy
\begin{equation}
  \label{eq:char-smatrix-relation}
  \yout_{\mathrm{bulk}}
  = \refl{\kbulk\theta_n}\,\bigl( \Jbulk \yin_{\mathrm{bulk}} \bigr),
  \qquad
  \yout_{\mathrm{surf}} = \Jsurf \yin_{\mathrm{surf}} .
\end{equation}
This is the germ of the $S$-matrix already noted in
\eqref{eq:smatrix-germ}. It will be given its proper place in
\S\ref{sec:char-smatrix}, once the characterization theorem has shown that
the incoming datum determines the solution. At this stage only the
consequence $|\refl{\cdot}| = 1$ is needed, and it makes the two amplitudes in
\eqref{eq:ydata-norms} equal.

\subsubsection{Orthogonality of the two channels}
\label{sec:char-channel}

To compute the shell average of the sum \eqref{eq:eigen-far-field} we must
know that the two channels do not interfere. It is tempting to argue that
they decouple because the phases $e^{i\kbulk r}$ and $e^{i\ksurf\rho}$
oscillate at different rates, the gap $\ksurf - \kbulk > 0$ being available
at every energy by \eqref{eq:asymp-wavenumbers}. Such an argument would require
pointwise far-field information about $\Fbulk(E)^{*}g_{\mathrm{bulk}}$ in the
layer $\{x_n = O(1/\alpha)\}$ where the surface channel is concentrated, that
is, precisely in the grazing region where that far field is least accessible.
We use instead an exact mechanism, which avoids the grazing region
altogether: by \eqref{eq:fiber-recall} the transverse profiles
$\gphi_k$ and $\psizero$ are orthogonal in $\Ltwo{\Rhalf}$ for \emph{every}
$k > 0$.

\begin{proposition}[$\Bst$-orthogonality of the channels]
\label{prop:char-channel-orth}
Let $E > 0$, $g_{\mathrm{bulk}} \in \Ltwo{\Sm}$,
$g_{\mathrm{surf}} \in \Ltwo{\Ssurf}$, and put
$\Ubulk = \Fbulk(E)^{*}g_{\mathrm{bulk}}$,
$\Usurf = \Fsurf(E)^{*}g_{\mathrm{surf}}$, the capital letters distinguishing
these from the model waves \eqref{eq:ubulk}--\eqref{eq:usurf}. Then
\begin{equation}
  \label{eq:char-channel-orth}
  \lim_{R\to\infty}\frac{1}{R}
  \int_{\Ball{R}}\Ubulk(x)\,\overline{\Usurf(x)}\,dx = 0 .
\end{equation}
\end{proposition}

\begin{proof}
Assume first $g_{\mathrm{bulk}} \in C(\overline{\Sm})$ and
$g_{\mathrm{surf}} \in C^{\infty}(\Ssurf)$, and write
$\Usurf = w \otimes \psizero$ as in \eqref{eq:surf-factorization}.

\emph{Step 1: exact transverse orthogonality.} By \eqref{eq:fiber-recall} we
have $\int_0^{\infty}\gphi_k(t)\psizero(t)\,dt = 0$ for every $k > 0$. Since
$|\gphi_k| \le \sqrt{2/\pi}$ and $\psizero \in L^1(\Rhalf)$, Fubini's theorem
applies to \eqref{eq:Psib-def}--\eqref{eq:Fbulk-adjoint} and gives, for
\emph{every} $\xpar \in \Rnm$,
\begin{equation}
  \label{eq:transverse-orth}
  \int_0^{\infty}\Ubulk(\xpar,x_n)\,\psizero(x_n)\,dx_n
  = \frac{E^{\frac{n-2}{4}}}{\sqrt2}(2\pi)^{-\frac{n-1}{2}}
    \int_{\Sm}e^{i\kbulk\xpar\cdot\sigma'}g_{\mathrm{bulk}}(\sigma)
    \Bigl[ \int_0^{\infty}\gphi_{\kbulk\sigma_n}\psizero\,dt \Bigr]
    d\sigma = 0 ,
\end{equation}
the function $\psizero$ being real-valued.

\emph{Step 2: the cylinder contributes nothing.} Let
$Z_R = \{|\xpar| < R\} \times (0,\infty)$. Both $\Ubulk$ and $w$ are bounded
(the former by $\norm{g_{\mathrm{bulk}}}_{L^1(\Sm)}$, since
$|\gphi_k| \le \sqrt{2/\pi}$ uniformly in $k$ by \eqref{eq:fiber-recall}), so
$\Ubulk\overline{\Usurf} \in L^1(Z_R)$, and by \eqref{eq:transverse-orth}
\begin{equation*}
  \int_{Z_R}\Ubulk\,\overline{\Usurf}\,dx
  = \int_{\{|\xpar|<R\}}\overline{w(\xpar)}
    \Bigl[ \int_0^{\infty}\Ubulk(\xpar,x_n)\psizero(x_n)\,dx_n \Bigr]
    d\xpar = 0 .
\end{equation*}
Hence
$\int_{\Ball{R}}\Ubulk\overline{\Usurf}
= -\int_{Z_R \setminus \Ball{R}}\Ubulk\overline{\Usurf}$, and it suffices to
estimate the latter. Explicitly
$Z_R \setminus \Ball{R}
= \{ |\xpar| < R,\ x_n > \sqrt{R^2-|\xpar|^2} \}$.

\emph{Step 3: the region $x_n \ge R$.} By \eqref{eq:psi0}
$\psizero(x_n) = \sqrt{2\alpha}\,e^{-\alpha x_n}$, and the ball
$\{|\xpar| < R\}$ of $\Rnm$ has volume
$\int_{\{|\xpar|<R\}}d\xpar = (n-1)^{-1}R^{n-1}\int_{\Ssurf}dS(\omega)$, so
\begin{equation*}
\begin{aligned}
  \int_{\{|\xpar|<R\}}\int_R^{\infty}|\Ubulk|\,|\Usurf|\,dx_n\,d\xpar
  &\le \norm{\Ubulk}_{L^{\infty}}\norm{w}_{L^{\infty}}
      \sqrt{2\alpha}\,\Bigl(\int_{\{|\xpar|<R\}}d\xpar\Bigr)
      \int_R^{\infty}e^{-\alpha x_n}\,dx_n
  \\[1mm]
  &= O\bigl( R^{n-1}e^{-\alpha R} \bigr).
  \end{aligned}
\end{equation*}

\emph{Step 4: the region $x_n < R$.} This part of
$Z_R \setminus \Ball{R}$ is contained in $\Ball{\sqrt2 R}$, so by
Lemma~\ref{lem:shell-average}(i)
\begin{equation*}
  \int_{(Z_R \setminus \Ball{R}) \cap \{x_n<R\}}|\Ubulk|^2\,dx
  \le 4\sqrt2\,R\,\Bstnorm{\Ubulk}^2 .
\end{equation*}
For $\Usurf$ we use the uniform pointwise bound
$|w(\xpar)| \le C_{g}(1+\rho)^{-(n-2)/2}$, which is
Lemma~\ref{lem:sphere-statphase} applied on the closed manifold $\Ssurf$ with
the smooth density $g_{\mathrm{surf}}$. Then, integrating $x_n$ first,
\begin{equation*}
  \int_{(Z_R \setminus \Ball{R}) \cap \{x_n<R\}}|\Usurf|^2\,dx
  \le \int_{\{|\xpar|<R\}}|w(\xpar)|^2
      e^{-2\alpha\sqrt{R^2-|\xpar|^2}}\,d\xpar
  \le C\int_0^{R}e^{-2\alpha\sqrt{R^2-\rho^2}}\,d\rho ,
\end{equation*}
where the constant absorbed $\rho^{n-2}(1+\rho)^{-(n-2)}\le 1$ and the area
of $\Ssurf$. Substituting $s = \sqrt{R^2-\rho^2}$ gives
\begin{equation*}
  \int_0^{R}e^{-2\alpha\sqrt{R^2-\rho^2}}\,d\rho
  = \int_0^{R}e^{-2\alpha s}\frac{s\,ds}{\sqrt{R^2-s^2}}
  \le \frac{2}{\sqrt3\,R}\int_0^{\infty}se^{-2\alpha s}\,ds
    + e^{-\alpha R}\frac{\sqrt3}{2}R
  = O(R^{-1}),
\end{equation*}
the two terms coming from $s \le R/2$, where
$\sqrt{R^2-s^2} \ge \tfrac{\sqrt3}{2}R$, and from $s > R/2$, where
$e^{-2\alpha s} \le e^{-\alpha R}$ and
$\int_{R/2}^{R}s(R^2-s^2)^{-1/2}ds = \tfrac{\sqrt3}{2}R$. By
Cauchy--Schwarz the contribution of this region to
$\int_{Z_R\setminus\Ball{R}}|\Ubulk||\Usurf|$ is therefore
$O(R^{1/2})\cdot O(R^{-1/2}) = O(1)$.

Combining Steps 3 and 4 and dividing by $R$ gives
\eqref{eq:char-channel-orth} with rate $O(R^{-1})$ for the dense class of
data considered.

\emph{Step 5: density.} Write
$\Lambda_R(g_{\mathrm{bulk}},g_{\mathrm{surf}})
= R^{-1}\int_{\Ball{R}}\Ubulk\overline{\Usurf}\,dx$. By
Cauchy--Schwarz and Lemma~\ref{lem:shell-average}(i),
\begin{equation*}
\begin{aligned}
  |\Lambda_R(g_{\mathrm{bulk}},g_{\mathrm{surf}})|
  &\le \MavR{\Ubulk}{R}^{1/2}\MavR{\Usurf}{R}^{1/2}
   \le 4\,\Bstnorm{\Ubulk}\Bstnorm{\Usurf}
  \\[1mm]
  &\le C\norm{g_{\mathrm{bulk}}}_{\Ltwo{\Sm}}
       \norm{g_{\mathrm{surf}}}_{\Ltwo{\Ssurf}},
  \end{aligned}
\end{equation*}
uniformly in $R \ge 1$, the last step by
Proposition~\ref{prop:F-adjoint}. A sesquilinear family that is uniformly
bounded and tends to $0$ on a dense subspace tends to $0$ everywhere, which
completes the proof.
\end{proof}

The cancellation \eqref{eq:transverse-orth} is exact and pointwise in
$\xpar$, and holds at every energy including, formally, the threshold. Only
the difference between the ball $\Ball{R}$ and the cylinder $Z_R$ has to be
estimated, and there the exponential localization of $\psizero$ does all the
work.

\subsubsection{The shell-average identity and the two-sided bound}
\label{sec:char-lower-bound}

\begin{theorem}[Shell-average identity for eigenfunctions]
\label{thm:shell-average-identity}
Let $E \in \Ereg$ and $\phi = (\phi_{\mathrm{bulk}},\phi_{\mathrm{surf}})
\in \Echan{E}$. Then $\Mav{\Fzero(E)^{*}\phi}$ exists and
\begin{equation}
  \label{eq:shell-average-identity}
  \Mav{\Fzero(E)^{*}\phi}
  = \frac{1}{2\pi\kbulk}\,\norm{\phi_{\mathrm{bulk}}}^2_{\Ltwo{\Sm}}
    + \frac{1}{2\pi\ksurf}\,\norm{\phi_{\mathrm{surf}}}^2_{\Ltwo{\Ssurf}},
\end{equation}
the first term being absent when $-\alpha^2 < E < 0$. The same identity holds
with $\Fzero(E)^{*}$ replaced by $\Fzero^{(\pm)}(E)^{*}$.
\end{theorem}

\begin{proof}
Let first $E > 0$ and write $u = \Ubulk + \Usurf$ as in
Proposition~\ref{prop:char-channel-orth}. By that proposition,
\begin{equation*}
  \MavR{u}{R} = \MavR{\Ubulk}{R} + \MavR{\Usurf}{R}
    + 2\re\Lambda_R \quad\text{with}\quad \Lambda_R \to 0,
\end{equation*}
so it suffices to compute the two diagonal terms. By
Lemma~\ref{lem:shell-average}(iii) each may be computed on the far-field
representative furnished by Theorem~\ref{thm:eigen-far-field}.

For the bulk term let
$v = r^{-(n-1)/2}\bigl[\gambulk e^{i\kbulk r}\yout_{\mathrm{bulk}}(\theta)
+ \overline{\gambulk}e^{-i\kbulk r}\yin_{\mathrm{bulk}}(\theta)\bigr]$, which
is locally square integrable since $r^{-(n-1)}$ is. In polar coordinates
$dx = r^{n-1}dr\,dS(\theta)$ on $\Rnp$, with $\theta$ ranging over $\Sm$, the
weight $r^{-(n-1)}$ cancels the Jacobian and
\begin{equation*}
  \MavR{v}{R}
  = \frac{1}{R}\int_0^{R}\!\!\int_{\Sm}
    \Bigl[ |\gambulk|^2\bigl( |\yout_{\mathrm{bulk}}|^2
      + |\yin_{\mathrm{bulk}}|^2 \bigr)
      + 2\re\bigl( \gambulk^2 e^{2i\kbulk r}
        \yout_{\mathrm{bulk}}\overline{\yin_{\mathrm{bulk}}} \bigr)
    \Bigr] dS(\theta)\,dr .
\end{equation*}
The cross term carries the non-stationary radial phase $e^{2i\kbulk r}$ and
contributes
$R^{-1}\int_0^{R}e^{2i\kbulk r}dr = O\bigl((\kbulk R)^{-1}\bigr) \to 0$,
while $R^{-1}\int_0^{R}dr = 1$. Hence, by \eqref{eq:ydata-norms} and
\eqref{eq:gamma-moduli},
\begin{equation*}
  \Mav{\Ubulk} = \Mav{v}
  = 2|\gambulk(E)|^2\norm{\phi_{\mathrm{bulk}}}^2_{\Ltwo{\Sm}}
  = \frac{1}{2\pi\kbulk}\norm{\phi_{\mathrm{bulk}}}^2_{\Ltwo{\Sm}} .
\end{equation*}
For the surface term the far-field representative is
$v' \otimes \psizero$, where
\begin{equation*}
  v' = \rho^{-\frac{n-2}{2}}
    \Bigl[ \gamsurf(E)\,e^{i\ksurf\rho}\,\yout_{\mathrm{surf}}(\omega)
      + \overline{\gamsurf(E)}\,e^{-i\ksurf\rho}\,
        \yin_{\mathrm{surf}}(\omega) \Bigr].
\end{equation*}
The identical computation performed in $\Rnm$, where
$d\xpar = \rho^{n-2}d\rho\,dS(\omega)$ with $\omega$ ranging over the whole
of $\Ssurf$, gives
$\Mavt{v'} = 2|\gamsurf(E)|^2
\norm{\phi_{\mathrm{surf}}}^2_{\Ltwo{\Ssurf}}$, and
Lemma~\ref{lem:tensor-shell} together with
$\norm{\psizero}_{\Ltwo{\Rhalf}} = 1$ yields
\begin{equation*}
  \Mav{\Usurf} = \Mav{v' \otimes \psizero} = \Mavt{v'}
  = \frac{1}{2\pi\ksurf}\norm{\phi_{\mathrm{surf}}}^2_{\Ltwo{\Ssurf}} .
\end{equation*}
Adding the two gives \eqref{eq:shell-average-identity}. For
$-\alpha^2 < E < 0$ only the surface computation occurs and there is no cross
term. Finally
$\Fzero^{(\pm)}(E)^{*}\phi = \Fzero(E)^{*}\bigl(m^{\mp1}\phi\bigr)$ for the
unimodular multiplier $m = \cph{\kbulk\theta_n} \oplus 1$ of
\eqref{eq:Fout-def}, and the right-hand side of
\eqref{eq:shell-average-identity} is unchanged under multiplication of
$\phi$ by a unimodular function.
\end{proof}

\begin{corollary}[Two-sided estimate and closed range]
\label{cor:two-sided}
Let $E \in \Ereg$ and let $\Fzero^{\bullet}$ denote any one of $\Fzero$,
$\Fzero^{(+)}$, $\Fzero^{(-)}$. Then for every $\phi \in \Echan{E}$
\begin{equation}
  \label{eq:two-sided}
  \frac{1}{\sqrt{8\pi\ksurf}}\;\norm{\phi}_{\Echan{E}}
  \;\le\; \Bstnorm{\Fzero^{\bullet}(E)^{*}\phi}
  \;\le\; \norm{\Fzero(E)^{*}}_{\Lop{\Echan{E}}{\Bst}}\,
    \norm{\phi}_{\Echan{E}} .
\end{equation}
In particular each of the three operators
$\Fzero^{\bullet}(E)^{*} : \Echan{E} \to \Bst$ is injective, is bounded below,
and has closed range.
\end{corollary}

\begin{proof}
The upper bound is Proposition~\ref{prop:F-adjoint}. For the lower one,
observe that $\ksurf > \kbulk > 0$ for $E > 0$, so that
$(2\pi\ksurf)^{-1}$ is the smaller of the two coefficients in
\eqref{eq:shell-average-identity} and
\begin{equation*}
  \Mav{\Fzero^{\bullet}(E)^{*}\phi}
  \ge \frac{1}{2\pi\ksurf}
    \Bigl( \norm{\phi_{\mathrm{bulk}}}^2_{\Ltwo{\Sm}}
      + \norm{\phi_{\mathrm{surf}}}^2_{\Ltwo{\Ssurf}} \Bigr)
  = \frac{1}{2\pi\ksurf}\norm{\phi}^2_{\Echan{E}} .
\end{equation*}
For $-\alpha^2 < E < 0$ this holds with equality. Now apply
Lemma~\ref{lem:shell-average}(v), the identity
\eqref{eq:shell-average-identity} being available for all three normalizations
by Theorem~\ref{thm:shell-average-identity}. An operator bounded below is
injective, and its range is closed: if $\Fzero^{\bullet}(E)^{*}\phi_j \to u$ in
$\Bst$, then $(\phi_j)$ is Cauchy in $\Echan{E}$ by the lower bound, so
$\phi_j \to \phi$ and $u = \Fzero^{\bullet}(E)^{*}\phi$ by continuity.
\end{proof}

\begin{remark}[Interpretation: total reflection]
\label{rem:total-reflection}
In \eqref{eq:shell-average-identity} the incoming and the outgoing halves of
the far field contribute equally, each of them
$(4\pi\kappa)^{-1}\norm{\phi_{\text{channel}}}^2$ with $\kappa$ the effective
wavenumber of the channel. This equality is the total reflection
$|\refl{\cdot}| = 1$ of \eqref{eq:phase-and-reflection}: the Robin boundary
absorbs nothing, and $\Fzero(E)^{*}\phi$ is a standing wave. It also explains
why Corollary~\ref{cor:two-sided} does not follow from Stone's formula
\eqref{eq:stone}, which controls $\Fzero(E)^{*}$ only from above: the outgoing
resolvent $\Ralpha{E+i0}f$ carries only the $e^{+i\kbulk r}$ component, and
its shell average \eqref{eq:shell-flux} has one term per channel instead of
two.
\end{remark}

\subsubsection{The shell average of the outgoing resolvent}
\label{sec:char-flux}

The second pillar of the characterization requires the shell average of
$\Ralpha{E+i0}f$ for $f \in \Bsp$ and for every $E \in \Ereg$. For $E > 0$
this follows from Theorem~\ref{thm:asymp-uniform}
(Lemma~\ref{lem:shell-flux-B} below). On the surface band the corresponding
expansion is not yet available, because \S\ref{subsec:asymptotics} is written
under the standing assumption $E > 0$. We supply it first.

There the bulk channel is closed, and the expansion is simpler than for
$E > 0$: it needs neither stationary phase nor
any analysis of the grazing region, because the whole bulk part of the
resolvent is an $\Ltwo{\Rnp}$ function. The mechanism is the exact reduction
of Proposition~\ref{prop:surface-reduction}. We write
$\Vsurf : \Ltwo{\Rnm} \to \Ltwo{\Rnp}$ for the isometry
$(\Vsurf g)(\xpar,x_n) = g(\xpar)\psizero(x_n)$, so that
$\Vsurf^{*} = T_{\psizero}$ in the notation of
Corollary~\ref{cor:tensor}, $\psizero$ being real, that
$\Vsurf^{*}\Vsurf = I$ because $\norm{\psizero}_{\Ltwo{\Rhalf}} = 1$, and
that $\Vsurf\Vsurf^{*} = I \otimes \Psurf$ is the orthogonal projection onto
$\Hsurf$. Let $\Pibulk = I - \Vsurf\Vsurf^{*}$ be the projection onto
$\Hbulk$.

\begin{proposition}[Expansion of the resolvent on the surface band]
\label{prop:char-surface-band}
Let $-\alpha^2 < E < 0$ and $f \in \Bsp$. Then
\begin{equation}
  \label{eq:char-sb-splitting}
  \Ralpha{E+i0}f
  = \Vsurf\Bigl( \Rfree{n-1}{E+\alpha^2+i0}\,\Vsurf^{*}f \Bigr)
    \;+\; \Ralpha{E}\,\Pibulk f,
\end{equation}
and the second summand satisfies
\begin{equation}
  \label{eq:char-sb-L2}
  \norm{\Ralpha{E}\,\Pibulk f}_{\Ltwo{\Rnp}}
  \le \frac{1}{|E|}\,\norm{f}_{\Ltwo{\Rnp}}
  \le \frac{1}{|E|}\,\Bnorm{f} .
\end{equation}
Consequently
\begin{equation}
  \label{eq:char-sb-expansion}
  \Ralpha{E+i0}f \;\simeq\; \Ws{+}{C},
  \qquad
  C = c_n'\,\bigl( \Fsurf(E)f \bigr),
\end{equation}
with $c_n' = \sqrt{\pi/\ksurf}\,e^{-i(n-4)\pi/4}$ as in \eqref{eq:usurf}, all
of whose ingredients depend on $E$ only through the surface wavenumber $\ksurf$.
\end{proposition}

\begin{proof}
By Proposition~\ref{prop:surface-reduction} the orthogonal decomposition
$\Ltwo{\Rnp} = \Hsurf \oplus \Hbulk$ reduces $\Halpha$, so that for
$z \notin \mathbb{R}$ the resolvent splits accordingly, and on the first
summand it is given by \eqref{eq:surface-resolvent}. This is
\eqref{eq:char-sb-splitting} for $z$ off the axis.

It remains to pass to the boundary value, and here the two summands behave
differently. By \eqref{eq:bulk-part} the restriction $\Halpha|_{\Hbulk}$ is
unitarily equivalent to multiplication by $|\xipar|^2 + k^2$ on
$\Ltwo{\Rnm \times \Rhalf}$, so
$\sigma\bigl(\Halpha|_{\Hbulk}\bigr) = [0,\infty)$. Hence every $E < 0$ lies
in the resolvent set of $\Halpha|_{\Hbulk}$, the map
$z \mapsto \Ralpha{z}|_{\Hbulk}$ is holomorphic in the $\Ltwo{}$ operator norm
in a neighbourhood of $E$, and by the spectral theorem
\begin{equation*}
  \bigl\| \Ralpha{E}\big|_{\Hbulk} \bigr\|_{\Lop{\Ltwo{\Rnp}}{\Ltwo{\Rnp}}}
  \le \dist\bigl( E, [0,\infty) \bigr)^{-1} = |E|^{-1} .
\end{equation*}
The bulk summand of the boundary value $\Ralpha{E+i0}$ of
Theorem~\ref{thm:lap} therefore coincides with $\Ralpha{E}\Pibulk$, and
\eqref{eq:char-sb-L2} follows, the second inequality because
$\norm{f}^2_{\Ltwo{\Rnp}} = \sum_{j\ge0}\norm{f}^2_{\Ltwo{\shellp{j}}}
\le \bigl( \sum_{j\ge0}2^{j/2}\norm{f}_{\Ltwo{\shellp{j}}} \bigr)^2
= \Bnorm{f}^2$ by \eqref{eq:B-norms}. On the surface summand
$E + \alpha^2 > 0$ is a non-threshold energy for $-\Delta'$ on $\Rnm$ and
$\Vsurf^{*}f \in \Bspd{\Rnm}$ by \eqref{eq:tensor-bounds}, so the boundary
value $\Rfree{n-1}{E+\alpha^2+i0}\Vsurf^{*}f \in \Bstd{\Rnm}$ is furnished by
the free limiting absorption principle in dimension $n-1$.

For \eqref{eq:char-sb-expansion}, note first that
$\Ltwo{\Rnp} \subset \Bstz$, since
$R^{-1}\int_{\Ball{R}}|h|^2 \le R^{-1}\norm{h}^2_{\Ltwo{\Rnp}} \to 0$. By
\eqref{eq:char-sb-L2} the second summand of \eqref{eq:char-sb-splitting} is
therefore negligible, and only the first has to be expanded. Assume
$f \in C_c^{\infty}(\Rnp)$. Proposition~\ref{prop:asymp-surface} applies
verbatim: its proof invokes Lemma~\ref{lem:hankel} with $m = n-1$ and
$k = \ksurf > 0$, uses \eqref{eq:surface-formula} and \eqref{eq:Fsurf-def},
and nowhere uses the sign of $E$. It gives
\begin{equation*}
  \Bigl[ \Vsurf\Bigl( \Rfree{n-1}{E+\alpha^2+i0}\Vsurf^{*}f \Bigr) \Bigr](x)
  = \Ws{+}{C}(x) + O\bigl( \rho^{-\frac n2} \bigr)e^{-\alpha x_n},
  \qquad \rho \to \infty,
\end{equation*}
uniformly in $\omega \in \Ssurf$ and in $x_n \ge 0$. Exactly as in
\eqref{eq:asymp-split}\,ff., but now with no wedge to isolate,
\begin{equation*}
  \frac{1}{R}\int_{\Ball{R} \cap \{\rho>1\}}
    O\bigl( \rho^{-n} \bigr)e^{-2\alpha x_n}\,dx
  \le \frac{C}{R}\int_{\Ssurf}\!\!\int_1^{R}\rho^{-n}\rho^{n-2}
      \Bigl( \int_0^{\infty}e^{-2\alpha x_n}dx_n \Bigr) d\rho\,dS(\omega)
  = O(R^{-1}),
\end{equation*}
the normal integral being finite independently of $R$. On the tube
$\{\rho \le 1\}$ both functions are square integrable, the first because it
is smooth with transverse decay $e^{-\alpha x_n}$ and the second because
$\int_{\{\rho\le1\}}\rho^{-(n-2)}\rho^{n-2}d\rho < \infty$. That region
therefore contributes $O(R^{-1})$ as well. Letting $R \to \infty$ gives
\eqref{eq:simeq}. The extension to $f \in \Bsp$ is a density argument:
$f \mapsto \Ralpha{E+i0}f$ is continuous from $\Bsp$ to $\Bst$ by
Theorem~\ref{thm:lap}(b), $f \mapsto \Ws{+}{C}$ is continuous by
Proposition~\ref{prop:F-bounded}, and $\Bstz$ is closed in $\Bst$.
\end{proof}

The surface band is thus the one place in this paper where the far field is
obtained with no oscillatory-integral analysis at all: for $E > 0$ the
three-term decomposition \eqref{eq:splitting-preview} of
\S\ref{subsec:lap} has to be used together with a stationary-phase analysis
of the reflected Sommerfeld integral and the shell-negligibility of the
grazing wedge (Proposition~\ref{prop:asymp-grazing}), whereas for $E < 0$ the
entire bulk part of the resolvent is replaced by a single bounded operator.
The same reduction, read as a statement about $\alpha$ rather than about $E$,
is the content of Lemma~\ref{lem:inv-surface-free}.

\begin{lemma}[Shell average of the resolvent on all of $\Bsp$]
\label{lem:shell-flux-B}
Let $E \in \Ereg$ and $f \in \Bsp$, and put $v = \Ralpha{E+i0}f$. Then
$\Mav{v}$ exists and
\begin{equation}
  \label{eq:shell-flux-B}
  \Mav{v}
  = \frac{\pi}{\kbulk}\bigl\| \Fbulk(E)f \bigr\|^2_{\Ltwo{\Sm}}
    + \frac{\pi}{\ksurf}\bigl\| \Fsurf(E)f \bigr\|^2_{\Ltwo{\Ssurf}} ,
\end{equation}
the first term being absent when $-\alpha^2 < E < 0$. In particular
$\Fzero(E)f = 0$ implies $v \in \Bstz$.
\end{lemma}

\begin{proof}
Let first $E > 0$, and put $W_1 = \Wb{+}{A_f}$ and $W_2 = \Ws{+}{C_f}$. By
Theorem~\ref{thm:asymp-uniform}, $v \simeq W_1 + W_2$, so by
Lemma~\ref{lem:shell-average}(iii) it suffices to show that $\Mav{W_1+W_2}$
exists and equals the right-hand side of \eqref{eq:shell-flux-B}.

In polar coordinates $dx = r^{n-1}dr\,dS(\theta)$ on $\Rnp$ the weight
$r^{-(n-1)}$ of $|W_1|^2$ cancels the Jacobian, so that for every $R$
\begin{equation*}
  \MavR{W_1}{R}
  = \frac{1}{R}\int_0^{R}\!\!\int_{\Sm}|A_f(\theta)|^2\,dS(\theta)\,dr
  = \norm{A_f}^2_{\Ltwo{\Sm}}
  = |c_n|^2\,\bigl\| \Fbulk(E)f \bigr\|^2_{\Ltwo{\Sm}},
\end{equation*}
the last step by \eqref{eq:ubulk} and $|\cph{\cdot}| = 1$. In particular
$W_1 \in \Bst$. The computation made below for $E < 0$ applies verbatim to
$W_2$ and gives $\Mav{W_2} = |c_n'|^2\,\norm{\Fsurf(E)f}^2_{\Ltwo{\Ssurf}}$.
It remains to show that
$\Lambda_R = R^{-1}\int_{\Ball{R}}W_1\overline{W_2}\,dx \to 0$. Fix $L > 0$
and split $\Ball{R}$ along the hyperplane $\{x_n = L\}$. By the
Cauchy--Schwarz inequality,
\begin{equation*}
  |\Lambda_R|
  \le \Bigl( \frac{1}{R}\int_{\Ball{R}\cap\{x_n<L\}}|W_1|^2\,dx \Bigr)^{1/2}
      \MavR{W_2}{R}^{1/2}
  + \MavR{W_1}{R}^{1/2}
      \Bigl( \frac{1}{R}\int_{\Ball{R}\cap\{x_n \ge L\}}|W_2|^2\,dx \Bigr)^{1/2}.
\end{equation*}
In the second term
$|W_2|^2 = |C_f(\omega)|^2\rho^{-(n-2)}\psizero(x_n)^2$ and
$\int_L^{\infty}\psizero(x_n)^2\,dx_n = e^{-2\alpha L}$, so that, integrating
first in $x_n$ and then in polar coordinates on $\{|\xpar| < R\}$, the last
factor is at most $\norm{C_f}_{\Ltwo{\Ssurf}}\,e^{-\alpha L}$. In the first
term the set $\{x_n < L\}$ meets the sphere $\{|x| = r\}$ in
$\{\theta_n < L/r\}$, so that for every $r_0 > 0$
\begin{equation*}
  \frac{1}{R}\int_{\Ball{R}\cap\{x_n<L\}}|W_1|^2\,dx
  \le \frac{r_0}{R}\,\norm{A_f}^2_{\Ltwo{\Sm}}
    + \int_{\{\theta \in \Sm:\ \theta_n < L/r_0\}}|A_f(\theta)|^2\,dS(\theta),
\end{equation*}
and the last integral tends to $0$ as $r_0 \to \infty$ by dominated
convergence. Since $\MavR{W_1}{R}$ and $\MavR{W_2}{R}$ are bounded in $R$,
letting first $R \to \infty$, then $r_0 \to \infty$, and finally
$L \to \infty$ gives $\Lambda_R \to 0$. Hence $\Mav{W_1+W_2}$ exists and
equals $\Mav{W_1} + \Mav{W_2}$, and $|c_n|^2 = \pi/\kbulk$,
$|c_n'|^2 = \pi/\ksurf$ by \eqref{eq:constants}. This proves
\eqref{eq:shell-flux-B}, and in particular \eqref{eq:shell-flux}.

Let now $-\alpha^2 < E < 0$. By Proposition~\ref{prop:char-surface-band} and
Lemma~\ref{lem:shell-average}(iii) it suffices to compute
$\Mav{\Ws{+}{C}}$. Write $\Ws{+}{C} = w \otimes \psizero$ with
$w(\xpar) = C(\omega)\rho^{-(n-2)/2}e^{i\ksurf\rho}$. In polar coordinates
$d\xpar = \rho^{n-2}d\rho\,dS(\omega)$ the weight $\rho^{-(n-2)}$ cancels the
Jacobian exactly, so that for every $R$
\begin{equation*}
  \MavtR{w}{R}
  = |c_n'|^2\,\frac{1}{R}\int_0^{R}\!\!\int_{\Ssurf}
    \bigl| \bigl(\Fsurf(E)f\bigr)(\omega) \bigr|^2\,dS(\omega)\,d\rho
  = |c_n'|^2\,\bigl\| \Fsurf(E)f \bigr\|^2_{\Ltwo{\Ssurf}} ;
\end{equation*}
in particular $w \in \Bstd{\Rnm}$ and $\Mavt{w}$ exists. By
Lemma~\ref{lem:tensor-shell} and $\norm{\psizero}_{\Ltwo{\Rhalf}} = 1$ we get
$\Mav{\Ws{+}{C}} = \Mavt{w}$, and $|c_n'|^2 = \pi/\ksurf$ by
\eqref{eq:constants}. This is \eqref{eq:shell-flux-B} without the bulk term.

In both cases the last assertion follows from
Lemma~\ref{lem:shell-average}(iii).
\end{proof}

\subsubsection{Orthogonality to the kernel of the transform}
\label{sec:char-kernel}

The second pillar of the characterization is the statement that every
$\Bst$-solution annihilates $\ker\Fzero^{(-)}(E)$. This is the analogue of
\cite[Lemma~7.5]{Isozaki2001}. It is here that the Robin condition is felt
twice, and in two opposite ways.

\begin{lemma}[Orthogonality to $\ker\Fzero(E)$]
\label{lem:kernel-orth}
Let $E \in \Ereg$ and let $u \in \Bst \cap H^2_{\mathrm{loc}}(\overline{\Rnp})$
satisfy
\begin{equation}
  \label{eq:solution-hypotheses}
  (-\Delta - E)u = 0 \ \text{ in } \Rnp,
  \qquad
  \gamma_1 u - \alpha\gamma_0 u = 0 \ \text{ on } \Gamma .
\end{equation}
If $f \in \Bsp$ satisfies $\Fzero(E)f = 0$, then $(u,f) = 0$.
\end{lemma}

\begin{proof}
Fix $\psi \in C_c^{\infty}\bigl((0,\infty)\bigr)$ with $\psi \ge 0$ and
$\operatorname{supp}\psi \subset [c_1,c_2]$, $0 < c_1 < c_2 < \infty$, and set
\begin{equation}
  \label{eq:cutoff-choice}
  \chi = \psi^2,
  \qquad
  \int_0^{\infty}\chi(t)\,dt = 1,
  \qquad
  \chi_1(t) = \int_t^{\infty}\chi(s)\,ds,
\end{equation}
so that $\chi \ge 0$, $\chi_1' = -\chi$, $\chi_1 \equiv 1$ near $t = 0$ and
$\chi_1 \equiv 0$ for $t \ge c_2$. Writing
$\chi_R(x) = \chi(|x|/R)$ and $\chi_{1,R}(x) = \chi_1(|x|/R)$, both are
smooth and compactly supported on $\overline{\Rnp}$, and
$\operatorname{supp}\chi_R \subset S_R := \{c_1R \le |x| \le c_2R\}\cap\Rnp$.
Put $v = \Ralpha{E+i0}f$, so that $v \in \Bst \cap
H^2_{\mathrm{loc}}(\overline{\Rnp})$ solves $(-\Delta-E)v = f$ with the same
boundary condition, by Theorem~\ref{thm:lap}(e). Throughout we abbreviate
\begin{equation*}
  \mathcal{I}_R = \frac{1}{R}\int_{\Rnp}\chi_R|\nabla v|^2\,dx,
  \qquad
  \mathcal{J}_R = \frac{1}{R}\int_{\Gamma}\chi_R|v|^2\,dS,
  \qquad
  \mathcal{K}_R = \frac{1}{R}\int_{\Rnp}\chi_R|v|^2\,dx,
\end{equation*}
all three finite for each $R$ because $\chi_R$ is compactly supported and
$v \in H^2_{\mathrm{loc}}(\overline{\Rnp})$. Since $\chi$ is bounded and
supported in $[c_1,c_2]$, Lemma~\ref{lem:shell-average} gives
\begin{equation}
  \label{eq:KR-bounds}
  \mathcal{K}_R \le c_2\norm{\chi}_{L^{\infty}}\,\MavR{v}{c_2R},
  \qquad
  \frac{1}{R}\int_{\Rnp}\chi_R|u|^2\,dx \le 4c_2\norm{\chi}_{L^{\infty}}
    \Bstnorm{u}^2 ,
\end{equation}
and by Lemma~\ref{lem:shell-flux-B} the hypothesis $\Fzero(E)f = 0$ forces
\begin{equation}
  \label{eq:KR-vanishes}
  \mathcal{K}_R \longrightarrow 0 \qquad (R \to \infty).
\end{equation}

\emph{Step 1: the boundary terms cancel, and $(u,f)$ is a radial flux.}
Apply Green's second identity on the bounded set
$\operatorname{supp}\chi_{1,R}\cap\Rnp$ to the pair $u$ and
$\Phi = \chi_{1,R}\overline{v}$:
\begin{equation*}
  \int_{\Rnp}\bigl( u\,\Delta\Phi - \Phi\,\Delta u \bigr)dx
  = \int_{\Gamma}\bigl( u\,\gamma_1\Phi - \Phi\,\gamma_1 u \bigr)dS ,
\end{equation*}
$\gamma_1 = \partial_{\nu}$ being the derivative along the outward normal
$\nu = -e_n$ of $\Rnp$. On the left, $\Delta u = -Eu$ and
$\Delta\overline{v} = -E\overline{v} - \overline{f}$, so that the two terms
carrying the factor $E$ cancel and
\begin{equation}
  \label{eq:green-lhs}
  \int_{\Rnp}\bigl( u\,\Delta\Phi - \Phi\,\Delta u \bigr)dx
  = \int_{\Rnp}u\,(\Delta\chi_{1,R})\,\overline{v}\,dx
    + 2\int_{\Rnp}u\,\bigl( \nabla\chi_{1,R}\cdot\nabla\overline{v} \bigr)dx
    - \int_{\Rnp}\chi_{1,R}\,u\,\overline{f}\,dx .
\end{equation}
On the right, the normal derivative of the cut-off vanishes identically on
$\Gamma$,
\begin{equation}
  \label{eq:cutoff-normal}
  \gamma_1\chi_{1,R}
  = -\partial_{x_n}\Bigl[ \chi_1\bigl( |x|/R \bigr) \Bigr]\Big|_{x_n=0}
  = -\frac{1}{R}\,\chi_1'\bigl( |x|/R \bigr)\frac{x_n}{|x|}\bigg|_{x_n=0}
  = 0 ,
\end{equation}
because the shells are centred at a point of the boundary. Hence
$\gamma_1\Phi = \chi_{1,R}\,\gamma_1\overline{v}
= \alpha\chi_{1,R}\overline{v}$, while
$\Phi\,\gamma_1 u = \alpha\chi_{1,R}\overline{v}u$, and the boundary integral
vanishes: \emph{both} $u$ and $v$ obey the same Robin condition, so the two
contributions cancel whatever the sign of $\alpha$. Therefore, by
\eqref{eq:green-lhs},
\begin{equation}
  \label{eq:uf-flux}
  \int_{\Rnp}\chi_{1,R}\,u\,\overline{f}\,dx
  = \int_{\Rnp}u\,(\Delta\chi_{1,R})\,\overline{v}\,dx
    - \frac{2}{R}\int_{\Rnp}\chi_R\,u\,\partial_r\overline{v}\,dx ,
\end{equation}
where we used $\nabla\chi_{1,R} = R^{-1}\chi_1'(|x|/R)\theta
= -R^{-1}\chi_R\theta$. As $R \to \infty$ the left-hand side converges to
$(u,f)$ by dominated convergence, the pairing being absolutely convergent by
\eqref{eq:AH-pairing} and $\chi_{1,R} \to 1$ pointwise. On the right,
$|\Delta\chi_{1,R}| \le CR^{-2}$ on $S_R$, so that
\begin{equation*}
  \Bigl| \int_{\Rnp}u\,(\Delta\chi_{1,R})\,\overline{v}\,dx \Bigr|
  \le \frac{C}{R^2}\int_{S_R}|u|\,|v|\,dx
  \le \frac{C'}{R}\,\Bstnorm{u}\Bstnorm{v}
  \longrightarrow 0
\end{equation*}
by Lemma~\ref{lem:shell-average}(i), while by Cauchy--Schwarz and
\eqref{eq:KR-bounds}
\begin{equation}
  \label{eq:flux-cs}
  \Bigl| \frac{2}{R}\int_{\Rnp}\chi_R\,u\,\partial_r\overline{v}\,dx \Bigr|
  \le 2\Bigl( \frac{1}{R}\int\chi_R|u|^2 \Bigr)^{1/2}\mathcal{I}_R^{1/2}
  \le C\,\Bstnorm{u}\,\mathcal{I}_R^{1/2} .
\end{equation}
It therefore suffices to prove $\mathcal{I}_R \to 0$.

\emph{Step 2: a weighted trace inequality.} For every $\eta > 0$,
\begin{equation}
  \label{eq:weighted-trace}
  \mathcal{J}_R \le \eta\,\mathcal{I}_R + \eta^{-1}\mathcal{K}_R
    + \frac{C}{R}\,\Bstnorm{v}^2 .
\end{equation}
Indeed, for each $\xpar \in \Rnm$ the function
$x_n \mapsto \chi_R(\xpar,x_n)|v(\xpar,x_n)|^2$ vanishes for $x_n$ large, so
\begin{equation*}
  \int_{\Gamma}\chi_R|v|^2\,dS
  = -\int_{\Rnp}\partial_{x_n}\bigl[ \chi_R|v|^2 \bigr]dx
  = -\int_{\Rnp}(\partial_{x_n}\chi_R)|v|^2\,dx
    - 2\int_{\Rnp}\chi_R\,\re\bigl( \overline{v}\,\partial_{x_n}v \bigr)dx .
\end{equation*}
Since $|\partial_{x_n}\chi_R| \le \norm{\chi'}_{L^{\infty}}/R$ and
$\int_{S_R}|v|^2 \le 4c_2R\Bstnorm{v}^2$, the first term is bounded by
$C\Bstnorm{v}^2$. For the second we use
$2\chi_R|v||\partial_{x_n}v| \le \eta\chi_R|\nabla v|^2
+ \eta^{-1}\chi_R|v|^2$. Dividing by $R$ gives
\eqref{eq:weighted-trace}.

\emph{Step 3: the energy identity.} Multiply $(-\Delta - E)v = f$ by
$\chi_R\overline{v}$ and integrate. Green's first identity, in the form
$\int_{\Rnp}(-\Delta v)h\,dx = \int_{\Rnp}\nabla v \cdot \nabla h\,dx
- \int_{\Gamma}h\,\gamma_1 v\,dS$ with $h = \chi_R\overline{v}$, gives
\begin{equation*}
  \int_{\Rnp}\chi_R|\nabla v|^2\,dx
    + \frac{1}{R}\int_{\Rnp}\chi'\bigl(|x|/R\bigr)\,
      \overline{v}\,\partial_r v\,dx
    - \alpha\int_{\Gamma}\chi_R|v|^2\,dS
    - E\int_{\Rnp}\chi_R|v|^2\,dx
  = \int_{\Rnp}\chi_R\,\overline{v}f\,dx ,
\end{equation*}
where the boundary term was evaluated by $\gamma_1 v = \alpha\gamma_0 v$.
Taking real parts and dividing by $R$,
\begin{equation}
  \label{eq:energy-scaled}
  \mathcal{I}_R
  \le \alpha\mathcal{J}_R + |E|\,\mathcal{K}_R
    + \frac{1}{R}\Bigl| \int_{\Rnp}\chi_R\overline{v}f\,dx \Bigr|
    + \frac{1}{R^2}\Bigl| \int_{\Rnp}\chi'\bigl(|x|/R\bigr)
      \overline{v}\,\partial_r v\,dx \Bigr| .
\end{equation}
The third term is $O(R^{-1})$ by \eqref{eq:AH-pairing}. For the fourth we use
the choice $\chi = \psi^2$ of \eqref{eq:cutoff-choice}, which gives
$|\chi'| = 2\psi|\psi'| \le C\psi = C\chi^{1/2}$ and hence, by
Cauchy--Schwarz,
\begin{equation*}
  \frac{1}{R^2}\Bigl| \int\chi'\,\overline{v}\,\partial_r v \Bigr|
  \le \frac{C}{R^2}\Bigl( \int_{S_R}|v|^2 \Bigr)^{1/2}
      \Bigl( \int\chi_R|\nabla v|^2 \Bigr)^{1/2}
  \le \frac{C'}{R}\,\Bstnorm{v}\,\mathcal{I}_R^{1/2} .
\end{equation*}
Inserting \eqref{eq:weighted-trace} with $\eta = 1/(2\alpha)$ into
\eqref{eq:energy-scaled} we obtain
\begin{equation*}
  \mathcal{I}_R
  \le \tfrac12\mathcal{I}_R
    + \bigl( 2\alpha^2 + |E| \bigr)\mathcal{K}_R
    + \frac{C}{R} + \frac{C}{R}\,\mathcal{I}_R^{1/2},
\end{equation*}
and absorbing the last term by $CR^{-1}\mathcal{I}_R^{1/2}
\le \tfrac14\mathcal{I}_R + C^2R^{-2}$,
\begin{equation*}
  \tfrac14\,\mathcal{I}_R
  \le \bigl( 2\alpha^2 + |E| \bigr)\mathcal{K}_R
    + \frac{C}{R} + \frac{C^2}{R^2} .
\end{equation*}
By \eqref{eq:KR-vanishes} the right-hand side tends to $0$, so
$\mathcal{I}_R \to 0$ and, by \eqref{eq:weighted-trace},
$\mathcal{J}_R \to 0$ as well. With \eqref{eq:uf-flux} and
\eqref{eq:flux-cs} this gives $(u,f) = 0$. The energy entered only
through $|E|$, so that the argument is insensitive to the sign of $E$.
\end{proof}

\begin{remark}[The sign of the boundary term]
\label{rem:boundary-sign}
The Robin condition enters the proof twice, with opposite effects. In Step 1
the boundary integral cancels identically, because $u$ and $v$ satisfy the
same condition. This reflects the symmetry of $\Halpha$ and does not depend on
the sign of $\alpha$. The only geometric input is \eqref{eq:cutoff-normal}: the
dyadic shells are centred at a point of $\Gamma$, so their defining function is
constant along the normal there.

In Step 3, by contrast, the boundary term has an unfavourable sign. With the
convention $\gamma_1 = \partial_{\nu}$ of \eqref{eq:solution-hypotheses} the
quadratic form of $\Halpha$ is
$\hform[w] = \norm{\nabla w}^2 - \alpha\norm{\gamma_0 w}^2$, which is not
nonnegative for $\alpha > 0$, and this produces the surface-wave band
$(-\alpha^2,0)$. Consequently \eqref{eq:energy-scaled} presents
$\mathcal{I}_R$ and $\alpha\mathcal{J}_R$ with opposite signs, and one cannot
conclude, as in the Dirichlet or Neumann case, that both vanish by positivity.
The trace inequality \eqref{eq:weighted-trace} bounds $\mathcal{J}_R$ by an
arbitrarily small multiple of $\mathcal{I}_R$ plus a multiple of
$\mathcal{K}_R$, which is already known to vanish: the surface energy is
subordinate to the bulk energy at infinity.
\end{remark}

\subsubsection{The characterization theorem}
\label{sec:char-theorem}

\begin{theorem}[Characterization of $\Bst$-solutions]
\label{thm:characterization}
Let $E \in \Ereg$ and put
\begin{equation}
  \label{eq:solution-space}
  \Nsol{E} = \Bigl\{ u \in \Bst \cap H^2_{\mathrm{loc}}(\overline{\Rnp}) :
    (-\Delta - E)u = 0 \text{ in } \Rnp,\
    \gamma_1 u - \alpha\gamma_0 u = 0 \text{ on } \Gamma \Bigr\}.
\end{equation}
Then
\begin{equation}
  \label{eq:characterization}
  \Nsol{E} = \Ran\bigl( \Fzero^{(\pm)}(E)^{*} \bigr)
           = \Ran\bigl( \Fzero(E)^{*} \bigr),
\end{equation}
and for each $u \in \Nsol{E}$, and each of the three normalizations
$\Fzero$, $\Fzero^{(+)}$, $\Fzero^{(-)}$, there is exactly one
$\phi \in \Echan{E}$ representing $u$. The correspondence is bicontinuous, with
\begin{equation}
  \label{eq:characterization-norms}
  \frac{1}{\sqrt{8\pi\ksurf}}\,\norm{\phi}_{\Echan{E}}
  \le \Bstnorm{u}
  \le \norm{\Fzero(E)^{*}}_{\Lop{\Echan{E}}{\Bst}}\,
      \norm{\phi}_{\Echan{E}} .
\end{equation}
On the surface band $-\alpha^2 < E < 0$ this reads
$\Nsol{E} = \Ran\bigl( \Fsurf(E)^{*} \bigr)$ with
$\Echan{E} = \Ltwo{\Ssurf}$.
\end{theorem}

\begin{proof}
The three ranges in \eqref{eq:characterization} coincide because
$\Fzero^{(\pm)}(E)^{*}\phi = \Fzero(E)^{*}\bigl( m^{\mp1}\phi \bigr)$ for the
unimodular multiplier $m = \cph{\kbulk\theta_n} \oplus 1$ of
\eqref{eq:Fout-def}, so that multiplication by $m^{\mp1}$ is a unitary
automorphism of $\Echan{E}$. For $-\alpha^2 < E < 0$ the multiplier reduces
to $1$ and $\Fzero^{(\pm)}(E) = \Fzero(E) = \Fsurf(E)$.

\emph{The inclusion $\Ran \subset \Nsol{E}$.} By
Proposition~\ref{prop:F-adjoint}, $\Fzero(E)^{*}\phi$ lies in $\Bst$ and
satisfies \eqref{eq:eigenoperator}, which is \eqref{eq:solution-hypotheses}.
Membership in $H^2_{\mathrm{loc}}(\overline{\Rnp})$ follows from local
elliptic regularity for the Robin problem, or directly from the fact that
$\Psib$ and $\Psis$ are smooth and bounded on $\overline{\Rnp}$ together with
their derivatives.

\emph{Injectivity and the norm bounds.} These are
Corollary~\ref{cor:two-sided}.

\emph{The inclusion $\Nsol{E} \subset \Ran$.} Put
$T = \Fzero^{(-)}(E) \in \Lop{\Bsp}{\Echan{E}}$, which is bounded by
Proposition~\ref{prop:F-bounded}. Its Banach-space adjoint is a bounded map
$T^{*} : \Echan{E}^{*} \to \Bsp^{*}$. Identifying $\Echan{E}^{*}$ with
$\Echan{E}$ by the Riesz isomorphism and $\Bsp^{*}$ with $\Bst$ by the
Agmon--H\"ormander duality \eqref{eq:AH-pairing}, $T^{*}$ becomes the
eigenoperator $\Fzero^{(-)}(E)^{*}$ of Proposition~\ref{prop:F-adjoint}. By
Corollary~\ref{cor:two-sided} the operator $T^{*}$ is bounded below, so
$\Ran(T^{*})$ is a closed subspace of $\Bst$. By Banach's closed range
theorem \cite[Ch.~VII, \S5]{Yosida} the range of $T$ is then closed as well,
and
\begin{equation}
  \label{eq:closed-range}
  \Ran\bigl( \Fzero^{(-)}(E)^{*} \bigr)
  = \bigl( \ker\Fzero^{(-)}(E) \bigr)^{\perp}
  := \Bigl\{ w \in \Bst : (w,f) = 0 \ \text{ for all }
    f \in \ker\Fzero^{(-)}(E) \Bigr\} .
\end{equation}
Finally $\ker\Fzero^{(-)}(E) = \ker\Fzero(E)$, again because the two
transforms differ by a unimodular multiplier, so
Lemma~\ref{lem:kernel-orth} says precisely that every $u \in \Nsol{E}$ lies in
the right-hand side of \eqref{eq:closed-range}.
\end{proof}

\begin{remark}[What each hypothesis is doing]
\label{rem:characterization-hypotheses}
The theorem is a statement about behaviour at infinity only: no decay is
imposed beyond membership in $\Bst$, and no radiation condition is assumed.
The two thresholds enter in three different places. The lower bound of Corollary~\ref{cor:two-sided} has
constant $(8\pi\ksurf)^{-1/2}$, which stays bounded away from $0$ as
$E \downarrow 0$ because $\ksurf \to \alpha$. The bulk threshold is felt
instead in the upper bound, where \eqref{eq:F-bound} gives only
$\norm{\Fbulk(E)} \le \CAH{n}E^{-1/4}$. The essential obstruction is
Lemma~\ref{lem:shell-flux-B}, which rests on the asymptotic expansion of
\S\ref{subsec:asymptotics} as $E \downarrow 0$ and on the free limiting
absorption principle in dimension $n-1$ as $E \downarrow -\alpha^2$, and
neither survives at its endpoint. This is where the exclusion of
$\Thr = \{-\alpha^2,0\}$ originates.
\end{remark}

\subsubsection{Far field of an arbitrary solution, and the \texorpdfstring{$S$}{S}-matrix}
\label{sec:char-smatrix}

Combining Theorem~\ref{thm:characterization} with
Theorem~\ref{thm:eigen-far-field} we obtain the analogue of
\cite[Theorem~8.2]{Isozaki2001}: every solution in $\Bst$ has a complete
far field, consisting of an incoming and an outgoing part in each open
channel.

For $E > 0$ the channel space $\Echan{E} = \Ltwo{\Sm} \oplus \Ltwo{\Ssurf}$
of \eqref{eq:channel-space} is an orthogonal direct sum of two summands, and
the far-field data of Definition~\ref{def:far-field-data} are correspondingly
column vectors $\yin = (\yin_{\mathrm{bulk}},\yin_{\mathrm{surf}})^{\mathsf T}$.
Every bounded operator on $\Echan{E}$, the scattering matrix in particular,
is therefore a $2 \times 2$ operator matrix
\begin{equation*}
  \begin{pmatrix}
    \Smat_{\mathrm{bb}}(E) & \Smat_{\mathrm{bs}}(E) \\[1mm]
    \Smat_{\mathrm{sb}}(E) & \Smat_{\mathrm{ss}}(E)
  \end{pmatrix},
  \qquad
  \begin{aligned}
    \Smat_{\mathrm{bb}}(E) &: \Ltwo{\Sm} \to \Ltwo{\Sm},
    &\quad
    \Smat_{\mathrm{bs}}(E) &: \Ltwo{\Ssurf} \to \Ltwo{\Sm},
    \\[1mm]
    \Smat_{\mathrm{sb}}(E) &: \Ltwo{\Sm} \to \Ltwo{\Ssurf},
    &\quad
    \Smat_{\mathrm{ss}}(E) &: \Ltwo{\Ssurf} \to \Ltwo{\Ssurf},
  \end{aligned}
\end{equation*}
the two off-diagonal entries being the conversion of an incident bulk wave
into a surface wave and conversely. The corollary below shows that in the
flat half-space this matrix is \emph{diagonal}, and that its
lower-right entry is the identity.

\begin{corollary}[Far field of a $\Bst$-solution]
\label{cor:solution-far-field}
Let $E \in \Ereg$ and $u \in \Nsol{E}$. Then there exist unique
$\yin,\yout \in \Echan{E}$ such that
\begin{equation}
  \label{eq:solution-far-field}
  \begin{aligned}
  u \;\simeq\;
  & r^{-\frac{n-1}{2}}
    \Bigl[ \gambulk(E)\,e^{i\kbulk r}\,\yout_{\mathrm{bulk}}(\theta)
      + \overline{\gambulk(E)}\,e^{-i\kbulk r}\,
        \yin_{\mathrm{bulk}}(\theta) \Bigr]
  \\[1mm]
  + \;& \rho^{-\frac{n-2}{2}}
    \Bigl[ \gamsurf(E)\,e^{i\ksurf\rho}\,\yout_{\mathrm{surf}}(\omega)
      + \overline{\gamsurf(E)}\,e^{-i\ksurf\rho}\,
        \yin_{\mathrm{surf}}(\omega) \Bigr]\,\psizero(x_n) ,
  \end{aligned}
\end{equation}
the first line being absent when $-\alpha^2 < E < 0$, and they satisfy
\begin{equation}
  \label{eq:far-field-isometry}
  \norm{\yin}_{\Echan{E}} = \norm{\yout}_{\Echan{E}}
  = \norm{\phi}_{\Echan{E}},
  \qquad
  \yout = \Smat(E)\,\Jchan\,\yin ,
\end{equation}
where $\phi$ is the datum of Theorem~\ref{thm:characterization} in the
standing normalization, $u = \Fzero(E)^{*}\phi$, and
\begin{equation}
  \label{eq:smatrix-def}
  \begin{aligned}
  \Smat(E)
  &= \begin{pmatrix}
       \Smat_{\mathrm{bulk}}(E) & 0 \\[1mm]
       0 & \mathrm{Id}_{\Ltwo{\Ssurf}}
     \end{pmatrix}
     \quad\text{on } \Echan{E} = \Ltwo{\Sm} \oplus \Ltwo{\Ssurf},
  &&\quad E > 0,
  \\[2mm]
  \Smat(E)
  &= \mathrm{Id}_{\Ltwo{\Ssurf}}
     \quad\text{on } \Echan{E} = \Ltwo{\Ssurf},
  &&\quad -\alpha^2 < E < 0,
  \end{aligned}
\end{equation}
where $\Smat_{\mathrm{bulk}}(E)$ is the operator of multiplication by the
reflection coefficient,
\begin{equation*}
  \bigl( \Smat_{\mathrm{bulk}}(E)F \bigr)(\theta)
  = \refl{\kbulk\theta_n}\,F(\theta),
  \qquad \theta \in \Sm .
\end{equation*}
Moreover each of the maps $u \mapsto \yin$ and $u \mapsto \yout$ is a
bijection of $\Nsol{E}$ onto $\Echan{E}$, and $\Smat(E)$ is unitary on
$\Echan{E}$.
\end{corollary}

\begin{proof}
Existence of the expansion is Theorem~\ref{thm:eigen-far-field} applied to
the datum $\phi$ supplied by Theorem~\ref{thm:characterization}, with
$\yin,\yout$ as in \eqref{eq:ydata-bulk}--\eqref{eq:ydata-surf}. The
identities \eqref{eq:far-field-isometry} are \eqref{eq:ydata-norms} and
\eqref{eq:char-smatrix-relation}. On the surface band the latter reads
$\yout = \Jsurf\yin$, which is \eqref{eq:smatrix-def} with
$\Smat(E) = \mathrm{Id}$. Next, $\phi$ is recovered from $\yin$ by
\begin{equation}
  \label{eq:phi-from-yin}
  \phi_{\mathrm{bulk}}
  = \cph{\kbulk\theta_n}\,\bigl( \Jbulk\yin_{\mathrm{bulk}} \bigr),
  \qquad
  \phi_{\mathrm{surf}} = \Jsurf\yin_{\mathrm{surf}} ,
\end{equation}
because $\Jbulk$ commutes with multiplication by $\cph{\kbulk\theta_n}$ and
$\Jbulk^2 = \mathrm{id}$. The map $\yin \mapsto \phi$ so defined is unitary,
and $\Smat(E)$ is unitary because $|\refl{\cdot}| = 1$.

Uniqueness now follows. If $\bigl( \tilde y^{(-)},\tilde y^{(+)} \bigr)$ is a
second pair satisfying \eqref{eq:far-field-isometry} and representing $u$ in
\eqref{eq:solution-far-field}, then by \eqref{eq:phi-from-yin} it arises from
some $\tilde\phi \in \Echan{E}$, and the two right-hand sides of
\eqref{eq:solution-far-field} are the far-field models of
$\Fzero(E)^{*}\phi$ and $\Fzero(E)^{*}\tilde\phi$ furnished by
Theorem~\ref{thm:eigen-far-field}. Their difference lies in $\Bstz$, so
$\Fzero(E)^{*}(\phi - \tilde\phi) \in \Bstz$ and
$\Mav{\Fzero(E)^{*}(\phi-\tilde\phi)} = 0$ by
Lemma~\ref{lem:shell-average}(iii). Theorem~\ref{thm:shell-average-identity}
then forces $\phi = \tilde\phi$. Since $\yin \mapsto \phi$ is unitary and
$\phi \mapsto u$ is bijective onto $\Nsol{E}$ by
Theorem~\ref{thm:characterization}, the map $u \mapsto \yin$ is a bijection
onto $\Echan{E}$, and likewise for $u \mapsto \yout$.
\end{proof}

\begin{remark}[Reading \eqref{eq:smatrix-def}]
\label{rem:smatrix-reading}
By Corollary~\ref{cor:solution-far-field} the scattering problem for $\Halpha$
in the flat half-space is exactly solvable: the incoming data determine the
solution, and the scattering matrix is diagonal in the channel decomposition.
The bulk block is multiplication by $\refl{\kbulk\theta_n}$, which depends on
the direction only through $\theta_n$, as it must for a boundary invariant
under tangential translations and rotations. The surface block is the
identity, since a surface wave travelling along $\Gamma$ meets no obstacle.
In particular $\Smat_{\mathrm{bs}}(E) = \Smat_{\mathrm{sb}}(E) = 0$, and on the
surface band the matrix reduces to $\mathrm{Id}$. A local perturbation of the
boundary can break this degeneracy.
\end{remark}

\begin{remark}[Comparison with the literature]
\label{rem:char-literature}
For $n = 3$ and with a Sommerfeld-type radiation condition in place of the
membership $u \in \Bst$, results of this kind go back to
\cite{DuranMugaNedelec2009}. Here the hypotheses are imposed in $\Bst$ rather
than through a pointwise decay rate, so that no exponent has to be chosen and
the two channels are treated on the same footing, at the price that the far
field is obtained modulo $\Bstz$ rather than pointwise. The two channels are
handled by the same argument in dimensions $m = n$ and $m = n-1$.
Lemma~\ref{lem:sphere-statphase} requires $m \ge 2$, and this is where
$n \ge 3$ enters the present subsection. The abstract skeleton, namely
two-sided bound, orthogonality to the kernel and closed range, is that of
\cite[\S8]{Isozaki2001}, where it is carried out for $N$-body Schr\"odinger
operators (see also \cite[Ch.~1]{IsozakiKurylev2014}). What is specific to the
present setting is the source of the lower bound, the total reflection of
Remark~\ref{rem:total-reflection}, and the sign of the boundary term in
Lemma~\ref{lem:kernel-orth} (Remark~\ref{rem:boundary-sign}).
\end{remark}

%% file: Part1_026_radiation.tex
\label{subsec:radiation}

This subsection completes Section~\ref{sec:unperturbed} by isolating, among all
solutions of the impedance problem that live in $\Bst$, the single one produced
by the limiting absorption principle. Three statements are proved, in the order
in which they depend on one another.

The first is a \emph{Rellich-type uniqueness theorem}
(Theorem~\ref{thm:rad-rellich}): a solution of the Robin problem which is
defined only near infinity and whose shell averages tend to zero must vanish
identically near infinity. Its hypotheses involve neither a radiation condition
nor any behaviour of $u$ on a compact set. All it uses is the equation, the
boundary condition, and the single scalar decay statement $u \in \Bstz$. The second is
the \emph{outgoing radiation condition} (Definition~\ref{def:rad-condition}),
imposed directly on $u$ in the two-region form of Dur\'an, Muga and
N\'ed\'elec~\cite{DuranMugaNedelec2009}, together with the proof that the
resolvent satisfies it (Proposition~\ref{prop:rad-resolvent-satisfies}) and the
resulting uniqueness theorem for the homogeneous problem
(Theorem~\ref{thm:rad-uniqueness-hom}). The third is the identification of the
radiation condition with the \emph{intrinsic}, parameter-free requirement that
the far field of $u$ be purely outgoing in both channels
(Theorem~\ref{thm:rad-equivalence}), from which the existence and uniqueness
theorem for the inhomogeneous equation follows at once
(Theorem~\ref{thm:rad-main}).

The logical order is forced: the uniqueness proof converts the radiation
condition into the decay hypothesis of the Rellich theorem and then invokes it,
so the Rellich theorem must come first. In the surface channel the Rellich
argument is carried out in dimension $n-1$, and this is where the standing
assumption $n \ge 3$ is used (Remark~\ref{rem:standing-dimension}). The
passage from the half-space to the whole space in the bulk
channel is effected not by an odd or even reflection, neither of which is
available for a Robin condition, but by an explicit integral extension formula
\eqref{eq:rad-clean-extension} whose boundedness at $x_n \to -\infty$ holds
precisely because the surface component has already been shown to vanish.

One preparatory result is of a different nature from everything else here and
is proved separately in \S\ref{sec:rad-reflected}. Verifying that the resolvent
satisfies the radiation condition requires pointwise control of $\partial_r u$
in the boundary layer as well as away from it, and the pointwise expansions of
\S\ref{subsec:asymptotics} are uniform only on cones $\{x_n \ge \delta r\}$
with $\delta$ fixed. The intermediate wedge $\{r^{\alpha_1} < x_n < \delta r\}$
is left uncovered by them for every $\delta$. We therefore establish, in
Lemma~\ref{lem:rad-uniform}, a bound on the radiation defect of the reflected
integral which is uniform in the direction, down to and including the boundary.
This is the one place in Part~I where the grazing region has to be resolved
rather than averaged over. The expansion of \S\ref{subsec:asymptotics} needed
only its shell-negligibility (Remark~\ref{rem:no-grazing}).

\subsubsection{Standing hypotheses and scale averages at an arbitrary ratio}
\label{sec:rad-prelim}

Throughout this subsection $n \ge 3$, $\alpha > 0$, and $E > 0$ unless stated
otherwise. The coordinates are those of \eqref{eq:asymp-polar}, so that
$x = (\xpar,x_n)$, $r = |x|$, $\rho = |\xpar|$, $\theta = x/r \in \Sm$ and
$\omega = \xpar/\rho \in \Ssurf$. The wavenumbers $\kbulk$ and $\ksurf$ are
those of \eqref{eq:asymp-wavenumbers}. For $R_0 \ge 0$ we write
\begin{equation*}
  \Rnp(R_0) = \{x \in \Rnp : |x| > R_0\},
  \qquad
  \Gamma(R_0) = \{x \in \Gamma : |\xpar| > R_0\},
\end{equation*}
and we say that $u$ \emph{solves the Robin problem outside $R_0$} when
\begin{equation}
  \label{eq:rad-exterior-problem}
  (-\Delta - E)u = 0 \ \text{ in } \Rnp(R_0),
  \qquad
  \partial_{x_n}u + \alpha u = 0 \ \text{ on } \Gamma(R_0),
\end{equation}
which is the boundary condition \eqref{eq:robin-alpha} written out with
$\gamma_1 = \partial_\nu$, $\nu = -e_n$. Note that
$|\xpar| > R_0$ implies $|x| > R_0$ for every $x_n > 0$, so the whole vertical
ray above a point of $\Gamma(R_0)$ lies in $\Rnp(R_0)$. This is used
repeatedly.

\begin{remark}[Regularity]
\label{rem:rad-regularity}
It suffices to assume that $u \in H^1_{\mathrm{loc}}(\overline{\Rnp(R_0)})$ is a
weak solution, that is
\begin{equation*}
  \int_{\Rnp(R_0)}\bigl(\nabla u \cdot \nabla\overline{v} - Eu\overline{v}\bigr)dx
  - \alpha\int_{\Gamma(R_0)}u\overline{v}\,dS = 0
\end{equation*}
for all $v \in H^1_c(\overline{\Rnp(R_0)})$ vanishing near $\{|x| = R_0\}$,
which encodes the equation and the Robin condition simultaneously. Since the
coefficients are constant and the boundary is flat, interior elliptic
regularity and elliptic regularity up to the boundary, applied on bounded
subdomains, upgrade $u$ to $C^\infty(\overline{\Rnp(R_0')})$ for every
$R_0' > R_0$. See \cite[\S6.3.1]{Evans2010} for the interior statement and
\cite[Thm.~4.18(ii)]{McLean2000} for the regularity up to the boundary. The
latter is stated for the Neumann problem and applies to the Robin condition
$\gamma_1 u = \alpha\gamma_0 u$ with the Neumann datum $g = \alpha\gamma_0 u$:
if $u \in H^{r+1}$ near a boundary point, then $g \in H^{r+1/2}$ there, hence
$u \in H^{r+2}$ on a smaller neighbourhood, and induction on $r$ together with
the Sobolev embedding gives smoothness up to the boundary. For
constant coefficients and a flat boundary this can also be seen directly from
a Fourier representation in the tangential variables. All integrations by
parts, traces and pointwise manipulations below are therefore classical. The
membership $u \in \Bst$ is a growth restriction at infinity, not a solution
concept, and does not replace the assumption that $u$ solves the problem.
\end{remark}

The decay hypothesis of the Rellich theorem is the statement $u \in \Bstz$. In
the proof it is convenient to test decay on annuli of ratio $2$ (in the variable
$x$) and of ratio $\sqrt3$ (in the variable $\xpar$), and the following
elementary lemma shows that all these formulations coincide with each other and
with \eqref{eq:Bzero}. Recording it once removes the need for any ad hoc
covering argument later.

\begin{lemma}[Scale averages at an arbitrary ratio]
\label{lem:rad-scale-ratio}
Let $m \ge 1$, let $\Omega$ be either $\mathbb{R}^m$ or a half-space in
$\mathbb{R}^m$, let $\lambda > 1$, and let
$v \in L^2_{\mathrm{loc}}(\overline{\Omega})$ satisfy
$\sup_{R>1}R^{-1}\int_{\{|x|<R\}\cap\Omega}|v|^2\,dx < \infty$. Then
\begin{equation*}
  \lim_{R\to\infty}\frac1R\int_{\{|x|<R\}\cap\Omega}|v|^2\,dx = 0
  \quad\Longleftrightarrow\quad
  \lim_{R\to\infty}\frac1R\int_{\{R<|x|<\lambda R\}\cap\Omega}|v|^2\,dx = 0 .
\end{equation*}
In particular the right-hand side is, for every $\lambda > 1$, equivalent to
$v \in \Bstz(\Omega)$ in the sense of \eqref{eq:Bzero}.
\end{lemma}

\begin{proof}
Write $I(R) = \int_{\{|x|<R\}\cap\Omega}|v|^2dx$ and
$Q(R) = \int_{\{R<|x|<\lambda R\}\cap\Omega}|v|^2dx$, so that
$Q(R) \le I(\lambda R)$.

If $I(R) = o(R)$ then $Q(R) \le I(\lambda R) = o(\lambda R) = o(R)$, which is
the implication from left to right.

Conversely, assume $Q(R) = o(R)$ and let $\eps > 0$. Choose $R_1 \ge 1$ with
$Q(t) \le \eps t$ for all $t \ge R_1$. Let $R > \lambda R_1$ and set
$t_j = \lambda^{-j}R$ for $j \ge 0$. Let $N \ge 1$ be the largest integer with
$t_{N-1} \ge R_1$. Then $t_{N-1} < \lambda R_1$, since otherwise $t_N \ge R_1$
would contradict maximality. Because
$\{t_j \le |x| < t_{j-1}\} = \{t_j \le |x| < \lambda t_j\}$, the annuli with
$j = 1,\dots,N-1$ cover $\{t_{N-1} \le |x| < R\}$, whence
\begin{equation*}
  I(R) \le I(\lambda R_1) + \sum_{j=1}^{N-1}Q(t_j)
  \le I(\lambda R_1) + \eps\sum_{j\ge1}\lambda^{-j}R
  = I(\lambda R_1) + \frac{\eps R}{\lambda-1}.
\end{equation*}
The first term is a finite constant depending only on $\eps$, since $v$ is
locally square integrable. Dividing by $R$ and letting $R \to \infty$ gives
$\limsup_{R\to\infty}R^{-1}I(R) \le \eps/(\lambda-1)$, and $\eps > 0$ was
arbitrary. Finally, the equivalence of $R^{-1}I(R) \to 0$ with
\eqref{eq:Bzero} is the equivalence of norms recorded in
\eqref{eq:Bstar-equivalent}.
\end{proof}

The second preliminary converts the integral bound $u \in \Bst$ into the
pointwise growth bound needed to discard boundary terms at $x_n \to \infty$.

\begin{lemma}[Pointwise growth of solutions in $\Bst$]
\label{lem:rad-pointwise}
Let $u$ solve the Robin problem outside $R_0$ and let $u \in \Bst$. Then there
is $C = C(n,E,\alpha)$ such that
\begin{equation*}
  |u(x)| + |\nabla u(x)| \le C\,\Bstnorm{u}\,|x|^{1/2}
  \qquad (|x| > R_0 + 2,\ x_n > 1).
\end{equation*}
The same bound holds, with a possibly larger $C$, on the whole of
$\{|x| > R_0+2\}$: on the slab $0 \le x_n \le 1$ one applies the elliptic
estimate up to the boundary of Remark~\ref{rem:rad-regularity} instead of the
interior one.
\end{lemma}

\begin{proof}
By Remark~\ref{rem:rad-regularity}, $u$ is smooth there. Interior elliptic
estimates for the constant-coefficient operator $-\Delta - E$ on the ball
$B_1(x) \subset \Rnp(R_0)$ give
$|u(x)| + |\nabla u(x)| \le C\norm{u}_{\Ltwo{B_1(x)}}$ with $C$ independent of
$x$. Since $B_1(x) \subset \{|y| < |x|+1\} \cap \Rnp$, the equivalence
\eqref{eq:Bstar-equivalent} yields
$\norm{u}_{\Ltwo{B_1(x)}}^2 \le C'\Bstnorm{u}^2(|x|+1)$, and
$|x|+1 \le 2|x|$ in the stated range. For $0 \le x_n \le 1$ one replaces
$B_1(x)$ by the half-ball $B_2(x)\cap\Rnp$ and the interior estimate by the
estimate up to the flat boundary for the Robin problem, the right-hand side
being of the same form.
\end{proof}

\subsubsection{The Rellich-type theorem}
\label{sec:rad-rellich}

\begin{theorem}[Rellich-type theorem for the Robin half-space]
\label{thm:rad-rellich}
Let $n \ge 3$, $\alpha > 0$, $E > 0$, and let $R_0 > 0$. Suppose that $u$
solves the Robin problem \eqref{eq:rad-exterior-problem} outside $R_0$ in the
sense of Remark~\ref{rem:rad-regularity} and that
\begin{equation}
  \label{eq:rad-rellich-hypothesis}
  u \in \Bstz,
  \qquad\text{that is}\qquad
  \lim_{R\to\infty}\frac1R\int_{\{|x|<R\}\cap\Rnp}|u(x)|^2\,dx = 0 .
\end{equation}
Then there is $R_1 > R_0$ with
\begin{equation}
  \label{eq:rad-rellich-conclusion}
  u(x) = 0 \qquad (|x| > R_1).
\end{equation}
\end{theorem}

Two features of the hypotheses should be isolated before the proof, since they
are what distinguish the theorem from the Rellich-type results available in the
literature.

The hypotheses are imposed \emph{only near infinity}. Neither the equation nor
the boundary condition is assumed on the compact set $\{|x| \le R_0\}$, and
nothing at all is assumed there about $u$. The theorem therefore applies in
the perturbed setting of Part~II, where an obstacle of unknown shape
occupies that region. See Remark~\ref{rem:rad-outlook}.

The boundary condition, on the other hand, is used, and must be. For the
half-space, decay at infinity alone never implies uniqueness: the evanescent
superpositions displayed in \S\ref{subsec:intro-novelty} solve the Helmholtz
equation in all of $\Rnp$, are rapidly decreasing, and satisfy every hypothesis
imposed at infinity anywhere in this paper. The comparison with the conical
geometries of Bonnet-Ben Dhia, Fliss, Hazard and
Tonnoir~\cite{BonnetBenDhia2016}, where no boundary condition is needed and the
half-space is exactly the excluded borderline case, is made there as well. For
Rellich-type theorems in other geometries with unbounded boundary see
\cite{CossettiKrejcirik2020} and the references there.

The proof occupies the rest of \S\ref{sec:rad-rellich}. It decomposes $u$ along
the transverse spectral splitting of Lemma~\ref{lem:1d-transform} and treats the
two channels separately: the surface component is killed by an $(n-1)$-dimensional
Rellich theorem (Steps 1--3), and only then can the bulk component be handled by
an $n$-dimensional one (Steps 4--6). Both invoke the classical Rellich
uniqueness theorem \cite{Rellich1943} for the free Helmholtz equation, in the
shell-averaged form of Agmon and H\"ormander \cite{AgmonHormander1976}, which we
use as a black box in the following formulation: if $m \ge 2$, $\kappa > 0$, and
$U \in \Bstd{\mathbb{R}^m}$ satisfies $(-\Delta - \kappa^2)U = 0$ outside a
compact set together with
$\lim_{R\to\infty}R^{-1}\int_{|x|<R}|U|^2 = 0$, then $U$ vanishes outside a
compact set.

\medskip
\noindent\textbf{The transverse splitting.}
By Lemma~\ref{lem:1d-transform} the half-line Robin operator has the single
bound state $\psizero(x_n) = \sqrt{2\alpha}\,e^{-\alpha x_n}$, normalized in
$\Ltwo{\Rhalf}$, together with a continuum $\{\gphi_k\}_{k>0}$ orthogonal to it.
Accordingly we split $u$, for each fixed $\xpar$ with $|\xpar| > R_0$, as
\begin{equation}
  \label{eq:rad-splitting}
  u(\xpar,x_n) = \usurf(\xpar)\psizero(x_n) + \ubulk(\xpar,x_n),
  \qquad
  \usurf(\xpar) = \int_0^\infty u(\xpar,x_n)\psizero(x_n)\,dx_n,
\end{equation}
so that $\int_0^\infty\psizero(x_n)\ubulk(\xpar,x_n)\,dx_n = 0$ identically. The
integral defining $\usurf$ converges absolutely by
Lemma~\ref{lem:rad-pointwise}: the integrand is $O(x_n^{1/2}e^{-\alpha x_n})$.
The same argument, run with higher-order interior and boundary elliptic
estimates on balls of radius $1$, bounds $\partial_{\xpar}^{\beta}u(\xpar,x_n)$
for $|\beta| \le 2$ by $C(1+x_n)^{1/2}$ locally uniformly in $\xpar$, so that
$\usurf$ may be
differentiated twice under the integral sign. This is used without further
comment in Step~1.

\medskip
\noindent\textbf{Step 1: the equation satisfied by the surface component.}
Multiply $(-\Delta - E)u = 0$ by $\psizero$ and integrate over
$x_n \in (0,\infty)$ at fixed $\xpar$ with $|\xpar| > R_0$. Integrating the
term $\partial_{x_n}^2u$ by parts twice,
\begin{equation*}
  \int_0^\infty(\partial_{x_n}^2u)\psizero\,dx_n
  = \Bigl[\partial_{x_n}u\,\psizero - u\,\partial_{x_n}\psizero\Bigr]_0^\infty
    + \int_0^\infty u\,\partial_{x_n}^2\psizero\,dx_n .
\end{equation*}
At $x_n \to \infty$ the bracket vanishes, because $\psizero$ and
$\partial_{x_n}\psizero$ decay like $e^{-\alpha x_n}$ while $u$ and
$\partial_{x_n}u$ grow at most like $x_n^{1/2}$ by
Lemma~\ref{lem:rad-pointwise}. At $x_n = 0$ the bracket vanishes as well: the
mode $\psizero$ satisfies the same Robin condition as $u$, that is
$\partial_{x_n}\psizero(0) + \alpha\psizero(0) = 0$, so that
\begin{equation*}
  -\bigl[\partial_{x_n}u(\xpar,0)\psizero(0)
    - u(\xpar,0)\partial_{x_n}\psizero(0)\bigr]
  = -\sqrt{2\alpha}\bigl(\partial_{x_n}u(\xpar,0) + \alpha u(\xpar,0)\bigr) = 0 .
\end{equation*}
The surface channel is thus reached through this exact cancellation of the
boundary term rather than through a positivity property (compare
Remark~\ref{rem:boundary-sign}). Since $\partial_{x_n}^2\psizero = \alpha^2\psizero$,
what remains is $-\Delta_{\xpar}\usurf - \alpha^2\usurf - E\usurf = 0$, that is
\begin{equation}
  \label{eq:rad-surface-helmholtz}
  \bigl(-\Delta_{\xpar} - \ksurf^2\bigr)\usurf(\xpar) = 0
  \qquad (|\xpar| > R_0),
\end{equation}
the $(n-1)$-dimensional Helmholtz equation with the surface wavenumber
$\ksurf$. This is the analytic origin of the second
channel: the transverse binding energy $-\alpha^2$ is converted into an
increment of the tangential wavenumber.

\medskip
\noindent\textbf{Step 2: transfer of the decay condition to $\usurf$.}
The decay hypothesis \eqref{eq:rad-rellich-hypothesis} concerns $u$ on balls in
$\Rnp$, whereas \eqref{eq:rad-surface-helmholtz} needs decay of $\usurf$ on
balls in $\Rnm$. The two are related by cutting the $x_n$-integration in
\eqref{eq:rad-splitting} at the height at which the vertical segment leaves the
ball, and estimating the discarded tail by the exponential decay of $\psizero$.

\begin{lemma}[The truncated surface component]
\label{lem:rad-truncation}
For $\sigma > 0$ and $\xpar$ with $\sigma < |\xpar| < \sqrt3\sigma$ put
\begin{equation}
  \label{eq:rad-Zplus}
  Z(\xpar;\sigma) = \sqrt{4\sigma^2 - |\xpar|^2} \ \ge\ \sigma,
  \qquad
  \usurf^{(\sigma)}(\xpar) = \int_0^{Z(\xpar;\sigma)}
    \psizero(x_n)u(\xpar,x_n)\,dx_n .
\end{equation}
Then
\begin{equation}
  \label{eq:rad-truncation-main}
  \int_{\sigma<|\xpar|<\sqrt3\sigma}\bigl|\usurf^{(\sigma)}(\xpar)\bigr|^2 d\xpar
  \ \le\ \int_{\{\sigma<|x|<2\sigma\}\cap\Rnp}|u(x)|^2\,dx,
\end{equation}
and, if $u \in \Bst$,
\begin{equation}
  \label{eq:rad-truncation-tail}
  \int_{\sigma<|\xpar|<\sqrt3\sigma}
    \bigl|\usurf(\xpar) - \usurf^{(\sigma)}(\xpar)\bigr|^2 d\xpar
  \ \le\ C\,\Bstnorm{u}^2\,\sigma^2e^{-2\alpha\sigma}
  \qquad (\sigma \ge 1).
\end{equation}
\end{lemma}

\begin{proof}
For \eqref{eq:rad-truncation-main}, Cauchy--Schwarz together with
$\int_0^{Z}\psizero^2\,dx_n \le \int_0^\infty\psizero^2\,dx_n = 1$ gives
\begin{equation*}
  |\usurf^{(\sigma)}(\xpar)|^2
  \le \int_0^{Z(\xpar;\sigma)}|u(\xpar,x_n)|^2\,dx_n .
\end{equation*}
The region swept out is contained in the annulus, since
$|\xpar| > \sigma$ forces $|x| > \sigma$, while
$|x|^2 = |\xpar|^2 + x_n^2 < |\xpar|^2 + Z(\xpar;\sigma)^2 = 4\sigma^2$.
Integration in $\xpar$ then gives \eqref{eq:rad-truncation-main}.

For \eqref{eq:rad-truncation-tail}, note
$\usurf - \usurf^{(\sigma)} = \int_{Z(\xpar;\sigma)}^\infty\psizero u\,dx_n$
and $Z(\xpar;\sigma) \ge \sigma$, so we may split $[\sigma,\infty)$ into the
intervals $[2^k\sigma,2^{k+1}\sigma)$, $k \ge 0$. On each of them
$\psizero \le \sqrt{2\alpha}\,e^{-\alpha2^k\sigma}$, so Cauchy--Schwarz in
$x_n$ over an interval of length $2^k\sigma$ gives
\begin{equation*}
  \int_{2^k\sigma}^{2^{k+1}\sigma}\psizero|u|\,dx_n
  \le \sqrt{2\alpha}\,e^{-\alpha2^k\sigma}(2^k\sigma)^{1/2}
      \Bigl(\int_{2^k\sigma}^{2^{k+1}\sigma}|u|^2dx_n\Bigr)^{1/2}.
\end{equation*}
Take the $\Ltwo{\{\sigma<|\xpar|<\sqrt3\sigma\}}$ norm in $\xpar$ and apply
Minkowski's inequality to the sum over $k$. Writing
$D_k = \{\sigma<|\xpar|<\sqrt3\sigma\} \times [2^k\sigma,2^{k+1}\sigma)$, we
have for $k \ge 0$
\begin{equation*}
  |x|^2 = |\xpar|^2 + x_n^2 < 3\sigma^2 + 4\cdot4^k\sigma^2
  \le 7\cdot4^k\sigma^2
  \qquad (x \in D_k),
\end{equation*}
so that $D_k \subset \{|x| < \sqrt7\cdot2^k\sigma\} \cap \Rnp$ and therefore,
by \eqref{eq:Bstar-equivalent},
$\norm{u}_{\Ltwo{D_k}}^2 \le C\Bstnorm{u}^2\,2^k\sigma$. No covering by
dyadic shells is needed: the ball form of the $\Bst$ norm applies to
$D_k$ directly. Consequently
\begin{equation*}
  \norm{\usurf - \usurf^{(\sigma)}}_{\Ltwo{\{\sigma<|\xpar|<\sqrt3\sigma\}}}
  \le C\Bstnorm{u}\sum_{k\ge0}2^k\sigma\,e^{-\alpha2^k\sigma}
  \le C'\Bstnorm{u}\,\sigma e^{-\alpha\sigma},
\end{equation*}
the last step because for $k \ge 1$ one has $2^k - 1 \ge 2^{k-1}$, so that the
$k$-th term is at most $2^k\sigma e^{-\alpha\sigma}e^{-\alpha2^{k-1}\sigma}$
and the series converges, for $\sigma \ge 1$, to a bounded multiple of its
first term. Squaring gives \eqref{eq:rad-truncation-tail}.
\end{proof}

By Lemma~\ref{lem:rad-scale-ratio} with $\lambda = 2$ the right-hand side of
\eqref{eq:rad-truncation-main} is $o(\sigma)$, hence so is the left-hand side.
Dividing
\eqref{eq:rad-truncation-tail} by $\sigma$ gives $O(\sigma e^{-2\alpha\sigma})$.
Hence, by the triangle inequality in
$\Ltwo{\{\sigma<|\xpar|<\sqrt3\sigma\}}$,
\begin{equation*}
  \lim_{\sigma\to\infty}\frac1\sigma
    \int_{\sigma<|\xpar|<\sqrt3\sigma}|\usurf(\xpar)|^2 d\xpar = 0 .
\end{equation*}
The same computation with $\sup$ in place of $\lim$ shows
$\usurf \in \Bstd{\Rnm}$. Applying Lemma~\ref{lem:rad-scale-ratio} once more,
now on $\Rnm$ and with the ratio $\lambda = \sqrt3$, we conclude
\begin{equation}
  \label{eq:rad-surface-Bzero}
  \usurf \in \Bstz(\Rnm).
\end{equation}
The ratio $\sqrt3$ arises as follows, and it costs nothing.
The geometric cut \eqref{eq:rad-Zplus} requires
$4\sigma^2 - |\xpar|^2 \ge \sigma^2$, hence $|\xpar| < \sqrt3\sigma$, so a
single application of \eqref{eq:rad-truncation-main} covers only an annulus of
ratio $\sqrt3 < 2$ in $\Rnm$. Lemma~\ref{lem:rad-scale-ratio}, being stated for
an arbitrary ratio $\lambda > 1$, absorbs this mismatch. No gluing of shifted
base points is required.

\medskip
\noindent\textbf{Step 3: the surface component vanishes.}
By \eqref{eq:rad-surface-helmholtz} the function $\usurf$ solves the free
Helmholtz equation in $\Rnm$ outside the compact set $\{|\xpar| \le R_0\}$ with
wavenumber $\ksurf > 0$, and by \eqref{eq:rad-surface-Bzero} it satisfies the
Rellich decay condition there. Since $n \ge 3$ we have $n - 1 \ge 2$, so the
classical Rellich theorem applies in dimension $n-1$ and produces $R_1' > R_0$
with $\usurf(\xpar) = 0$ for $|\xpar| > R_1'$. As $n - 1 \ge 2$, the set
$\{|\xpar| > R_0\}$ is connected, and $\usurf$ is real analytic there, being a
solution of a constant-coefficient elliptic equation. By unique continuation
\cite{Aronszajn1957},
\begin{equation}
  \label{eq:rad-surface-zero}
  \usurf \equiv 0 \qquad (|\xpar| > R_0),
  \qquad\text{equivalently}\qquad
  \int_0^\infty e^{-\alpha x_n}u(\xpar,x_n)\,dx_n = 0 .
\end{equation}
Thus $u = \ubulk$ carries no surface-wave component. The requirement
$n - 1 \ge 2$ is genuine: for $n = 2$ the surface channel would be governed by a
one-dimensional Helmholtz equation, whose solutions $e^{\pm i\ksurf\xpar}$ do
not decay at all, and the argument breaks down at this point rather than merely
becoming technically harder.

\medskip
\noindent\textbf{Step 4: the auxiliary function and its odd extension.}
For the bulk channel we pass to the whole space. Neither the even nor the odd
reflection preserves a Robin condition, but the function
\begin{equation*}
  v := \partial_{x_n}u + \alpha u
\end{equation*}
satisfies $(-\Delta - E)v = 0$ in $\Rnp(R_0)$, because $-\Delta - E$ has
constant coefficients and hence commutes with $\partial_{x_n}$ and with
multiplication by $\alpha$, and it satisfies $v = 0$ on $\Gamma(R_0)$ by the
Robin condition. Its odd extension
\begin{equation*}
  \widetilde v(\xpar,x_n) =
  \begin{cases}
    v(\xpar,x_n), & x_n > 0,\\
    -v(\xpar,-x_n), & x_n < 0,
  \end{cases}
\end{equation*}
is therefore $C^1$ across $\{x_n = 0\}$ and solves $(-\Delta-E)\widetilde v = 0$
on $\{|x| > R_0\} \subset \Rn$, the equation being invariant under
$x_n \mapsto -x_n$. Note that $\widetilde v$ is defined on all of
$\{|x| > R_0\}$, since $|(\xpar,-x_n)| = |x|$.

\medskip
\noindent\textbf{Step 5: the Robin extension formula.}
Suppose first that $|\xpar| > R_0$, so that the whole vertical line above and
below $\xpar$ meets $\Rnp$ only in $\Rnp(R_0)$. We extend $u$ by solving
$\partial_{x_n}\widetilde u + \alpha\widetilde u = \widetilde v$ for $x_n < 0$
with the initial value $\widetilde u(\xpar,0) = u(\xpar,0)$. With the
integrating factor $e^{\alpha x_n}$ we have
$\partial_{x_n}(e^{\alpha x_n}\widetilde u) = e^{\alpha x_n}\widetilde v$, so
integrating from $0$ to $x_n < 0$ and substituting $\tau = -s$,
\begin{equation*}
  e^{\alpha x_n}\widetilde u(\xpar,x_n) - u(\xpar,0)
  = -\int_0^{x_n}e^{\alpha s}v(\xpar,-s)\,ds
  = \int_0^{-x_n}e^{-\alpha\tau}v(\xpar,\tau)\,d\tau .
\end{equation*}
Inserting $v = \partial_\tau u + \alpha u$ and integrating the first summand by
parts,
\begin{equation*}
  \int_0^{h}e^{-\alpha\tau}v\,d\tau
  = e^{-\alpha h}u(\xpar,h) - u(\xpar,0)
    + 2\alpha\int_0^{h}e^{-\alpha\tau}u(\xpar,\tau)\,d\tau
  \qquad (h = -x_n).
\end{equation*}
The terms $u(\xpar,0)$ cancel, and dividing by $e^{\alpha x_n}$ we obtain the
extension formula
\begin{equation}
  \label{eq:rad-extension}
  \widetilde u(\xpar,x_n) = u(\xpar,-x_n)
    + 2\alpha\int_0^{-x_n}e^{-\alpha(x_n+\tau)}u(\xpar,\tau)\,d\tau
  \qquad (x_n < 0).
\end{equation}
For $\alpha \to 0$ the second term disappears and \eqref{eq:rad-extension}
degenerates to the even reflection, which is the Neumann one. For
$\alpha \to \infty$ the second term does not disappear. It contributes,
formally, twice the even reflection with the opposite sign, so that
\eqref{eq:rad-extension} degenerates to the odd reflection, which is the
Dirichlet one. The Robin case interpolates between the two by a memory term
along the vertical ray.

Formula \eqref{eq:rad-extension} is not yet usable, for two distinct reasons.
Its second term grows like $e^{-\alpha x_n}$ as $x_n \to -\infty$ in general.
It also requires the values of $u$ on the whole segment $0 \le \tau \le -x_n$,
so that it defines $\widetilde u$ only over $\{|\xpar| > R_0\}$, whereas the
Rellich theorem to be applied in Step~6 needs a function defined outside a
compact subset of $\Rn$. Both defects are removed at once by the result of
Step~3. Writing $g(h) = \int_0^he^{-\alpha\tau}u(\xpar,\tau)\,d\tau$, relation
\eqref{eq:rad-surface-zero} says exactly $g(\infty) = 0$, so
$g(-x_n) = -\int_{-x_n}^\infty e^{-\alpha\tau}u\,d\tau$. Substituting
$\tau = -x_n+s$ with $s \ge 0$ turns \eqref{eq:rad-extension} into
\begin{equation}
  \label{eq:rad-clean-extension}
  \widetilde u(\xpar,x_n) = u(\xpar,-x_n)
    - 2\alpha\int_0^\infty e^{-\alpha s}u(\xpar,-x_n+s)\,ds
  \qquad (x_n < 0),
\end{equation}
in which the extension is an exponentially weighted average of the values of
$u$ \emph{above} the reflected point, with no growing factor at all. That the
rewriting is not cosmetic is seen from the formal test function $u \equiv 1$,
for which the right-hand side of \eqref{eq:rad-extension} equals
$2e^{-\alpha x_n}-1 \to \infty$: the vanishing of the surface component is
precisely what excludes this behaviour.

\emph{We take \eqref{eq:rad-clean-extension} as the definition of $\widetilde u$
on $\{x_n<0\}$}, and \eqref{eq:rad-extension} is retained only as the derivation
which identifies it. The gain is that \eqref{eq:rad-clean-extension} uses the
values $u(\xpar,\tau)$ for $\tau \ge -x_n$ only. If $-x_n > R_0$ then
$|(\xpar,\tau)| \ge \tau \ge -x_n > R_0$ for every such $\tau$, whatever
$\xpar$. Hence $\widetilde u$ is defined on
\begin{equation*}
  \Omega_{R_0} :=
  \bigl(\Rnp(R_0)\bigr) \cup \{|\xpar| > R_0\} \cup \{x_n < -R_0\},
\end{equation*}
and $\Rn \setminus \Omega_{R_0} \subset \{|x| \le \sqrt2\,R_0\}$ is bounded. A
point outside $\Omega_{R_0}$ has $|\xpar| \le R_0$ and $-x_n \le R_0$, and
either $x_n \le 0$, in which case $|x|^2 \le 2R_0^2$, or $x_n > 0$, in which
case it fails to lie in $\Rnp(R_0)$ only if $|x| \le R_0$.

\begin{lemma}[$C^2$ matching and validity of the equation]
\label{lem:rad-matching}
The function $\widetilde u$ of \eqref{eq:rad-clean-extension} matches $u$ across
$\{x_n = 0\}$ in the $C^2$ sense, and $(-\Delta - E)\widetilde u = 0$ holds
classically on $\{|x| > \sqrt2\,R_0\} \subset \Rn$.
\end{lemma}

\begin{proof}
We first record that \eqref{eq:rad-clean-extension} satisfies the same ordinary
differential equation as \eqref{eq:rad-extension} wherever it is defined. Put
$h = -x_n > 0$ and $G(h) = u(\xpar,h) - 2\alpha\int_0^\infty
e^{-\alpha s}u(\xpar,h+s)\,ds$, so that $\widetilde u(\xpar,x_n) = G(-x_n)$ and
$(\partial_{x_n}+\alpha)\widetilde u = -G'(h) + \alpha G(h)$. Since
$\int_0^\infty e^{-\alpha s}\partial_hu(\xpar,h+s)\,ds
= -u(\xpar,h) + \alpha\int_0^\infty e^{-\alpha s}u(\xpar,h+s)\,ds$, the boundary
term at $s = \infty$ vanishing by Lemma~\ref{lem:rad-pointwise}, we get
\begin{equation*}
  -G'(h) + \alpha G(h)
  = -\partial_hu(\xpar,h) - \alpha u(\xpar,h)
  = -v(\xpar,h) = \widetilde v(\xpar,x_n),
\end{equation*}
as claimed. Consequently, with $\mathcal{G} := (-\Delta - E)\widetilde u$ on
$\{x_n<0\}\cap\Omega_{R_0}$, and since $-\Delta - E$ commutes with
$\partial_{x_n} + \alpha$,
\begin{equation*}
  (\partial_{x_n}+\alpha)\mathcal{G} = (-\Delta-E)\widetilde v = 0,
  \qquad\text{hence}\qquad
  \mathcal{G}(\xpar,x_n) = \mathcal{G}(\xpar,x_n^0)\,e^{-\alpha(x_n-x_n^0)}
\end{equation*}
along each vertical segment contained in the region.

\emph{Vertical lines with $|\xpar| > R_0$.} Here the segment reaches
$x_n = 0^-$, and we show $\mathcal{G}(\xpar,0^-) = 0$ by $C^2$ matching.
The two formulas agree on this range. Differentiating
\eqref{eq:rad-extension} once and using
$\frac{d}{dx_n}\int_0^{-x_n}e^{-\alpha\tau}u\,d\tau = -e^{\alpha x_n}u(\xpar,-x_n)$,
\begin{equation*}
  \partial_{x_n}\widetilde u(\xpar,x_n)
  = -u'(\xpar,-x_n) - 2\alpha u(\xpar,-x_n)
    - 2\alpha^2\int_0^{-x_n}e^{-\alpha(x_n+\tau)}u(\xpar,\tau)\,d\tau,
\end{equation*}
where $u'$ denotes $\partial_{x_n}u$. Letting $x_n \to 0^-$ and using the Robin
condition $u'(\xpar,0) = -\alpha u(\xpar,0)$,
\begin{equation*}
  \partial_{x_n}\widetilde u(\xpar,0^-)
  = -u'(\xpar,0) - 2\alpha u(\xpar,0)
  = -\alpha u(\xpar,0) = \partial_{x_n}u(\xpar,0^+).
\end{equation*}
Differentiating once more,
\begin{equation*}
  \partial_{x_n}^2\widetilde u(\xpar,x_n)
  = u''(\xpar,-x_n) + 2\alpha u'(\xpar,-x_n) + 2\alpha^2u(\xpar,-x_n)
    + 2\alpha^3\int_0^{-x_n}e^{-\alpha(x_n+\tau)}u(\xpar,\tau)\,d\tau,
\end{equation*}
so that, again by the Robin condition,
$\partial_{x_n}^2\widetilde u(\xpar,0^-)
= u''(\xpar,0) + 2\alpha u'(\xpar,0) + 2\alpha^2u(\xpar,0) = u''(\xpar,0)$.
Since the tangential derivatives match trivially, $\widetilde u$ is $C^2$ across
$\{x_n=0\}$, so $\mathcal{G}(\xpar,0^-)$ coincides with the limit of
$(-\Delta-E)\widetilde u$ from above, which is $0$. Hence $\mathcal{G}$ vanishes
on every such vertical line.

\emph{Vertical lines with $|\xpar| \le R_0$.} Here the segment is
$\{x_n < -R_0\}$ and does not reach the boundary. By
Lemma~\ref{lem:rad-pointwise} and \eqref{eq:rad-clean-extension}, $\widetilde u$
and its derivatives up to order two grow at most polynomially as
$x_n \to -\infty$, whence so does $\mathcal{G}$. But
$\mathcal{G}(\xpar,x_n) = \mathcal{G}(\xpar,-2R_0)e^{-\alpha(x_n+2R_0)}$ grows
exponentially unless $\mathcal{G}(\xpar,-2R_0) = 0$. Therefore
$\mathcal{G} \equiv 0$ there as well.

Since $\{x_n>0\} \cap \{|x|>R_0\}$ carries the original equation and
$\Rn\setminus\Omega_{R_0} \subset \{|x|\le\sqrt2R_0\}$, the conclusion holds on
$\{|x| > \sqrt2R_0\}$.
\end{proof}

\medskip
\noindent\textbf{Step 6: decay on the whole space, and conclusion.}

\begin{lemma}[Decay of the extension on the whole space]
\label{lem:rad-full-decay}
Under the hypotheses of Theorem~\ref{thm:rad-rellich}, the function
$\widetilde u$ satisfies
\begin{equation}
  \label{eq:rad-full-decay}
  \lim_{R\to\infty}\frac1R\int_{|x|<R}|\widetilde u(x)|^2\,dx = 0 .
\end{equation}
\end{lemma}

\begin{proof}
On $\{x_n>0\}$ we have $\widetilde u = u$, and the claim is
\eqref{eq:rad-rellich-hypothesis}. For $\{x_n<0\}$ put $h = -x_n > 0$. The
region $\Omega_R^- = \{(\xpar,-h) : R<|(\xpar,h)|<2R,\ h>0\}$ is congruent to
its counterpart in $\Rnp$, and by \eqref{eq:rad-clean-extension}
\begin{equation*}
  \widetilde u(\xpar,-h) = u(\xpar,h) - 2\alpha(\Aav u)(\xpar,h),
  \qquad
  (\Aav u)(\xpar,h) = \int_0^\infty e^{-\alpha s}u(\xpar,h+s)\,ds .
\end{equation*}
The first term contributes $o(R)$ by \eqref{eq:rad-rellich-hypothesis} and
Lemma~\ref{lem:rad-scale-ratio}. For the second, the Minkowski integral
inequality gives
$\norm{\Aav u}_{\Ltwo{\Omega_R^-}} \le \int_0^\infty e^{-\alpha s}
\norm{u(\cdot,\cdot+s)}_{\Ltwo{\Omega_R^-}}\,ds$. The shift
$(\xpar,h) \mapsto (\xpar,h+s)$ preserves Lebesgue measure and, since
$|(\xpar,h+s)| \ge |(\xpar,h)|$ and
$|(\xpar,h+s)| \le |(\xpar,h)| + s$, maps $\Omega_R^-$ into
$\{R<|x|<2R+s\} \cap \Rnp$. Hence
\begin{equation*}
  \norm{u(\cdot,\cdot+s)}_{\Ltwo{\Omega_R^-}}^2 \le N(R,s)
  := \int_{\{R<|x|<2R+s\}\cap\Rnp}|u|^2\,dx .
\end{equation*}
For $s \ge R$ we use only $u \in \Bst$, which by \eqref{eq:Bstar-equivalent}
gives $N(R,s) \le C\Bstnorm{u}^2(R+s)$ and therefore
$\int_R^\infty e^{-\alpha s}N(R,s)^{1/2}ds = O\bigl(e^{-\alpha R}R^{1/2}\bigr)$.
For $0 \le s \le R$ we have $2R+s \le 3R$, so that
\begin{equation*}
  N(R,s) \le \int_{\{R<|x|<3R\}\cap\Rnp}|u|^2\,dx = o(R)
\end{equation*}
by
\eqref{eq:rad-rellich-hypothesis} and Lemma~\ref{lem:rad-scale-ratio} with
$\lambda = 3$, uniformly in $s$. It follows that
$\int_0^Re^{-\alpha s}N(R,s)^{1/2}ds \le \alpha^{-1}o(R^{1/2})$. Adding the two
ranges, $\norm{\Aav u}_{\Ltwo{\Omega_R^-}} = o(R^{1/2})$, that is
$R^{-1}\norm{\Aav u}_{\Ltwo{\Omega_R^-}}^2 = o(1)$. The bounded set
$\Rn\setminus\Omega_{R_0}$, on which $\widetilde u$ is undefined, may be given
the value $0$. It contributes $O(1/R)$. Combining the two half-spaces
and applying Lemma~\ref{lem:rad-scale-ratio} in $\Rn$ gives
\eqref{eq:rad-full-decay}.
\end{proof}

\begin{proof}[Conclusion of the proof of Theorem~\ref{thm:rad-rellich}]
By Lemma~\ref{lem:rad-matching}, $\widetilde u$ solves
$(-\Delta - E)\widetilde u = 0$ on $\{|x|>\sqrt2R_0\} \subset \Rn$ with $E>0$,
and by Lemma~\ref{lem:rad-full-decay} it satisfies the Rellich decay condition
there. The classical $n$-dimensional Rellich theorem yields $R_1 > R_0$ with
$\widetilde u = 0$ for $|x|>R_1$. Since $\widetilde u = u$ on
$\{x_n>0\}$, this is \eqref{eq:rad-rellich-conclusion}.
\end{proof}

\subsubsection{The outgoing radiation condition}
\label{sec:rad-condition}

We now formulate the radiation condition. Following
\cite{DuranMugaNedelec2009} we impose it directly on $u$, not on the summands of
any decomposition of $u$: a boundary condition at infinity should be a statement
about the solution, and the splitting \eqref{eq:rad-splitting} is a device of
the transverse spectral analysis. Since the two channels propagate with the
different wavenumbers $\kbulk$ and $\ksurf$, and since the surface channel lives
in a layer of thickness $O(1/\alpha)$ adjacent to $\Gamma$, the condition is
stated separately in a near-boundary region and its complement.

\begin{definition}[Outgoing radiation condition]
\label{def:rad-condition}
Let $E > 0$, $\alpha > 0$, and fix $\alpha_1 \in (0,1)$. Split
\begin{equation}
  \label{eq:rad-regions}
  \Rnp(\alpha_1+) = \{x \in \Rnp : x_n > r^{\alpha_1}\},
  \qquad
  \Rnp(\alpha_1-) = \{x \in \Rnp : 0 \le x_n \le r^{\alpha_1}\},
\end{equation}
and write $\Sigma_1(R) = \{R<|x|<2R\} \cap \Rnp(\alpha_1+)$ and
$\Sigma_2(R) = \{R<|x|<2R\} \cap \Rnp(\alpha_1-)$. A function $u \in \Bst$,
smooth outside a compact set, satisfies the \emph{outgoing radiation condition}
if
\begin{equation}
  \label{eq:rad-bulk}
  \lim_{R\to\infty}\frac1R\int_{\Sigma_1(R)}
    \bigl|\partial_ru - i\kbulk u\bigr|^2 dx = 0,
\end{equation}
\begin{equation}
  \label{eq:rad-surf}
  \lim_{R\to\infty}\frac1R\int_{\Sigma_2(R)}
    \bigl|\partial_ru - i\ksurf u\bigr|^2 dx = 0 .
\end{equation}
For solutions of an inhomogeneous equation with compactly supported right-hand
side the limits are understood along $R$ large enough that
$\{R<|x|<2R\}$ misses the support, so that $\partial_ru$ is classically
defined there.
\end{definition}

The exponent $\alpha_1$ is a purely geometric parameter, unrelated to the
impedance $\alpha$. It is the exponent that \cite{DuranMugaNedelec2009} call
$\alpha$. The subscript is added here because that letter is already taken. The
notation $\Rnp(\alpha_1\pm)$ for the two regions is theirs. Two remarks on the
shape of the definition are in order, and both concern the comparison with that
paper.

\begin{lemma}[The radiation condition passes to differences]
\label{lem:rad-linear}
Let $\kappa > 0$, let $\Sigma \subset \Rnp$, and let $u_1,u_2$ satisfy
$\lim_{R\to\infty}R^{-1}\int_{\{R<|x|<2R\}\cap\Sigma}
|\partial_ru_j - i\kappa u_j|^2dx = 0$ for $j=1,2$. Then $u_1-u_2$ satisfies
the same condition. Consequently, if $u_1$ and $u_2$ both satisfy
Definition~\ref{def:rad-condition}, so does $u_1-u_2$.
\end{lemma}

\begin{proof}
Put $F_j = \partial_ru_j - i\kappa u_j$. Then
$\partial_r(u_1-u_2) - i\kappa(u_1-u_2) = F_1 - F_2$, and the triangle
inequality in $\Ltwo{\{R<|x|<2R\}\cap\Sigma}$ gives
$\norm{F_1-F_2} \le \norm{F_1}+\norm{F_2}$. Divide by $R^{1/2}$, square, and
let $R \to\infty$. The last assertion follows by applying this with
$\Sigma = \Rnp(\alpha_1+)$, $\kappa = \kbulk$ and with
$\Sigma = \Rnp(\alpha_1-)$, $\kappa = \ksurf$.
\end{proof}

\begin{remark}[The exponent $\alpha_1$ and the range $(1/4,1/2)$ of
\cite{DuranMugaNedelec2009}]
\label{rem:rad-alpha1}
Dur\'an, Muga and N\'ed\'elec split $\mathbb{R}^3_+$ by the same surface
$x_n = r^{\alpha_1}$ but restrict $\alpha_1$ to $(1/4,1/2)$. Their condition is
the pair of pointwise bounds
\begin{equation*}
  \bigl|\partial_ru - i\kbulk u\bigr| \le \frac{C}{r^{2\alpha_1+\frac12}}
  \ \ \text{in } \Rnp(\alpha_1+),
  \qquad
  \bigl|\partial_ru - i\ksurf u\bigr| \le \frac{C}{r^{\frac32-\alpha_1}}
  \ \ \text{in } \Rnp(\alpha_1-),
\end{equation*}
written in the notation of \eqref{eq:rad-regions}. Both endpoints disappear
here, for different reasons.

\emph{The upper endpoint.} Weakening the norm gains nothing by itself: for a
quantity of pointwise size $r^{-\gamma}$ on a region whose intersection with
$\{r \sim R\}$ has volume $\sim R^{d}$, both the weighted $L^2$ condition
underlying their variational theory and our shell-averaged condition amount to
$2\gamma > d-1$. The gain comes from integrating at all. Their pointwise bounds
ask for $\gamma > 1$ in each region, and in the layer the rate is
$\frac{n-2}{2}+1-\alpha_1$, which for $n = 3$ gives $\alpha_1 < 1/2$. A
pointwise condition does not see that the surface wave, whose radiation defect
is of relative size $r^{\alpha_1-1}$ at the top of the layer, is itself
exponentially small there. Under an integrated condition the weight
$\psizero^2$ makes the $x_n$-integration converge, and part~(ii) of the proof
of Proposition~\ref{prop:rad-resolvent-satisfies} shows that the surface wave
imposes no restriction on $\alpha_1$. What remains in the layer is the bulk
part of the resolvent, which does not decay in $x_n$, and its volume factor
$R^{n-1+\alpha_1}$ gives $\alpha_1 < 1$.

\emph{The lower endpoint.} In the interior region their rate
$r^{-(2\alpha_1+1/2)}$ requires $\alpha_1 > 1/4$ under any of these
formulations, the volume factor being the same there, but that rate is not
optimal. The apparent singularity of the reflected integral \eqref{eq:v-def}
on the circle $|\xipar| = \kbulk$ is one of the coordinate $|\xipar|$ and not
of the integral (Lemma~\ref{lem:branch}). In the variable $\mu$ the amplitude is
holomorphic across that circle, and the radiation defect vanishes to second
order at the stationary point (Lemma~\ref{lem:rad-phase}). The resulting bound
of Lemma~\ref{lem:rad-uniform} is $O(r^{-n/2})$ uniformly in the direction and
does not involve $\alpha_1$.

Accordingly, neither Lemma~\ref{lem:rad-linear} nor the uniqueness proof of
Theorem~\ref{thm:rad-uniqueness-hom} uses the value of $\alpha_1$, only the
fact that \eqref{eq:rad-regions} partitions each shell. The value of
$\alpha_1$ enters only in Proposition~\ref{prop:rad-resolvent-satisfies},
through the constraint $\alpha_1 < 1$, uniformly in $n \ge 3$, and
Theorem~\ref{thm:rad-equivalence} will show that the resulting class of
solutions does not depend on $\alpha_1$ at all.
\end{remark}

\subsubsection{The reflected integral: a uniform radiation estimate}
\label{sec:rad-reflected}

The pointwise expansions of \S\ref{subsec:asymptotics} are uniform on cones
$\{x_n \ge \delta r\}$ with $\delta$ fixed, the constants degenerating as
$\delta \downarrow 0$. The interior region $\Rnp(\alpha_1+)$ of
\eqref{eq:rad-regions} is not a cone: for every $\delta > 0$ it contains the
wedge $\{r^{\alpha_1} < x_n < \delta r\}$, whose intersection with a shell has
volume of order $\delta R^n$. The boundary layer $\Rnp(\alpha_1-)$ lies entirely
inside the grazing region. Neither can be reached by those expansions, and this
subsection supplies what is missing: a bound on the radiation defect of the
remainder $\Krem{E+i0}f$ of the exact splitting
\eqref{eq:splitting-preview}, valid uniformly in $\theta \in \Sm$ down to and
including the boundary.

Throughout, $f \in C_c^{\infty}(\Rnp)$ is fixed, $E > 0$, and we use the
representation of Proposition~\ref{prop:reflected}. Recall from
\eqref{eq:mu-branches} that $\mu = \mu(\xipar)$ is the normal wavenumber, that
$\Gamma$ of \eqref{eq:branch-contour} is the path traversed by $\mu$ as
$|\xipar|$ increases from $0$ to $\infty$, and that
\begin{equation*}
  s(\mu) = \sqrt{E-\mu^2} = |\xipar|,
  \qquad
  s\ \text{holomorphic near}\ \Gamma\cap\{|\mu|<\kbulk\},
  \qquad
  s(i\tau) = \sqrt{E+\tau^2} .
\end{equation*}
Writing $\xipar = s(\mu)\omega'$ with $\omega' \in \Ssurf$ and using
$|\xipar|^{n-2}d|\xipar| = -s(\mu)^{n-3}\mu\,d\mu$, the reflected integral
\eqref{eq:v-def} becomes, for $x \in \Rnp$ with $r$ large,
\begin{equation}
  \label{eq:rad-mu-representation}
  v(x) = \int_{\Ssurf}\int_{\Gamma}
    a_f(\mu,\omega')\;e^{\,i r\,\Phi_{\omega'}(\mu)}\,d\mu\,dS(\omega'),
  \qquad
  a_f = (2\pi)^{-\frac{n-2}{2}}\,
        \frac{s(\mu)^{n-3}\,\mu\;\widehat{\zext f}\bigl(s(\mu)\omega',-\mu\bigr)}
             {i\mu + \alpha},
\end{equation}
where the phase, normalized by $r$, is
\begin{equation}
  \label{eq:rad-phase}
  \Phi_{\omega'}(\mu)
  = \frac{\rho}{r}\,s(\mu)\,(\omega'\cdot\omega) + \theta_n\,\mu,
  \qquad
  \frac{\rho}{r} = \sqrt{1-\theta_n^2},
\end{equation}
so that $r\Phi_{\omega'}$ is the phase \eqref{eq:branch-phase} of
Lemma~\ref{lem:branch}. The quotient in \eqref{eq:rad-mu-representation} is
understood, on the sphere $|\xipar| = \ksurf$ where $\mu = i\alpha$, through
\eqref{eq:v-pole}. The singular part contributes exactly
$\widetilde u_{\mathrm{surf}}$, so that
$\Krem{E+i0}f = v - \widetilde u_{\mathrm{surf}}$.

Three features of \eqref{eq:rad-mu-representation} are used, and they are what
make the estimate possible.

First, $\widehat{\zext f}$ is entire on $\mathbb{C}^n$ by the Paley--Wiener
theorem, and $\mu \mapsto (s(\mu)\omega',-\mu)$ is holomorphic near
$\{|\mu| < \kbulk\}$. Hence $a_f(\cdot,\omega')$ is holomorphic there except for
the simple pole at $\mu = i\alpha$, with all $\mu$-derivatives bounded uniformly
in $\omega'$. In particular the corner $\mu = 0$ of $\Gamma$, which is the
branch circle $|\xipar| = \kbulk$, is not a singularity. Second, the Jacobian
supplies the factor $\mu$, so that $a_f(0,\omega') = 0$. Third, the amplitude
does not depend on $x$. Only the phase does. Consequently
\begin{equation}
  \label{eq:rad-defect-identity}
  \partial_r v - i\kbulk v
  = i\int_{\Ssurf}\int_{\Gamma}
    \bigl(\Phi_{\omega'}(\mu) - \kbulk\bigr)\,a_f(\mu,\omega')\;
    e^{\,ir\Phi_{\omega'}(\mu)}\,d\mu\,dS(\omega'),
\end{equation}
the differentiation being at fixed $\theta$. \emph{Verifying the radiation
condition is thus nothing but multiplying the amplitude of
\eqref{eq:rad-mu-representation} by $\Phi_{\omega'} - \kbulk$}, and the whole
question is how small that factor is where the integral concentrates.

\begin{lemma}[The radiation defect vanishes to second order at the stationary
point]
\label{lem:rad-phase}
Let $\theta \in \Sm$ and $\omega' \in \Ssurf$, and let
$\mu \in [0,\kbulk]$ be real. Then
$\Phi_{\omega'}(\mu) \le \kbulk$, with equality if and only if
$\omega' = \omega$ and $\mu = \mu_* := \kbulk\,\theta_n$. Moreover, writing
$\Phi_{\omega}$ for the value of $\Phi_{\omega'}$ at $\omega' = \omega$ and the
prime for $d/d\mu$, one has $\Phi_{\omega}'(\mu_*) = 0$ and
$\Phi_{\omega}''(\mu) = -\rho E/(r\,s(\mu)^3)$, and
\begin{equation*}
  \Phi_{\omega}(\mu) - \kbulk
  = -\frac{(\mu-\mu_*)^2}{2\kbulk\,(1-\theta_n^2)}
    + O\bigl(|\mu-\mu_*|^3\bigr) ,
  \qquad
  \Phi_{\omega}''(\mu_*) = -\frac{1}{\kbulk\,(1-\theta_n^2)} .
\end{equation*}
\end{lemma}

\begin{proof}
On the real segment $s(\mu)^2 + \mu^2 = E$, so by the Cauchy--Schwarz
inequality in $\mathbb{R}^2$, applied to the unit vector
$(\rho/r,\theta_n)$ and to $(s(\mu),\mu)$,
\begin{equation*}
  \Phi_{\omega'}(\mu)
  \le \frac{\rho}{r}s(\mu) + \theta_n\mu
  \le \sqrt{\tfrac{\rho^2}{r^2}+\theta_n^2}\;\sqrt{s(\mu)^2+\mu^2}
  = \kbulk .
\end{equation*}
The first inequality is strict unless $\omega'\cdot\omega = 1$, and the second
unless $(s(\mu),\mu)$ is a positive multiple of $(\rho/r,\theta_n)$, that is
unless $\mu/s(\mu) = \theta_n r/\rho$, which by $s^2+\mu^2 = E$ means
$\mu = \kbulk\theta_n$. Since $s'(\mu) = -\mu/s(\mu)$,
\begin{equation*}
  \Phi_{\omega}'(\mu) = \theta_n - \frac{\rho}{r}\,\frac{\mu}{s(\mu)},
  \qquad
  \Phi_{\omega}''(\mu) = -\frac{\rho}{r}\,\frac{s(\mu)^2+\mu^2}{s(\mu)^3}
  = -\frac{\rho E}{r\,s(\mu)^3},
\end{equation*}
and $\Phi_{\omega}'(\mu_*) = 0$ because $s(\mu_*) = \kbulk\rho/r$.
Substituting this value into the formula for $\Phi_{\omega}''$ gives
\begin{equation*}
  \Phi_{\omega}''(\mu_*)
  = -\frac{\rho}{r}\,\frac{E}{(\kbulk\rho/r)^3}
  = -\frac{E}{\kbulk^3}\Bigl(\frac{r}{\rho}\Bigr)^{2}
  = -\frac{1}{\kbulk\,(1-\theta_n^2)},
\end{equation*}
and Taylor expansion at $\mu_*$ yields the stated form. In particular
$|\Phi_{\omega}''(\mu_*)| \asymp 1$ as long as $\theta_n$ stays away from $1$.
\end{proof}

\begin{lemma}[Uniform radiation estimate for the remainder]
\label{lem:rad-uniform}
Let $n \ge 3$, $E > 0$, $\alpha > 0$ and $f \in C_c^{\infty}(\Rnp)$. There is
$C = C(n,E,\alpha,f)$ such that, for all $x \in \Rnp$ with $r \ge 1$ and
uniformly in $\theta \in \Sm$,
\begin{equation}
  \label{eq:rad-K-radiation}
  \bigl|\Krem{E+i0}f\bigr| \le C\,r^{-\frac{n-1}{2}},
  \qquad
  \bigl|\partial_r\Krem{E+i0}f - i\kbulk\,\Krem{E+i0}f\bigr|
    \le C\,r^{-\frac{n}{2}} .
\end{equation}
\end{lemma}

\begin{proof}
Both bounds are proved at once: by \eqref{eq:rad-defect-identity} the second is
the first with the amplitude $a_f$ replaced by
$(\Phi_{\omega'}-\kbulk)a_f$, and we shall indicate at each step which of the
two exponents the argument produces. Constants are allowed to depend on
$n,E,\alpha,f$ and are denoted by $C$.

\emph{Step 1: removal of the pole and of the surface term.}
Let $\chi \in C_c^{\infty}(\mathbb{C})$ be equal to $1$ on
$\{|\mu - i\alpha| \le \alpha/4\}$ and supported in
$\{|\mu-i\alpha| \le \alpha/2\}$, so that $\operatorname{supp}\chi$ meets
$\Gamma$ in a compact subinterval of the imaginary arm, at distance
$\ge \alpha/2$ from the corner $\mu = 0$. Split
$\Krem{E+i0}f = I_1 - I_2 + I_3$, where $I_1$ is the contribution of
$(1-\chi)a_f$ to \eqref{eq:rad-mu-representation}, $I_2$ is the contribution of
$(1-\chi)$ to $\widetilde u_{\mathrm{surf}}$, and $I_3$ is the contribution of
$\chi$ to $v - \widetilde u_{\mathrm{surf}}$.

The amplitude of $I_1$ is holomorphic in $\mu$ near $\Gamma\cap\{|\mu|<\kbulk\}$
and rapidly decreasing along the imaginary arm, the pole having been cut out.
Here $I_1$ is the main term and occupies Steps 2--5.

For $I_2$, by \eqref{eq:usurf-exact} the surface term is
$\psizero(x_n)$ times the inverse Fourier transform in $\Rnm$ of
$\widehat{T_{\psizero}f}(\xipar)\bigl(|\xipar|^2-\ksurf^2-i0\bigr)^{-1}$.
Multiplying by $1-\chi$ removes the singular sphere, and what is left is the
product of a Schwartz function with a smooth symbol of order $-2$, hence a
Schwartz function of $\xpar$. Therefore
$|I_2| \le C_N\rho^{-N}e^{-\alpha x_n}$ for every $N$. As
$\max(\rho,x_n) \ge r/2$, this is $O_N(r^{-N})$. The same applies after
multiplication by $\Phi_{\omega'}-\kbulk$, which is bounded on
$\operatorname{supp}(1-\chi)\cap\Gamma\cap\{|\mu|\le\ksurf\}$ and is compensated
by the rapid decay elsewhere.

For $I_3$, the numerator of \eqref{eq:K-explicit} vanishes at $\mu = i\alpha$,
so the amplitude of $v - \widetilde u_{\mathrm{surf}}$ is regular there. On
$\operatorname{supp}\chi\cap\Gamma$ the variable is $\mu = i\tau$ with
$\tau$ in a compact subinterval of $(0,\infty)$, the amplitude is smooth, and
$|e^{ir\Phi_{\omega'}}| = e^{-r\theta_n\tau}$. If $\theta_n \ge 1/2$ this is
$O(e^{-cr})$. If $\theta_n < 1/2$ then $\rho > r\sqrt3/2$, and the
$\tau$-derivative of the phase $r\Phi_{\omega'}(i\tau)$ has real part
$\rho\,\tau(\omega'\cdot\omega)/\sqrt{E+\tau^2}$ while the $\omega'$-phase has
gradient of size $\asymp\rho$ away from $\omega' = \pm\omega$. Integrating by
parts repeatedly in $\omega'$ and in $\tau$ gives $|I_3| \le C_N\rho^{-N}$. The
amplitude and $\chi$ are smooth, so that no boundary terms occur. In
both cases $I_3 = O_N(r^{-N})$, and again the factor $\Phi_{\omega'}-\kbulk$ is
bounded there.

\emph{Step 2: contour deformation near the corner.}
Fix $\delta_0 \in (0,1/2]$, to be chosen at the end of this step, and assume
first $\theta_n \le \delta_0$, so that $\rho/r \ge \sqrt{3}/2$. Let $C_0 = C_0(E)$ be
such that
\begin{equation*}
  |\im s(\mu)| \le C_0\,\bigl|\im(\mu^2)\bigr|
  \qquad (|\mu| \le \kbulk/2),
\end{equation*}
which is possible because $s(\mu) = \kbulk\sqrt{1-\mu^2/E}$ is holomorphic there
with $s(0) = \kbulk$ and $\im s$ vanishes when $\im(\mu^2)$ does. Put
\begin{equation*}
  \epsilon_0 = \epsilon_0(\theta_n)
  = \min\Bigl\{\frac{\theta_n}{8C_0},\ \frac{\kbulk\theta_n}{2},\
    \frac{\kbulk}{2},\ \frac{\alpha}{2}\Bigr\},
\end{equation*}
and let $\delta_0$ be small enough that the first two entries are the smaller
ones for every $\theta_n \le \delta_0$, so that
$\epsilon_0 = c_1\theta_n$ with $c_1 = \min\{(8C_0)^{-1},\kbulk/2\}$ there.
Replace the corner $[0,\epsilon_0] \cup i[0,\epsilon_0]$ of $\Gamma$ by the
quarter circle
$A = \{\epsilon_0e^{i\vartheta} : 0 \le \vartheta \le \pi/2\}$. By Cauchy's
theorem this does not change the integral: the amplitude of $I_1$ is holomorphic on $\{|\mu|\le\kbulk/2\}$,
whose intersection with the closed first quadrant contains the region swept out,
and the pole $i\alpha$ lies outside it because $\epsilon_0 \le \alpha/2$. Since
the contour is deformed and not cut, no endpoint contributions arise. The
deformed path is $A \cup [\epsilon_0,\kbulk] \cup i[\epsilon_0,\infty)$.

We record the sign of the imaginary part of the phase on each piece. On the real
segment $\Phi_{\omega'}$ is real. On the imaginary arm
$\im\Phi_{\omega'}(i\tau) = \theta_n\tau \ge 0$, since $s(i\tau)$ is real. On
$A$, writing $\mu = \epsilon_0e^{i\vartheta}$ and using
$\im(\mu^2) = \epsilon_0^2\sin2\vartheta \le 2\epsilon_0^2\sin\vartheta$
together with $\omega'\cdot\omega \le 1$ and $\rho/r \le 1$,
\begin{equation*}
  \im\Phi_{\omega'}(\mu)
  \ \ge\ \theta_n\epsilon_0\sin\vartheta - 2C_0\epsilon_0^2\sin\vartheta
  \ =\ \epsilon_0\sin\vartheta\,\bigl(\theta_n - 2C_0\epsilon_0\bigr)
  \ \ge\ \tfrac34\,\theta_n\,\epsilon_0\,\sin\vartheta \ \ge\ 0,
\end{equation*}
by $\epsilon_0 \le \theta_n/(8C_0)$. Hence $|e^{ir\Phi_{\omega'}}| \le 1$
everywhere on the deformed path. It is essential that the radius $\epsilon_0$ be
proportional to $\theta_n$: for a radius fixed independently of the direction the
first term above would be dominated by the second and the integrand would grow
exponentially as $\theta_n \downarrow 0$.

\emph{Step 3: the quarter circle.}
On $A$ we have $|a_f| \le C|\mu| = C\epsilon_0$ by the factor $\mu$ in
\eqref{eq:rad-mu-representation}, and $|d\mu| = \epsilon_0\,d\vartheta$. For the
integration in $\omega'$ we separate the real and imaginary parts of $s$: the
factor $e^{ir(\rho/r)(\re s)(\omega'\cdot\omega)}$ is oscillatory with parameter
$r(\rho/r)\re s \asymp r$, whose phase $\omega' \mapsto \omega'\cdot\omega$ on
$\Ssurf$ has exactly the two nondegenerate critical points $\omega' = \pm\omega$,
while the factor $e^{-r(\rho/r)(\im s)(\omega'\cdot\omega)}$ is not oscillatory
and satisfies
\begin{equation*}
  \Bigl|\partial_{\omega'}^{\beta}
    e^{-r\frac{\rho}{r}(\im s)(\omega'\cdot\omega)}\Bigr|
  \le C_\beta\,(1+\Lambda)^{|\beta|}\,e^{\Lambda},
  \qquad
  \Lambda := r\,|\im s|
  \le 2C_0 r\epsilon_0^2\sin\vartheta
  \le \tfrac14\,r\,\theta_n\epsilon_0\sin\vartheta .
\end{equation*}
Stationary phase on $\Ssurf$ therefore gives a factor $r^{-\frac{n-2}{2}}$ at
the cost of finitely many such derivatives, and combining with the bound
$\im\Phi_{\omega'} \ge \frac34\theta_n\epsilon_0\sin\vartheta$ of Step 2, the
surviving exponential is
$(1+\Lambda)^{M}e^{\Lambda}e^{-\frac34 r\theta_n\epsilon_0\sin\vartheta}
\le C_Me^{-\frac12 r\theta_n\epsilon_0\sin\vartheta}$. Consequently, using
$\sin\vartheta \ge 2\vartheta/\pi$ on $[0,\pi/2]$,
\begin{equation*}
  \Bigl|\int_{\Ssurf}\int_A\Bigr|
  \le C\,r^{-\frac{n-2}{2}}\,\epsilon_0^2
      \int_0^{\pi/2}e^{-\frac12 r\theta_n\epsilon_0\sin\vartheta}\,d\vartheta
  \le C\,r^{-\frac{n-2}{2}}\,\epsilon_0^2\,
      \min\Bigl\{1,\ \frac{1}{r\theta_n\epsilon_0}\Bigr\}.
\end{equation*}
Since $\epsilon_0 \asymp \theta_n$ for $\theta_n \le \delta_0$, the last factor
is $\theta_n^2\min\{1,(r\theta_n^2)^{-1}\} \le r^{-1}$: if $r\theta_n^2 \le 1$
then $\theta_n^2 \le r^{-1}$, and otherwise
$\theta_n^2(r\theta_n^2)^{-1} = r^{-1}$. The quarter circle therefore
contributes $O(r^{-\frac{n-2}{2}-1}) = O(r^{-\frac n2})$, which is admissible
for the second bound of \eqref{eq:rad-K-radiation} and more than enough for the
first. The factor $\Phi_{\omega'}-\kbulk$, being bounded on $A$, changes
nothing. The maximum is attained at $\theta_n \asymp r^{-1/2}$, that is at
height $x_n \asymp \sqrt r$: this is the lateral wave, and the estimate shows
that it is harmless.

\emph{Step 4: the real segment.}
On $[\epsilon_0,\kbulk]$ the phase is real, and the $\omega'$-integration is a
classical stationary phase with parameter $\rho\,s(\mu) $. This parameter
degenerates as $s(\mu) \to 0$, that is as $|\xipar| \to 0$, where polar
coordinates are singular. We therefore insert a smooth cut-off
$\eta(s(\mu))$ with $\eta = 1$ for $s \ge 2\lambda_0$ and $\eta = 0$ for
$s \le \lambda_0$, and dispose of the complementary piece at the end of the
step. On $\operatorname{supp}\eta$ we have $\rho s(\mu) \ge \lambda_0\rho \asymp
r$, and stationary phase at $\omega' = \pm\omega$ yields
\begin{equation*}
  \int_{\Ssurf}a_f\,e^{ir\Phi_{\omega'}}dS(\omega')
  = r^{-\frac{n-2}{2}}\Bigl[A_+(\mu)e^{ir\Phi_{\omega}(\mu)}
      + A_-(\mu)e^{ir\Phi_{-\omega}(\mu)}\Bigr] + O\bigl(r^{-\frac n2}\bigr),
\end{equation*}
with $A_\pm$ smooth in $\mu$, $|A_\pm| \le C|\mu|$, and with an error whose
$\mu$-integral over a bounded interval is again $O(r^{-\frac n2})$.

For the $+$ term, $\Phi_{\omega}$ has by Lemma~\ref{lem:rad-phase} the single
critical point $\mu_* = \kbulk\theta_n$, which lies in $[\epsilon_0,\kbulk]$
because $\epsilon_0 \le \kbulk\theta_n/2$, and there
$|\Phi_{\omega}''| = \rho E/(r s^3) \asymp 1$ on
$\operatorname{supp}\eta$. Stationary phase in $\mu$ gives $O(r^{-1/2})$, whence
the contribution $O(r^{-\frac{n-1}{2}})$ to the first bound of
\eqref{eq:rad-K-radiation}. For the second bound the amplitude carries the extra
factor $\Phi_{\omega}-\kbulk$, which by Lemma~\ref{lem:rad-phase} vanishes to
second order at $\mu_*$. The leading term of the stationary phase expansion
therefore vanishes and the contribution is $O(r^{-3/2})$, that is
$O(r^{-\frac{n+1}{2}})$ in all.

For the $-$ term, $\Phi_{-\omega}'(\mu) = \theta_n + (\rho/r)\mu/s(\mu) > 0$ has
no zero, and $|\Phi_{-\omega}'| \ge (\rho/r)\mu/s(\mu) \ge c\mu$ on
$\operatorname{supp}\eta$. One integration by parts, the amplitude being
$O(\mu)$ and its derivative $O(1)$, gives a contribution bounded by
\begin{equation*}
  \frac{C}{r}\Bigl[\frac{\mu}{|\Phi_{-\omega}'|}\Bigr]_{\mu=\epsilon_0}
  + \frac{C}{r}\int_{\epsilon_0}^{\kbulk}
      \Bigl|\partial_\mu\frac{A_-(\mu)}{\Phi_{-\omega}'(\mu)}\Bigr|d\mu
  \le \frac{C}{r},
\end{equation*}
uniformly in $\theta_n$. Indeed, write $A_-(\mu) = \mu B_-(\mu)$ with $B_-$
bounded together with its derivative, and
$\Phi_{-\omega}'(\mu) = \theta_n + c(\mu)\mu$ with $c = (\rho/r)/s(\mu) \asymp 1$
on $\operatorname{supp}\eta$. Then $A_-/\Phi_{-\omega}' \le C/c$ and
\begin{equation*}
  \partial_\mu\frac{\mu}{\theta_n + c(\mu)\mu}
  = \frac{\theta_n - c'(\mu)\mu^2}{\bigl(\theta_n + c(\mu)\mu\bigr)^2},
  \qquad
  \int_0^{\kbulk}\frac{\theta_n\,d\mu}{\bigl(\theta_n+c(\mu)\mu\bigr)^2}
  \le \frac1c ,
\end{equation*}
the second summand of the numerator contributing $O(1)$ as well. It is the
first bound that matters: using only $\Phi_{-\omega}' \ge c\mu$ would leave
$\int_{\epsilon_0}^{\kbulk}d\mu/\mu \asymp \log(1/\theta_n)$ and destroy the
uniformity. Hence the $-$ term contributes
$O(r^{-\frac{n-2}{2}-1}) = O(r^{-\frac n2})$, and it does so for both bounds,
the factor $\Phi_{-\omega}-\kbulk$ being bounded and smooth. \emph{This is the
term that fixes the exponent $n/2$}: near the antipode $\omega' = -\omega$ the
defect factor is close to $-2\kbulk$ and provides no gain, so the rate
$r^{-\frac{n+1}{2}}$ obtained at the stationary point is not attained uniformly.

Finally, the piece with $s(\mu) \le 2\lambda_0$: reverting to the variable
$\xipar$, its phase is $\xpar\cdot\xipar + \mu x_n$ and its gradient is
$\xpar - x_n\xipar/\mu$, of modulus at least
$\rho - x_n\,2\lambda_0/\sqrt{E-4\lambda_0^2} \ge r(\sqrt3/2 - C\lambda_0)$.
Choosing $\lambda_0$ small makes this $\ge r/2$, and repeated integration by
parts in $\xipar$ gives $O_N(r^{-N})$.

\emph{Step 5: the imaginary arm.}
On $i[\epsilon_0,\infty)$ we have $|e^{ir\Phi_{\omega'}}| = e^{-r\theta_n\tau}$,
$s(i\tau) = \sqrt{E+\tau^2} \ge \kbulk$, and, by the half-space Paley--Wiener
bound \eqref{eq:PW-halfplane}, $|a_f| \le C_N\,\tau\,(1+\tau)^{-N}$ for every
$N$. Stationary phase in $\omega'$ again produces $r^{-\frac{n-2}{2}}$, the
parameter $\rho\,s(i\tau)$ being at least $\kbulk\rho \asymp r$, and leaves a
one-dimensional integral with phase $\psi_\pm(\tau)
= \pm r(\rho/r)\sqrt{E+\tau^2}$ and amplitude
$g(\tau) = \tau\,\tilde a(\tau)\,e^{-r\theta_n\tau}$, where $\tilde a$ and all
its derivatives are $O_N((1+\tau)^{-N})$.

For $\tau \ge 1$ one has $|\psi_\pm'| = r(\rho/r)\tau/\sqrt{E+\tau^2} \ge cr$,
and repeated integration by parts gives $O_N(r^{-N})$.

For $\epsilon_0 \le \tau \le 1$ the essential point is that the factor $\tau$
supplied by the Jacobian in \eqref{eq:rad-mu-representation} cancels exactly the
degeneracy of $\psi_\pm'$ at $\tau = 0$: since
$\psi_\pm'(\tau) = \pm\rho\,\tau/\sqrt{E+\tau^2}$,
\begin{equation*}
  \frac{g(\tau)}{\psi_\pm'(\tau)}
  = \pm\frac{\sqrt{E+\tau^2}}{\rho}\;\tilde a(\tau)\,e^{-r\theta_n\tau}
\end{equation*}
is smooth up to $\tau = 0$ and bounded by $C/r$, and its derivative is bounded
by $C r^{-1}(1+r\theta_n)e^{-r\theta_n\tau}$. One integration by parts therefore
gives
\begin{equation*}
  \Bigl|\int_{\epsilon_0}^{1}g\,e^{i\psi_\pm}\,d\tau\Bigr|
  \le \frac{C}{r}
    + \frac{C}{r}\int_{0}^{1}(1+r\theta_n)\,e^{-r\theta_n\tau}\,d\tau
  \le \frac{C}{r},
\end{equation*}
because $(1+r\theta_n)\int_0^1e^{-r\theta_n\tau}d\tau
\le (1+r\theta_n)\min\{1,(r\theta_n)^{-1}\} \le 2$. No logarithm appears, and
the bound is uniform in $\theta_n \in (0,\delta_0]$. The defect factor
$\Phi_{\omega'} - \kbulk$ is bounded on $\tau \le 1$ and is absorbed by the
rapid decay of $\tilde a$ for $\tau \ge 1$, so the same estimate holds for it.
The imaginary arm thus contributes $O(r^{-\frac{n-2}{2}-1}) = O(r^{-\frac n2})$.

\emph{Step 6: the cone $\theta_n \ge \delta_0$.}
Here the stationary point $\xipar_* = \kbulk\theta'$ of the phase in
\eqref{eq:v-def} lies at distance $\ge c\,\delta_0^2$ from the branch circle,
and Lemmas~\ref{lem:statphase} and \ref{lem:branch} apply directly, with the
amplitude of \eqref{eq:rad-defect-identity} in place of that of
\eqref{eq:rad-mu-representation}: both are stated for an arbitrary amplitude
holomorphic near the corner with $\mu$-derivatives bounded uniformly in
$\omega'$, which $(\Phi_{\omega'}-\kbulk)a_f$ is. They give
$O_{\delta_0}(r^{-\frac{n+1}{2}})$ for the defect and
$O_{\delta_0}(r^{-\frac{n-1}{2}})$ for the function itself, the constants being
harmless since $\delta_0$ is now fixed. Collecting Steps 1--6 proves
\eqref{eq:rad-K-radiation}.
\end{proof}

The exponent $n/2$ in \eqref{eq:rad-K-radiation} is one half-power weaker than
the $r^{-\frac{n+1}{2}}$ which the stationary point alone would suggest, and by
Step~4 the loss occurs at the antipodal critical point $\omega' = -\omega$,
where the defect factor $\Phi_{\omega'} - \kbulk$ is bounded away from zero. The
rate $r^{-\frac n2}$ is the size of the remainder in the
cone expansion of \S\ref{subsec:asymptotics}. The content of
Lemma~\ref{lem:rad-uniform} is that the same rate holds for the radiation
defect, uniformly in the direction, including in the boundary layer. That is all
that the shell averages of Definition~\ref{def:rad-condition} require.

\subsubsection{The outgoing waves satisfy the radiation condition}
\label{sec:rad-verify}

It is convenient to have the four elementary waves that can appear in a
two-channel far field at hand. Recalling
Definition~\ref{def:model-waves}, for $A \in \Ltwo{\Sm}$ and
$C \in \Ltwo{\Ssurf}$,
\begin{equation}
  \label{eq:rad-model}
  \Wb{\pm}{A}(x) = A(\theta)\,r^{-\frac{n-1}{2}}e^{\pm i\kbulk r},
  \qquad
  \Ws{\pm}{C}(x) = C(\omega)\,\rho^{-\frac{n-2}{2}}
    e^{\pm i\ksurf\rho}\,\psizero(x_n).
\end{equation}
Only the outgoing pair has occurred so far. In this notation
Theorem~\ref{thm:asymp-uniform} reads
$\Ralpha{E+i0}f \simeq \Wb{+}{A_f} + \Ws{+}{C_f}$ with
\begin{equation}
  \label{eq:rad-resolvent-amplitudes}
  A_f = c_n\,\Fbulk^{(+)}(E)f,
  \qquad
  C_f = c_n'\,\Fsurf(E)f .
\end{equation}

\begin{proposition}[The resolvent satisfies the radiation condition]
\label{prop:rad-resolvent-satisfies}
Let $n \ge 3$, $E > 0$, $\alpha_1 \in (0,1)$, and let
$f \in C_c^\infty(\Rnp)$. Then $u = \Ralpha{E+i0}f$ satisfies the outgoing
radiation condition of Definition~\ref{def:rad-condition}, with genuine limits
rather than merely $\liminf$.
\end{proposition}

\begin{proof}
We use the exact splitting of Proposition~\ref{prop:three-terms},
\begin{equation*}
  u = u_{\mathrm{D}} + \widetilde u_{\mathrm{surf}} + \Krem{E+i0}f,
\end{equation*}
and estimate the radiation defect of each summand separately. Cross terms are
then handled by the Cauchy--Schwarz inequality, which is why it suffices to
bound each shell average individually. Throughout, $r$ is large enough that
$\{R<|x|<2R\}$ misses $\operatorname{supp}f$, so all three summands are smooth
there. No pointwise expansion on a cone is used: the two regions
\eqref{eq:rad-regions} are not cones, and this is precisely what
Lemma~\ref{lem:rad-uniform} is for.

\emph{(i) The Dirichlet term.}
By \eqref{eq:usurf-exact}, $u_{\mathrm{D}}$ is the restriction to $\Rnp$ of the
outgoing free resolvent in $\Rn$ applied to $f_{\mathrm{odd}} \in
C_c^\infty(\Rn)$. By Lemma~\ref{lem:hankel}, applied with $m = n$ and
$k = \kbulk$ and used together with its final clause on first derivatives, the
expansion
$u_{\mathrm{D}} = A_{\mathrm{D}}(\theta)r^{-\frac{n-1}{2}}e^{i\kbulk r}
(1+O(r^{-1}))$ holds \emph{uniformly in all directions} $\theta \in S^{n-1}$,
and so does the differentiated one. Hence
\begin{equation*}
  |u_{\mathrm{D}}| \le C\,r^{-\frac{n-1}{2}},
  \qquad
  \bigl|\partial_ru_{\mathrm{D}} - i\kbulk u_{\mathrm{D}}\bigr|
    \le C\,r^{-\frac{n+1}{2}}
  \qquad\text{on } \Rnp .
\end{equation*}
The set $\Sigma_1(R)$ has volume $O(R^n)$, so the shell average of the
squared defect there is
\begin{equation*}
  R^{-1}\,O(R^{n})\,O\bigl(R^{-(n+1)}\bigr) = O(R^{-2}).
\end{equation*}
On $\Sigma_2(R)$ the comparison wavenumber is $\ksurf$, and we write
\begin{equation*}
  \partial_ru_{\mathrm{D}} - i\ksurf u_{\mathrm{D}}
  = \bigl(\partial_ru_{\mathrm{D}} - i\kbulk u_{\mathrm{D}}\bigr)
    + i(\kbulk-\ksurf)u_{\mathrm{D}},
\end{equation*}
whose square is $O(r^{-(n-1)})$. That region has volume
$O(R^{n-1+\alpha_1})$, an $(n-1)$-dimensional annular base of volume
$O(R^{n-1})$ times a layer of height $O(R^{\alpha_1})$, so its shell average
is
\begin{equation*}
  R^{-1}\,O\bigl(R^{n-1+\alpha_1}\bigr)\,O\bigl(R^{-(n-1)}\bigr)
  = O\bigl(R^{\alpha_1-1}\bigr) \longrightarrow 0
\end{equation*}
for every $\alpha_1 < 1$.

\emph{(ii) The surface term.}
By \eqref{eq:usurf-exact}, $\widetilde u_{\mathrm{surf}}
= \bigl(\Rfree{n-1}{E+\alpha^2+i0}T_{\psizero}f\bigr)(\xpar)\,\psizero(x_n)$.
Lemma~\ref{lem:hankel} with $m = n-1$ and $k = \ksurf$ gives, uniformly in
$\omega \in \Ssurf$ and in $x_n \ge 0$,
\begin{equation*}
  \widetilde u_{\mathrm{surf}} = \Ws{+}{C_f}\bigl(1+O(\rho^{-1})\bigr),
\end{equation*}
together with the corresponding statement for the first derivatives. On
$\Sigma_1(R)$ we have $x_n > r^{\alpha_1}$ and hence
$\psizero(x_n) \le \sqrt{2\alpha}\,e^{-\alpha r^{\alpha_1}}$, while every
derivative produces only polynomial factors. The stretched-exponential factor
defeats the volume $O(R^n)$ for every $\alpha_1 > 0$, so the contribution to
\eqref{eq:rad-bulk} vanishes.

On $\Sigma_2(R)$ we must compare with $\ksurf$. There
$\rho = r + O(r^{2\alpha_1-1})$, $\theta_n = O(r^{\alpha_1-1})$ and
$1 - \sqrt{1-\theta_n^2} = O(r^{2\alpha_1-2})$. Differentiating $\Ws{+}{C_f}$
in $r$ at fixed direction, with $\partial_r\rho = \sqrt{1-\theta_n^2}$ and
$\partial_rx_n = \theta_n$,
\begin{equation*}
\begin{aligned}
  \partial_r\Ws{+}{C_f} - i\ksurf\Ws{+}{C_f}
  &= -i\ksurf\bigl(1-\sqrt{1-\theta_n^2}\bigr)\Ws{+}{C_f}\\
  &\quad - \alpha\theta_n\Ws{+}{C_f}
   + O\bigl(\rho^{-\frac{n-2}{2}-1}\bigr).
\end{aligned}
\end{equation*}
Relative to the size $\rho^{-(n-2)/2}$ of $\Ws{+}{C_f}$ itself, the three terms
are of sizes $O(\theta_n^2)$, $O(\theta_n)$ and $O(\rho^{-1})$, and the middle
one dominates on $\Sigma_2(R)$. The tilting of the surface wavefront away from
the radial direction, measured by $\theta_n$, is thus the dominant defect.

It must not be estimated by its supremum. Every one of the three terms carries
the factor $\psizero(x_n)$, so the layer's height enters the $x_n$-integration
weighted by $\psizero^2$, and
\begin{equation*}
  \int_0^{\infty}x_n^2\,\psizero(x_n)^2\,dx_n = \frac{1}{2\alpha^2}
  < \infty .
\end{equation*}
Writing $\theta_n = x_n/r$ and using $\rho \asymp r \asymp R$ on $\Sigma_2(R)$,
together with $d\xpar = \rho^{n-2}d\rho\,dS(\omega)$, which cancels the weight
$\rho^{-(n-2)}$ exactly,
\begin{equation*}
  \int_{\Sigma_2(R)}\bigl|\alpha\theta_n\Ws{+}{C_f}\bigr|^2dx
  \;\le\; \frac{C\alpha^2}{R^{2}}\,
    \bigl\|C_f\bigr\|^2_{\Ltwo{\Ssurf}}
    \int_R^{2R}\!\!d\rho
    \int_0^{\infty}\!\!x_n^2\,\psizero(x_n)^2\,dx_n
  = O\bigl(R^{-1}\bigr),
\end{equation*}
so that the shell average is $O(R^{-2})$. The other two terms are smaller by a
further factor $O(\theta_n)$ resp.\ $O(\rho^{-1})$ and are treated identically.
\emph{The surface term therefore imposes no constraint on $\alpha_1$ at all},
and the dimension has cancelled throughout.

The crude bound gives a different answer. Replacing the two integrations by their suprema over
$\Sigma_2(R)$ would give a squared integrand $O(r^{2\alpha_1-n})$ and, after
multiplication by the volume $O(R^{n-1+\alpha_1})$, a shell average
$O(R^{3\alpha_1-2})$, which requires $\alpha_1 < 2/3$. That bound is attained
at no point of $\Sigma_2(R)$: the height $x_n \asymp r^{\alpha_1}$ at which
$\theta_n$ is largest is one at which $\psizero$ is exponentially small, while
at the heights $x_n = O(1/\alpha)$ where $\psizero \asymp 1$ one has
$\theta_n = O(1/r)$. The function $x_n\psizero(x_n)$ attains its maximum at
$x_n = 1/\alpha$, at the bottom of the layer and not at its top. The same crude
estimate is, by contrast, sharp for the two bulk terms below, which do not
decay in $x_n$.

\emph{(iii) The remainder.}
Lemma~\ref{lem:rad-uniform} gives, uniformly in $\theta$,
\begin{equation*}
  \bigl|\Krem{E+i0}f\bigr| \le C\,r^{-\frac{n-1}{2}},
  \qquad
  \bigl|\partial_r\Krem{E+i0}f - i\kbulk\Krem{E+i0}f\bigr|
    \le C\,r^{-\frac n2}.
\end{equation*}
On $\Sigma_1(R)$ the shell average of the squared
defect is $R^{-1}O(R^n)O(R^{-n}) = O(R^{-1})$. On $\Sigma_2(R)$ we write
again
\begin{equation*}
  \partial_r w - i\ksurf w
  = (\partial_rw - i\kbulk w) + i(\kbulk-\ksurf)w,
  \qquad w = \Krem{E+i0}f,
\end{equation*}
so that the defect is $O(r^{-\frac{n-1}{2}})$ and the shell average is
$O(R^{\alpha_1-1})$, exactly as in part~(i). Here too the volume factor is
genuine, $\Krem{E+i0}f$ being a bulk quantity with no decay in $x_n$.
\emph{The two bulk terms \textup{(i)} and \textup{(iii)} impose the same
constraint $\alpha_1 < 1$, and it is the only constraint on $\alpha_1$ in the
whole proof}.

Adding the three contributions in each region, and using Cauchy--Schwarz for the
cross terms, proves \eqref{eq:rad-bulk} and \eqref{eq:rad-surf} for every
$\alpha_1 \in (0,1)$.
\end{proof}

\begin{remark}[Why the verification is restricted to smooth compactly supported
data]
\label{rem:rad-smooth-data}
Proposition~\ref{prop:rad-resolvent-satisfies} is a statement about
$\partial_ru$, and two of its ingredients require $f \in C_c^\infty(\Rnp)$
rather than only $f \in \Bsp$ with compact support. Lemma~\ref{lem:hankel} needs a
smooth compactly supported density. The deformation of the contour $\Gamma$
in Step~2 of the proof of Lemma~\ref{lem:rad-uniform} rests on the holomorphy of
$\mu \mapsto \widehat{\zext f}(s(\mu)\omega',-\mu)$, which is the Paley--Wiener
theorem for a compactly supported $f$. Moreover there is no reason to expect
\eqref{eq:rad-bulk}--\eqref{eq:rad-surf} to be stable under $\Bst$-limits of
$u$, since a radial derivative is involved, so the class cannot be enlarged by a
density argument at this point. The intrinsic condition of \S\ref{sec:rad-intrinsic} is free of this
restriction (Remark~\ref{rem:rad-intrinsic-better}).
\end{remark}

\subsubsection{Uniqueness for the homogeneous problem}
\label{sec:rad-uniqueness}

\begin{theorem}[Uniqueness of the outgoing solution, homogeneous case]
\label{thm:rad-uniqueness-hom}
Let $n \ge 3$, $E > 0$, $\alpha > 0$, and $\alpha_1 \in (0,1)$. Suppose that
$u \in \Bst$ solves
\begin{equation}
  \label{eq:rad-hom-problem}
  (-\Delta - E)u = 0 \ \text{ in } \Rnp,
  \qquad
  \partial_{x_n}u + \alpha u = 0 \ \text{ on } \Gamma,
\end{equation}
in the sense of Remark~\ref{rem:rad-regularity}, and satisfies the outgoing
radiation condition of Definition~\ref{def:rad-condition}. Then $u \equiv 0$ on
$\Rnp$.
\end{theorem}

\begin{proof}
\emph{Step 1: an exact energy identity.}
By Remark~\ref{rem:rad-regularity}, $u \in C^\infty(\overline{\Rnp})$. Apply
Green's second identity on the half-ball $\Ball{t} = \{|x|<t\} \cap \Rnp$.
Since $E$ is real,
$\int_{\Ball{t}}(u\Delta\overline u - \overline u\Delta u)dx
= \int_{\Ball{t}}(-Eu\overline u + E\overline uu)dx = 0$. The boundary
$\partial\Ball{t}$ consists of the hemisphere
$S_t^+ = \{|x|=t\} \cap \Rnp$ and the disc
$\Gamma_t = \{x_n=0,\ |\xpar|<t\}$, whose outward normal is $\nu = -e_n$. On
$\Gamma_t$ the Robin condition gives $\partial_\nu u = -\partial_{x_n}u = \alpha u$
with $\alpha$ real, so that
\begin{equation}
  \label{eq:rad-flat-cancels}
  u\,\overline{\partial_\nu u} - \overline u\,\partial_\nu u
  = \alpha\bigl(u\overline u - \overline uu\bigr) \equiv 0
\end{equation}
identically on $\Gamma_t$: the boundary contribution does not merely have
vanishing imaginary part, it vanishes outright. Hence
\begin{equation}
  \label{eq:rad-energy}
  \im\int_{S_t^+}\overline u\,\partial_ru\,dS = 0
  \qquad\text{for every } t > 0 .
\end{equation}

\emph{Step 2: shell average and Young's inequality.}
Integrating \eqref{eq:rad-energy} over $t \in (R,2R)$ and using
$\int_R^{2R}\norm{u}_{\Ltwo{S_t^+}}^2dt = \int_{\{R<|x|<2R\}\cap\Rnp}|u|^2dx$
gives
$\im\int_{\{R<|x|<2R\}\cap\Rnp}\overline u\,\partial_ru\,dx = 0$. Split the
shell into $\Sigma_1(R)$ and $\Sigma_2(R)$ and write
$\overline u\partial_ru = \overline uF_1 + i\kbulk|u|^2$ on the first and
$\overline u\partial_ru = \overline uF_2 + i\ksurf|u|^2$ on the second, where
$F_1 = \partial_ru - i\kbulk u$ and $F_2 = \partial_ru - i\ksurf u$. Taking
imaginary parts,
\begin{equation*}
  \kbulk\int_{\Sigma_1(R)}|u|^2dx + \ksurf\int_{\Sigma_2(R)}|u|^2dx
  = -\im\int_{\Sigma_1(R)}\overline uF_1\,dx
    - \im\int_{\Sigma_2(R)}\overline uF_2\,dx .
\end{equation*}
Estimating the right-hand side by Cauchy--Schwarz and then Young's inequality
$ab \le \frac{c}{2}a^2 + \frac{1}{2c}b^2$ with $c = \kbulk$, respectively
$c = \ksurf$, we obtain
\begin{equation*}
  \frac{\kbulk}{2}\int_{\Sigma_1(R)}|u|^2dx
  + \frac{\ksurf}{2}\int_{\Sigma_2(R)}|u|^2dx
  \le \frac{1}{2\kbulk}\int_{\Sigma_1(R)}|F_1|^2dx
    + \frac{1}{2\ksurf}\int_{\Sigma_2(R)}|F_2|^2dx .
\end{equation*}
Divide by $R$ and let $R \to \infty$. The right-hand side tends to $0$ by
\eqref{eq:rad-bulk}--\eqref{eq:rad-surf}, and the left-hand side is a sum of
non-negative terms, so each tends to $0$ separately. Adding them, and recalling
that $\Sigma_1(R) \cup \Sigma_2(R)$ is the whole shell,
\begin{equation}
  \label{eq:rad-decay-obtained}
  \lim_{R\to\infty}\frac1R\int_{\{R<|x|<2R\}\cap\Rnp}|u|^2dx = 0,
\end{equation}
which by Lemma~\ref{lem:rad-scale-ratio} is exactly
$u \in \Bstz$, the hypothesis \eqref{eq:rad-rellich-hypothesis} of
Theorem~\ref{thm:rad-rellich}. Only the partition of
the shell was used here, and not the value of $\alpha_1$ nor the shape of the
interface.

\emph{Step 3: Rellich and unique continuation.}
By \eqref{eq:rad-decay-obtained} and $u \in \Bst$,
Theorem~\ref{thm:rad-rellich} applies with any $R_0 > 0$ and yields $R_1$ with
$u = 0$ on $\{|x|>R_1\} \cap \Rnp$. The set $\Rnp$ is connected and open, and
$u$ is real analytic there, being a smooth solution of a constant-coefficient
elliptic equation. Vanishing on a non-empty open subset,
$u \equiv 0$ on $\Rnp$ by unique continuation \cite{Aronszajn1957}.
\end{proof}

\begin{remark}[The flat half-space supports no purely outgoing mode]
\label{rem:rad-no-outgoing-modes}
Theorem~\ref{thm:rad-uniqueness-hom} may be read as the statement
$\{u \in \Nsol{E} : u \text{ outgoing}\} = \{0\}$, where $\Nsol{E}$ is the space
of $\Bst$ solutions of \eqref{eq:rad-hom-problem} introduced in
\S\ref{subsec:characterization}. This is consistent with, and sharper than, the
absence of eigenvalues: by Theorem~\ref{thm:characterization} every element of
$\Nsol{E}$ is a standing wave $\Fzero(E)^*\phi$, whose far field carries
incoming and outgoing parts of equal magnitude in each channel, so a non-zero
element can never be purely outgoing. The content of the theorem is that the
radiation condition of Definition~\ref{def:rad-condition}, which refers only to
radial derivatives on shells, is strong enough to detect this.
\end{remark}

\subsubsection{The intrinsic form of the radiation condition}
\label{sec:rad-intrinsic}

The condition of Definition~\ref{def:rad-condition} still mentions the
auxiliary exponent $\alpha_1$ and the artificial interface
$x_n = r^{\alpha_1}$. We now show that it selects the same solutions as a
condition in which neither appears: that the far field of $u$ be, in the sense
of $\simeq$, a purely outgoing two-channel wave. The key is that the four
amplitudes of a two-channel far field are determined by its
$\Bstz$-equivalence class.

\begin{lemma}[Rigidity of the four amplitudes]
\label{lem:rad-rigidity}
For $A,B \in \Ltwo{\Sm}$ and $C,D \in \Ltwo{\Ssurf}$ put
\begin{equation*}
  \Xi = \Wb{+}{A} + \Wb{-}{B} + \Ws{+}{C} + \Ws{-}{D}.
\end{equation*}
Then $\Xi \in \Bst$, its shell average exists, and
\begin{equation}
  \label{eq:rad-rigidity}
  \Mav{\Xi}
  = \norm{A}_{\Ltwo{\Sm}}^2 + \norm{B}_{\Ltwo{\Sm}}^2
    + \norm{C}_{\Ltwo{\Ssurf}}^2 + \norm{D}_{\Ltwo{\Ssurf}}^2 .
\end{equation}
In particular $\Xi \in \Bstz$ if and only if $A = B = C = D = 0$.
\end{lemma}

\begin{proof}
Assume first that $A,B$ are continuous on $\overline{\Sm}$ and $C,D$ on
$\Ssurf$, a dense class. For the bulk pair, polar coordinates
$dx = r^{n-1}dr\,d\theta$ over $\theta \in \Sm$ cancel the amplitude weight
exactly:
\begin{equation*}
  \frac1R\int_{\{|x|<R\}\cap\Rnp}\bigl|\Wb{+}{A}+\Wb{-}{B}\bigr|^2dx
  = \frac1R\int_{\Sm}\int_0^R
    \bigl|A(\theta)e^{i\kbulk r}+B(\theta)e^{-i\kbulk r}\bigr|^2 dr\,d\theta,
\end{equation*}
and the integrand equals
$|A|^2+|B|^2+2\re\bigl(A\overline Be^{2i\kbulk r}\bigr)$, whose oscillatory part
contributes $O(1/R)$ after integration in $r$. Hence the bulk pair contributes
$\norm{A}^2+\norm{B}^2$ in the limit. For the surface pair, use
$d\xpar = \rho^{n-2}d\rho\,d\omega$ and
$\int_0^\infty\psizero^2dx_n = 1$: over the cylinder $\{\rho<R\}$ the same
computation gives $\norm{C}^2+\norm{D}^2$, and the passage from the cylinder to
the ball is Lemma~\ref{lem:tensor-shell}, applied with
$\Psi = \psizero$ and with $w$ the tangential factor, the tensor being written
in the order $(w\otimes\Psi)(x) = w(\xpar)\Psi(x_n)$ used there. The cross terms
between the two channels are $O(R^{-1/2})$: in the layer where the surface
factor is not exponentially small one has $r \asymp \rho$, so the product of
the two amplitudes is $O(\rho^{-\frac{n-1}{2}-\frac{n-2}{2}})$, and
$\rho^{n-2}\rho^{-(n-3/2)} = \rho^{-1/2}$ integrates to $O(R^{1/2})$ over
$\rho < R$, which divided by $R$ tends to $0$.

For general $L^2$ amplitudes, observe that by the computation just
performed the map
\begin{equation*}
  (A,B,C,D) \longmapsto
  \Wb{+}{A}+\Wb{-}{B}+\Ws{+}{C}+\Ws{-}{D}
\end{equation*}
is bounded from $\Ltwo{\Sm}^2\oplus\Ltwo{\Ssurf}^2$ into $\Bst$, while
$\Mav{\cdot}^{1/2}$ is Lipschitz on $\Bst$ with respect to $\Bstnorm{\cdot}$
by Lemma~\ref{lem:shell-average}(ii),(iv). Both sides of
\eqref{eq:rad-rigidity} are therefore continuous in $(A,B,C,D)$, and the
identity extends from the dense class. The last assertion is immediate, the
right-hand side of \eqref{eq:rad-rigidity} being a sum of squares.
\end{proof}

\begin{theorem}[Equivalence of the radiation conditions]
\label{thm:rad-equivalence}
Let $n \ge 3$, $E > 0$, $\alpha > 0$, let $f \in C_c^\infty(\Rnp)$, and let
$u \in \Bst$ solve
\begin{equation}
  \label{eq:rad-inhom-problem}
  (-\Delta - E)u = f \ \text{ in } \Rnp,
  \qquad
  \partial_{x_n}u + \alpha u = 0 \ \text{ on } \Gamma
\end{equation}
in the sense of Remark~\ref{rem:rad-regularity}. Then the following are
equivalent.
\begin{enumerate}
\item[\textup{(a)}] $u$ satisfies the outgoing radiation condition of
Definition~\ref{def:rad-condition} for some, equivalently every,
$\alpha_1 \in (0,1)$.
\item[\textup{(b)}] There are $A \in \Ltwo{\Sm}$ and $C \in \Ltwo{\Ssurf}$ with
$u \simeq \Wb{+}{A} + \Ws{+}{C}$.
\item[\textup{(c)}] $u = \Ralpha{E+i0}f$.
\end{enumerate}
Moreover, in that case $A = A_f$ and $C = C_f$ are the amplitudes
\eqref{eq:rad-resolvent-amplitudes}, so that the far field of $u$ is
$\Fzero^{(+)}(E)f$ up to the constants $c_n$, $c_n'$.
\end{theorem}

\begin{proof}
Write $u_0 = \Ralpha{E+i0}f$, which by Theorem~\ref{thm:lap} belongs to $\Bst$
and solves \eqref{eq:rad-inhom-problem}.

\emph{(c) $\Rightarrow$ (a)} is
Proposition~\ref{prop:rad-resolvent-satisfies}, and
\emph{(c) $\Rightarrow$ (b)} is Theorem~\ref{thm:asymp-uniform} together with
\eqref{eq:rad-resolvent-amplitudes}.

\emph{(a) $\Rightarrow$ (c).} Let $u$ satisfy the radiation condition for some
$\alpha_1 \in (0,1)$. By the implication already proved, $u_0$ satisfies it
for the same $\alpha_1$. Hence $w = u - u_0$ satisfies it too, by
Lemma~\ref{lem:rad-linear}, and $w \in \Bst$ solves the homogeneous problem
\eqref{eq:rad-hom-problem} since the right-hand sides and the boundary
conditions cancel. Theorem~\ref{thm:rad-uniqueness-hom} gives $w \equiv 0$.

\emph{(b) $\Rightarrow$ (c).} Let $u \simeq \Wb{+}{A}+\Ws{+}{C}$. Since
$u_0 \simeq \Wb{+}{A_f}+\Ws{+}{C_f}$ and $\simeq$ is compatible with
subtraction, $w = u - u_0 \in \Nsol{E}$ satisfies
\begin{equation}
  \label{eq:rad-w-outgoing}
  w \simeq \Wb{+}{A-A_f} + \Ws{+}{C-C_f}.
\end{equation}
On the other hand $w \in \Nsol{E}$, so by
Corollary~\ref{cor:solution-far-field} its far field is a standing wave
carrying both phases in each channel:
\begin{equation}
  \label{eq:rad-w-standing}
  w \simeq \Wb{+}{\gambulk(E)\yout_{\mathrm{bulk}}}
    + \Wb{-}{\overline{\gambulk(E)}\,\yin_{\mathrm{bulk}}}
    + \Ws{+}{\gamsurf(E)\yout_{\mathrm{surf}}}
    + \Ws{-}{\overline{\gamsurf(E)}\,\yin_{\mathrm{surf}}},
\end{equation}
where $(\yout,\yin)$ are the outgoing and incoming far-field data of $w$.
Subtracting \eqref{eq:rad-w-outgoing} from \eqref{eq:rad-w-standing}, the
difference lies in $\Bstz$, and Lemma~\ref{lem:rad-rigidity} forces every
amplitude to vanish. In particular
$\overline{\gambulk(E)}\yin_{\mathrm{bulk}} = 0$ and
$\overline{\gamsurf(E)}\yin_{\mathrm{surf}} = 0$. Since
$\gambulk(E),\gamsurf(E) \ne 0$, we get $\yin[w] = 0$, and by the bijectivity of
the correspondence $w \leftrightarrow \yin[w]$ of
Corollary~\ref{cor:solution-far-field} this gives $w = 0$.

Finally, the equivalence with ``every $\alpha_1 \in (0,1)$'' in (a) follows
because (c) does not mention $\alpha_1$, and the identification of $A$ and $C$
follows from \eqref{eq:rad-w-outgoing} with $w = 0$ and
Lemma~\ref{lem:rad-rigidity}.
\end{proof}

\begin{remark}[Advantages of the intrinsic form]
\label{rem:rad-intrinsic-better}
Condition (b) is preferable to (a) on three counts, and it is (b) that will be
carried into Part~II. It contains no auxiliary exponent and no artificial
interface, so the class of admissible solutions is manifestly independent of
both. It involves no derivative of $u$, being a statement about the
$\Bstz$-equivalence class alone, hence is stable under $\Bst$-limits. In
particular it makes sense, and by Theorem~\ref{thm:asymp-uniform} is satisfied
by the resolvent, for every $f \in \Bsp$ and not only for
$f \in C_c^\infty(\Rnp)$, whereas the verification of (a) is confined to the
smaller class by Remark~\ref{rem:rad-smooth-data}. And it continues to make
sense on the surface band $-\alpha^2 < E < 0$, where the bulk channel is closed:
there $\Echan{E} = \Ltwo{\Ssurf}$, the term $\Wb{+}{A}$ is simply absent, and
(b) reads $u \simeq \Ws{+}{C}$, whereas the two-region split
\eqref{eq:rad-regions} loses its meaning because there is no second wavenumber
to separate. The equivalence of (b) and (c) persists there. For $E > 0$ the
implication (c) $\Rightarrow$ (b) rests on Theorem~\ref{thm:asymp-uniform}, and
on the band it rests on Proposition~\ref{prop:char-surface-band}. The converse
survives verbatim, since Corollary~\ref{cor:solution-far-field} and
Lemma~\ref{lem:rad-rigidity} hold in their single-channel form on the whole of
$\Ereg$.
\end{remark}

\subsubsection{Existence and uniqueness of the outgoing solution}
\label{sec:rad-main}

\begin{theorem}[The outgoing solution of $(\Halpha-E)u = f$]
\label{thm:rad-main}
Let $n \ge 3$, $\alpha > 0$, $E > 0$, and let $f \in C_c^\infty(\Rnp)$.
\begin{enumerate}
\item[\textup{(a)}] \textup{(Existence.)} The function
$u = \Ralpha{E+i0}f \in \Bst$ solves \eqref{eq:rad-inhom-problem} and satisfies
the outgoing radiation condition of Definition~\ref{def:rad-condition} for
every $\alpha_1 \in (0,1)$, as well as its intrinsic form
Theorem~\ref{thm:rad-equivalence}\textup{(b)}.
\item[\textup{(b)}] \textup{(Uniqueness.)} If $u_1,u_2 \in \Bst$ both solve
\eqref{eq:rad-inhom-problem} and both satisfy the radiation condition, in
either of the two equivalent forms, then $u_1 = u_2$.
\end{enumerate}
Thus $\Ralpha{E+i0}f$ is the unique outgoing solution of
\eqref{eq:rad-inhom-problem}, and its far field is
$\Fzero^{(+)}(E)f$ in the sense of \eqref{eq:outgoing-transform}. For
every $f \in \Bsp$ the same statements hold with the radiation
condition taken in the intrinsic form.
\end{theorem}

\begin{proof}
Part (a) is Theorem~\ref{thm:lap} together with
Proposition~\ref{prop:rad-resolvent-satisfies} and
Theorem~\ref{thm:asymp-uniform}. For (b), each $u_j$ equals
$\Ralpha{E+i0}f$ by Theorem~\ref{thm:rad-equivalence}, so $u_1 = u_2$.
Alternatively, and without appealing to \S\ref{subsec:characterization}:
$w = u_1-u_2$ solves the homogeneous problem, satisfies the radiation condition
by Lemma~\ref{lem:rad-linear}, and vanishes by
Theorem~\ref{thm:rad-uniqueness-hom}. For the final assertion, note that
Theorem~\ref{thm:asymp-uniform} holds for every
$f \in \Bsp$, so that the proof of (b) $\Rightarrow$ (c) in
Theorem~\ref{thm:rad-equivalence} applies verbatim to that larger class. Only
the verification of the two-region form uses $f \in C_c^\infty(\Rnp)$.
\end{proof}

\begin{remark}[What is used again in the perturbed problem]
\label{rem:rad-outlook}
Only two ingredients of this subsection use the flat geometry. The first is the
cancellation \eqref{eq:rad-flat-cancels}. For an obstacle
$D \subset \{|x| \le R_0\}$ carrying its own boundary condition, the right-hand
side of \eqref{eq:rad-energy} becomes, for $t$ large enough to enclose $D$, the
constant $\im\int_{\partial D}\overline u\,\partial_\nu u\,dS$, which vanishes
when that condition is self-adjoint (real Robin, Dirichlet or Neumann). Steps
2--3 of Theorem~\ref{thm:rad-uniqueness-hom} are then unchanged. The second is
the extension formula \eqref{eq:rad-clean-extension}, which uses the flatness
of $\Gamma$ and the constancy of $\alpha$ but is applied only in
$\{|x| > R_0\}$. Since the hypotheses of Theorem~\ref{thm:rad-rellich} concern
only a neighbourhood of infinity, that theorem transfers verbatim. What does
not transfer is the block-diagonal form of $\Smat(E)$ in
\eqref{eq:smatrix-def}: for a perturbed boundary nothing forces the
off-diagonal blocks to vanish, and that is the subject of Part~II.
\end{remark}

%% file: Part1_027_inverse_uniqueness.tex
This section closes Part~I by drawing the one consequence of
\S\ref{sec:unperturbed} that bears on the inverse problem. For a flat boundary
the Robin condition $\gamma_1 u = \alpha\gamma_0 u$ prescribes the same
constant at every point of $\Gamma$, so the boundary carries a single scalar
and no shape, and Corollary~\ref{cor:inv-alpha} below recovers that scalar
from the scattering matrix. The recovery of the shape of a locally perturbed
boundary is left to Part~II.

The scattering matrix of \eqref{eq:smatrix-def} is block diagonal, and only one
of its two blocks depends on the impedance. This section makes that statement
precise and draws its two consequences. The first is negative: the surface--surface block is a fixed operator,
the same for every $\alpha > 0$, so surface-wave scattering data alone cannot
determine the impedance of a flat boundary. The second is positive and
elementary: the bulk block determines $\alpha$ from a single energy and a
single non-grazing direction. The two rest on different computations, and we
take them in turn: the reflection coefficient of the fiber problem in
\S\ref{subsec:inv-reflection}, and the reduction of the surface channel to a
free Laplacian in \S\ref{subsec:inv-surface-free}.

Throughout this section $\alpha$ is allowed to vary, and we display the
dependence by a subscript, placed before the channel name:
$\HalphaA{\alpha}$, $\HsurfA{\alpha}$, $\HbulkA{\alpha}$, $\psizeroA{\alpha}$,
$\cphA{\alpha}{k}$, $\reflA{\alpha}{k}$, $\FsurfA{\alpha}$,
$\SmatA{\alpha}{E}$, $\SmatbulkA{\alpha}{E}$, $\SmatsurfA{\alpha}{E}$. Two
further quantities of \S\ref{sec:unperturbed} depend on the impedance and
receive the same treatment: the set of non-threshold energies
\begin{equation*}
  \EregA{\alpha} = \bigl( -\alpha^2,0 \bigr) \cup \bigl( 0,\infty \bigr)
\end{equation*}
and the surface wavenumber $\ksurfA{\alpha}$, whose value at energy $E$ is
$\sqrt{E+\alpha^2}$. The bulk wavenumber $\kbulk$ carries no subscript,
because it does not involve the impedance. This asymmetry
between the two channels is the subject of this section. In \S\ref{sec:unperturbed} the impedance was fixed and all these
subscripts were suppressed. A statement quoted from there is to be read with
$\alpha$ fixed.

\subsection{The Reflection Coefficient as a Function of the Impedance}
\label{subsec:inv-reflection}

\begin{lemma}[The reflection coefficient determines the impedance]
\label{lem:inv-reflection}
Fix $k > 0$ and let
$\reflA{\alpha}{k} = \dfrac{ik-\alpha}{ik+\alpha} = \dfrac{k+i\alpha}{k-i\alpha}$
be the Robin reflection coefficient of \eqref{eq:phase-and-reflection}. Then:
\begin{enumerate}
  \item[\textup{(i)}] $\arg\reflA{\alpha}{k} = 2\arctan(\alpha/k)$, and the map
  $\alpha \mapsto \reflA{\alpha}{k}$ is a real-analytic bijection of
  $(0,\infty)$ onto the open upper unit semicircle
  $\bigl\{ e^{i\vartheta} : 0 < \vartheta < \pi \bigr\}$, with inverse
  \begin{equation}
    \label{eq:inv-alpha-from-R}
    \alpha = k\,\tan\Bigl( \tfrac12 \arg\reflA{\alpha}{k} \Bigr).
  \end{equation}
  \item[\textup{(ii)}] The argument is strictly increasing in $\alpha$, with
  \begin{equation}
    \label{eq:inv-sensitivity}
    \frac{\partial}{\partial\alpha}\arg\reflA{\alpha}{k}
    = \frac{2k}{k^2+\alpha^2}
    \le \frac1\alpha ,
  \end{equation}
  with equality exactly at $k = \alpha$.
  \item[\textup{(iii)}] $\reflA{\alpha}{k} \to 1$ as $\alpha \downarrow 0$ and
  $\reflA{\alpha}{k} \to -1$ as $\alpha \to \infty$, uniformly on compact
  subsets of $(0,\infty)$ in $k$. For fixed $\alpha$,
  $\reflA{\alpha}{k} \to -1$ as $k \downarrow 0$ and
  $\reflA{\alpha}{k} \to 1$ as $k \to \infty$.
\end{enumerate}
\end{lemma}

\begin{proof}
By \eqref{eq:phase-and-reflection} we have
$\cphA{\alpha}{k} = (k+i\alpha)/\sqrt{k^2+\alpha^2}$ with
$|\cphA{\alpha}{k}| = 1$, so
$\reflA{\alpha}{k} = \cphA{\alpha}{k}/\overline{\cphA{\alpha}{k}}
= \cphA{\alpha}{k}^2$, and
$\arg\cphA{\alpha}{k} = \arctan(\alpha/k) \in (0,\pi/2)$ because $k>0$ and
$\alpha>0$. Doubling gives the formula for the argument, and
$|\reflA{\alpha}{k}| = 1$. As $\alpha$ runs over $(0,\infty)$ the quantity
$\arctan(\alpha/k)$ runs strictly monotonically and real-analytically over
$(0,\pi/2)$, whence (i), with \eqref{eq:inv-alpha-from-R} obtained by solving
$\vartheta = 2\arctan(\alpha/k)$ for $\alpha$. No branch has to be chosen,
since $\arg\reflA{\alpha}{k}$ takes values in $(0,\pi)$ and $\tan$ is
injective on $(0,\pi/2)$. Differentiating,
$\partial_\alpha\,2\arctan(\alpha/k) = 2k^{-1}(1+\alpha^2/k^2)^{-1}
= 2k/(k^2+\alpha^2)$, and $2k/(k^2+\alpha^2) \le 1/\alpha$ is the inequality
$(k-\alpha)^2 \ge 0$. This is (ii). Part (iii) is immediate from
$\arg = 2\arctan(\alpha/k)$.
\end{proof}

Part (iii) reproduces the two degenerate limits used as independent checks in
the proof of Lemma~\ref{lem:1d-transform}: as $\alpha \to \infty$ the boundary
condition
becomes the Dirichlet one and $\reflA{\alpha}{k} \to -1$, while as
$\alpha \downarrow 0$ it becomes the Neumann one and $\reflA{\alpha}{k} \to
+1$. The grazing limit $k \downarrow 0$ at fixed $\alpha$ gives $-1$ as well,
which is the structural fact recorded in \S\ref{subsec:lap}: at the bulk
threshold the Robin problem is asymptotically Dirichlet, and that is why
comparison against the Dirichlet resolvent, rather than against the Neumann
one, produces a weight that stays bounded there.

\subsection{The Surface Channel Is a Free Channel}
\label{subsec:inv-surface-free}

The rigidity of the surface block expresses the fact that the surface channel of $\HalphaA{\alpha}$ is, after an
energy shift, the free Laplacian of the boundary hyperplane, and that the
impedance enters it only through that shift and through a unitary conjugation
in the transverse variable. Neither of the two is visible in the far field.
For $m \ge 2$ and $t > 0$ we write
\begin{equation}
  \label{eq:inv-Ffree}
  \bigl( \Ffree{m}(t)g \bigr)(\omega)
  = \frac{t^{\frac{m-2}{4}}}{\sqrt2}\,\widehat{g}\bigl( \sqrt{t}\,\omega \bigr),
  \qquad
  \widehat{g}(\xi) = (2\pi)^{-\frac m2}\int_{\mathbb{R}^m}
    e^{-ix\cdot\xi}g(x)\,dx,
  \qquad \omega \in S^{m-1},
\end{equation}
for the free generalized Fourier transform of $-\Delta$ on $\mathbb{R}^m$ at
energy $t$, the Fourier transform being normalized as in
\eqref{eq:fourier-partial}, of which it is the $m$-dimensional analogue. The
constant in \eqref{eq:inv-Ffree} is the one appearing in \eqref{eq:Fbulk-def}
for $m = n$ and in \eqref{eq:Fsurf-def} for $m = n-1$. This is the same ``one
step apart'' pattern as for the constants $c_n,c_n'$ of \eqref{eq:constants}.

\begin{lemma}[Reduction and far-field compatibility of the surface channel]
\label{lem:inv-surface-free}
For $\alpha > 0$ let
$\VsurfA{\alpha} : \Ltwo{\Rnm} \to \Ltwo{\Rnp}$ be given by
$\bigl( \VsurfA{\alpha}g \bigr)(\xpar,x_n) = g(\xpar)\,\psizeroA{\alpha}(x_n)$.
Then:
\begin{enumerate}
  \item[\textup{(i)}] $\VsurfA{\alpha}$ is an isometry with
  $\VsurfA{\alpha}^{*} = T_{\psizeroA{\alpha}}$,
  $\VsurfA{\alpha}^{*}\VsurfA{\alpha} = \mathrm{Id}_{\Ltwo{\Rnm}}$ and
  $\Ran\VsurfA{\alpha} = \HsurfA{\alpha}$, the subspace
  $\HsurfA{\alpha}$ reduces $\HalphaA{\alpha}$, and
  \begin{equation}
    \label{eq:inv-reduction}
    \VsurfA{\alpha}^{*}\,\HalphaA{\alpha}\,\VsurfA{\alpha}
    = -\Delta' - \alpha^2
    \qquad \text{on } \Ltwo{\Rnm} .
  \end{equation}
  \item[\textup{(ii)}] For every $E > -\alpha^2$,
  \begin{equation}
    \label{eq:inv-Fsurf-factorization}
    \FsurfA{\alpha}(E)
    = \Ffree{n-1}\bigl( E+\alpha^2 \bigr) \circ \VsurfA{\alpha}^{*},
  \end{equation}
  and the same identity holds for the outgoing and incoming transforms, since
  $\FsurfA{\alpha}^{(\pm)}(E) = \FsurfA{\alpha}(E)$ by \eqref{eq:Fout-def}.
  \item[\textup{(iii)}] Let $\alpha,\tilde\alpha > 0$, let $q > 0$ satisfy
  $q \ne \alpha$ and $q \ne \tilde\alpha$, and put $E = q^2 - \alpha^2$ and
  $\tilde E = q^2 - \tilde\alpha^2$, so that $E \in \EregA{\alpha}$ and
  $\tilde E \in \EregA{\tilde\alpha}$, the two surface channels carrying the
  common wavenumber $\ksurfA{\alpha} = \ksurfA{\tilde\alpha} = q$ at these
  respective energies. Then the single, energy-independent operator
  \begin{equation*}
    \Wint = \VsurfA{\tilde\alpha}\VsurfA{\alpha}^{*}
  \end{equation*}
  is unitary from $\HsurfA{\alpha}$ onto $\HsurfA{\tilde\alpha}$, satisfies
  \begin{equation}
    \label{eq:inv-intertwine-H}
    \Wint\bigl( \HalphaA{\alpha} + \alpha^2 \bigr)\big|_{\HsurfA{\alpha}}
    = \bigl( \HalphaA{\tilde\alpha} + \tilde\alpha^2 \bigr)\big|_{\HsurfA{\tilde\alpha}}\Wint ,
  \end{equation}
  and matches the far-field identifications of the two surface channels:
  \begin{equation}
    \label{eq:inv-intertwine-F}
    \FsurfA{\tilde\alpha}^{(\pm)}(\tilde E)\,\Wint
    = \FsurfA{\alpha}^{(\pm)}(E) .
  \end{equation}
\end{enumerate}
\end{lemma}

\begin{proof}
(i) All of this is Proposition~\ref{prop:surface-reduction}, whose
$\Hsurf = \Ltwo{\Rnm}\otimes\mathbb{C}\psizero$ is exactly
$\Ran\VsurfA{\alpha}$ and whose intertwining relation
\eqref{eq:surface-reduction} is \eqref{eq:inv-reduction}. The adjoint is
$T_{\psizeroA{\alpha}}$ in the notation of Corollary~\ref{cor:tensor} because
$\psizeroA{\alpha}$ is real, and
$\VsurfA{\alpha}^{*}\VsurfA{\alpha} = \mathrm{Id}$ because
$\norm{\psizeroA{\alpha}}_{\Ltwo{\Rhalf}} = 1$. We recall why the
reduction holds, since the same computation underlies the rigidity
below: $-\psizeroA{\alpha}'' = -\alpha^2\psizeroA{\alpha}$, and for
$g \in H^2(\Rnm)$,
\begin{equation*}
  \bigl( \gamma_1 - \alpha\gamma_0 \bigr)
    \bigl( g \otimes \psizeroA{\alpha} \bigr)
  = -g(\xpar)\bigl[ \psizeroA{\alpha}'(0) + \alpha\,\psizeroA{\alpha}(0) \bigr]
  = -g(\xpar)\sqrt{2\alpha}\,\bigl[ -\alpha + \alpha \bigr]
  = 0,
\end{equation*}
by \eqref{eq:fiber-recall} and the identity
$\gamma_1 u - \alpha\gamma_0 u
= -\bigl( \partial_{x_n}u(\xpar,0) + \alpha u(\xpar,0) \bigr)$ of the
convention $\gamma_1 = \partial_\nu$, $\nu = -e_n$. The transverse profile
satisfies the Robin condition for its own $\alpha$ and for no other, which is
the whole of the impedance dependence of this channel.

(ii) By \eqref{eq:Psis-def} and the reality of $\psizeroA{\alpha}$, writing
$t = E + \alpha^2 > 0$, so that $\sqrt t = \ksurfA{\alpha}$,
\begin{equation*}
\begin{aligned}
  \bigl( \FsurfA{\alpha}(E)f \bigr)(\omega)
  &= \frac{t^{\frac{n-3}{4}}}{\sqrt2}
     \int_{\Rnp} (2\pi)^{-\frac{n-1}{2}}
       e^{-i\xpar\cdot\sqrt{t}\,\omega}\,\psizeroA{\alpha}(x_n)\,f(x)\,dx
  \\
  &= \frac{t^{\frac{n-3}{4}}}{\sqrt2}\,(2\pi)^{-\frac{n-1}{2}}
     \int_{\Rnm} e^{-i\xpar\cdot\sqrt{t}\,\omega}
       \Bigl( \int_0^{\infty}\psizeroA{\alpha}(x_n)f(\xpar,x_n)\,dx_n \Bigr)
       d\xpar
  \\
  &= \frac{t^{\frac{n-3}{4}}}{\sqrt2}\,
     \widehat{\VsurfA{\alpha}^{*}f}\bigl( \sqrt{t}\,\omega \bigr)
   = \bigl( \Ffree{n-1}(t)\,\VsurfA{\alpha}^{*}f \bigr)(\omega),
\end{aligned}
\end{equation*}
which is \eqref{eq:inv-Fsurf-factorization}. The interchange of the two
integrations is justified by Fubini's theorem on the dense class of
\eqref{eq:Fsurf-def}. The last assertion of (ii) is the fact, recorded in
\eqref{eq:Fout-def}, that the surface transform carries no phase factor,
because the surface mode is a bound state of the transverse problem and is not
reflected.

(iii) Since $q \ne \alpha$ and $q \ne \tilde\alpha$, neither $E$ nor $\tilde E$
vanishes, and $E > -\alpha^2$, $\tilde E > -\tilde\alpha^2$ because $q^2 > 0$.
Hence $E \in \EregA{\alpha}$ and $\tilde E \in \EregA{\tilde\alpha}$. By (i)
each of $\VsurfA{\alpha}$, $\VsurfA{\tilde\alpha}$ is an isometry onto
$\HsurfA{\alpha}$, $\HsurfA{\tilde\alpha}$ respectively, so
$\Wint = \VsurfA{\tilde\alpha}\VsurfA{\alpha}^{*}$ is unitary from
$\HsurfA{\alpha}$ onto $\HsurfA{\tilde\alpha}$, with inverse
$\VsurfA{\alpha}\VsurfA{\tilde\alpha}^{*}$. Relation
\eqref{eq:inv-intertwine-H} follows from \eqref{eq:inv-reduction}, both sides
being $\VsurfA{\tilde\alpha}(-\Delta')\VsurfA{\alpha}^{*}$. Finally, by
\eqref{eq:inv-Fsurf-factorization} and
$\VsurfA{\tilde\alpha}^{*}\VsurfA{\tilde\alpha} = \mathrm{Id}$, and because
$E+\alpha^2 = \tilde E+\tilde\alpha^2 = q^2$,
\begin{equation*}
  \FsurfA{\tilde\alpha}(\tilde E)\,\Wint
  = \Ffree{n-1}(q^2)\,\VsurfA{\tilde\alpha}^{*}
      \VsurfA{\tilde\alpha}\VsurfA{\alpha}^{*}
  = \Ffree{n-1}(q^2)\,\VsurfA{\alpha}^{*}
  = \FsurfA{\alpha}(E),
\end{equation*}
and the outgoing and incoming transforms coincide with these by (ii).
\end{proof}

The normalization of part (iii) deserves a word, because it is the only point
at which the two impedances have to be brought into a common frame, and the
frame is not the one a first reading suggests. The comparison is made at a
common surface \emph{wavenumber} $q$, hence at the two \emph{different}
energies $q^2-\alpha^2$ and $q^2-\tilde\alpha^2$. It is in this frame, and only
in this frame, that the surface channel is $\alpha$-independent. (The letter
$q$ is used here, and $k$ reserved for the normal wavenumber of the bulk
channel as in \eqref{eq:phase-and-reflection}. The two are wavenumbers of
different channels and are never equal to one another.) Matching the energies
instead would require comparing $\Ffree{n-1}(E+\alpha^2)$ with
$\Ffree{n-1}(E+\tilde\alpha^2)$. These are indeed related, through the dilation
$g \mapsto \lambda^{\frac{n-1}{2}}g(\lambda\,\cdot)$ of $\Ltwo{\Rnm}$ with
$\lambda^2 = (E+\tilde\alpha^2)/(E+\alpha^2)$, which is unitary. But the
resulting map of channel spaces multiplies the far-field amplitude by
$\bigl( (E+\alpha^2)/(E+\tilde\alpha^2) \bigr)^{1/2}$, so it is not an
isometry unless $\alpha = \tilde\alpha$. The wavenumber, not the energy, is the
variable in which the two surface channels are the same object. Nothing in the
next subsection depends on this choice. The statements proved there are made
for one impedance at a time. It shows, however, that the identity block is not
an artefact of the identification: the whole surface scattering system is
carried from $\alpha$ to $\tilde\alpha$ by the single operator $\Wint$.

\subsection{Rigidity of the Surface Block, and Recovery from the Bulk Block}
\label{subsec:inv-rigidity}

\begin{proposition}[Rigidity of the surface channel]
\label{prop:inv-rigidity}
Let $n \ge 3$.
\begin{enumerate}
  \item[\textup{(i)}] For every $\alpha > 0$ and every $E > 0$ the scattering
  matrix $\SmatA{\alpha}{E}$ of \eqref{eq:smatrix-def} is block diagonal with
  respect to $\Echan{E} = \Ltwo{\Sm} \oplus \Ltwo{\Ssurf}$, its off-diagonal
  blocks vanish, and
  \begin{equation}
    \label{eq:inv-surf-block}
    \SmatsurfA{\alpha}{E} = \mathrm{Id}_{\Ltwo{\Ssurf}} .
  \end{equation}
  In particular $\alpha \mapsto \SmatsurfA{\alpha}{E}$ is a constant map.
  \item[\textup{(ii)}] For every $\alpha > 0$ and every
  $E \in (-\alpha^2,0)$ one has, by
  Corollary~\ref{cor:solution-far-field},
  $\SmatA{\alpha}{E} = \mathrm{Id}_{\Ltwo{\Ssurf}}$: on the surface band the
  \emph{entire} scattering matrix is the identity. Consequently, if
  $\alpha \ne \tilde\alpha$ and $I$ is any nonempty open interval with
  $I \subset \bigl( -\min(\alpha,\tilde\alpha)^2,\,0 \bigr)$, then
  \begin{equation}
    \label{eq:inv-nonuniqueness}
    \SmatA{\alpha}{E} = \SmatA{\tilde\alpha}{E}
    \qquad \text{for every } E \in I .
  \end{equation}
  No functional of the surface-band scattering data on $I$, and no analytic
  continuation of it in the energy, can distinguish $\alpha$ from
  $\tilde\alpha$.
  \item[\textup{(iii)}] By contrast, fix any $E > 0$. Then
  $\SmatbulkA{\alpha}{E}$ is the operator of multiplication by
  $\theta \mapsto \reflA{\alpha}{\kbulk\theta_n}$, and for each fixed
  $\theta \in \Sm$ with $\theta_n > 0$ the map
  $\alpha \mapsto \reflA{\alpha}{\kbulk\theta_n}$ is injective on
  $(0,\infty)$. Hence $\alpha \mapsto \SmatbulkA{\alpha}{E}$ is injective,
  and $\SmatA{\alpha}{E} = \SmatA{\tilde\alpha}{E}$ for a single $E > 0$
  already forces $\alpha = \tilde\alpha$.
\end{enumerate}
\end{proposition}

\begin{proof}
(i) The block structure and \eqref{eq:inv-surf-block} are
\eqref{eq:smatrix-def}, in which the surface summand of the scattering matrix
is the identity. That this is a statement about $\alpha$, and not about the
identification of the channel space, is Lemma~\ref{lem:inv-surface-free}. Indeed, by
\eqref{eq:inv-Fsurf-factorization} the surface far-field identification
$\FsurfA{\alpha}^{(\pm)}(E)$ factors through the isometry
$\VsurfA{\alpha}^{*}$, which depends on $\alpha$, and through a member of the
family $\Ffree{n-1}$ of free transforms of the boundary hyperplane. That
family is defined without reference to the impedance, and it enters here only
through the value $E+\alpha^2$ of its argument. Since
$\FsurfA{\alpha}^{(+)} = \FsurfA{\alpha}^{(-)}$, the incoming and outgoing
amplitudes of a surface wave are computed by one and the same map, so by
\eqref{eq:ydata-surf} they satisfy $\yout = \Jsurf\yin$ and, in the
normalization \eqref{eq:far-field-isometry}, the surface block is the
identity. The vanishing of the off-diagonal blocks is
Proposition~\ref{prop:char-channel-orth} together with
\eqref{eq:smatrix-def}: the two transverse profiles are exactly orthogonal, and
a flat boundary converts neither channel into the other.

(ii) The first assertion is the second line of \eqref{eq:smatrix-def}. For the
second, let $I \subset (-\min(\alpha,\tilde\alpha)^2,0)$ be a nonempty open
interval. For each $E \in I$ we have $-\alpha^2 < E < 0$ and
$-\tilde\alpha^2 < E < 0$,
so $\Echan{E} = \Ltwo{\Ssurf}$ for both impedances and both scattering
matrices equal $\mathrm{Id}_{\Ltwo{\Ssurf}}$, whence
\eqref{eq:inv-nonuniqueness}. The two data sets are equal and, moreover,
constant in $E$. Every quantity computed from
$\bigl\{ \SmatA{\alpha}{E} : E \in I \bigr\}$ is therefore the same for
$\alpha$ and $\tilde\alpha$. That includes any derivative in $E$ and any
continuation to a complex neighbourhood.

(iii) That the bulk block is multiplication by
$\reflA{\alpha}{\kbulk\theta_n}$ is \eqref{eq:smatrix-def}. Fix $\theta$ with
$\theta_n > 0$ and put $k = \kbulk\theta_n$. Then $k > 0$ because $E > 0$, and
Lemma~\ref{lem:inv-reflection}(i) applied with this $k$ gives the injectivity
of $\alpha \mapsto \reflA{\alpha}{k}$. If
$\SmatA{\alpha}{E} = \SmatA{\tilde\alpha}{E}$ for some $E > 0$, then the two
multipliers agree for almost every $\theta \in \Sm$, hence for some $\theta$
with $\theta_n > 0$, and $\alpha = \tilde\alpha$ follows.
\end{proof}

\begin{corollary}[Recovery of the impedance from the bulk block]
\label{cor:inv-alpha}
Let $E > 0$, let $\theta \in \Sm$ satisfy $\theta_n > 0$, and put
$k = \sqrt{E}\,\theta_n > 0$. Write
\begin{equation*}
  \vartheta(\alpha) = \arg \SmatbulkA{\alpha}{E}(\theta) \in (0,\pi),
  \qquad
  \SmatbulkA{\alpha}{E}(\theta) = \reflA{\alpha}{k},
\end{equation*}
for the phase of the multiplier of \eqref{eq:smatrix-def} at $\theta$. The
argument lies in $(0,\pi)$, so no branch has to be chosen. Then:
\begin{enumerate}
  \item[\textup{(i)}] The impedance is recovered from that single matrix
  element by
  \begin{equation}
    \label{eq:inv-reconstruction}
    \alpha = \Arec{k}\bigl( \vartheta(\alpha) \bigr),
    \qquad
    \Arec{k}(\vartheta) := k\,\tan\bigl( \tfrac12\vartheta \bigr),
  \end{equation}
  and $\Arec{k} : (0,\pi) \to (0,\infty)$ is a real-analytic increasing
  bijection, inverse to $\alpha \mapsto \vartheta(\alpha)$, with
  \begin{equation}
    \label{eq:inv-condition-number}
    \Arec{k}'(\vartheta)
    = \frac k2\Bigl( 1 + \tan^2\bigl( \tfrac12\vartheta \bigr) \Bigr)
    = \frac{k^2 + \Arec{k}(\vartheta)^2}{2k}
    \;\ge\; \Arec{k}(\vartheta),
  \end{equation}
  with equality exactly when $\Arec{k}(\vartheta) = k$.
  \item[\textup{(ii)}] \textup{(Stability.)} Let $\alpha,\tilde\alpha > 0$ and
  put $m = \min(\alpha,\tilde\alpha)$, $M = \max(\alpha,\tilde\alpha)$. Then
  \begin{equation}
    \label{eq:inv-bilipschitz}
    \frac{k^2+m^2}{2k}\,
      \bigl| \vartheta(\alpha) - \vartheta(\tilde\alpha) \bigr|
    \;\le\; \bigl| \alpha - \tilde\alpha \bigr|
    \;\le\; \frac{k^2+M^2}{2k}\,
      \bigl| \vartheta(\alpha) - \vartheta(\tilde\alpha) \bigr| .
  \end{equation}
  In particular $\Arec{k}$ is Lipschitz on
  $\Arec{k}^{-1}\bigl( (0,M] \bigr)$ with constant $(k^2+M^2)/(2k)$, and that
  constant is smallest, for given $M$, when $k = M$. It is not Lipschitz on
  all of $(0,\pi)$: the constant grows like $M^2/(2k)$ as $M \to \infty$.
\end{enumerate}
\end{corollary}

\begin{proof}
By Lemma~\ref{lem:inv-reflection}(i) with this $k$ one has
$\vartheta(\alpha) = 2\arctan(\alpha/k)$, which is a real-analytic increasing
bijection of $(0,\infty)$ onto $(0,\pi)$. Solving for $\alpha$ gives
\eqref{eq:inv-reconstruction}, and differentiating
$\Arec{k}(\vartheta) = k\tan(\vartheta/2)$ gives the first equality in
\eqref{eq:inv-condition-number}. The second follows from
$\tan(\vartheta/2) = \Arec{k}(\vartheta)/k$, and the inequality
$(k^2+\alpha^2)/(2k) \ge \alpha$ is $(k-\alpha)^2 \ge 0$. Equivalently,
\eqref{eq:inv-condition-number} is the reciprocal of
\eqref{eq:inv-sensitivity}. This is (i).

For (ii), the mean value theorem applied to
$\alpha \mapsto \vartheta(\alpha)$ on the interval with endpoints
$\alpha,\tilde\alpha$ gives, by \eqref{eq:inv-sensitivity}, an
$\alpha_* \in [m,M]$ with
\begin{equation*}
  \vartheta(\alpha) - \vartheta(\tilde\alpha)
  = \frac{2k}{k^2+\alpha_*^2}\,\bigl( \alpha - \tilde\alpha \bigr).
\end{equation*}
Since $\alpha \mapsto 2k/(k^2+\alpha^2)$ is strictly decreasing on
$(0,\infty)$, its value at $\alpha_*$ lies between $2k/(k^2+M^2)$ and
$2k/(k^2+m^2)$, and \eqref{eq:inv-bilipschitz} follows on taking absolute
values and inverting. If $\alpha,\tilde\alpha \le M$ then $M$ may be replaced
by the given bound, which is the Lipschitz assertion. The function
$k \mapsto (k^2+M^2)/(2k)$ attains its minimum $M$ at $k = M$.
\end{proof}

\begin{remark}[Stability of the reconstruction]
\label{rem:inv-applied}
A measurement of the phase in \eqref{eq:inv-reconstruction} to within $\eps$
determines $\alpha$ to within $(k^2+M^2)\eps/(2k)$, once an a priori bound
$\alpha \le M$ is available. This is smallest for $k = \sqrt{E}\theta_n = M$,
and by the left-hand inequality in \eqref{eq:inv-bilipschitz} the estimate
cannot be improved in order. The Lipschitz constant $(k^2+M^2)/(2k)$ blows up
both as $k \downarrow 0$ and as $k \to \infty$, in accordance with the
degeneration of the reconstruction at both ends of the range: the reflection
coefficient tends to $-1$ at grazing incidence and to $+1$ at normal incidence
as $E \to \infty$ (Lemma~\ref{lem:inv-reflection}(iii)).
\end{remark}

\begin{remark}[Why propagation data see the impedance and scattering data do
not]
\label{rem:inv-propagation}
Proposition~\ref{prop:inv-rigidity}(ii) does not contradict the fact that a
surface wave depends visibly on $\alpha$: each of
\begin{equation}
  \label{eq:inv-propagation-data}
  \psizeroA{\alpha}(x_n) = \sqrt{2\alpha}\,e^{-\alpha x_n},
  \qquad
  \inf\sigma\bigl( \HalphaA{\alpha} \bigr) = -\alpha^2,
\end{equation}
determines $\alpha$, through the penetration depth of the transverse profile
and through the position of the surface threshold respectively. Neither is
scattering data. The scattering matrix is defined after the free propagation of
each channel has been divided out, and by Lemma~\ref{lem:inv-surface-free} the
quantities \eqref{eq:inv-propagation-data} are what that division removes.

The dispersion relation belongs to the same list. At an energy $E > 0$,
\begin{equation*}
  \alpha^2 = \ksurfA{\alpha}^2 - E
  = \bigl( \ksurfA{\alpha} - \kbulk \bigr)
    \bigl( \ksurfA{\alpha} + \kbulk \bigr),
\end{equation*}
so the impedance is determined by the gap between the two channel wavenumbers
rather than by either channel alone: reading $\alpha$ off $\ksurfA{\alpha}$
requires the origin of the energy axis, which is fixed by the bulk threshold.
Within the surface channel the shift $-\alpha^2$ is absorbed into the spectral
parameter, which is the content of \eqref{eq:inv-intertwine-H}. Thus $\alpha$
is determined by the domain $\EregA{\alpha}$ of the scattering matrix and by
the values of its bulk block, but not by the values of its surface block.
\end{remark}

\begin{remark}[Unitary equivalence and the far-field identification]
\label{rem:inv-unitary-trap}
Lemma~\ref{lem:inv-surface-free}(iii) should not be summarized as ``the two
surface channels are unitarily equivalent, hence indistinguishable''. By
\eqref{eq:bulk-part} the bulk parts $\HalphaA{\alpha}\big|_{\HbulkA{\alpha}}$
and $\HalphaA{\tilde\alpha}\big|_{\HbulkA{\tilde\alpha}}$ are also unitarily
equivalent, both to multiplication by $|\xipar|^2+k^2$ on
$\Ltwo{\Rnm\times\Rhalf}$, and yet the bulk block determines $\alpha$ by
Proposition~\ref{prop:inv-rigidity}(iii). What distinguishes the two channels
is the compatibility of an intertwiner with the far-field identification. For
the surface channel a single, energy-independent intertwiner matches the
incoming \emph{and} the outgoing amplitude maps, by
\eqref{eq:inv-intertwine-F}. For the bulk channel no unitary can do so: by
\eqref{eq:Fout-def} the reflection coefficient is the ratio of the outgoing to
the incoming bulk transform, so such an intertwiner would force
$\reflA{\alpha}{\cdot} = \reflA{\tilde\alpha}{\cdot}$, hence
$\alpha = \tilde\alpha$ by Lemma~\ref{lem:inv-reflection}. The full operators
$\HalphaA{\alpha}$ and $\HalphaA{\tilde\alpha}$ are never unitarily equivalent
for $\alpha \ne \tilde\alpha$, since their spectra $[-\alpha^2,\infty)$ and
$[-\tilde\alpha^2,\infty)$ differ.
\end{remark}

\begin{remark}[On the convention for the antipodal map]
\label{rem:inv-convention}
The statement $\SmatsurfA{\alpha}{E} = \mathrm{Id}$ refers to the
normalization \eqref{eq:far-field-isometry},
$\yout = \Smat(E)\,\Jchan\,\yin$, in which the antipodal map
$\Jchan = \Jbulk \oplus \Jsurf$ of the free propagation is factored out
separately. In the convention that omits $\Jchan$, the surface block is
$\Jsurf$, $(\Jsurf G)(\omega) = G(-\omega)$. Nothing in
Proposition~\ref{prop:inv-rigidity} depends on the choice: $\Jsurf$ is
likewise independent of $\alpha$, being the free relation between incoming and
outgoing amplitudes on $\Rnm$, and the rigidity statement is that the surface
block does not depend on the impedance, whichever of the two fixed operators
one calls it.
\end{remark}

Part~I has developed the direct scattering theory of $\Halpha$ in full:
self-adjointness and spectral structure, the limiting absorption principle
in $\Bsp$--$\Bst$, the generalized eigenfunction expansion, the asymptotic
expansion of the resolvent, the characterization of $\Bst$-solutions, and a
Rellich-type uniqueness theorem whose hypotheses are imposed only near
infinity. Proposition~\ref{prop:inv-rigidity} draws from this theory the one
consequence that bears on the inverse problem for a flat boundary. The two
channels behave asymmetrically. The bulk block determines $\alpha$ from a
single energy and direction, and the surface block determines nothing at all.
This asymmetry is not a defect of the flat geometry but its most informative
feature. Because the flat surface--surface block is the identity for
\emph{every} $\alpha > 0$, it provides a baseline against which a perturbation
of the boundary can be read: any departure from the identity is attributable
to the perturbation and not to the impedance. Part~II will treat such
perturbations, a compactly supported deformation of $\Gamma$, an obstacle, or
both. There nothing forces the two channels to decouple.

%% file: Part1_a1_app_selfadjoint.tex
Throughout this appendix $\gamma_0 : H^1(\Rnp) \to H^{1/2}(\Gamma)$ denotes the
Dirichlet trace and $\gamma_1 = \partial_\nu = -\partial_{x_n}$ the Neumann
trace on $H^2(\Rnp)$. Recall that $\gamma_0$ is bounded and surjective, so that
$\gamma_0\bigl(H^1(\Rnp)\bigr) = H^{1/2}(\Gamma)$ is dense in $\Ltwo{\Gamma}$
\cite[Ch.~1]{Grisvard1985}. Nothing below requires more than $n \ge 2$, which
is the
hypothesis of Theorem~\ref{thm:sa-spec}. For the places where the standing
assumption $n \ge 3$ of the body of the paper is used, see
Remark~\ref{rem:standing-dimension}.

All three parts of Theorem~\ref{thm:sa-spec} are known.
Remark~\ref{rem:literature} gives the references and the reasons for proving
them here in full, and two of those reasons govern what follows. The trace
inequality \eqref{eq:trace-ineq} produces the sharp constant $-\alpha^2$
directly, and its one-dimensional form is used again in \S\ref{app:A2}. The
diagonalization of \S\ref{app:A3} fixes the fiber objects
$\bigl(\hfib(\xipar),\psizero,\gphi_k\bigr)$ in the normalisation used
throughout the paper. The four subsections prove, in order, part~(a), the fiber
decomposition of Proposition~\ref{prop:fibration}, the explicit spectral
transform of the fiber operator, and finally parts~(b) and~(c).

\subsection{The Quadratic Form: Semiboundedness and Self-Adjointness}
\label{app:A1}

\subsubsection*{A sharp trace inequality}

For $u \in C_c^\infty(\overline{\Rnp})$ and any $\eps > 0$,
\begin{equation}
  \label{eq:trace-ineq}
  \norm{\gamma_0 u}_{\Ltwo{\Gamma}}^2
  \le \eps\,\norm{\partial_{x_n} u}_{\Ltwo{\Rnp}}^2
     + \eps^{-1}\,\norm{u}_{\Ltwo{\Rnp}}^2 .
\end{equation}
Indeed, since $u(\xpar, x_n)$ vanishes for large $x_n$,
\begin{equation*}
  |u(\xpar,0)|^2
  = -\int_0^\infty \partial_{x_n} |u(\xpar,x_n)|^2 \, dx_n
  = -2\operatorname{Re} \int_0^\infty
        \overline{u}\,\partial_{x_n} u \, dx_n .
\end{equation*}
Integrating over $\xpar \in \Rnm$ and applying the Cauchy--Schwarz and Young
inequalities,
\begin{equation*}
  \norm{\gamma_0 u}_{\Ltwo{\Gamma}}^2
  \le 2\,\norm{u}_{\Ltwo{\Rnp}}\,\norm{\partial_{x_n} u}_{\Ltwo{\Rnp}}
  \le \eps\,\norm{\partial_{x_n} u}_{\Ltwo{\Rnp}}^2
     + \eps^{-1}\,\norm{u}_{\Ltwo{\Rnp}}^2 ,
\end{equation*}
which is \eqref{eq:trace-ineq}. By density \eqref{eq:trace-ineq} extends to all
$u \in H^1(\Rnp)$. Since $\norm{\partial_{x_n} u} \le \norm{\nabla u}$ and
$\eps>0$ is arbitrary, the boundary form
$q[u] := \norm{\gamma_0 u}_{\Ltwo{\Gamma}}^2$ is \emph{infinitesimally
form-bounded} with respect to the closed nonnegative form
$u \mapsto \norm{\nabla u}_{\Ltwo{\Rnp}}^2$ on $H^1(\Rnp)$.

\subsubsection*{Semiboundedness}

Choosing $\eps = 1/\alpha$ in \eqref{eq:trace-ineq} gives, for every
$u \in H^1(\Rnp)$,
\begin{equation}
  \label{eq:sharp-bound}
  \hform[u]
  = \norm{\nabla u}^2 - \alpha\, q[u]
  \ge \norm{\nabla u}^2 - \norm{\partial_{x_n} u}^2 - \alpha^2\,\norm{u}^2
  \ge -\alpha^2\,\norm{u}^2 ,
\end{equation}
which is \eqref{eq:lower-bound}. That the constant $-\alpha^2$ is optimal will
follow from \eqref{eq:spectrum}, proved in \S\ref{app:A4}.

\subsubsection*{Self-adjointness}

By the preceding paragraphs $-\alpha\,q$ is an infinitesimally form-bounded
perturbation of the closed nonnegative Neumann form
$u \mapsto \norm{\nabla u}^2$ on $H^1(\Rnp)$. By the KLMN theorem
\cite[Thm.~X.17]{ReedSimon2}, the form $\hform$ is therefore closed, symmetric
and bounded from below by $-\alpha^2$ on the form domain $H^1(\Rnp)$, and it
defines a unique self-adjoint operator $A$ with $A \ge -\alpha^2$. It remains to
identify $A$ with $\Halpha$.

Let $u \in \Dom(A)$ and put $f = Au$, so that
\begin{equation}
  \label{eq:weak}
  \hform[u,v] = (f, v)_{\Ltwo{\Rnp}}
  \qquad \text{for all } v \in H^1(\Rnp) .
\end{equation}
Testing \eqref{eq:weak} against $v \in C_c^\infty(\Rnp)$ yields $-\Delta u = f$
in the sense of distributions, so that $-\Delta u \in \Ltwo{\Rnp}$. Since the
boundary $\Gamma = \{x_n = 0\}$ is flat and the coefficients are constant,
elliptic regularity up to the boundary applies in its simplest form. The
half-space is invariant under translations in $\xpar$, so difference quotients
in the tangential directions may be used freely, and they give
$\nabla' u \in H^1(\Rnp)$. The equation then supplies the one remaining second
derivative, $\partial_{x_n}^2 u = -\Delta' u - f \in \Ltwo{\Rnp}$. Hence
$u \in H^2(\Rnp)$ (cf.\ \cite[Thm.~4.18(ii)]{McLean2000}). For $u \in H^2(\Rnp)$ and
$v \in H^1(\Rnp)$, Green's first identity reads
\begin{equation*}
  (\nabla u, \nabla v)
  = (-\Delta u, v) + (\gamma_1 u, \gamma_0 v)_{\Ltwo{\Gamma}} ,
\end{equation*}
so that, subtracting $\alpha\,(\gamma_0 u, \gamma_0 v)_{\Ltwo{\Gamma}}$,
\begin{equation*}
  \hform[u,v]
  = (-\Delta u, v)
    + (\gamma_1 u - \alpha\,\gamma_0 u,\; \gamma_0 v)_{\Ltwo{\Gamma}} .
\end{equation*}
Comparing with \eqref{eq:weak} and using $-\Delta u = f$ we obtain
$(\gamma_1 u - \alpha\gamma_0 u, \gamma_0 v)_{\Ltwo{\Gamma}} = 0$ for all
$v \in H^1(\Rnp)$, and since $\gamma_0\bigl(H^1(\Rnp)\bigr)$ is dense in
$\Ltwo{\Gamma}$ this forces $\gamma_1 u - \alpha\gamma_0 u = 0$. Hence
$u \in \Dom(\Halpha)$ and $Au = -\Delta u = \Halpha u$, i.e.\ $A \subseteq
\Halpha$. The reverse inclusion $\Dom(\Halpha) \subseteq \Dom(A)$ is immediate
from the same Green identity, so $A = \Halpha$. This proves
Theorem~\ref{thm:sa-spec}(a). \qed

\subsection{Fiber Decomposition}
\label{app:A2}

Since $\Halpha$ is invariant under translations in $\xpar$, the partial Fourier
transform $\mathcal{F}'$ of \S\ref{sec:fourier} is the natural tool. Under the
identification $\Ltwo{\Rnp} \cong \Ltwo{\Rnm; \Ltwo{\Rhalf}}$, the map
$\mathcal{F}'$ is unitary onto the constant-fiber direct integral
$\int_{\Rnm}^{\oplus} \Ltwo{\Rhalf}\, d\xipar$. It intertwines $-\Delta'$ with
multiplication by $|\xipar|^2$, and it commutes with $\partial_{x_n}$ and with
the traces $\gamma_0, \gamma_1$, since these involve the normal variable alone
and leave $\xpar$ untouched. Hence it preserves the Robin condition fiberwise.
Consequently $\mathcal{F}'\Halpha(\mathcal{F}')^{-1}$ acts on the fiber over
$\xipar$ as the half-line operator $\hfib(\xipar)$ of \eqref{eq:fibration} with
the domain \eqref{eq:fiber-domain}.

Each $\hfib(\xipar)$ is self-adjoint on $\Ltwo{\Rhalf}$: this is the
one-dimensional case of \S\ref{app:A1}. The trace inequality on $\Rhalf$,
\begin{equation*}
  |v(0)|^2 \le \eps\,\norm{v'}_{\Ltwo{\Rhalf}}^2
             + \eps^{-1}\,\norm{v}_{\Ltwo{\Rhalf}}^2
  \qquad \bigl(v \in H^1(\Rhalf),\ \eps > 0\bigr),
\end{equation*}
is proved exactly as \eqref{eq:trace-ineq}, and the argument of \S\ref{app:A1}
identifies the operator associated with the form
$v \mapsto \norm{v'}^2 + |\xipar|^2\norm{v}^2 - \alpha|v(0)|^2$ on $H^1(\Rhalf)$
with $\hfib(\xipar)$. The family $\xipar \mapsto \hfib(\xipar)$ is measurable,
so \eqref{eq:fibration} holds as a direct integral of self-adjoint operators
\cite[\S XIII.16]{ReedSimon4}. The spectral assertions
\eqref{eq:fiber-spectrum}--\eqref{eq:psi0} are proved in \S\ref{app:A3}. \qed

\subsection{The Half-Line Robin Operator: Explicit Spectral Transform}
\label{app:A3}

Since $\hfib(\xipar) = h_0 + |\xipar|^2$ with
\begin{equation}
  \label{eq:h0}
  h_0 = -\frac{d^2}{dt^2},
  \qquad
  \Dom(h_0)
  = \bigl\{\, v \in H^2(\Rhalf) \;\bigm|\; v'(0) + \alpha\,v(0) = 0 \,\bigr\},
\end{equation}
it suffices to diagonalize $h_0$. The term $|\xipar|^2$ then shifts the spectrum
rigidly. We use the variable $t = x_n > 0$.

\begin{lemma}[Spectral transform of $h_0$]
\label{lem:1d-transform}
Let $\alpha > 0$ and let $h_0$ be as in \eqref{eq:h0}. Put
\begin{equation}
  \label{eq:1d-eigenfunctions}
  \psizero(t) = \sqrt{2\alpha}\, e^{-\alpha t},
  \qquad
  \gphi_k(t) = \sqrt{\frac{2}{\pi}}\;
    \frac{k\cos(kt) - \alpha\sin(kt)}{\sqrt{k^2 + \alpha^2}}
  \qquad (k > 0).
\end{equation}
Then the following hold.
\begin{enumerate}
  \item[\textup{(i)}] $\psizero \in \Dom(h_0)$, $\norm{\psizero}_{\Ltwo{\Rhalf}} = 1$
  and $h_0\psizero = -\alpha^2\psizero$. Moreover $-\alpha^2$ is the only
  eigenvalue of $h_0$, and it is simple.
  \item[\textup{(ii)}] For every $k > 0$ the function $\gphi_k$ is a bounded
  solution of $-\gphi_k'' = k^2\gphi_k$ on $\Rhalf$ satisfying the Robin
  condition $\gphi_k'(0) + \alpha\,\gphi_k(0) = 0$, and
  $(\gphi_k, \psizero)_{\Ltwo{\Rhalf}} = 0$ for every $k > 0$.
  \item[\textup{(iii)}] The transform
  \begin{equation}
    \label{eq:1d-transform}
    (\mathcal{U}v)(k) := \int_0^\infty \gphi_k(t)\, v(t)\, dt
  \end{equation}
  extends from $C_c^\infty(\Rhalf)$ to a bounded operator on $\Ltwo{\Rhalf}$
  with kernel $\mathbb{C}\psizero$, and its restriction
  $\mathcal{U}|_{\{\psizero\}^{\perp}}$ is unitary onto
  $\Ltwo{\Rhalf, dk}$. It diagonalizes $h_0$ there:
  \begin{equation}
    \label{eq:1d-diag}
    \bigl(\mathcal{U} h_0 v\bigr)(k) = k^2\,(\mathcal{U}v)(k),
    \qquad v \in \Dom(h_0) \cap \{\psizero\}^{\perp}.
  \end{equation}
\end{enumerate}
Consequently the orthogonal decomposition
$\Ltwo{\Rhalf} = \mathbb{C}\psizero \oplus \{\psizero\}^{\perp}$ reduces $h_0$.
The restriction of $h_0$ to $\{\psizero\}^{\perp}$ is purely absolutely
continuous with spectrum $[0,\infty)$, and
\begin{equation}
  \label{eq:h0-spectrum}
  \sigma(h_0) = \{-\alpha^2\} \cup [0,\infty) .
\end{equation}
\end{lemma}

\begin{proof}
(i) A direct computation gives $\psizero'(0) + \alpha\psizero(0) = 0$,
$h_0\psizero = -\alpha^2\psizero$ and
$\norm{\psizero}^2 = 2\alpha\int_0^\infty e^{-2\alpha t}\,dt = 1$. Conversely,
let $\mu$ be an eigenvalue with eigenfunction $v \in \Dom(h_0)$. If $\mu \ge 0$
no nonzero solution of $-v'' = \mu v$ lies in $\Ltwo{\Rhalf}$. If $\mu < 0$,
write $\mu = -\kappa^2$ with $\kappa > 0$. The only $\Ltwo{\Rhalf}$ solution is
$v(t) = c\,e^{-\kappa t}$, and the Robin condition reads
$(-\kappa + \alpha)c = 0$, so that $c \neq 0$ forces $\kappa = \alpha$. Hence
$\mu = -\alpha^2$ is the unique eigenvalue and it is simple.

(ii) With $\gphi_k$ as in \eqref{eq:1d-eigenfunctions} one has, up to the
positive normalizing factor, $\gphi_k(0) \propto k$ and
$\gphi_k'(0) \propto -\alpha k$, whence
$\gphi_k'(0) + \alpha\gphi_k(0) = 0$. Using
\begin{equation*}
  \int_0^\infty e^{-\alpha t}\cos(kt)\,dt = \frac{\alpha}{\alpha^2+k^2},
  \qquad
  \int_0^\infty e^{-\alpha t}\sin(kt)\,dt = \frac{k}{\alpha^2+k^2},
\end{equation*}
we get
$\int_0^\infty e^{-\alpha t}\bigl(k\cos(kt) - \alpha\sin(kt)\bigr)dt
 = \bigl(k\alpha - \alpha k\bigr)/(\alpha^2+k^2) = 0$,
which is the stated orthogonality.

(iii) The assertion is the classical Weyl--Titchmarsh eigenfunction expansion
for a regular endpoint carrying a Robin boundary condition. See
\cite[Ch.~II]{Titchmarsh1962} or \cite[Ch.~14]{Weidmann1987}. What has to be
checked is that the normalization chosen in \eqref{eq:1d-eigenfunctions} is the
one for which the spectral density is exactly $dk$, since that normalization
propagates through the whole of \S\ref{subsec:gft}. It is, and the reason is
visible if $\gphi_k$ is written as a phase shift of the cosine kernel. Setting
\begin{equation*}
  \delta_k := \arctan(\alpha/k) \in (0,\pi/2),
  \qquad\text{so that}\qquad
  \cos\delta_k = \frac{k}{\sqrt{k^2+\alpha^2}},
  \quad
  \sin\delta_k = \frac{\alpha}{\sqrt{k^2+\alpha^2}},
\end{equation*}
the addition formula turns \eqref{eq:1d-eigenfunctions} into
\begin{equation}
  \label{eq:phi-phase-form}
  \gphi_k(t) = \sqrt{\tfrac{2}{\pi}}\;\cos\bigl(kt + \delta_k\bigr) .
\end{equation}
The amplitude $\sqrt{2/\pi}$ is thus independent of $k$ and of $\alpha$: it is
the amplitude of the cosine transform kernel, which is unitary from
$\Ltwo{\Rhalf}$ onto $\Ltwo{\Rhalf,dk}$. The Robin condition changes the phase
of the oscillation at infinity and nothing else, and the spectral density is
therefore the same $dk$ as in the Neumann case. Both limiting cases can be read
off \eqref{eq:phi-phase-form}: as $\alpha \downarrow 0$ one has
$\delta_k \to 0$ and $\gphi_k$ becomes the Neumann kernel
$\sqrt{2/\pi}\cos(kt)$, while as $\alpha \to \infty$ one has
$\delta_k \to \pi/2$ and $\gphi_k$ becomes $-\sqrt{2/\pi}\sin(kt)$, the
Dirichlet kernel up to a sign.

The remaining assertions follow. The kernel of $\mathcal{U}$ is
$\mathbb{C}\psizero$ by~(ii) together with the completeness of the expansion,
and identity \eqref{eq:1d-diag} follows from two integrations by parts. The
boundary terms at $t = 0$ cancel because $v$ and $\gphi_k$ satisfy the same
Robin condition. On $\{\psizero\}^{\perp}$ the operator $h_0$ is then unitarily
equivalent to multiplication by $k^2$ on $\Ltwo{\Rhalf,dk}$, which is purely
absolutely continuous with spectrum $[0,\infty)$. Adjoining the simple
eigenvalue of~(i) gives \eqref{eq:h0-spectrum}.
\end{proof}

Adding the constant $|\xipar|^2$ to \eqref{eq:h0-spectrum} gives
\eqref{eq:fiber-spectrum}, and the eigenfunction \eqref{eq:psi0} does not depend
on $\xipar$. This completes the proof of
Proposition~\ref{prop:fibration}. \qed

\subsection{Proof of Parts (b) and (c)}
\label{app:A4}

Because the fiber eigenfunction $\psizero$ is independent of $\xipar$, the
decomposition of Lemma~\ref{lem:1d-transform} can be carried out
\emph{before} taking the partial Fourier transform, which makes the whole
spectral picture explicit. Write
\begin{equation}
  \label{eq:tensor-splitting}
  \Ltwo{\Rnp} \cong \Ltwo{\Rnm} \otimes \Ltwo{\Rhalf}
  = \Hsurf \oplus \Hbulk,
  \qquad
  \begin{aligned}
    \Hsurf &:= \Ltwo{\Rnm} \otimes \mathbb{C}\psizero, \\
    \Hbulk &:= \Ltwo{\Rnm} \otimes \{\psizero\}^{\perp},
  \end{aligned}
\end{equation}
and let $P_{\mathrm{surf}} = I \otimes (\,\cdot\,, \psizero)\psizero$ be the
orthogonal projection onto $\Hsurf$, that is
$P_{\mathrm{surf}}u = \usurf \otimes \psizero$ with
$\usurf(\xpar) = \int_0^\infty u(\xpar,t)\,\psizero(t)\,dt$.

\subsubsection*{The decomposition \eqref{eq:tensor-splitting} reduces $\Halpha$}

Let $u \in \Dom(\Halpha)$. For a.e.\ $\xpar$ we have $u(\xpar,\cdot) \in
H^2(\Rhalf)$, so that $u(\xpar,t)$ and $\partial_t u(\xpar,t)$ tend to $0$ as
$t \to \infty$, while $\psizero$ and $\psizero'$ decay exponentially. Two
integrations by parts therefore give
\begin{equation*}
  \int_0^\infty \bigl(\partial_t^2 u\bigr)\psizero\, dt
  = \Bigl[\, \partial_t u\,\psizero - u\,\psizero' \,\Bigr]_{t=0}^{t=\infty}
    + \int_0^\infty u\,\psizero''\, dt
  = \alpha^2 \usurf ,
\end{equation*}
the boundary term at $t = 0$ vanishing because $u$ and $\psizero$ satisfy the
same Robin condition, namely
$\partial_t u(\xpar,0)\psizero(0) - u(\xpar,0)\psizero'(0)
= \psizero(0)\bigl(\partial_t u(\xpar,0) + \alpha u(\xpar,0)\bigr) = 0$.
Pairing $\Halpha u = -\Delta' u - \partial_t^2 u$ with $\psizero$ and using this
identity gives
\begin{equation}
  \label{eq:surf-reduction}
  P_{\mathrm{surf}}\,\Halpha u
  = \Bigl(-\Delta' \usurf - \alpha^2 \usurf\Bigr) \otimes \psizero ,
\end{equation}
where $-\Delta'\usurf$ is taken in the sense of distributions on $\Rnm$. Its
left-hand side lies in $\Ltwo{\Rnp}$, so $-\Delta'\usurf \in \Ltwo{\Rnm}$, and
since $\usurf \in \Ltwo{\Rnm}$ this forces $\usurf \in H^2(\Rnm)$. Hence
$P_{\mathrm{surf}}u = \usurf \otimes \psizero$ belongs to $\Dom(\Halpha)$ and
the right-hand side of \eqref{eq:surf-reduction} equals
$\Halpha P_{\mathrm{surf}} u$. Thus $P_{\mathrm{surf}}$ commutes with $\Halpha$,
and both $\Hsurf$ and $\Hbulk$ are reducing subspaces.

\subsubsection*{The surface part}

By \eqref{eq:surf-reduction} the restriction of $\Halpha$ to $\Hsurf$ is
unitarily equivalent, via $a \otimes \psizero \mapsto a$, to
\begin{equation}
  \label{eq:surf-part}
  \Halpha\big|_{\Hsurf} \;\cong\; -\Delta' - \alpha^2
  \quad\text{on } \Ltwo{\Rnm},
\end{equation}
that is, after the Fourier transform in $\xpar$, to multiplication by
$|\xipar|^2 - \alpha^2$ on $\Ltwo{\Rnm}$. The free Laplacian on $\Rnm$ is purely
absolutely continuous with spectrum $[0,\infty)$. Hence
$\Halpha|_{\Hsurf}$ is purely absolutely continuous with spectrum
$[-\alpha^2,\infty)$.

\subsubsection*{The bulk part}

Applying $I \otimes \mathcal{U}$ with $\mathcal{U}$ as in
\eqref{eq:1d-transform}, and then the Fourier transform in $\xpar$, we obtain a
unitary map of $\Hbulk$ onto
$\Ltwo{\Rnm \times \Rhalf;\, d\xipar\, dk}$ under
which, by \eqref{eq:1d-diag},
\begin{equation}
  \label{eq:bulk-part}
  \Halpha\big|_{\Hbulk}
  \;\cong\; \text{multiplication by } |\xipar|^2 + k^2 .
\end{equation}
The function $(\xipar,k) \mapsto |\xipar|^2 + k^2$ has nonvanishing gradient off
the single point $(\xipar,k) = (0,0)$, so the push-forward of Lebesgue measure
under it is absolutely continuous with respect to $dE$ on $(0,\infty)$. Hence
$\Halpha|_{\Hbulk}$ is purely absolutely continuous with spectrum
$\overline{\Ran(|\xipar|^2+k^2)} = [0,\infty)$.

\subsubsection*{Conclusion}

Combining \eqref{eq:surf-part} and \eqref{eq:bulk-part},
\begin{equation*}
  \sigma(\Halpha)
  = \sigma\bigl(\Halpha|_{\Hsurf}\bigr) \cup \sigma\bigl(\Halpha|_{\Hbulk}\bigr)
  = [-\alpha^2,\infty) \cup [0,\infty)
  = [-\alpha^2,\infty),
\end{equation*}
which is \eqref{eq:spectrum}. In particular $\inf\sigma(\Halpha) = -\alpha^2$,
so the bound \eqref{eq:sharp-bound} is sharp. Both summands are purely
absolutely continuous, hence so is $\Halpha$, and
\eqref{eq:ac} follows:
$\sac(\Halpha) = [-\alpha^2,\infty)$ and
$\spp(\Halpha) = \ssc(\Halpha) = \varnothing$. This proves
Theorem~\ref{thm:sa-spec}(b) and (c). \qed

\begin{remark}
\label{rem:no-eigenvalues-direct}
The absence of eigenvalues can also be seen directly from
\eqref{eq:fibration}, without diagonalizing the bulk part. If
$\Halpha u = Eu$ with $u \in \Ltwo{\Rnp}$, then
$\hfib(\xipar)\widehat{u}(\xipar,\cdot) = E\,\widehat{u}(\xipar,\cdot)$ for
a.e.\ $\xipar$. By Lemma~\ref{lem:1d-transform}(i), $E$ is an eigenvalue of
$\hfib(\xipar)$ if and only if $|\xipar|^2 = E + \alpha^2$, and for every other
$\xipar$ the only solution is $\widehat{u}(\xipar,\cdot) = 0$. The exceptional
set is a sphere in $\Rnm$ (or empty), hence a Lebesgue null set, so $u = 0$.
This is the mechanism by which the fiberwise eigenvalue branch
$\lambda_{\mathrm{surf}}$ is converted into absolutely continuous spectrum,
as anticipated in Remark~\ref{rem:thresholds}.
\end{remark}

%% file: Part1_a2_app_mourre.tex
Nothing in this appendix is used in the body of the paper: the limiting
absorption principle of \S\ref{subsec:lap} is proved from the exact
decomposition \eqref{eq:splitting-preview}, and no commutator estimate enters
it. We record the Mourre theory for $\Halpha$ for three reasons. The commutator
identity is of interest in itself, and is short. It exhibits both thresholds of
$\Halpha$ as the zeros of a single explicit constant, which is a second and
quite different way of seeing the structure described in
Remark~\ref{rem:thresholds}. And, unlike the argument of
\S\ref{subsec:lap}, the commutator method does not rest on the fibration of
Proposition~\ref{prop:fibration}, so it remains available in principle when
that fibration is not. Whether the estimate survives a perturbation of the
geometry is not examined here. The comparison with
Theorem~\ref{thm:lap} was summarized in Remark~\ref{rem:mourre}. Here we supply
the computations behind it.

\subsection{The Conjugate Operator and the Commutator}

Take for conjugate operator the generator of dilations,
$A = \tfrac12(x \cdot D + D \cdot x)$ with $D = -i\nabla$, whose unitary group is
$(e^{i\theta A}u)(x) = e^{n\theta/2}u(e^{\theta}x)$. Since $\Rnp$ is a cone, this
group leaves $\Rnp$ and $H^1(\Rnp)$ invariant. It does not, however, leave
$\Dom(\Halpha)$ invariant. The Robin condition is not scale invariant, the
parameter $\alpha$ having the dimension of an inverse length. In fact
\begin{equation}
  \label{eq:dilation-alpha}
  e^{i\theta A}\,\Dom(\Halpha) = \Dom(H_{e^{\theta}\alpha}).
\end{equation}
The commutator must therefore be computed in the form sense. From
$\norm{\nabla u_\lambda}^2 = \lambda^2\norm{\nabla u}^2$ and
$\norm{\gamma_0 u_\lambda}_{\Ltwo{\Gamma}}^2
 = \lambda\,\norm{\gamma_0 u}_{\Ltwo{\Gamma}}^2$ for
$u_\lambda(x) = \lambda^{n/2}u(\lambda x)$ one obtains
\begin{equation}
  \label{eq:scaled-form}
  \hform[u_\lambda]
  = \lambda^2\norm{\nabla u}^2
    - \lambda\,\alpha\,\norm{\gamma_0 u}_{\Ltwo{\Gamma}}^2,
\end{equation}
and differentiating at $\lambda = 1$,
\begin{equation}
  \label{eq:commutator}
  i[\Halpha, A] = 2\Halpha + \alpha\,\gamma_0^{*}\gamma_0
  \qquad\text{as quadratic forms on } H^1(\Rnp).
\end{equation}
The boundary term is relatively form bounded with respect to $\Halpha$ with
relative bound $0$, by the trace inequality
$\norm{\gamma_0 u}_{\Ltwo{\Gamma}}^2
 \le \delta\,\norm{\nabla u}^2 + C_\delta\,\norm{u}^2$. Hence
$\Halpha \in C^1(A)$ and the theory of \cite{Mourre1981,ABG1996} is available.

In fact the regularity is unlimited, and for a reason that costs nothing to
record. By \eqref{eq:scaled-form} the map $\theta \mapsto \hform[u_{e^{\theta}}]$
is $e^{2\theta}\norm{\nabla u}^2 - e^{\theta}\alpha\norm{\gamma_0 u}^2$, so every
derivative in $\theta$ is again a combination of the same two forms,
\begin{equation*}
  \operatorname{ad}_A^m(\Halpha)
  = 2^m\,\norm{\nabla\,\cdot\,}^2 - \alpha\,\gamma_0^{*}\gamma_0
  \qquad (m \ge 1),
\end{equation*}
each of which is form bounded as above. Hence $\Halpha \in C^{\infty}(A)$ in the
form sense, and in particular the hypothesis $C^{1,1}(A)$ under which the
optimal weighted $L^2$ conclusions of \cite{ABG1996} are stated is satisfied.

\subsection{The Boundary Term and the Mourre Constant}

Were the commutator equal to $2\Halpha$ alone, the Mourre estimate
$E_I\,i[\Halpha,A]\,E_I \ge c\,E_I$ would fail outright on the surface branch,
where $E < 0$. It is the term $\alpha\gamma_0^{*}\gamma_0 \ge 0$ that repairs it.
Indeed, for $u = g \otimes \psizero \in \Hsurf$ one has
$\norm{\gamma_0 u}_{\Ltwo{\Gamma}}^2 = 2\alpha\,\norm{g}^2$, since
$|\psizero(0)|^2 = 2\alpha$, so that \eqref{eq:commutator} gives
\begin{equation*}
  \bigl( i[\Halpha,A]u, u \bigr)
  = 2\bigl(\Halpha u, u\bigr) + 2\alpha^2\,\norm{u}^2 ,
\end{equation*}
whereas for $I \subset (0,\infty)$, where $\Ran E_I$ meets both channels, one
simply discards the nonnegative boundary term. Below the bulk threshold there is
no bulk spectrum, so $E_I$ projects onto $\Hsurf$ when
$I \subset (-\alpha^2,0)$, and the display above applies there as an identity.
The best Mourre constant is
\begin{equation}
  \label{eq:mourre-constant}
  c(I) =
  \begin{cases}
    2\inf I, & I \subset (0,\infty), \\[1mm]
    2\bigl( \inf I + \alpha^2 \bigr), & I \subset (-\alpha^2,0).
  \end{cases}
\end{equation}
On the surface branch the first display is an identity, so the second line of
\eqref{eq:mourre-constant} is optimal by inspection. That the first line cannot
be improved either is a statement about grazing directions. For
$u \in \Hbulk$ whose transform is supported in $\{k < \eps\}$, formula
\eqref{eq:phi-phase-form} gives
$|\gphi_k(0)|^2 = \tfrac{2}{\pi}\,k^2/(k^2+\alpha^2)$, whence by
Cauchy--Schwarz
\begin{equation*}
  \norm{\gamma_0 u}_{\Ltwo{\Gamma}}^2
  \le \Bigl(\int_0^{\eps} |\gphi_k(0)|^2\,dk\Bigr)\,\norm{u}^2
  \le \frac{2\eps^3}{3\pi\alpha^2}\,\norm{u}^2 .
\end{equation*}
Such $u$ exist with spectral support in any prescribed $I \subset (0,\infty)$,
since $|\xipar|^2 + k^2 \in I$ is compatible with $k < \eps$. Let
$\eps \downarrow 0$ and concentrate the energy at $\inf I$. The boundary term
then becomes negligible and $\bigl(i[\Halpha,A]u,u\bigr)/\norm{u}^2$ tends to
$2\inf I$. The bulk channel is thus blind to the impedance in the grazing
limit, the transverse profile $\gphi_k$ tending to the Dirichlet one, and no
constant larger than $2\inf I$ will serve.

Thus both thresholds appear as the zeros of the Mourre constant, and the failure
\eqref{eq:dilation-alpha} of scale invariance produces the
surface branch and fixes the position $-\alpha^2$ of its threshold. This is the
commutator counterpart of the pole at $\kappa = \alpha$ in
\eqref{eq:g-minus-gD}.

\subsection{Limitations of the Conclusion}

Remark~\ref{rem:mourre} gave two reasons: the conclusion is a weighted $L^2$
bound valid only for $s > 1/2$, and the later sections need the kernel rather
than the norm. Both can be sharpened.

First, the optimal abstract versions of the theory \cite{ABG1996} do produce
Besov-type spaces, but these are the spaces $\mathcal{B}_A$ defined by the
spectral decomposition of the conjugate operator $A$, not the spaces
$\Bsp = \mathcal{B}_{|x|}$ defined by the dyadic decomposition in $|x|$.
Identifying the two is a separate matter.

Secondly, the two threshold rates are not two estimates of one quantity, one of
them less precise than the other. They measure different things. This is
clearest at the lower threshold. As $I \uparrow 0^{-}$ the constant
\eqref{eq:mourre-constant} tends to $2\alpha^2$ and does not degenerate. The
bound \eqref{eq:lap-explicit} does degenerate, like $|E|^{-1/2}$, because its
first term is the free resolvent in dimension $n$ at energy $E$ and the bulk
threshold sits at $0$. The reason for the discrepancy is that the
Mourre constant registers the opening of the channels that are \emph{present}
in $I$. For $-\alpha^2 < E < 0$ the bulk channel contributes no spectrum at
all, such $E$ lying in the resolvent set of the Dirichlet Laplacian. The
splitting \eqref{eq:splitting-preview} nevertheless retains the bulk term
there, and with it the bulk threshold, since it is an identity between
resolvents and not a statement about spectra. Which of the two behaviours is
the relevant one depends on what is to be done next, and for the purposes of
\S\S\ref{subsec:gft}--\ref{subsec:radiation} it is the splitting.

\subsection{Advantages of the Mourre Estimate}

Two advantages should be recorded. The estimate \eqref{eq:mourre-constant} is
strict, no compact remainder having to be subtracted. The virial theorem
therefore applies directly and yields the absence of eigenvalues in $\Ereg$ as
a by-product of the commutator computation. Here that fact is imported from
Theorem~\ref{thm:sa-spec}(c) instead. If the strictness were lost under a
perturbation and a compact remainder appeared, the conclusion would degrade
only to local finiteness of the point spectrum, which is still more than a
resolvent estimate gives by itself.

More importantly, the argument of \S\ref{subsec:lap} uses in an essential way
that $\Gamma$ is flat and $\alpha$ constant, since it rests on the exact
fibration of Proposition~\ref{prop:fibration}. Under a variable Robin
coefficient $\alpha = \alpha(\xpar)$, a curved boundary, or the insertion of an
obstacle, the fibration is destroyed and the method of \S\ref{subsec:lap} does
not apply, whereas the commutator identity itself does not presuppose it. We do
not pursue that route. The perturbed problem of Part~II is treated there by
other means, and what it requires from \S\ref{subsec:lap} is
Theorem~\ref{thm:lap}, for the flat operator, in the endpoint spaces.